\documentclass[twoside,11pt]{article}

\usepackage{dsfont}
\usepackage{amssymb, amsmath,amsthm, amsfonts}
 \usepackage[abbrvbib, nohyperref,preprint]{arxiv-ver}
 \usepackage[hyperindex,breaklinks]{hyperref}
\usepackage[shortlabels]{enumitem}

\usepackage{setspace}
\usepackage{enumitem}
\usepackage[dvipsnames]{xcolor}
\usepackage{graphicx}
\usepackage{epstopdf}
\usepackage{csquotes}
\usepackage{mathtools}
\usepackage{ifthen}
\usepackage{tikz}
\usepackage{cite}
\usepackage{pifont}
\usepackage{bm}  
\usepackage{stackengine}
\usepackage{bbm}

\allowdisplaybreaks
\usepackage{accents}
\usepackage{graphicx}
\usepackage{subcaption}
\usepackage{enumitem}
\usepackage{url}
\usepackage{wrapfig,framed}
\usepackage{setspace}
\usepackage{graphicx}
\usepackage{epstopdf}
\usepackage{csquotes}
\usepackage{mathtools}
\usepackage{ifthen}
\usepackage{tikz}
\usepackage{cite}
\usepackage{pifont}
\usepackage{bm}  
\usepackage{stackengine}
\usepackage{bbm}

\allowdisplaybreaks
\usepackage{accents}

\DeclareMathOperator*{\argmax}{arg\,max}

\DeclareSymbolFont{bbold}{U}{bbold}{m}{n}
\DeclareSymbolFontAlphabet{\mathbbold}{bbold}

\newcommand{\cB}{\mathcal{B}}
\newcommand{\cC}{\mathcal{C}}

\newcommand{\cF}{\mathcal{F}}
\newcommand{\cG}{\mathcal{G}}

\newcommand{\cL}{\mathcal{L}}

\newcommand{\cN}{\mathcal{N}}

\newcommand{\cP}{\mathcal{P}}

\newcommand{\cR}{\mathcal{R}}
\newcommand{\cS}{\mathcal{S}}

\newcommand{\cX}{\mathcal{X}}

\newcommand{\EE}{\mathbb{E}}

\newcommand{\NN}{\mathbb{N}}

\newcommand{\PP}{\mathbb{P}}

\newcommand{\RR}{\mathbb{R}}

\newcommand{\ZZ}{\mathbb{Z}}
\newcommand*{\dd}{\mathrm{d}}

\newcommand{\vasti}{\bBigg@{3.5 }}
\newcommand{\vast}{\bBigg@{4}}
\newcommand{\Vast}{\bBigg@{5}}
\newcommand{\Vastt}{\bBigg@{7}}

\numberwithin{equation}{section}
\newcommand{\be}{\begin{equation}}
\newcommand{\ee}{\end{equation}}
\newcommand{\ben}{\begin{equation*}}
\newcommand{\een}{\end{equation*}}
\newcommand{\ba}{\begin{eqnarray}}
\newcommand{\ea}{\end{eqnarray}}
\newcommand{\norm}[1]{\left\Vert#1\right\Vert}
\newcommand{\abs}[1]{\left\vert#1\right\vert}

\newcommand{\supp}[1]{\mathsf{supp}\left(#1\right)}

\newcommand{\ind}{\mathbbm 1}

\newcommand{\Ucal}{\mathcal{U}}

\newcommand{\X}{\mathcal{X}}

\newcommand{\tv}[2]{\delta_{\mathsf{TV}}\mspace{-4 mu}\left(#1,#2\right)}
\newcommand{\tvf}{\hat{\delta}}
\newcommand{\tvn}[3]{\delta_{#3}\mspace{-5 mu}\left(#1,#2\right)}
\newcommand{\kl}[2]{\mathsf{D}_\mathsf{KL}\left(#1\middle\|#2\right)}
\newcommand{\chisq}[2]{\chi^2\left(#1\middle\|#2\right)}

\def\h2{\tilde h}

\def\hm1{\hat{\mathsf{D}}_{-1}}

\usepackage{xpatch}
\makeatletter
\AtBeginDocument{\xpatchcmd{\@thm}{\thm@headpunct{.}}{\thm@headpunct{}}{}{}}
\makeatother

\jmlrheading{1}{2000}{1-48}{4/00}{10/00}{meila00a}{Sreejith Sreekumar and Ziv Goldfeld}

\ShortHeadings{Neural Estimation of  Statistical Divergences}{Sreekumar and Goldfeld}
\firstpageno{1}

\begin{document}

\title{Neural Estimation of Statistical Divergences}

\author{\name Sreejith Sreekumar \email sreejithsreekumar@cornell.edu \\
       \addr Electrical and Computer Engineering Department\\
       Cornell University\\
       Ithaca, NY 14850, USA
       \AND
       \name Ziv Goldfeld \email goldfeld@cornell.edu \\
       \addr Electrical and Computer Engineering Department\\
       Cornell University\\
       Ithaca, NY 14850, USA}

\editor{Jean-Philippe Vert}

\maketitle


%

%
%

\begin{abstract}
Statistical divergences (SDs), which quantify the dissimilarity between probability distributions, are a basic constituent of statistical inference and machine learning. A modern method for estimating those divergences relies on parametrizing an empirical variational form by a neural network (NN) and optimizing over parameter space. Such neural estimators are abundantly used in practice, but corresponding performance guarantees are partial and call for further exploration. We establish non-asymptotic absolute error bounds for a neural estimator realized by a shallow NN, focusing on four popular $\mathsf{f}$-divergences---Kullback-Leibler, chi-squared, squared Hellinger, and total variation. 
Our analysis relies on non-asymptotic function approximation theorems and tools from empirical process theory to bound the two sources of error involved: function approximation and empirical estimation. The bounds characterize the effective error in terms of NN size and the number of samples, and reveal scaling rates that ensure consistency. For compactly supported distributions, we further show that neural estimators of the first three divergences above with appropriate NN growth-rate are minimax rate-optimal, achieving the parametric convergence rate.
\end{abstract}


\section{Introduction}
Statistical divergences (SDs) measure the discrepancy between probability distributions. A variety of inference tasks, from generative modeling \citep{kingma2013auto,nowozin2016f,arjovsky2017wasserstein,tolstikhin2018wasserstein,Goldfeld2020limit_wass,Nietert2021} to homogeneity/goodness-of-fit/independence testing \citep{Kac-1955,QZhang-2018,Hallin-2021} 
can be posed as measuring or optimizing a SD between the data distribution and the model. Popular SDs include $\mathsf{f}$-divergences 
\citep{ali1966general,csiszar1967information}, integral probability metrics (IPMs) \citep{zolotarev1983probability,muller1997integral}, and Wasserstein distances \citep{villani2008optimal,santambrogio2015}. A common formulation that captures many of these is\footnote{Specifically, \eqref{Mestgen} accounts for $\mathsf{f}$-divergences, IPMs and the 1-Wasserstein distance.}
\begin{equation}
     \mathsf{D}_{h, \mathcal{F}}(\mu,\nu)=\sup_{f \in \mathcal{F}} \EE_{\mu}[f]-\EE_{\nu}[h \circ f], \label{Mestgen}
\end{equation}
where $\cF$ is a function class of `discriminators' and $h$ is sometimes called a `measurement function'  \citep[cf., e.g.,][]{arora2017generalization}. This variational form is at the core of various learning algorithms implemented based on SDs \citep{nowozin2016f,arjovsky2017wasserstein}, and has been recently leveraged for estimating SDs from samples---a technique termed neural estimation. While neural estimators (NEs) are popular in practice due to their computational scalability, a theoretic account of corresponding performance guarantees is missing. To address the deficit, this work provides a through study of consistency and non-asymptotic absolute error bounds for NEs realized by shallow neural networks (NNs).

\subsection{Neural Estimation of Statistical Divergences}

Typical applications to machine learning, e.g., generative adversarial networks (GANs) 
\citep{goodfellow2014generative, arjovsky2017wasserstein} or anomaly detection \citep{Poczos-liang-2011,Zenati-2018,SCHLEGL-2019}, favor estimators whose computation scales well with number of samples and is compatible with backpropagation and minibatch-based optimization. Neural estimation is a modern technique that adheres to these requirements \citep{arora2017generalization,zhang2018discrimination,belghazi2018,Mroueh-2021}. Neural estimators (NEs) parameterize the discriminator class $\mathcal{F}$ in \eqref{Mestgen} by a NN, approximate expectations by sample means, and then optimize the resulting empirical objective over parameter space. Denoting the samples from $\mu$ and $\nu$ by $X^n:=(X_1,\ldots,X_n)$ and $Y^n:=(Y_1,\ldots,Y_n)$, respectively, the said NE is
\begin{equation}
    \hat{\mathsf{D}}_{h,\cG}(X^n,Y^n) := \sup_{g \in \cG} \frac 1n \sum_{i=1}^n\Big[ g(X_i)- h \circ g(Y_i)\Big], \label{Mest-emp}
\end{equation}
where $\cG$ is the class of functions realized by a NN.

The performance of a NE is dictated by the quality of the NN approximation to the original function class $\cF$ from \eqref{Mestgen}, and the sample size needed to accurately estimate the parametrized form $\mathsf{D}_{h, \mathcal{G}}(\mu,\nu)$. The former is measured by the \textit{approximation error}, $\big|\mathsf{D}_{h, \mathcal{F}}(\mu,\nu)-\mathsf{D}_{h, \mathcal{G}}(\mu,\nu)\big|$, whereas the latter by the \textit{empirical estimation error}, $\big|\hat{\mathsf{D}}_{h,\mathcal{G}}(X^n,Y^n)-\mathsf{D}_{h, \mathcal{G}}(\mu,\nu)\big|$. While approximation needs $\cG$ to be rich and expressive, efficient estimation relies on controlling its complexity. Past works on NEs provide only a partial account of estimation performance. \citet{belghazi2018} proved consistency of mutual information neural estimation, which boils down to estimating KL divergence, but do not quantify approximation errors. Non-asymptotic sample complexity bounds for the parameterized form, i.e., when $\cF$ in \eqref{Mestgen} is the NN class $\cG$ to begin with, were derived in \citep{arora2017generalization,zhang2018discrimination}.  These objects are known as NN distances and, by definition, overlook the approximation error.  
Also related is \citep{Nguyen-2010}, where KL divergence estimation rates are provided under the assumption that the approximating class is large enough to contain an optimizer of \eqref{Mestgen}. This assumption is often violated in practice, e.g., when using a NN class as done herein, or a reproducing kernel Hilbert space, as considered in \citep{Nguyen-2010}.

Quantification of the approximation error, alongside the empirical estimation error, is pivotal for a complete account of neural estimation performance. This work thus studies non-asymptotic effective (approximation plus empirical estimation) error bounds for NEs realized by a $k$-neuron shallow NN and $n$ samples from each distribution. Results are specialized to four popular $\mathsf{f}$-divergences: Kullback-Leibler (KL), chi-squared ($\chi^2$), squared Hellinger ($\mathsf{H}^2$) distance, and total variation (TV)~distance.

\subsection{Contributions}

This work extends our earlier conference paper \citep{SS-2021}, where the first non-asymptotic effective error bounds for NEs of $f$-divergences was derived. Consistency results for appropriate scaling rates of the NN and the sample sizes were also provided. However, the analysis therein resulted in sub-optimal error rates, only considered  compactly supported distributions, and was not applicable for TV distance estimation. These aspects are key for a complete account of the neural estimation performance, and serve to motivate the present work, which closes all the aforementioned gaps.

We first consider compactly supported distributions and show that the effective  error of a NE based on $k$ neurons and $n$ samples for the KL divergence, $\chi^2$ divergence, or the $\mathsf{H}^2$ distance scales~as
\begin{equation} O\left(k^{-1/2}+n^{-1/2}\right).\label{EQ:result_intro}
\end{equation}  
Our bound is sharp in the sense that by scaling $k$ proportional to $n$, NEs achieve  minimax optimality, converging at the parametric $n^{-1/2}$ rate. 
The results assume a spectral norm bound on the optimal potential (i.e., maximizer of \eqref{Mestgen}) of the SD, which, in particular, is satisfied when the distributions have sufficiently smooth densities . Notably, this condition suffices to avoid the so-called curse of dimensionality (CoD) and attain parametric rates that do not degrade exponentially with dimension.\footnote{A similar behavior was observed in \citep{kandasamy2015nonparametric} for classic $\sf{f}$-divergence estimators between densities with high (H\"older) smoothness.}

The derivation of \eqref{EQ:result_intro} relies on two key technical results that separately account for the approximation and estimation errors. The first is a sup-norm $O(k^{-1/2})$ universal approximation bound for shallow NNs \citep{Klusowski-2018}, and the second is a $ O(n^{-1/2})$ bound on the empirical estimation error of the parametrized form. Derivation of the latter result leverages tools from empirical process theory and bounds the entropy integral \citep{AVDV-book} associated with the NN class. To that end, we bound the covering number of the NN class by noting that it can be represented as (a subset of) the symmetric convex hull \citep{AVDV-book}) of a composition of a monotone function with a VC subgraph class.

Equipped with these results, we treat neural estimation of the KL and $\chi^2$ divergences, and the $\mathsf{H}^2$ and TV distances. We establish consistency and obtain \eqref{EQ:result_intro} as a finite-sample absolute-error bound by combining the approximation and empirical estimation bounds and identifying the appropriate scaling of the NN width $k$ and parameter norms with the sample size $n$ for each $\mathsf{f}$-divergence. Our analysis results in the parametric absolute-error convergence rate for the NEs of KL divergence, $\chi^2$ divergence, and $\mathsf{H}^2$ distance. We also show an $\Omega(n^{-1/2})$ lower bound on the minimax absolute error risk for  KL divergence estimation problem by reducing it to differential entropy estimation and using the lower bound from \citet{Goldfeld-2020} for the latter problem. This establishes minimax optimality of KL divergence NE, with similar claims holding for NEs of $\chi^2$ divergence and $\mathsf{H}^2$ distance. Our method also accounts for the mutual information neural estimator (MINE) \citep{belghazi2018}, and provides the first non-asymptotic effective error bound and minimax optimality claim for it. Different from these, the TV distance NE requires a modified approach because the spectral norm of the optimal potential is infinite. To circumvent the issue, we apply Gaussian smoothing to this potential, and control the approximation error as the smoothing parameters shrinks with the NN size $k$.
This results in an approximation-estimation error bound that depends on dimension, i.e., the CoD applies in this case.  

We then extend our results to distributions with unbounded support. To that end, we exploit the fact that our approximation error bound depends on the support of the target function only via its spectral norm. Thus, bounds on the effective error in the unbounded case are obtained by quantifying the spectral norm of the optimal potential inside a ball and growing its radius appropriately with $k$. The resulting bound depends on the scaling of the radius and the tail decay of the underlying distributions (as quantified by the Orlicz norm of the densities). The results are specialized to the aforementioned divergences, focusing on Gaussian and sub-Gaussian distributions. We note that our analysis applies to distributions whose densities need not be bounded away from zero  \citep[see also,][]{berrett2019efficient, berrett2019-efficientfunctional}---an assumption that is often imposed for $\mathsf{f}$-divergence estimation.

\subsection{Related Work} \label{Sec:Relwork}
Many non-parametric estimators of SDs are available in the literature  \citep{Wang-2005,Perez-2008,Bharath-2012,krishnamurthy2014nonparametric,Singh-Poczos-2014,Singh-Poczos2-2014,kandasamy2015nonparametric,Singh-Poczos-2016,Noshad-2017,Moon-2018,Wisler-2018,berrett2019efficient,berrett2019-efficientfunctional,liang2019estimating,Yanjun-2020}. These estimators typically rely on classic methods (or their variants) such as plug-in,  kernel density estimation (KDE) or $k$-nearest neighbors (kNN) techniques,  and are known to achieve optimal estimation error rates for specific SDs, subject to smoothness and/or regularity conditions on the densities (see Remark \ref{rem:relation-otherworks}). 
However, kernel-based methods usually require at least one of the following to achieve optimal rates: (i)  boundary bias correction mechanism such as the usage of a mirror image kernel which assumes knowledge of the boundaries of the support of the distributions \citep[cf., e.g.,][]{Singh-Poczos-2014,Singh-Poczos2-2014}; (ii) assumptions ensuring smooth behaviour of densities at the boundaries \citep[cf., e.g.,][]{Moon-2018, Yanjun-2020}. As will be evident later, smoothness assumptions imposed by the aforementioned spectral norm condition are sufficient for \eqref{EQ:result_intro} to hold for the NE. Thus,  knowledge of the support of distributions or boundary bias correction is not needed.

Focussing on NEs, the tradeoffs between approximation and estimation errors was previously studied for non-parametric regression using NNs \citep[cf., e.g.,][]{Barron-1994,Bach-2017,suzuki2018adaptivity}. The goal there is to fit the best NN proxy to an (unknown) target function based on data generated from it by minimizing a prescribed loss function. Assuming that the target function satisfies certain smoothness or spectral norm constraints, the approximation-estimation tradeoff in such problems has been analyzed for different loss functions. In particular, \citet{Barron-1994} derived upper bounds on the minimax mean squared error rate for shallow NN models under a spectral norm condition on the Fourier transform of the target function. 
Density estimation under general loss functions was considered in \citep{Yang-1999}, where minimax rate bounds in terms of covering/packing entropy were established. In \citep{suzuki2018adaptivity}, the minimax rate for non-parametric regression using deep NNs (DNNs) when the target function is Besov was determined. More recently, \citep{Uppal-2019} established the minimax rate for density estimation under a so-called Besov IPM loss.

\subsection{Organization}
The paper is organized as follows. Section \ref{sec:prbform} provides background and preliminary definitions.   
Technical results characterizing the approximation error and empirical estimation error are stated in Section \ref{Sec:prelim}. In Section \ref{Sec:NESDs}, we apply these results to obtain upper bounds on the neural estimation error of the aforementioned $\mathsf{f}$-divergences. 
Corresponding error bounds for  distributions with unbounded support are the topic of Section \ref{sec:extncsupp}. 
Section \ref{sec:conlcusion} provides concluding remarks and discusses future research directions. Proofs are deferred to appendices.

\section{Background and Definitions} \label{sec:prbform}

\subsection{Notation} \label{sec:notation}
Let $\norm{\cdot}$ denote the Euclidean norm on $\RR^d$ and $x \cdot y$ designate the inner product.
The  $\ell^m$ ball  of radius $r\geq 0$ in $\RR^d$  centered at 0 is $B_d^m(r)$; in particular, the Euclidean ball is designated as $B_d(r)$. 
We use $\bar{\RR}:=\RR\cup \{-\infty,\infty\}$ for the extended reals. 
For $1\leq r < \infty$, the $L^r$ space over $\cX\subseteq\RR^d$ with respect to (w.r.t.) the measure $\mu$ is denoted by $L^r(\cX,\mu)$, with $\| \cdot \|_{r,\mu}$ 
representing the norm. When $\mu$ is the Lebesgue measure 
$\lambda$, we use the shorthand $L^r(\cX)$ with norm  $\| \cdot \|_{r,\X}$, or even $L^r$ and $\| \cdot \|_{r}$ when $\X$ is clear from the context. For $r=\infty$, we use $\|\cdot\|_{\infty,\mu}$ and $\|\cdot\|_{\infty,\cX}$ for the essential supremum norm and the standard sup-norm, respectively. Slightly abusing notation, for  $\X \subseteq \RR^d$, we set $\norm{\X}:=\sup_{x \in \X}  \norm{x}_{\infty}$.  

The probability space on which all random variables are defined is denoted by  $(\Omega,\mathcal{A},\PP)$ (assumed to be sufficiently rich), with $\EE$ designating the corresponding expectation.  The class of Borel probability measures on $\X \subseteq \RR^d$ is denoted by $\cP(\X)$. To stress that the expectation of $f$ is taken w.r.t. $ \mu \in\cP(\X)$, we write $\EE_{\mu}[f]:=\int f\dd\mu$.  
For $\mu,\nu\in\cP(\X)$ with $\mu \ll \nu$, i.e., $\mu$ is absolutely continuous w.r.t. $\nu$, we use $\frac{\dd \mu}{\dd \nu}$ for the Radon-Nikodym derivative of $\mu$ w.r.t. $\nu$. For $n \in \mathbb{N}$, $\mu^{\otimes n}$ denotes the $n$-fold product measure of $\mu$.

We assume that all functions are Borel measurable. 
For a multi-index $\alpha=   (\alpha_{_1}, \cdots,\alpha_d)\in\ZZ_{\geq 0}^d$, the partial derivative operator of order $\norm{\alpha}_1:=\sum_{j=1}^d \alpha_{_j}$
is designated by  $D^{\alpha}:=\frac{\partial^{\alpha_{_1}}}{\partial^{\alpha_{1}} x_1}\cdots \frac{\partial^{\alpha_d}}{\partial^{\alpha_d} x_d}$.
For an open set $\Ucal \subseteq \RR^d$ and an integer $m \geq 0$, the class of functions such that all partial derivatives of order $m$ exist and are continuous on $\Ucal$ are denoted by $\mathsf{C}^m(\Ucal)$. In particular, $\mathsf{C}(\Ucal):=\mathsf{C}^0(\Ucal)$ and $\mathsf{C}^{\infty}(\Ucal)$ denotes the class of continuous functions and infinitely differentiable functions. 
For $b\geq 0$ and an integer $m \geq 0$, 
$\mathsf{C}_b^{m}(\Ucal):=\big\{f \in \mathsf{C}^m(\Ucal):
\max_{\alpha:\norm{\alpha}_1  \leq m}  \norm{D^{\alpha}f}_{\infty,\Ucal} \leq b \big\}$ denotes the subclass of $\mathsf{C}^m(\Ucal)$ with partial derivatives of order up to $m$ uniformly bounded by $b$.  
The restriction of $f:\RR^d\to\RR$ to a subset $\cX \subseteq \RR^d$ is denoted by $f|_\cX$. The Fourier transform of $f\in L^1(\cX)$ is denoted by $\mathfrak{F}[f]$. 
For a function class $\cF$ and a function $g$, 
$g \circ \cF:=\{g \circ f:f \in \cF\}$ and $\abs{g \circ \cF}:=\{\abs{g \circ f}:f \in \cF\}$, where $\circ$ denotes function composition (domains assumed to be compatible for composition). 

We denote universal constants by $c$ (or $c_1$, $c_2$,  etc.) while constants that depend on a parameter $x$ are denoted by $c_x$.  
\emph{The values of $c$ and $c_x$ may change between different instances even within the same line of an equation.}   
We  use the shorthand $a\lesssim_{x} b$ for $a\leq c_x b$ for some $c_x>0$ ($a\lesssim b$ means $a\leq cb$ for a universal constant~$c>0$); similarly, $a \asymp_{x} b$ stands for $a= c_x b$.  We also employ standard asymptotic notations such as $O$, $\Omega$, $\tilde O$, etc., where the tilde designates hidden logarithmic factors. For $a,b \in \RR$, $a \vee b:=\max\{a,b\}$ and $a \wedge b:=\min\{a,b\}$. We proceed with preliminary definitions and technical background.

\subsection{Statistical Divergences} \label{SDsvarform}
Let $\cX\subseteq  \RR^d$. A common variational formulation of a SD between $\mu,\nu \in  \cP(\X)$ is
\begin{align}
     \mathsf{D}_{h, \mathcal{F}}(\mu,\nu)=\sup_{f \in \mathcal{F}} \EE_{\mu}[f]-\EE_{\nu}[h\circ f], \label{Mestgen2}
\end{align}
where $h:\mathbb{R}\rightarrow \bar{\mathbb{R}}$, and $\mathcal{F}$ is a class of measurable functions $f:\RR^d\rightarrow \mathbb{R}$ for which the last expectation is finite. This formulation captures  $\mathsf{f}$-divergences, IPMs (for $h(x)=x$), as well as the 1-Wasserstein distance (which is an IPM w.r.t. the 1-Lipschitz function class). We next specialize the above variational form to the $\mathsf{f}$-divergences for which we derive neural estimation error bounds.

\underline{KL divergence:} The KL divergence between $\mu,\nu\in\cP(\X)$ is
\begin{align}
  \kl{\mu}{\nu}:=\begin{cases}
 \EE_{\mu} \left[\log \frac{\dd \mu}{\dd \nu}\right] ,&\mu \ll \nu,\\
  \infty,& \mbox{otherwise}.\end{cases} \notag
\end{align}
 A variational form for $\kl{\mu}{\nu}$ is obtained via Legendre-Fenchel duality, yielding: 
\begin{equation}
\kl{\mu}{\nu}=\sup_{f :\X \rightarrow \mathbb{R}}\EE_{\mu}[f]-\EE_{\nu}\big[e^{f}-1\big], \label{CC-charact}
\end{equation}
where the supremum is over all measurable functions such that the last expectation in \eqref{CC-charact} is finite. 
This fits the framework of \eqref{Mestgen2} with $h(x)=h_{\mathsf{KL}}(x):=e^x-1$. When $\mu \ll \nu$, the supremum in \eqref{CC-charact} is achieved by $f_{\mathsf{KL}}:=\log\frac{\dd \mu}{\dd \nu}$.

\underline{$\chi^2$ divergence:}
The $\chi^2$ divergence between $\mu,\nu\in\mathcal{P}(\X)$ is
\begin{align}
    \chisq{\mu}{\nu}:=\begin{cases}\EE_{\nu}\left[\Big(\frac{\dd \mu}{\dd \nu}-1\right)^2\Big],&\mu \ll \nu,\\
  \infty,& \mbox{otherwise}.\end{cases} \notag
\end{align}
It admits the dual~form:
\begin{equation}
     \chisq{\mu}{\nu}=\sup_{\substack{f:\X\rightarrow \mathbb{R}}} \EE_{\mu}[f]-\EE_{\nu}\left[f+f^2/4\right], \label{chisqdistvarchar}
\end{equation}
where the supremum is over all measurable functions such that the last expectation in \eqref{chisqdistvarchar} is finite. 
This dual form corresponds to \eqref{Mestgen2} with $h(x)=h_{\chi^2}(x):=x+x^2/4$ and the supremum is achieved by $f_{\chi^2}:=2\left(\frac{\dd \mu}{\dd \nu}-1\right)$, whenever $\mu \ll \nu$.

\underline{Squared Hellinger distance:} Let $\eta \in \cP(\X)$ be a probability measure that dominates both $\mu,\nu \in \mathcal{P}(\X)$, i.e., $\mu,\nu \ll \eta$ (e.g., $\eta=(\mu+\nu)/2$), and denote the corresponding densities by $p=\frac{\dd \mu}{\dd \eta}$ and $q=\frac{\dd \nu}{\dd \eta}$. The squared Hellinger distance between~$\mu,\nu$ is\footnote{The standard definition of the squared Hellinger distance has an extra factor of $0.5$. We use the current definition as it simplifies the statements of some results and proofs, while clearly having no effect on the qualitative conclusions. The same applies for the TV distance given in \eqref{tvdefn}.} 
\begin{align}
     \mathsf{H}^2(\mu,\nu):=\EE_{\eta}\left[\big(\sqrt{p}-\sqrt{q}\big)^2\right]. \label{Heldistdef}
\end{align}
When $\mu \ll \nu$, the above expression can be written as
\begin{align}
     \mathsf{H}^2(\mu,\nu)=\EE_{\nu}\left[\left(\sqrt{\frac{\dd \mu}{\dd \nu}}-1\right)^2\right], \notag
\end{align}
with the corresponding dual form 
\begin{equation}
 \mathsf{H}^2(\mu,\nu)=\sup_{\substack{f:\X \rightarrow \mathbb{R},\\ f(x)< 1, \forall x\in \X}} \EE_{\mu}[f]-\EE_{\nu}\left[\frac{f}{1-f}\right], \label{helvarchar}
\end{equation}
where the supremum is over all functions such that the expectations are finite. This form corresponds to~\eqref{Mestgen2} with $h(x)=h_{\mathsf{H}^2}(x):=x/(1-x)$, and the supremum in~\eqref{helvarchar} is achieved by $f_{\mathsf{H}^2}:=1-\left(\frac{\dd \mu}{\dd \nu}\right)^{-1/2}$.  
 Also note that $\sqrt{\mathsf{H}^2}$ defines a metric on $\cP(\cX)$ and that $0 \leq \mathsf{H}^2(\mu,\nu) \leq 2$, for any $\mu,\nu\in\cP(\cX)$.

\underline{Total variation distance:}
The TV distance 
between  $\mu,\nu \in \cP(\X)$ is 
 \begin{align}
     \tv{\mu}{\nu}:=\sup_{\mathcal{C} } 2\abs{\mu(\mathcal{C})-\nu(\mathcal{C})},
     \label{tvdefn}
\end{align}
where the supremum is over all Borel subsets of $\X$. 
The corresponding variational form is
\begin{align}
     \tv{\mu}{\nu}=  \sup_{\substack{f:\X \rightarrow \mathbb{R},\\ \norm{f}_{\infty} \leq 1 }} \EE_{\mu}[f]-\EE_{\nu}[f],
   \label{tvdistvarcharac}
 \end{align}
which pertains to \eqref{Mestgen2} with $h(x)=h_{\mathsf{TV}}(x):=x$. Denoting the densities of $\mu$ and $\nu$ w.r.t. a common dominating measure $\eta \in \cP(\X)$ by $p$ and $q$, respectively,  the supremum in \eqref{tvdistvarcharac} is achieved by $f_{\mathsf{TV}}:=\ind_{\mathcal{C}^*}-\ind_{\X \setminus \mathcal{C}^*}$, where
\begin{equation}
    \mathcal{C}^*:=\big\{x \in \X: p(x) \geq q(x)\big\}. \label{setdefntvopt}
\end{equation} 
Furthermore, $\delta_{\mathsf{TV}}$ is a metric on $\cP(\cX)$ with $0\leq \tv{\mu}{\nu}\leq 2$.
\subsection{Stochastic Processes} 
Our analysis of the estimation error needs the following definitions.

\begin{definition}[Sub-Gaussian process] A real-valued stochastic process $(X_{\theta})_{\theta \in \Theta}$ on a metric space $(\Theta,\mathsf{d})$ is  sub-Gaussian if it is centered, i.e., $\EE[X_{\theta}]=0$ for all $\theta \in \Theta$, and 
\begin{equation}
    \mathbb{E}\big[e^{t(X_{\theta}-X_{\tilde{\theta}})}\big] \leq e^{\frac 12 t^2\mathsf{d}(\theta,\tilde{\theta})^2},\quad \forall ~\theta,\tilde{\theta} \in \Theta,~t \geq 0.\notag
\end{equation}
\end{definition}
\begin{definition}[Separable process]
A stochastic process $(X_{\theta})_{\theta \in \Theta}$ on a metric space $(\Theta,\mathsf{d})$ is said to be separable if there exists a null set $N$ and a countable subset  $\Theta_0 \subseteq \Theta $, such that for every $\omega\notin N$ and $\theta \in \Theta$, there is a sequence $(\theta_m)_{m \in \NN}$ in $\Theta_0$  with $\mathsf{d}(\theta_m, \theta) \rightarrow 0$ and $X_{\theta_m}(\omega) \rightarrow X_{\theta}(\omega)$.
\end{definition}    

\begin{definition}[Covering and packing numbers] \label{cov-pack-num}
Let $(\Theta,\mathsf{d})$ be a metric space. 
\begin{enumerate}[(i)]
    \item A set $\Theta'\subseteq\Theta$ is an $\epsilon$-covering of $(\Theta,\mathsf{d})$  if for every $\theta \in \Theta$, there exists $\tilde{\theta} \in \Theta' $ such that $\mathsf{d}(\theta,\tilde{\theta})\leq \epsilon$; the $\epsilon$-covering number is $N(\epsilon,\Theta,\mathsf{d}):=\inf \left\{|\Theta'|:\,\Theta' \mbox{ is an } \epsilon \mbox{-covering of }\Theta \right\}$.
    \item A set $\Theta' \subseteq \Theta$ is an $\epsilon$-packing of $(\Theta,\mathsf{d})$ if $\mathsf{d}(\theta,\tilde{\theta}) > \epsilon$ for every $\theta,\tilde{\theta} \in \Theta'$ such that $\theta \neq \tilde{\theta}$; the $\epsilon$-packing number is $T(\epsilon,\Theta,\mathsf{d}):=\sup\{|\Theta'|:\,\Theta' \mbox{ is an }\epsilon\mbox{-packing of }\Theta\}$.
\end{enumerate}
\end{definition}
\subsection{Function Classes} 
Our approximation result requires the target function on $\X$ to have an extension to $\RR^d$, whose spectral norm (as  introduced in \citep{Barron_1993} and \citep{Klusowski-2018}) is finite. The class of functions with such bounded spectral norm is defined next.
    \begin{definition}[Approximation class] \label{def:barronclass}
Let $m \in \NN$. Consider a function $f:\mathbb{R}^d \rightarrow \mathbb{R}$ that has a Fourier representation  $ f(x)=\int_{0}^{\infty} e^{i \omega \cdot x} F(\dd \omega)$, where $i=\sqrt{-1}$ is the imaginary unit and $F(\dd\omega)$ is a complex Borel measure over $\mathbb{R}^d$ with magnitude $\abs{F}(\dd\omega)$ that satisfies 
\begin{equation} 
   S_m(f):= \int_{\RR^d}  \norm{\omega}_1^m \abs{F}(\dd \omega) <\infty. \label{Cfconstdefval}
\end{equation}
For $c \geq 0$, $m =1,2$, and  $\cX\subseteq\RR^d$, define
\begin{align}
   \cB_{c,m,\cX}\big(\RR^d\big) &:= \left\{ f:\mathbb{R}^d \rightarrow \mathbb{R}: \norm{\cX} S_m(f) \vee \abs{f(0)} \vee   \norm{\nabla f (0)}_1 \ind_{\{m = 2\}}  \leq  c\right\}, \notag
 \end{align}  
   and  for $f:\X \rightarrow \RR$, set
   \begin{align}
  c^\star(f,m,\X)&:= \inf\left\{c:\exists ~ \tilde f \in  \cB_{c,m,\cX}\big(\RR^d\big),~f= \tilde f|_\cX \right\}. \notag
\end{align}
We refer to $\cB_{c,1,\cX}\big(\RR^d\big)$, $\cB_{c,2,\cX}\big(\RR^d\big)$, $c_{\mathsf{B}}^\star(f,\X):=c^\star(f,1,\X)$ and $c_{\mathsf{KB}}^\star(f,\X):=c^\star(f,2,\X)$ as the Barron class, Klusowski-Barron class, Barron coefficient, and Klusowski-Barron coefficient, respectively.  
\end{definition}
For TV distance neural estimation, analysis of the NN approximation error for step functions is required. Such functions naturally belong to the Lipschitz function class defined below. 
\begin{definition}[Lipschitz class] 
For  $r \in (0,\infty]$, $m \in \NN$, 
and $f \in L^r\big(\RR^d\big)$, the $m^{th}$ modulus of smoothness of $f$ is
\begin{equation}
  \xi_{m,r}(f,t):=\sup_{u \in \RR^d, \norm{u} \leq t}\norm{\Delta_{u}^m f}_{r,\RR^d}, \label{modsmdefn}
\end{equation}
where $\Delta_{u}^m f(x)=\sum_{j=0}^{m} (-1)^{m-j}f(x+ju)$. For $\X \subseteq \RR^d$ and $0<s \leq 1$,
 the Lipschitz class with smoothness parameter $s$ is
\begin{align}
\mathsf{Lip}_{s,r,b}(\X):=\big\{f \in L^r\big(\RR^d\big): \norm{f}_{\mathsf{Lip}(s,r)} \leq b, \supp{f}=\X \big\}, \notag
\end{align}
where $\norm{f}_{\mathsf{Lip}(s,r)}:=\norm{f}_{r}+\sup_{t>0} t^{-s}\xi_{1,r}(f,t)$ is the Lipschitz seminorm. 
\end{definition}
Note that norm in \eqref{modsmdefn} is taken over $\RR^d$ (despite the assumption that $f$ nullifies outside of $\X$). 
For $d=1$,  the class of functions of bounded variation over $\X \subset \RR$ is contained in $\cup_{b \in \RR}\mathsf{Lip}_{1,1,b}(\X)$.

\medskip
The Vapnik-Chervonenkis (VC) type class of functions will play a prominent role in our empirical estimation error analysis. 
\begin{definition}[VC-type class]\label{def:VCclass}
Let $\cF$ be a class of Borel measurable functions with domain $\X$ and  a finite measurable 
envelope $F$, i.e., $\sup_{f\in\cF}|f(x)| \leq F(x)<\infty,~ \forall x \in \X$. Then, $\cF$ is a VC-type class with envelope $F$ if there exists finite  constants $l_{\mathsf{vc}}(\cF)=l_{\mathsf{vc}}(\cF,F)$ and $u_{\mathsf{vc}}(\cF)=u_{\mathsf{vc}}(\cF,F)$ such that
\begin{align}
    \sup_{\gamma \in \cP(\X)} N\left(\epsilon \norm{F}_{2,\gamma},\cF,\|\cdot\|_{2,\gamma}\right) \leq \big(l_{\mathsf{vc}}(\cF)\epsilon^{-1}\big)^{u_{\mathsf{vc}}(\cF)},\quad \forall~0<\epsilon \leq 1. \label{def:VC-type class}
\end{align}
\end{definition}

Finally, we introduce the function class of shallow NNs. 
\begin{definition}[NN class] \label{def:NNclass}
Let $\phi:\RR \rightarrow \RR$ be a (non-linear) measurable activation function. The class of shallow NNs (i.e., with a single hidden layer) with $k$ neurons and bounds on its parameters  specified by  $\mathbf{a}=(a_1,a_2,a_3,a_4)$ $\in\RR^4_{\geq 0}$ is 
\begin{equation} 
  \cG_k(\mathbf{a},\phi) :=\left\{ g:\mathbb{R}^d\mspace{-5mu}\rightarrow \mathbb{R}:\begin{aligned}
    &\qquad\qquad g(x)=\sum_{i=1}^k \beta_i \phi\left(w_i\cdot x+b_i\right)+w_0 \cdot x + b_0, 
    \\&\max_{1 \leq i \leq k}\norm{w_{i}}_1 \vee \abs{b_i} \leq a_1,~  
 \max_{1 \leq i \leq k}|\beta_i| \leq a_2,~ \abs{b_0} \leq a_3, \norm{w_0}_1 \leq a_4 \end{aligned}\right\}. \notag
\end{equation}
Let $\phi_{\mathsf{S}}(z)=(1+e^{-z})^{-1}$ and  $\phi_{\mathsf{R}}(z)=z\vee 0$ denote the logistic sigmoid$\mspace{3 mu}$\footnote{The results that follow with $\phi_{\mathsf{S}}$ as activation straightforwardly applies to any continuous monotone  bounded  activation, e.g., any sigmoidal activation with $\phi(z) \rightarrow 1$ as $z \rightarrow \infty$ and $\phi(z) \rightarrow 0$ as $z \rightarrow -\infty$.} and the rectified linear unit (ReLU) activation functions, respectively. Further, for $a \geq 0$, define  the shorthands   $\cG_k^{\mathsf{S}}(a):=\cG_k\big(k^{1/2}\log k, 2k^{-1}a,a, 0,\phi_{\mathsf{S}}\big)$, $\cG_k^{\mathsf{R}}(a):=\cG_k\big(1,2k^{-1}a,a,a,\phi_{\mathsf{R}}\big)$, and  $\cG_k^*(\phi):=\cG_k\big(\mathbf{a}^*,\phi\big)$ with $\mathbf{a}^*=(1,1,1,0)$. Throughout, we will assume $\phi \in \{\phi_{\mathsf{S}},\phi_{\mathsf{R}}\}$.
\end{definition}

\subsection{Minimax Estimation Risk}

To investigate the decision-theoretic fundamental limit of estimating a SD $\mathsf{D}_{h, \mathcal{F}}$ as defined in  \eqref{Mestgen2}, we now define the minimax risk. Let $\cP_\cX^2\subseteq \cP(\cX)\times\cP(\cX)$ be a class of pairs of distributions between which $\mathsf{D}_{h, \mathcal{F}}$ is finite and fix $(\mu,\nu)\in\cP_\cX^2$. Let $X^n:=(X_1,\ldots,X_n)$ and $Y^n:=(Y_1,\ldots,Y_n)$ be $n$ independently and identically distributed (i.i.d.) samples from $\mu$ and $\nu$, respectively.\footnote{For simplicity, we restrict attention to the case where an equal number of samples is available from both $\mu$ and $\nu$, but our analysis readily extends to the mismatched scenario with the corresponding bounds obtained by replacing $n^{-1/2}$ by $(m^{-1}+n^{-1})^{1/2}$, where  $m$  denotes the number of samples from $\mu$ (say).} An estimator of $\mathsf{D}_{h, \mathcal{F}}$ based on these samples is denoted by $\hat{\mathsf{D}}_{h, \mathcal{F}}(X^n,Y^n)$. The minimax absolute-error risk is
\begin{equation}
     \mathcal{R}^\star_{h,\cF}(n,\cP_\cX^2):=\inf_{\hat{\mathsf{D}}_{h,\cF}}\sup\limits_{(\mu,\nu)\in\cP_\cX^2}\mathbb{E}\left[\left|\mathsf{D}_{h, \mathcal{F}}(\mu,\nu)-\hat{\mathsf{D}}_{h, \mathcal{F}}(X^n,Y^n)\right|\right].\label{EQ:risk}
\end{equation}
 We explore the performance of the NE
\begin{equation}
    \hat{\mathsf{D}}_{h,\cG_k(\mathbf{a}_k,\phi)}(X^n,Y^n) := \sup_{g \in \cG_k(\mathbf{a}_k,\phi)} \frac 1n \sum_{i=1}^n \Big[ g(X_i)- h \circ g(Y_i)\Big], \label{Mest-emp_def}
\end{equation}
under the above framework. By appropriately scaling the NN size $k$ (and parameter norm) with the sample size $n$, we show that NEs of KL and $\chi^2$ divergences as well as $\mathsf{H}^2$ distance converges at the parametric $n^{-\frac 12}$ rate uniformly over certain classes of distribution pairs satisfying regularity conditions. We further show (see, e.g., Corollary \ref{minimaxoptKL}) that the minimax risk is at least $\Omega(n^{-1/2})$ over this class, thus establishing the minimax optimality of NEs.


\section{Preliminary Technical Results} \label{Sec:prelim}

We next present two technical results that account for the NN approximation error and the empirical estimation error of the parametrized SD. 
These results are later leveraged to derive effective error bounds for neural estimation of KL and $\chi^2$ divergences, squared Hellinger distance and TV distance.


\subsection{Sup-norm Function Approximation}
We start with a bound on the approximation error of a target function $f$ with a compact domain $\X$ for which $ c^\star(f,m,\X)<\infty$, $m=1,2$. A reminiscent result for the case $m=1$ was given in \citep{Barron-1992}, albeit without explicitly quantifying the dependence on dimension or addressing how the NN parameters scale with $k$. The  bounds for $m=2$ are taken from  \citep{Klusowski-2018}.   
\begin{theorem}[Approximation error bound]\label{THM:approximation} 
Let $\X$ be compact. Given $f:\X \rightarrow \mathbb{R}$ with $ c_{\mathsf{KB}}^\star(f,\X)\leq a$, there exists $g \in \cG_k^{\mathsf{R}}(a)$   such that
\begin{align}
\|f-g\|_\infty \lesssim   a d^{\frac 12} k^{-\frac 12}. \label{approxrateklubar}
\end{align}
Similarly, given  $f:\X \rightarrow \mathbb{R}$ such that $ c_{\mathsf{B}}^\star(f,\X)\leq a$, there exists $g \in \cG_k^{\mathsf{S}}(a)$  satisfying \eqref{approxrateklubar}.
\end{theorem}
The above theorem states that a $k$-neuron shallow NN can approximate a function $f$ on  $\X$ within an $O(k^{-1/2})$ gap in the sup-norm, provided  $ f$ is the restriction of some $\tilde f$ from the Barron class or Klusowski-Barron class. The bound in \eqref{approxrateklubar} follows from  \citep[Theorem 2]{Klusowski-2018}, up to rescaling the domain therein.
The  proof of the second claim pertaining to approximation by NN class $\cG_k^{\mathsf{S}}(a)$ is provided in Appendix \ref{supnormapprox-proof}, and is based on ideas from \citep{Barron-1992,Barron_1993,Yukich-1995}.  
The  error bounds stated in  Theorem \ref{THM:approximation} are representative of the approximation capabilities of shallow NNs with ReLU (unbounded)  and sigmoid (bounded) activations, respectively. Note that $\cG_k^{\mathsf{R}}(a)$ has bounded parameters independent of $k$, albeit with an extra affine term (see Definition  \ref{def:NNclass}) compared to functions in $\cG_k^{\mathsf{S}}(a)$. On the other hand, achieving $O(k^{-1/2})$ approximation error using the latter class  requires the bounds on the hidden layer weights and biases  to scale as $k^{1/2}\log k$.   
\begin{remark}[Related approximation results] \label{Rem:relapproxres}
Several related approximation bounds to Theorem \ref{THM:approximation} are available in the literature, which can also be leveraged to analyze the approximation error of NEs. In particular, \citet[Theorem 2.2]{Yukich-1995} provides sup-norm error bounds for approximating a target function and its derivatives by a sigmoidal NN with unbounded input weights and biases. A further improvement over \citep[Theorem 2]{Barron-1992}   by a $k^{-1/2d}$ factor is reported in \citep{MAKOVOZ-1998} for NNs with step activation functions, under a different regularity condition on the Fourier transform of target function. A sup-norm approximation result for squared ReLU activation  is given in \citep[Theorem 3]{Klusowski-2018} for functions $f$ with bounded $S_3(f)$ (see \eqref{Cfconstdefval}). Also related are NN approximation bounds  derived in \citep{Domingo-2021} for a function with bounded  $\cR,\Ucal$-norm, where the latter is based on $\cR$-norm introduced in \citep{Ongie2020}. 
\end{remark}

The next proposition shows that a sufficiently  smooth function over a compact domain can be approximated to within $O(k^{-1/2})$ error by a shallow NN.
\begin{proposition}[Approximation of smooth functions] \label{prop:bndfourcoeff}
Let $\X\subseteq\RR^d$ be compact and $f:\X \rightarrow \RR$. Suppose that there exists an open set $~\Ucal \supset  \X$, $b \geq 0$, and $ \tilde f \in  \mathsf{C}_b^{s_{\mathsf{KB}}}(\Ucal)$, $s_{\mathsf{KB}}:=\lfloor d/2\rfloor+3$,  such that $f= \tilde f|_\cX$.  
Then, there exists $g \in \mathcal{G}_k^{\mathsf{R}}\left(\bar c_{b,d,\norm{\cX}}\right)$, where $\bar c_{b,d,\norm{\cX}}$ is given in \eqref{constapproxhold}, such~that $\|f-g\|_\infty\lesssim c_{b,d,\norm{\cX}}d^{1/2} k^{-1/2}$. The same holds with  $s_{\mathsf{KB}}$ and $\mathcal{G}_k^{\mathsf{R}}$ replaced with $s_{\mathsf{B}}:=\lfloor d/2\rfloor+2$ and $\mathcal{G}_k^{\mathsf{S}}$, respectively.
\end{proposition}

The proof of Proposition \ref{prop:bndfourcoeff} (see Appendix \ref{prop:bndfourcoeff-proof}) shows that any sufficiently smooth function on $\cX$ can be extended to a function in the Barron or the Klusowksi-Barron class with domain $\RR^d$. This is done by nullifying the partial derivatives of order $s_{\mathsf{KB}}$ (or $s_{\mathsf{B}}$) outside $\X$ and multiplying by a smooth bump function that equals 1 on $\X$ and smoothly decays outside. Note that for an integer $s \geq 0$ and a real number $\tilde s \geq s$, $\mathsf{C}_b^{s}(\Ucal)$ contains the H\"{o}lder  class with smoothness $\tilde s$ and radius $b$.

\subsection{Estimation of Parameterized Divergences}
For $\mu,\nu\in\cP(\X)$, consider the SD $\mathsf{D}_{h,\cF}(\mu,\nu)$ defined in \eqref{Mestgen2}. Let $X^n$ and $Y^n$ be $n$ i.i.d. samples from $\mu$ and $\nu$, respectively. Consider a NE for $\mathsf{D}_{h,\cF}(\mu,\nu)$ realized by a shallow NN, i.e., $\hat{\mathsf{D}}_{h,\mathcal{G}_k(\mathbf{a}_k,\phi)}(X^n,Y^n)$ (see \eqref{Mest-emp_def}). Our next result provides a tail inequality for the error in estimating the parametrized divergence $\mathsf{D}_{h, \cG_k^*(\phi)}(\mu,\nu)$ by $\hat{\mathsf{D}}_{h,\cG_k^*(\phi)}(X^n,Y^n)$, which will be used to prove consistency of the NE. To state it, given a class of functions $\cF$ with domain $\X$, define $\underaccent{\bar}{C}(\cF,\cX):=\inf_{x \in \X,f \in \cF} f(x)$ and $\bar C(\cF,\cX) :=\sup_{x \in \X, f \in \cF} f(x)$.

\begin{theorem}[Empirical estimation error tail bound] \label{empesterrbnd}
Let $\mu,\nu \in \cP(\X)$ and consider the NN class $\cG_k^*(\phi)$ given in Definition \ref{def:NNclass}.
Assume $\cX$ and $\phi$ are such that $\bar C\big(|\cG_k^*(\phi)|,\cX\big)<\infty$, $h$ is differentiable in $\big[\underaccent{\bar}{C}\big(\cG_k^*(\phi),\cX\big),\bar C\big(\cG_k^*(\phi),\cX\big)\big]$ with derivative $h'$, $\mathsf{D}_{h,\cG_k^*(\phi)}(\mu,\nu)<\infty$, and 
\begin{align}
 \bar C\left(\abs{h'\circ \cG_k^*(\phi)},\cX\right)
 <\infty. \label{maxdergamma}
\end{align}
Then there exists a constant $c>0$ such that for any $\delta \geq 0$, we have
\begin{flalign}
   & \sup_{\substack{\mu,\nu \in \cP(\X):\\\mathsf{D}_{h,\cG_k^*(\phi)}(\mu,\nu)<\infty}}\mathbb{P}\mspace{-1mu}\Big(\mspace{-4mu}\abs{ \hat{\mathsf{D}}_{h,\cG_k^*(\phi)}\mspace{-1.5mu}(X^n,Y^n\mspace{-1.5mu})\mspace{-3mu}-\mspace{-3mu}\mathsf{D}_{h,\cG_k^*(\phi)}\mspace{-2mu}(\mu,\nu)}\mspace{-2mu}\geq\mspace{-2mu} \delta\mspace{-2mu}+\mspace{-2mu} E_{k,h,\phi,\cX}n^{-\frac 12}\mspace{-2 mu}\Big)\mspace{-4 mu} \leq \mspace{-2 mu} c e^{-\frac{n\delta^2}{V_{k,h,\phi,\cX}}}, \label{bndesterremp} &&
\end{flalign}
with upper bounds for $V_{k,h,\phi,\cX}$ and $ E_{k,h,\phi,\cX}$ available in \eqref{Vkconstdef} and \eqref{Ekconstdef}, respectively. 
\end{theorem}
The proof of Theorem \ref{empesterrbnd} (see Appendix \ref{empesterrbnd-proof}) relies on upper bounding the estimation error by a separable sub-Gaussian process and invoking the chaining tail inequality (see Theorem \ref{thm:tailineq} in Appendix \ref{empesterrbnd-proof}).

The next theorem provides an upper bound on the expected  empirical estimation error. It will be used to obtain effective error bounds for the NE in the forthcoming sections. 
\begin{theorem}[Empirical estimation error bound]\label{thm:optkdepNN}
Let $a>0$ and $\cG_k \in \big\{\cG_k^{\mathsf{R}}(a),\cG_k^{\mathsf{S}}(a)\big\}$. 
Suppose $h$ is differentiable  in $\big[\underaccent{\bar}{C}\big(\cG_k,\cX\big),\bar C\big(\cG_k,\cX\big)\big]$ with derivative $h'$, $\bar C\left(\abs{h'\circ \cG_k},\cX\right) \vee \bar C(|h \circ \cG_k|,\cX)$ $ \vee \bar C\big(|\cG_k|,\cX\big) \lesssim_{a,h,\norm{\cX}}1
 $ for all $k \in \NN$. 
 Then, for all $k,n \in \NN$, 
\begin{flalign}
&\sup_{\mu,\nu \in \cP(\X)} 
\EE\left[\abs{ \hat{\mathsf{D}}_{h,\cG_k}(X^n,Y^n)-\mathsf{D}_{h, \cG_k}(\mu,\nu)}\right] 
 \lesssim_{h,a,\norm{\cX}}
d^{\frac 32}n^{-\frac 12}. \label{empestuppbnd}
\end{flalign}
\end{theorem}
Theorem \ref{thm:optkdepNN} follows from a more general result that we establish in Appendix \ref{thm:optkdepNN-proof} (namely, Theorem \ref{empesterrorbndgen}), where $\cG_k$ as above is replaced by an arbitrary VC-type class satisfying certain technical conditions. The proof of the latter relies on standard maximal inequalities from empirical process theory. To prove \eqref{empestuppbnd}, we also require a bound on the entropy integral of the NN class. This is obtained by noting that $\cG_k$ is a subset of the symmetric convex hull 
of the composition of a monotone function with a VC subgraph class, and upper bounding the covering numbers of such convex hulls. 
\begin{remark}[NN distances]
The SD $\mathsf{D}_{h, \cG_k(\mathbf{a},\phi)}(\mu,\nu)$ is the so-called NN distance, studied in \citep{arora2017generalization,zhang2018discrimination} in the context of GANs. Theorem \ref{empesterrbnd} and \ref{thm:optkdepNN} can thus be understood, respectively, as a tail bound and as an error bound for NN distance estimation from data, and implies that the  estimation error rate is parametric in $n$.
\end{remark}
For $\cG_k \in \big\{\cG_k^{\mathsf{R}}(a),\cG_k^{\mathsf{S}}(a)\big\}$, we have $\bar C\left(\abs{\cG_k}\right) \leq 3a(\norm{\cX}+1)$, $\bar C\left(\abs{h \circ \cG_k}\right) \leq \sup\{\abs{h(z)}, z \in [-3a(\norm{\cX}+1),3a(\norm{\cX}+1)]\}<\infty$ and $\mathsf{D}_{h, \cG_k}<\infty$ for all $k$ and $h \in \{h_{\mathsf{KL}},h_{\chi^2}\}$.  Similarly, $\bar C\left(\abs{h' \circ \cG_k}\right)$ is finite and bounded by a quantity independent of $k$ for these $h$ (see \eqref{derbndkl} and \eqref{derbndchisq}). 
Hence,  $h_{\mathsf{KL}}$ and $h_{\chi^2}$ satisfies the assumptions in Theorem \ref{thm:optkdepNN}, and consequently, \eqref{empestuppbnd} applies for KL and $\chi^2$ divergences. These bounds also hold for $\mathsf{H}^2$ and TV distances for appropriate NN classes (see Theorems  \ref{strongconshel} and \ref{TVerrbnd} below).    
In the next section, we use the above results to analyze the effective error for neural estimation of SDs.  
\section{Neural Estimation of $\mathsf{f}$-Divergences} \label{Sec:NESDs}
We now turn to analyze neural estimation performance of several important $\mathsf{f}$-divergences, encompassing KL, $\chi^2$, $\mathsf{H}^2$, and TV. Throughout this section, we assume for simplicity that $\X =[0,1]^d$, but the results and proof techniques readily extend to arbitrary compact domains. Further, we present results for ReLU NNs, although all statements also hold for sigmoid nets with a slightly modified spectral norm condition defining the class of distributions. We comment about this once in Remark \ref{rem:sigmoidcls} below, but omit further mention to avoid repetition.  
\subsection{KL Divergence} \label{sec:KLdiv}
Let $\hat{\mathsf{D}}_{\mathcal{G}_k(\mathbf{a}_k,\phi)}(X^n,Y^n):=\hat{\mathsf{D}}_{h_{\mathsf{KL}},\mathcal{G}_k(\mathbf{a}_k,\phi)}(X^n,Y^n)$ be a NE of $\kl{\mu}{\nu}$, where $\mathbf{a}_k\in\RR^4_{\geq 0}$ for all $k\in\NN$. 
To state performance guarantees for this NE, some definitions are needed. Let $\mathcal{P}^2_{\mathsf{KL}}(\X)$
be the set of all pairs $(\mu,\nu) \in \mathcal{P}(\X) \times \mathcal{P}(\X)$ such that $\mu \ll \nu$ and $\kl{\mu}{\nu}<\infty$, and for any $M\geq 0$ define
\begin{equation}
 \mathcal{P}^2_{\mathsf{KL}}(M,\X):= \left\{(\mu,\nu) \in \mathcal{P}^2_{\mathsf{KL}}(\X):  c_{\mathsf{KB}}^\star(f_{\mathsf{KL}},\X)\vee \kl{\mu}{\nu} \leq M  \right\}.\label{klsuffpotcond}
\end{equation}
For appropriately chosen $M,b \geq 0$, $\mathcal{P}^2_{\mathsf{KL}}(M,\X)$  contains  $(\mu,\nu) \in \mathcal{P}^2_{\mathsf{KL}}(\X)$ for which $\kl{\mu}{\nu}$ $\leq M$ and  $f_{\mathsf{KL}}=\log\frac{\dd\mu}{\dd\nu}\in\mathsf{C}_b^{s_{\mathsf{KB}}}(\Ucal)$ for some $ \Ucal \supseteq \X $. To see this, note that a smoothness order of $s_{\mathsf{KB}}$ for $f_{\mathsf{KL}}$ ensures that $c_{\mathsf{KB}}^\star(f_{\mathsf{KL}},\X) \leq \bar c_{b,d,\norm{\cX}}$ (see Proposition \ref{prop:bndfourcoeff}). Hence, for any $(\mu,\nu) \in \mathcal{P}^2_{\mathsf{KL}}(\X)$ and $M \geq \bar c_{b,d,\norm{\cX}} \vee \kl{\mu}{\nu}$,  $(\mu,\nu)  \in \mathcal{P}^2_{\mathsf{KL}}(M,\X)$.  
In particular, $\mathcal{P}^2_{\mathsf{KL}}(M,\X)$, for sufficiently large $M$, contains  Gaussian densities,  truncated and normalized to be supported on $\X$.

Since the class $ \mathcal{P}^2_{\mathsf{KL}}(M,\X)$ becomes larger as $M$ increases, it is to be expected that a larger NN class would be required for accurate neural estimation of KL divergence between distributions in this class. This means that the range of the  NN parameters has to be  selected depending on $M$. However, often it is hard to ascertain such an $M$ for the distributions of interest. To account for this,  we do not assume that $M$ is known in advance. Instead, we take a NN class  $\cG_k^{\mathsf{R}}(m_k)$ for some non-decreasing positive sequence $(m_k)_{k \in \NN}$ with $m_k \rightarrow \infty$,  for obtaining  neural estimation error bounds.

The following theorem establishes the consistency of  KL divergence NE and uniformly bounds the effective error in terms of the NN and sample sizes.

\begin{theorem}[KL divergence neural estimation]\label{strongcons} 
The following hold:
\begin{enumerate}[label = (\roman*),leftmargin=15 pt]
\item Let  $(\mu,\nu) \in \mathcal{P}^2_{\mathsf{KL}}(\X)$ be such that $f_{\mathsf{KL}} \in \mathsf{C}\left(\X\right)$. Then, for  any $0<\rho<1$,   $(k_n)_{n \in \NN}$ with $k_n\ \rightarrow  \infty$, $k_n \leq \frac 14(1-\rho) \log n$ and $\cG_n=\cG_{k_n}^*(\phi)$, 
 \begin{equation}
    \hat{\mathsf{D}}_{\cG_n}(X^n,Y^n)    \xrightarrow[n\rightarrow \infty]{} \kl{\mu}{\nu},\quad \mathbb{P}-\mbox{a.s.} \label{finbndascon}
 \end{equation}
\item For any $M\geq 0$, $m_k=\log \log k \vee 1$, $\cG_k=\cG_k^{\mathsf{R}}(m_k)$,
\begin{flalign}
 &  \sup_{(\mu,\nu) \in \mathcal{P}^2_{\mathsf{KL}}(M,\X)}\mathbb{E}\left[  \abs{\hat{\mathsf{D}}_{\mathcal{G}_{k}}(X^n,Y^n)  -\kl{\mu}{\nu}}\right]  \lesssim_{M}d^{\frac 12} k^{-\frac{1}{2}}+d^{\frac 32}(\log k)^7 n^{-\frac 12}.\label{KLeffbndsimp} &&
\end{flalign} 
\end{enumerate}
\end{theorem}
The proof of Theorem \ref{strongcons} is presented in Appendix \ref{strongcons-proof}. The consistency result in Part~$(i)$ relies on $\cG_k^*(\phi)$ being a universal approximator for the class of continuous functions on compact sets as $k \rightarrow \infty$ and Theorem \ref{empesterrbnd}. For Part (ii), we  derive \eqref{KLeffbndsimp} by utilizing Theorems \ref{THM:approximation} and \ref{thm:optkdepNN} to bound the sum of the approximation and estimation errors. From Theorem \ref{THM:approximation}, the former is $O(k^{-1/2})$ if $c_{\mathsf{KB}}^\star\left(f_{\mathsf{KL}},\X\right) \leq M$ and $k$ is such that $M \leq  \log  \log k \vee 1$. On the other hand, for $k$ violating this condition, the effective error is bounded by $\kl{\mu}{\nu} \leq M$. The growing NN parameters contribute an extra $\mathrm{polylog}(k)$ factor to the empirical estimation error bound.

\begin{remark}[Effective error bound for sigmoid NN class] \label{rem:sigmoidcls}
It can be seen from the proof of \eqref{KLeffbndsimp} that the same bound  applies to sigmoid NN class  $\cG_k^{\mathsf{S}}(m_k)$ when $\mathcal{P}^2_{\mathsf{KL}}(M,\X)$ is replaced by $\mathcal{P}^2_{\mathsf{KL,B}}(M,\X):=\left\{(\mu,\nu) \in \mathcal{P}^2_{\mathsf{KL}}(\X):  c_{\mathsf{B}}^\star(f_{\mathsf{KL}},\X)\vee \kl{\mu}{\nu} \leq M  \right\}$.  Similar remarks apply for all the effective error bounds henceforth, which we omit to avoid repetition.
\end{remark}
\begin{remark}[Effective error bound based on $M$] \label{KLNEratessimp}
If $M$ in the definition of the class $\mathcal{P}^2_{\mathsf{KL}}(M,\X)$ is known when picking the NN parameters (i.e., they can depend on $M$), then 
with $m_k=M$ and $\cG_k=\cG_k^{\mathsf{R}}(M)$, we have (see \eqref{finerrbndklapp} and the last statement in the proof of Theorem \ref{strongcons} in Appendix~\ref{strongcons-proof}) 
 \begin{flalign}
 & \sup_{(\mu,\nu) \in \mathcal{P}^2_{\mathsf{KL}}(M,\X)} \mathbb{E}\left[  \abs{\hat{\mathsf{D}}_{\mathcal{G}_{k}}(X^n,Y^n)  -\kl{\mu}{\nu}}\right] \lesssim_{M}~ d^{\frac 12} k^{-\frac{1}{2}}+ d^{\frac{3}{2}} n^{-\frac 12},\label{finbnderrrate-case1}
 \end{flalign}
which removes the 
polylog factor in the empirical estimation  bound (2nd~term in \eqref{KLeffbndsimp}).
\end{remark}

\begin{remark}[$L^2$ neural estimation of a function] A reminiscent approximation-estimation error analysis for learning a NN approximation of a bounded range function is presented in  \citep{Barron-1994}. This differs from our setup since SDs are given as a supremum over a function class, as opposed to a single function. As such, our results require stronger sup-norm approximation results, as opposed to the $L^2$ bound used in \citep{Barron-1994}. 
\end{remark}
The error bounds in \eqref{finbnderrrate-case1}  and \eqref{KLeffbndsimp}  imply that the KL divergence NE achieves the parametric and near parametric error rates, respectively. 
\begin{corollary}[Minimax optimality]\label{minimaxoptKL}
 The KL divergence NE $\hat{\mathsf{D}}_{\mathcal{G}_{n}}(X^n,Y^n)$ is minimax rate-optimal over $\mathcal{P}^2_{\mathsf{KL}}(M,\X)$ and  $\mathcal{P}^2_{\mathsf{KL,B}}(M,\X)$ with  $\cG_n=\cG_n^{\mathsf{R}}(M)$ and  $\cG_n=\cG_n^{\mathsf{S}}(M)$, respectively, achieving the $O\big(n^{-1/2}\big)$ minimax risk. If $M$ is unknown, then this NE with $M$ replaced by $m_n=\log \log n \vee 1$, is near minimax optimal achieving $\tilde O\big(n^{-1/2}\big)$ minimax risk.
\end{corollary}

The corollary is proven in Appendix \ref{minimaxoptKL-proof}, where the upper bound follows directly from  Theorem \ref{KLNEratessimp}, Remark \ref{rem:sigmoidcls} and \eqref{finbnderrrate-case1} by setting $k=n$. For the lower bound, we present a reduction of the KL divergence estimation problem to differential entropy estimation, and invoke the  $\Omega(n^{-1/2})$ lower bound from \citet{Goldfeld-2020} for the latter problem.

Theorem \ref{strongcons} and Corollary \ref{minimaxoptKL} impose conditions on $f_{\mathsf{KL}}$ to bound the effective neural estimation error (namely, assuming that $c_{\mathsf{KB}}^\star(f_{\mathsf{KL}},\X) \leq M$, for some $M$). A primitive sufficient condition in terms of the densities  $p$ and $q$ of $\mu$ and $\nu$, respectively, w.r.t. an arbitrary common dominating measure $\eta$ is given next.
\begin{proposition}[Sufficient condition for Theorem \ref{strongcons}] \label{prop:distconddirect}
For  $b\geq 0$ and $s_{\mathsf{KB}}=\lfloor d/2\rfloor+3$, consider the class $\widetilde{\cP}^2_{\mathsf{KL}}(b,\X)$ of pairs of distributions given by 
\begin{equation}
  \widetilde{\cP}^2_{\mathsf{KL}}(b,\X):=\left\{(\mu,\nu)\in \mathcal{P}^2_{\mathsf{KL}}(\X): \begin{aligned}    & \exists ~  \tilde{p},\tilde{q} \in \mathsf{C}_b^{s_{\mathsf{KB}}}(\Ucal)\mbox{ for some open set }\Ucal\supset\X \\&\mbox{ s.t. }\log p=\tilde{p}|_\X,~\log q=\tilde{q}|_\X\end{aligned} \right\}. \notag
\end{equation}
Then, \eqref{KLeffbndsimp} and \eqref{finbnderrrate-case1}  hold with  $M=2\bar c_{b,d,\norm{\cX}} \vee 2b$, where  $\bar c_{b,d,\norm{\cX}}$ is given in \eqref{constapproxhold},\footnote{Although $\X$ is taken to be $[0,1]^d$, we will retain the dependence of $\X$ in the error bounds, which will be used later for extending the results to the unbounded support case.} and $\widetilde{\cP}^2_{\mathsf{KL}}(b,\X)$ in place of $\mathcal{P}^2_{\mathsf{KL}}(M,\X)$. 
\end{proposition}
\begin{remark}[Feasible distributions]\label{kldivclassdist}
 $\widetilde{\cP}^2_{\mathsf{KL}}(\cdot,\X)$  contains distributions $(\mu,\nu) \in \mathcal{P}^2_{\mathsf{KL}}(\X)$ whose densities $(p,q)$ are  bounded (from above and below) on $\X$ with a smooth extension on an open set covering $\X$. In particular, this includes uniform distributions, truncated Gaussians, truncated Cauchy distributions, etc. 
\end{remark}
\begin{remark}[Relation to other works] \label{rem:relation-otherworks}
Corollary \ref{minimaxoptKL} and Proposition \ref{prop:distconddirect} together imply that the KL divergence NE achieves the parametric minimax error rate over the class of densities with at least $s_{\mathsf{KB}}=\lfloor d/2\rfloor+3$ derivatives (or $s_{\mathsf{B}}=\lfloor d/2\rfloor+2$ derivatives in the case of sigmoid activation). To compare with existing results, it is known that variants of classic kernel-based estimators \citep{kandasamy2015nonparametric,Singh-Poczos-2014,Moon-2018,berrett2019efficient}  achieve the optimal minimax risk of $O(n^{-1/2})$ when the densities are H\"{o}lder smooth with at least $d/2$ or $d$ derivatives. We also note that the parametric rate achieved by NE is an improvement over the $n^{-1/4}$ rate shown in \citet{SS-2021}. Furthermore, we observe that \eqref{KLeffbndsimp}, \eqref{finbnderrrate-case1} and minimax rate optimality  holds for the class of distributions  obtained by replacing   $c_{\mathsf{KB}}^\star(f_{\mathsf{KL}},\X) \leq M$ in \eqref{klsuffpotcond} with $\norm{\cX}\norm{f_{\mathsf{KL}}}_{\mathcal{R},\mathcal{U}} \vee |f_{\mathsf{KL}}(0)|\vee \norm{\nabla f_{\mathsf{KL}}(0)}_1 \leq M$,  where $\mathcal{R},\mathcal{U}$-norm is defined in \citep[Equation 6]{Domingo-2021}. This follows by using \citep[Theorem 2]{Domingo-2021} in place of Theorem \ref{THM:approximation} to analyze the approximation error. Similar conclusions  hold for NEs of other SDs considered below. 
\end{remark}

\subsubsection{Neural Estimation via Donsker-Varadhan Formula} 
Another well known variational representation for  KL divergence is the Donsker-Varadhan (DV) formula:
\begin{align}
    \kl{\mu}{\nu}=\sup_{f \in \cF} \EE_{\mu}[f]-\log \EE_{\nu}\big[e^f\big],  \notag
\end{align}
where the supremum is over all measurable $f$ such that the last expectation  is finite. Parametrizing $\cF$ by a NN and replacing expectation with sample means leads to the DV-NE for KL, given by 
\begin{equation}
    \check{\mathsf{D}}_{\mathsf{DV},\cG}(X^n,Y^n) := \sup_{g \in \cG} \frac 1n \sum_{i=1}^n g(X_i)- \log \frac 1n \sum_{i=1}^n e^{g(Y_i)}.\notag
\end{equation}

In  \citep{belghazi2018}, the authors studied the special case of DV-NE pertaining to estimation of mutual information, termed MINE. They established consistency along with sample complexity bounds (without accounting for the approximation error). In Appendix~\ref{App:DVNE-consefferr}, we show that consistency of the DV-NE holds under similar conditions as in Theorem \ref{strongcons}  (see \eqref{DVconsisproof}). We also prove that the effective error bound given in \eqref{KLeffbndsimp} applies to DV-NE, albeit with different constants (see \eqref{DVNEexperrorbnd}). 
In particular, the latter establishes the minimax optimality of DV-NE  with the scaling  $k=n$. Instantiating these results for $\mu=P_{AB}$ and $\nu=P_A\otimes P_B$ (i.e., a joint probability law versus the product of its marginals), translates these performance guarantees to MINE, now accounting for finite-size NNs, the associated approximation error, and minimax convergence rates.

\subsection{$\chi^2$ Divergence} 
Let $\hat{\chi}^2_{\mathcal{G}_k(\mathbf{a}_k,\phi)}(X^n,Y^n):=\hat{\mathsf{D}}_{h_{\chi^2},\mathcal{G}_k(\mathbf{a}_k,\phi)}(X^n,Y^n)$ denote the NE of $\chisq{\mu}{\nu}$. Set $\mathcal{P}_{\chi^2}^2(\X)$ as the collection of all $(\mu,\nu) \in \mathcal{P}(\X) \times \mathcal{P}(\X)$ such that $\mu \ll \nu$ and $\chisq{\mu}{\nu}<\infty$, and let
\begin{equation}
 \mathcal{P}^2_{\chi^2}(M,\X):= \left\{(\mu,\nu) \in \mathcal{P}_{\chi^2}^2(\X):  c_{\mathsf{KB}}^\star\big(f_{\chi^2},\X\big)\vee \chisq{\mu}{\nu} \leq M  \right\}.\notag
\end{equation}
The next theorem establishes consistency of the NE and bounds its effective absolute-error.
\begin{theorem}[$\chi^2$ divergence neural estimation]\label{strongconschisq}
The following hold:
\begin{enumerate}[label = (\roman*),leftmargin=15 pt]
     \item  Let  $(\mu,\nu) \in \mathcal{P}_{\chi^2}^2(\X)$ be such that $f_{\chi^2} \in \mathsf{C}\left(\X\right)$. Then, for any $0<\rho<1$, $(k_n)_{n \in \NN}$ with $k_n \rightarrow  \infty$, $k_n =O\left(n^{(1-\rho)/5}\right)$ and $\cG_n=\cG_{k_n}^*(\phi)$,  
     we have
 \begin{align}
 \hat{\chi}^2_{\cG_n}(X^n,Y^n)\xrightarrow[n\rightarrow \infty]{}  \chisq{\mu}{\nu}, \quad \mathbb{P}-\mbox{a.s.}  \label{finbndasconchisq}
 \end{align}
\item 
For any $M \geq 0$, $m_k=\log k$, and $\cG_k=\cG_k^{\mathsf{R}}(m_k)$, we have 
  \begin{flalign}
 & \sup_{(\mu,\nu)\in \mathcal{P}^2_{\chi^2}(M,\X)} \mathbb{E}\left[  \abs{\hat{\chi}^2_{\mathcal{G}_{k}}(X^n,Y^n) -\chisq{\mu}{\nu}}\right]  \lesssim_{M}d k^{-\frac{1}{2}}+  d^{\frac 32}(\log k)^{2} n^{-\frac 12}. \label{chisqsimprate} 
 \end{flalign}
   \end{enumerate}
  \end{theorem}
The proof strategy for Theorem \ref{strongconschisq} is similar to that of Theorem \ref{strongcons}, with appropriate adaptations to account for the difference between $f_{\chi^2}$ and $f_{\mathsf{KL}}$ (see Appendix \ref{strongconschisq-proof}). Comparing \eqref{finbndasconchisq}-\eqref{chisqsimprate} to \eqref{finbndascon}-\eqref{KLeffbndsimp}, we see that consistency for $\chi^2$ divergence estimation holds under milder conditions and that the effective error bound is better in terms of dependence on $k$ than for KL divergence.

\begin{remark}[Effective error based on $M$]
If the NN parameters can depend on $M$, then setting $m_k=M$ in \eqref{chisqsimprate} yields
\begin{flalign}
    & \sup_{(\mu,\nu)\in \mathcal{P}^2_{\chi^2}(M,\X)} \mathbb{E}\left[  \abs{\hat{\chi}^2_{\mathcal{G}_{k}}(X^n,Y^n) -\chisq{\mu}{\nu}}\right]  \lesssim_{M}~ dk^{-\frac 12}+d^{\frac 32} n^{-\frac 12}. \label{finbnderrratechisq-case1}
\end{flalign}
\end{remark}
Choosing $k=n$ in \eqref{finbnderrratechisq-case1} (resp. \eqref{chisqsimprate}), we have that the $\chi^2$ NE achieves the parametric (resp. near parametric) error rate over the class $ \mathcal{P}^2_{\chi^2}(M,\X)$. The proof is similar to that of Corollary \ref{minimaxoptKL}, and is omitted for brevity. Let $\mathcal{P}^2_{\chi^2,\mathsf{B}}(M,\X):= \big\{(\mu,\nu) \in \mathcal{P}_{\chi^2}^2(\X):  c_{\mathsf{B}}^\star\big(f_{\chi^2},\X\big)\vee \chisq{\mu}{\nu} \leq M  \big\}$.
\begin{corollary}[Minimax optimality] \label{chisqminimaxopt}
The $\chi^2$ NE $\hat{\chi}^2_{\mathcal{G}_{n}}(X^n,Y^n)$ is minimax rate-optimal over  $ \mathcal{P}^2_{\chi^2}(M,\X)$ and $\mathcal{P}^2_{\chi^2,\mathsf{B}}(M,\X)$  with $\cG_n =\cG_n^{\mathsf{R}}(M)$ and $\cG_n=\cG_n^{\mathsf{S}}(M)$, respectively, achieving the $O(n^{-1/2})$ risk.   This NE achieves the near parametric $\tilde O\big(n^{-1/2}\big)$ minimax risk when $M$ is replaced by $\log n$, which is applicable to the scenario of unknown $M$. 
\end{corollary}
Given next is the counterpart of Proposition \ref{prop:distconddirect} for $\chi^2$ divergence (proven in Appendix~\ref{prop:distconddirect-chisq-proof}), which provides primitive conditions in terms of densities under which the effective error bounds in Theorem \ref{strongconschisq} and Corollary \ref{chisqminimaxopt} hold.
\begin{proposition}[Sufficient condition for Theorem~\ref{strongconschisq}] \label{prop:distconddirect-chisq}
For  $b\geq 0$ and $s_{\mathsf{KB}}=\lfloor d/2\rfloor+3$, let 
\begin{equation}
\widetilde{\cP}^2_{\chi^2}(b,\X):=\left\{(\mu,\nu) \in \mathcal{P}_{\chi^2}^2(\X):\begin{aligned} &\exists ~\tilde p,\tilde q \in \mathsf{C}_b^{s_{\mathsf{KB}}}(\Ucal)\mbox{ for some }\mbox{open set }\Ucal \supset \X \\&\mbox{ s.t. }p=\tilde p|_\X,~ q^{-1}=\tilde q|_\X\end{aligned}\right\}. \notag
\end{equation}
Then, \eqref{chisqsimprate} and \eqref{finbnderrratechisq-case1} hold with  $M=(\kappa_d d^{3/2} \norm{\X} \vee 1)\big(2+2^{s_{\mathsf{KB}}+1} \bar c_{b,d,\norm{\cX}}^2\big)\vee ( b^2+1)$, where 
$\kappa_d$ and $\bar c_{b,d,\norm{\cX}}$ are given in \eqref{constkappa} and \eqref{constapproxhold}, respectively, and $\widetilde{\cP}^2_{\chi^2}(b,\X)$ in place of $ \mathcal{P}^2_{\chi^2}(M,\X)$,
\end{proposition}
\begin{remark}[Feasible distributions]The class $\widetilde{\cP}^2_{\chi^2}(\cdot,\X)$ contains $(\mu,\nu) \in \mathcal{P}_{\chi^2}^2(\X)$,  whose densities $p,q,$ are bounded (upper bounded for $p$ and bounded away from zero for $q$) on $\X$ with an extension that is sufficiently smooth on an open set covering $\X$. This includes the distributions mentioned in Remark~\ref{kldivclassdist}.
\end{remark}
\subsection{Squared Hellinger Distance}

Let $
\hat{\mathsf{H}}^2_{\tilde {\mathcal{G}}_{k,t}(\mathbf{a}_k,\phi)}(X^n,Y^n):=\hat{\mathsf{D}}_{h_{ \mathsf{H}^2},\tilde{\mathcal{G}}_{k,t}(\mathbf{a}_k,\phi)}(X^n,Y^n)$,
where for $t>0$, $\tilde{\mathcal{G}}_{k,t}(\mathbf{a},\phi)$ is the NN class
  \begin{equation}
 \tilde{\mathcal{G}}_{k,t}\left(\mathbf{a},\phi\right)
  :=\left\{g:\mathbb{R}^d \rightarrow \mathbb{R}:  g(x)=(1-t) \wedge \tilde g(x),~\tilde g \in \cG_k(\mathbf{a},\phi)\right\}.  \notag
  \end{equation}
Set $\mathcal{P}_{\mathsf{H}^2}^2(\X)$ as the collection of all $(\mu,\nu) \in \mathcal{P}(\X) \times \mathcal{P}(\X)$ such that $\mu \ll \nu$,  and  
\begin{equation}
\mathcal{P}^2_{\mathsf{H}^2}(M,\X):= \left\{  (\mu,\nu) \in \cP_{\mathsf{H}^2}(\X): c_{\mathsf{KB}}^\star(f_{\mathsf{H}^2},\X) \vee  \norm{\frac{\dd \mu}{\dd \nu}}_{\infty,\eta}\leq M \right\}. \label{Helingdistcls}
\end{equation}
Also, let $\tilde{\cG}_{k,t}^*(\phi):=\tilde{\cG}_{k,t}\big(1, 1,1,0,\phi\big)$, and  for $a \geq 0$, define
\begin{align}
\tilde{\cG}_{k,t}^{\mathsf{R}}(a):=\tilde{\cG}_{k,t}\big(1, 2k^{-1}a,a,a,\phi_{\mathsf{R}}\big). \label{Helfnclasserr}
\end{align}
 The next theorem establishes consistency of the NE and bounds its effective absolute-error.
\begin{theorem}[Squared Hellinger distance neural estimation]\label{strongconshel}
The following hold:
\begin{enumerate}[label = (\roman*),leftmargin=15 pt]
     \item If $(\mu,\nu) \in \mathcal{P}_{\mathsf{H}^2}^2(\X)$ is such that $f_{\mathsf{H}^2} \in \mathsf{C}\left(\X\right)$ and there exists $M>0$ such that  
     $\norm{\dd \mu/\dd \nu}_{\infty,\eta}$ $ \leq M$, then, for any $0<\rho<1$, $(k_n)_{n \in \NN}$
     with $ k_n \rightarrow \infty$,
     $k_n =O\left(n^{(1-\rho)/3}\right)$ and $\cG_n=\tilde{\cG}^*_{k_n,M^{-1/2}}(\phi)$, we have
\begin{align}
  \hat{\mathsf{H}}^2_{\cG_n}(X^n,Y^n)    \xrightarrow[n\rightarrow \infty]{}  \mathsf{H}^2(\mu,\nu), \quad \mathbb{P}-\mbox{a.s.}  \label{finbndasconhel}
\end{align}
\item For any $M \geq 0$, 
$m_k= \log k$, $t_k=(\log k)^{-1}$, and $ \cG_k=\tilde{\cG}_{k,t_k}^{\mathsf{R}}(m_k)$, we have
  \begin{flalign}
 &\sup_{(\mu,\nu) \in \mathcal{P}^2_{\mathsf{H}^2}(M,\X)}  \mathbb{E}\left[  \abs{\hat{\mathsf{H}}^2_{\cG_k}(X^n,Y^n) -\mathsf{H}^2(\mu,\nu)}\right] \lesssim_{M}  d^{\frac 12}k^{-\frac{1}{2}}\log k +d^{\frac 32}(\log k)^{3} n^{-\frac 12}. \label{helsimperrbnd}&&
 \end{flalign} 
   \end{enumerate}
  \end{theorem}
The proof of Theorem \ref{strongconshel} is presented in Appendix \ref{strongconshel-proof}. To prove consistency and  establish effective error bounds for $\mathsf{H}^2$ NE, we use a truncated NN class $\tilde{\mathcal{G}}_{k,t}\left(\mathbf{a},\phi\right)$ that saturates the NN output to $1-t$ for some $t>0$. This is done since $h_{\mathsf{H}^2}(x)$ has a singularity at $x=1$ and the NN outputs must be truncated below 1 so as to satisfy \eqref{maxdergamma} for bounding the empirical estimation error. To get the effective error bounds under this constraint, we take $t=t_k$ for some non-increasing positive sequence $t_k \rightarrow 0$. The bound in \eqref{helsimperrbnd} uses $t_k=(\log k)^{-1}$.

\begin{remark}[Effective error based on $M$] \label{rem:Helsampcomp}
 If $M$ is known when selecting the NN parameters, then for $ \cG_k=\tilde{\cG}_{k,M^{-1/2}}^{\mathsf{R}}(M)$, we have (see \eqref{finbndratehelcase2})
 \begin{flalign}
     & \sup_{(\mu,\nu) \in \mathcal{P}^2_{\mathsf{H}^2}(M,\X)}\mathbb{E}\left[  \abs{\hat{\mathsf{H}}^2_{\cG_k}(X^n,Y^n) -\mathsf{H}^2(\mu,\nu)}\right]   \lesssim_{M}~d^{\frac 12} k^{-\frac{1}{2}}+d^{\frac 32}n^{-\frac 12}.\notag
 \end{flalign}
\end{remark}

Addressing  minimax optimality, we set $k=n$ in the above equation (resp.\eqref{helsimperrbnd})  to attain the parametric (resp. near parametric) rate for the $\mathsf{H}^2$ NE over the class $\mathcal{P}^2_{\mathsf{H}^2}(M,\X)$. Let $\tilde{\cG}_{k,t}^{\mathsf{S}}(a):=\tilde{\cG}_{k,t}\big(k^{1/2}\log k, 2k^{-1}a,a,0,\phi_{\mathsf{S}}\big) $, and $\mathcal{P}^2_{\mathsf{H}^2,\mathsf{B}}(M,\X)$ denote the class of distribution pairs with $c_{\mathsf{KB}}^\star(f_{\mathsf{H}^2},\X)$ in \eqref{Helingdistcls} replaced with $c_{\mathsf{B}}^\star(f_{\mathsf{H}^2},\X)$. 
\begin{corollary}[Minimax optimality] \label{hsqminimaxcor}
The $\mathsf{H}^2$ NE $\hat{\mathsf{H}}^2_{\cG_n}(X^n,Y^n)$ is  minimax rate-optimal over the class $\mathcal{P}^2_{\mathsf{H}^2}(M,\X)$ and $\mathcal{P}^2_{\mathsf{H}^2,\mathsf{B}}(M,\X)$ with $ \cG_n= \tilde{\cG}_{n,M^{-1/2}}^{\mathsf{R}}(M)$ and $ \cG_n=\tilde{\cG}_{n,M^{-1/2}}^{\mathsf{S}}(M)$, respectively,  achieving  $ O\big(n^{-1/2}\big)$ minimax risk. Further, relevant to the case when $M$ is unknown, the same  NE achieves $\tilde O\big(n^{-1/2}\big)$ risk, when $M$ is replaced with  $\log n$ and $M^{-1/2}$ with $(\log n)^{-1}$.
\end{corollary}
Below, we provide a sufficient condition in terms of densities under which the effective error bounds in Theorem \ref{strongconshel} as well as Corollary \ref{hsqminimaxcor} applies, similar in spirit to Proposition \ref{prop:distconddirect} (see Appendix \ref{prop:distconddirect-hel-proof} for the  proof). 
\begin{proposition}[Sufficient condition for Theorem \ref{strongconshel}] \label{prop:distconddirect-hel}
 For  $b\geq 0$  and $s_{\mathsf{KB}}=\lfloor d/2\rfloor+3$, consider the class $\widetilde{\cP}^2_{\mathsf{H}^2}(b,\X)$ of pairs of distributions given by 
\begin{equation}
\widetilde{\cP}^2_{\mathsf{H}^2}(b,\X):=\left\{(\mu,\nu) \in \mathcal{P}_{\mathsf{H}^2}^2(\X):\begin{aligned}&\exists~  \tilde p,\tilde q \in \mathsf{C}_b^{s_{\mathsf{KB}}}(\Ucal)\mbox{ for some  open set }\Ucal \supset \X \\& \mbox{s.t. }~p^{-\frac 12}=\tilde p|_\X,~q^{\frac 12}=\tilde q|_\X, \mbox{ and } \norm{p \vee q^{-1}}_{\infty,\eta}\leq b\end{aligned}\right\}. \notag
\end{equation}
Then, \eqref{helsimperrbnd} and Remark \ref{rem:Helsampcomp}  hold with  $M=\big(\kappa_d d^{\frac 32}\norm{\X} \vee 1\big)\big(1+2^{s_{\mathsf{KB}}} \bar c^2_{b,d,\norm{\X}}\big) $ $\vee\mspace{2 mu} b^2$, where $\bar c_{b,d,\norm{\cX}}$ and $\kappa_d$ are given in \eqref{constapproxhold} and \eqref{constkappa}, respectively, and $\widetilde{\cP}^2_{\mathsf{H}^2}(b,\X)$ in place of $\mathcal{P}^2_{\mathsf{H}^2}(M,\X)$.
\end{proposition}
\begin{remark}[Feasible distributions] $\widetilde{\cP}^2_{\mathsf{H}^2}(\cdot,\X)$ includes $(\mu,\nu) \in \mathcal{P}_{\mathsf{H}^2}^2(\X)$,  whose densities $p,q,$ are bounded (from above and away from zero) on $\X$ with an extension that is sufficiently smooth on an open set covering $\X$. This contains the distributions mentioned in Remark~\ref{kldivclassdist}.
\end{remark}
\subsection{Total Variation Distance}
Consider the NN class obtained by truncating the functions in $\mathcal{G}_{k}(\mathbf{a},\phi)$ to $[-1,1]$, i.e.,
\begin{equation}
\bar{\mathcal{G}}_{k}(\mathbf{a},\phi):=\left\{g:~g(x)= \ind_{\{\abs{\tilde g(x)}\leq 1\}}\tilde g(x)+\ind_{\{\tilde g(x)>1\}}-\ind_{\{\tilde g(x)<-1\}} \mbox{ for some }\tilde g \in \mathcal{G}_{k}(\mathbf{a},\phi)\right\}.  \label{stepNNftv} 
\end{equation}
Also, let $\tvf_{\bar{\mathcal{G}}_{k}(\mathbf{a},\phi)}(X^n,Y^n):=\hat{\mathsf{D}}_{h_{ \mathsf{TV}},\bar{\mathcal{G}}_{k}(\mathbf{a},\phi)}(X^n,Y^n)$, and 
set $\bar{\cG}_k^{\mathsf{R}}(a):=\bar{\cG}_k\big(1, 2k^{-1}a,a,a,\phi_{\mathsf{R}}\big)$, 
and  $\bar{\cG}_k^*(\phi):=\bar{\cG}_k\big(1, 1,1,0,\phi\big)$.  Denote the densities of $\mu$ and $\nu$ w.r.t. $\lambda$ by $p$ and $q$, respectively, and for $M\geq 0$ define
\begin{equation}
 \cP_{\mathsf{TV}}^2(M,\X):=\left\{(\mu,\nu) \in \mathcal{P}(\X) \times \mathcal{P}(\X):\mu,\nu \ll \lambda, ~\norm{p \vee q}_{\infty,\X} \leq M \right\}. \label{tvsetdistr}
\end{equation}
The following theorem bounds the effective error for TV distance neural estimation. 
\begin{theorem}[TV distance neural estimation] \label{TVerrbnd}
The following hold:
\begin{enumerate}[label = (\roman*),leftmargin=15 pt]
\item  
For any $\mu,\nu\in\cP(\cX)$, $0<\rho<1$,  $(k_n)_{n \in \NN}$ with $k_n \rightarrow  \infty$,  $k_n=O\big(n^{(1-\rho)/2}\big)$ and $\cG_n=\bar{\cG}_{k_n}^*(\phi)$, we have
 \begin{align}
      \tvf_{\cG_n}(X^n,Y^n)  \xrightarrow[n\rightarrow \infty]{}  \tv{\mu}{\nu}, \quad \mathbb{P}-\mbox{a.s.}  \label{finbndascontv}
 \end{align}
\item For any $0<s\leq 1$, $M \geq 0$,   $\tilde c_{k,d,s,M,\norm{\X}}=O_{d,M,\norm{\cX}}\big(k^{(d+2)/2(s+d+2)}\big)$ as defined in \eqref{consNNparTV} and $\cG_k =\bar{\cG}_k^{\mathsf{R}}\big(\tilde c_{k,d,s,M,\norm{\X}}\big)$, we have
 \begin{flalign}
  \sup_{\substack{(\mu,\nu) \in\cP_{\mathsf{TV}}^2(M,\X):\\f_{\mathsf{TV}} \in \mathsf{Lip}_{s,1,M}(\X)}} \mathbb{E}\left[  \abs{\tvf_{\cG_k}(X^n,Y^n) -\tv{\mu}{\nu}}\right]  \lesssim_{d,M,s}  k^{-\frac{s}{2(s+d+2)}} +k^{\frac{d+2}{2(s+d+2)}}n^{-\frac 12}. \label{finbnderrratetv} &&
 \end{flalign}
  \end{enumerate}
 \end{theorem}
The proof of Theorem \ref{TVerrbnd} is provided in Appendix \ref{TVerrbnd-proof}. A key technical challenge arises from the fact that $ f_{\mathsf{TV}}= \ind_{\mathcal{C}^*}-\ind_{\cX \setminus \mathcal{C}^*}$ (see \eqref{setdefntvopt}) contains step discontinuities in its domain, and hence, it does not belong to the Klusowski-Barron class. Consequently, Theorem~\ref{THM:approximation} is not directly applicable for bounding the approximation error as was done for the SDs considered until now. To overcome this issue,  we apply a Gaussian smoothing kernel to $ f_{\mathsf{TV}}$ so that the smoothed version belongs to the Klusowski-Barron class. The width of the kernel is then adjusted as a function of $k$ such that $L^1$ norm of the difference between $ f_{\mathsf{TV}}$ and its smoothed version decreases as $k$ increases. The need for the  smoothing operation results in a slower approximation and empirical estimation error rate that depends on $d$. 
\begin{remark}[Curse of dimensionality]
Setting $k=n$ in \eqref{finbnderrratetv}, we achieve the effective error rate $O\big(n^{-s/2(s+d+2)}\big)$. Note that this rate suffers from CoD, different from NEs of other SDs considered above where the parametric rate is achieved.
\end{remark}
In practice, the condition $ f_{\mathsf{TV}} \in \mathsf{Lip}_{s,1,M}(\X)$  required  for \eqref{finbnderrratetv}  may be hard to verify. 
A simple sufficient condition in terms of the densities of $\mu$ and $\nu$ is given below. To state it, we need the following definition.

\begin{definition}[Critical zero]
Given $f:\X \rightarrow \RR$, a point $x_0\in \X$ is called a critical zero of $f$ if $f(x_0)=0$ and every  neighbourhood $\Ucal_{x_0}$  of $x_0$ contains an $x \in  \Ucal_{x_0} \cap \X $ such that $f(x) \neq 0$. In particular, if $f(x_0)=0$ and $f$ is differentiable at $x_0$ with derivative $f'(x_0)>0$, then $x_0$ is a critical zero. Let $\mathcal{Z}(f)$ denote the set of critical zeros of $f$. 
 \end{definition}

Based on the above, for $N \in \NN$ and $b\geq 0$, define
\begin{align}
    \mathcal{T}_{b,N}(\X):=\big\{f:\X \rightarrow \RR: \|x-x'\| \geq b, \forall~ x,x' \in \mathcal{Z}(f),\abs{\mathcal{Z}(f)} \leq N\big\}, \label{critzerodefset}
\end{align}
as the class of functions on $\X$ with at most $N$ critical zeros at pairwise (Euclidean) distance of at least $b$ from each other. We are now ready to state the sufficient condition for TV distance estimation; see Appendix \ref{TVpropsuffcond-proof} for proof.

\begin{proposition}[Sufficient condition for Theorem \ref{TVerrbnd}] \label{TVpropsuffcond}
For $N \in \NN$ and $b\geq 0$, consider the class
\begin{equation}
\widetilde{\cP}^2_{\mathsf{TV}}(b,N,\X):=\left\{(\mu,\nu) \in \cP_{\mathsf{TV}}^2(b,\X):\exists~  f \in  \mathcal{T}_{b,N}(\X)  \mbox{ s.t. }~p-q=f \right\}. \notag
\end{equation}
Then, for any $0<s \leq 1$, \eqref{finbnderrratetv}  holds with  $M=\lambda(\X)+\big( 2 b^{-s} \lambda(\X) \vee 2N\pi^{d/2}b^{d-s}\Gamma(d/2+1)^{-1} \big)$  
and supremum over $(\mu,\nu) \in \widetilde{\cP}^2_{\mathsf{TV}}(b,N,\X)$ in place of that over $(\mu,\nu) \in\cP_{\mathsf{TV}}^2(M,\X)$,  
where $\lambda(\X)$ is the Lebesgue measure of $\X$ and $\Gamma$ is the gamma function.
\end{proposition}

\begin{remark}[Feasible distributions]
The set $\widetilde{\cP}^2_{\mathsf{TV}}(\cdot,\cdot,\X)$ includes generalized Gaussian distributions, Gaussian mixtures,  exponential families, Cauchy distributions, etc., truncated and normalized to be supported on $\X$. It also includes distributions whose densities are analytic functions, e.g., non-negative polynomials on $\X$. These inclusions are easy to verify since $p-q$ has finitely many separated critical zeros for such distributions (cf., e.g., \citep{Smale-1986,Kalantari-2004} for the case of analytic functions). 
\end{remark}

\section{Neural Estimation for Distributions with Unbounded Support} \label{sec:extncsupp}
Thus far, we considered compactly supported $\mu$ and $\nu$. In this section, we consider neural estimation of KL, $\chi^2$, $\mathsf{H}^2$ and $\mathsf{TV}$  with $\mu,\nu \in \cP\big(\RR^d\big)$. Throughout, unless stated otherwise, we will assume that $\mu,\nu \ll \lambda$ with $p,q$ denoting the respective Lebesgue densities. For each SD, we first prove consistency of the NE under certain regularity conditions on the densities. Then, we present  effective error bounds under an  Orlicz norm constraint on the densities, which are subsequently specialized to multivariate  Gaussian distributions. 
We next introduce the required definitions below.
\begin{definition}[Orlicz space]\label{def:Orliczspace}
An  increasing 
convex function $\psi:[0,\infty) \rightarrow [0,\infty)$ with $\psi(0)=0$ and $\lim_{x \rightarrow \infty} \psi(x)=\infty$ is called an Orlicz function. For a given $\psi$ and  $M\geq 0$, the bounded Orlicz space$\mspace{2  mu}$\footnote{It is possible to generalize the results in this section to $\mu,\nu \ll \gamma$, where $\gamma$ is an arbitrary positive $\sigma$-finite Borel measure. Accordingly, the Orlicz norm in Definition \ref{def:Orliczspace} is replaced with $\norm{f}_{\psi,\gamma}:=\inf \left\{c\in [0,\infty]:\int_{\RR^d} \psi\big(\norm{x}/{c}\big)f(x)d\gamma(x) \leq 1 \right\}.$ We adopt the current definition for simplicity.}  is 
\begin{align}
    L_{\psi}(M)=\left\{f:\RR^d \rightarrow \RR: \norm{f}_{\psi} \leq M\right\},\notag
\end{align}
where $\norm{f}_{\psi}:=\inf \left\{c>0:\int_{\RR^d} \psi\left(\norm{x}/c\right)f(x)\dd x \leq 1 \right\}$.
\end{definition}
Examples of Orlicz functions include $\hat{\psi}_{r}(z)=z^r$ and  $\psi_{r}(z)=e^{z^r}-1$, $z \in \RR$, for   $r \geq 1$; in particular, $\psi_{r}$ with $r=2$ correspond to the sub-Gaussian class defined next.
\begin{definition}[Sub-Gaussian distribution]\label{def:subgaussvec}
A distribution $\mu \in \cP\big(\RR^d\big)$ is $\sigma^2$-sub-Gaussian for $\sigma>0$ if $X \sim \mu$ satisfies
\begin{align}
    \EE \left[e^{u \cdot(X-\EE[X])}\right] \leq e^{\frac{\sigma^2\norm{u}^2}{2}},\quad \forall~ u \in \RR^d.\notag
\end{align}
For $M \geq 0$, let $\mathcal{SG}(M)$ be the set of all $\sigma^2$-sub-Gaussian distributions with $\sigma^2 \vee \norm{\EE[X]} \leq M$.
\end{definition}

With some abuse of notation, we henceforth use boldface letters to denote infinite sequences, e.g., $\mathbf{v}=(v_k)_{k\in\NN}$; this will simplify some of the subsequent notation. In particular, we use $\mathbf{r}=(r_k)_{k \in \NN}$ for an increasing positive divergent sequence (i.e., $r_k \rightarrow \infty$) with $r_k \geq 1$, and $\mathbf{m}=(m_k)_{k \in \NN}$ for a non-decreasing positive sequence with $m_k \geq 1$. Let $\hat \cG_k(\mathbf{a},\phi,r):=\big\{g \ind_{B_d(r)}: g\in \cG_k(\mathbf{a},\phi)\big\}$, $\hat \cG_k^{\mathsf{R}}(a,r):=\big\{g \ind_{B_d(r)}: g\in \cG_k^{\mathsf{R}}(a)\big\}$, and $\hat \cG_k^*(\phi,r):=\big\{g \ind_{B_d(r)}: g\in \cG_k^*(\phi)\big\}$ denote the NN classes $\cG_k(\mathbf{a},\phi)$, $\cG_k^{\mathsf{R}}(a)$,  and $\cG_k^*(\phi)$, respectively, after nullifying the functions outside of $B_d(r)$. 
\subsection{KL Divergence} \label{kldivncsupp}
For $M \geq 0$, $\ell \in \NN$,  $\mathbf{r}$ and  $\mathbf{m}$ as above, consider the following class of  distributions: 
\begin{align}
\bar{\cP}^2_{\mathsf{KL},\psi}\left(M,\ell,\mathbf{r},\mathbf{m}\right)&:=\left\{(\mu,\nu) \in \mathcal{P}^2_{\mathsf{KL}}\big(\RR^d\big)  :\begin{aligned}  &\mu,\nu \ll \lambda,~p,q \in  L_{\psi}(M),~ \norm{f_{\mathsf{KL}}}_{\ell,\mu} \leq M, \\
& c_{\mathsf{KB}}^\star\left(f_{\mathsf{KL}}|_{B_d(r_k)},B_d(r_k)\right)  \leq m_k, ~k \in \NN \end{aligned}\right\}. \notag 
\end{align}
In words, the class above contains pairs of distributions whose (i) densities have a $\psi$-Orlicz norm bounded by $M$, (ii) $f_{\mathsf{KL}}$ has $L^\ell(\mu)$ norm at most $M$, and (iii) the restriction of $f_{\mathsf{KL}}$ to $B_d(r_k)$ has a Klusowski-Barron coefficient that is at most $m_k$.

The following is the counterpart of Theorem \ref{strongcons} for distributions supported on $\RR^d$; the proof is provided in Appendix \ref{Thm:KLNCsupp-proof}.
\begin{theorem}[KL divergence neural estimation]\label{Thm:KLNCsupp}
 For any $0<\rho<1$, the following hold:
 \begin{enumerate}[label = (\roman*),leftmargin=15 pt] 
 \item 
 Let  $(\mu,\nu) \in \mathcal{P}^2_{\mathsf{KL}}(\RR^d)$ be such that $f_{\mathsf{KL}} \in \mathsf{C}\left(\RR^d\right)$ and  $\big\|f_{\mathsf{KL}}\big\|_{1,\mu}  <\infty$. Then, for $k_n,r_n,n$ satisfying  $k_n \rightarrow  \infty$, $r_n \rightarrow  \infty$,  $k_n^{3/2}r_ne^{k_n(r_n+1)} =O \left(n^{(1-\rho)/2}\right)$ and $\cG_n=\hat{\cG}_{k_n}^*(\phi,r_n)$,
 \begin{equation}
 \hat{\mathsf{D}}_{\cG_n}(X^n,Y^n) \xrightarrow[n \rightarrow \infty]{} \kl{\mu}{\nu},\quad \mathbb{P}-\mbox{a.s.}  \notag
 \end{equation}
 \item Let  $\ell>1$, $M \geq 0$,  $\ell^*=\ell/(\ell-1)$, and $\mathbf{m}$  be such that  $1 \leq m_k \lesssim k^{(1-\rho)/2}$. Then, for $\cG_k=\hat{\mathcal{G}}_{k}^{\mathsf{R}}\left(m_k, r_k\right)$, we have
 \begin{flalign}
& \sup_{(\mu,\nu) \in \bar{\cP}^2_{\mathsf{KL},\psi}\left(M,\ell,\mathbf{r},\mathbf{m}\right)} \mathbb{E}\left[  \abs{\hat{\mathsf{D}}_{\cG_k}(X^n,Y^n)-\kl{\mu}{\nu}}\right]  \notag \\
& \qquad \qquad \qquad \lesssim_{d,M,\rho,\psi,\ell} 
 m_k  k^{-\frac 12}  + m_kr_ke^{3m_k(r_k+1)}~ n^{-\frac 12}+\big(\psi\big(r_kM^{-1}\big)\big)^{\frac {-1}{\ell^*}}. \label{errbndklncsup}&&
 \end{flalign}
\end{enumerate}
\end{theorem}
The proof of the consistency claim in Part $(i)$  follows similar to \eqref{finbndascon} by using the universal approximation property of $\hat{\cG}_{k_n}^*(\phi,r_n)$ on Euclidean balls, controlling the residual approximation error via integrability assumption on $f_{\mathsf{KL}}$, and  using  Theorem \ref{empesterrbnd} to bound the empirical estimation error. The proof of \eqref{errbndklncsup} is based on the following observations.  
 First, we note that if $c_{\mathsf{KB}}^\star\big(f_{\mathsf{KL}}|_{B_d(r_k)},B_d(r_k)\big)$ can be bounded for every $k$, then Theorem \ref{THM:approximation} implies that the NN class $ \hat{\mathcal{G}}_{k}^{\mathsf{R}}(m_k, r_k)$  with $m_k,r_k \rightarrow \infty$ at an appropriate rate  can approximate $f_{\mathsf{KL}}$  to within an error of $ \lesssim d^{1/2} m_k k^{-1/2}$  inside the Euclidean ball $B_d(r_k)$. An upper bound on $c_{\mathsf{KB}}^\star\big(f_{\mathsf{KL}}|_{B_d(r_k)},B_d(r_k)\big)$ is guaranteed, for instance, by Proposition \ref{prop:bndfourcoeff} when $f_{\mathsf{KL}}$ is sufficiently smooth on $B_d(r_k)$. Moreover, since every Borel probability measure on $\RR^d$ is tight, $\mu\big(B_d^c(r_k)\big) \vee \nu\big(B_d^c(r_k)\big) \rightarrow 0$ for every $r_k \rightarrow \infty$. The proof then follows by an analysis of the approximation error outside $B_d(r_k)$ under the Orlicz norm constraint on the densities of $\mu$ and $\nu$, along with an account of the empirical estimation error. The Orlicz norm constraint controls the rate of tail decay  of the densities.
\begin{remark}[Feasible distributions] \label{rem:KLfeasiblenc}
Based on  Proposition \ref{prop:bndfourcoeff}, \eqref{errbndklncsup} holds for distributions $ (\mu,\nu) \in  \mathcal{P}^2_{\mathsf{KL}}\big(\RR^d\big)$, $\mu,\nu \ll \lambda$, such that their densities are sufficiently smooth and bounded (from above and away from zero) on  Euclidean balls $B_d(r)$ for any $r > 0$, and $\norm{f_{\mathsf{KL}}}_{\ell,\mu}$ is finite for some $\ell>1$. This includes multivariate Gaussians, Gaussian mixtures, Cauchy distributions, etc., to name a few.
\end{remark}

As an instance of an explicit effective error bound, we now specialize Theorem \ref{Thm:KLNCsupp} to the important case of Gaussian distributions. Define the class
\begin{align}
\cP^2_{\mathsf{N}}(M)&:=\left\{\big(\cN(\mathsf{m}_p,\Sigma_p),\cN(\mathsf{m}_q,\Sigma_q)\big):\begin{aligned}
    &\norm{\mathsf{m}_p},\norm{\mathsf{m}_q} \leq M\\
    &\mspace{3mu} \|\Sigma_p\|_{\mathsf{op}},\|\Sigma_p^{-1}\|_{\mathsf{op}},\|\Sigma_q\|_{\mathsf{op}},\|\Sigma_q^{-1}\|_{\mathsf{op}} <M \end{aligned} \right\}, \notag
\end{align}
of pairs of non-singular multivariate Gaussian distributions with appropriate bounded operator norm (denoted by $\|\cdot\|_\mathsf{op}$). The following corollary quantifies the effective error for pairs of Gaussian distributions. However, as the proof (see Appendix \ref{cor:klgaussrate-proof}) requires a tedious  evaluation of a bound on the Klusowski-Barron coefficient, we restrict attention to isotropic Gaussians, i.e., whose covariance matrix is $\Sigma=\sigma^2 \mathrm{I}_d$, for some $\sigma>0$. The (sub)class of isotropic Gaussian measures is denoted by $\bar{\cP}^2_{\mathsf{N}}(M)$. Nevertheless, we stress that the argument can be generalized to account for the entire $\cP^2_{\mathsf{N}}(M)$ class above.
\begin{corollary}[Gaussian effective error] \label{cor:klgaussrate}
For any $1 <M < \infty$, there exists $c_{d,M}>0$ such that for $m_k\asymp_{d,M} (\log k)^{0.5(d+3)}$, $r_k:= 1 \vee M+\tilde r_k$,  $ \tilde  r_k \asymp_{d,M}
\sqrt{\log k}$ and $\cG_k=\hat{\mathcal{G}}_{k}^{\mathsf{R}}\left(m_k, r_k\right)$,
 \begin{flalign}
 \sup_{\substack{(\mu,\nu) \mspace{1 mu}\in  \bar{\cP}^2_{\mathsf{N}}(M)}}\mathbb{E}\mspace{-2 mu}\left[  \abs{\hat{\mathsf{D}}_{\cG_k}\mspace{-2 mu}(X^n,Y^n) \mspace{-3 mu}-\mspace{-3 mu}\kl{\mu}{\nu}}\right] \mspace{-3 mu}&\lesssim_{d,M}\mspace{-2 mu}
   (\log k)^{\frac{d+4}{2}} \Big( k^{-\frac 12} +k^{c_{d,M}(\log k)^{\frac{d+2}{2}}} n^{-\frac 12}\Big). \notag
 \end{flalign}
\end{corollary}
\begin{remark}[Gaussian error rate] \label{gausserrrateKL}
Optimizing over $k$ in the above equation yields an effective error rate of 
$n^{-(\log n)^{c_{d,M}}}\log n $ for some $c_{d,M}>-1$. Despite the dependence of this rate on $d$, in Appendix \ref{app:mildcodkl} we show that for certain classes of sub-Gaussian distributions, a NE effective error rate of $n^{-1/3}$  can be achieved independent of dimension. This is to stress that the NE can produce dimension-free convergence rates even when supports are unbounded. 
 \end{remark}

\subsection{$\chi^2$ Divergence}
We next consider $\chi^2$ divergence. Consider the following class of  distributions:
\begin{align}
&\bar{\cP}^2_{\chi^2,\psi}\left(M,\ell,\mathbf{r},\mathbf{m}\right):=\left\{(\mu,\nu) \in \mathcal{P}^2_{\chi^2}\big(\RR^d\big)  :\begin{aligned}  &\mu,\nu \ll \lambda,~p,q \in  L_{\psi}(M),~ \norm{f_{\chi^2}}_{\ell,\mu} \leq M, \\
& c_{\mathsf{KB}}^\star\left(f_{\chi^2}|_{B_d(r_k)},B_d(r_k)\right)  \leq m_k, ~k \in \NN \end{aligned}\right\}. \notag
\end{align}

The following theorem  states consistency of the $\chi^2$ NE and bounds the effective error.
\begin{theorem}[$\chi^2$ neural estimation]\label{Thm:chisqNCsupp}
The following hold:
 \begin{enumerate}[label = (\roman*),leftmargin=15 pt] 
\item  Let $(\mu,\nu) \in \mathcal{P}^2_{\chi^2}\big(\RR^d\big)$ satisfy $f_{\chi^2} \in \mathsf{C}\big(\RR^d\big)$ and $\big\|f_{\chi^2}\big\|_{1,\mu} \vee  \big\|h_{\chi^2} \circ f_{\chi^2}\big\|_{1,\nu} <\infty$. Then, for $k_n\ \rightarrow  \infty$, $r_n\ \rightarrow  \infty$, $n$ satisfying $k_n^{5/2}r_n^2= O \left(n^{(1-\rho)/2}\right)$  for some $0<\rho<1$ and $\cG_n=\hat{\cG}_{k_n}^*(\phi,r_n)$, we have
 \begin{equation}
 \hat{\chi}^2_{\cG_n}(X^n,Y^n) \xrightarrow[n \rightarrow \infty]{} \chisq{\mu}{\nu},\quad \mathbb{P}-\mbox{a.s.}  \notag
 \end{equation}
\item  For  any $M \geq 0$, $\ell>1$, $\ell^*=\ell/(\ell-1)$ and $\cG_k =\hat{\mathcal{G}}_{k}^{\mathsf{R}}\left(m_k, r_k\right)$, we have
 \begin{flalign}
  \sup_{\substack{(\mu,\nu) \in  \bar{\cP}^2_{\chi^2,\psi}\left(M,\ell,\mathbf{r},\mathbf{m}\right)}}\mathbb{E}\left[  \abs{\hat{\chi}^2_{\cG_k}(X^n,Y^n)  -\chisq{\mu}{\nu}}\right] &\lesssim_{M,\psi,\ell} m_k^2dk^{-\frac 12}  +d^{\frac 32}m_k^2r_k^2 n^{-\frac 12}\notag \\
 &\qquad \qquad \qquad \qquad \quad +\Big(\psi\big(r_kM^{-1}\big)\Big)^{-\frac {1}{\ell^*}}. \notag
 \end{flalign}
\end{enumerate}
\end{theorem}
The proof of Theorem \ref{Thm:chisqNCsupp} is similar to that of Theorem \ref{Thm:KLNCsupp} and is given in Appendix  \ref{Thm:chisqNCsupp-proof}.
\begin{remark}[Feasible distributions]
 Theorem \ref{Thm:chisqNCsupp} $(ii)$ holds for any distributions $ (\mu,\nu) \in \mathcal{P}^2_{\chi^2}\big(\RR^d\big)$, $\mu,\nu \ll \lambda$, such that their densities are sufficiently smooth and bounded (from above for $p$ and away from zero for $q$) on Euclidean balls,  and $\norm{f_{\chi^2}}_{\ell,\mu}$ is finite for some $\ell>1$. This encompasses the distributions mentioned in Remark \ref{rem:KLfeasiblenc} for certain parameter ranges.
\end{remark}
The  corollary below (see Appendix \ref{cor:chisqgaussrate-proof} for proof) provides effective error bounds for the following class of Gaussian distributions: 
\begin{align}
     \bar{\cP}^2_{\chi^2,\mathsf{N}}(M):=\left\{\left(\cN(\mathsf{m}_p,\sigma_p^2 \mathrm{I}_d),\cN(\mathsf{m}_q,\sigma_q^2 \mathrm{I}_d)\right):\begin{aligned}  &1/M <\sigma_p^2<2\sigma_q^2 <M, \\
& 2\sigma_q^2-\sigma_p^2>1/M,~\norm{\mathsf{m}_p} \vee \norm{\mathsf{m}_q} \leq M \end{aligned}\right\},\notag
\end{align}
where the constraint $\sigma_p^2<2\sigma_q^2$ is required for $\chisq{\mu}{\nu}$ to be finite.
\begin{corollary}[Gaussian effective error]\label{cor:chisqgaussrate}
For  $1 <M < \infty$, we have with $r_k= 1 \vee M+\tilde r_k$,  $\tilde r_k\asymp_M \sqrt{\log k}$, $m_k\asymp_{d,M}k^{2M^5/(4M^5+1)}(\log k)^{0.5(s_{\mathsf{KB}}+d+1)}$, and $\cG_k=\hat{\mathcal{G}}_{k}^{\mathsf{R}}\left(m_k, r_k\right)$ that 
\begin{flalign}
& \sup_{\substack{(\mu,\nu)\mspace{1 mu} \in  \bar{\cP}^2_{\chi^2,\mathsf{N}}(M)}} \mathbb{E}\left[  \abs{\hat{\chi}^2_{\cG_k}(X^n,Y^n) -\chisq{\mu}{\nu}}\right]  \lesssim_{d,M} (\log k)^{ 2(s_{\mathsf{KB}}+ d+1)}\Big(k^{-\frac{1}{2+8M^5}}+ k^{\frac{4M^5}{1+4M^5}}n^{-\frac{1}{2}}\Big). \notag
 \end{flalign}
\end{corollary}
\begin{remark}[Gaussian error rate]
The optimum in the right hand side (RHS) of the equation above  over $(k,n)$ is attained at  
 $k=n^{(1+4M^5)/(1+8M^5)}$, and results in an effective error rate of  $n^{-1/(2+16M^5)}$ $ (\log n)^{2(s_{\mathsf{KB}}+d+1)}$. Note that this rate degrades with increasing $d$ or $M$. Nevertheless, in Proposition \ref{chisqmildCoD} in Appendix \ref{mildcodchisq}, we show that a dimension-free improvement of $n^{-1/2}$ (up to logarithmic factors)  can be achieved  for a certain class of sub-Gaussian distributions with unbounded support.
\end{remark}
\subsection{Squared Hellinger Distance}
Next, we consider the squared Hellinger distance. For $M,\mathbf{r},\mathbf{m}$ as above, let
\begin{align}
&\bar{\cP}^2_{\mathsf{H}^2,\psi}\mspace{-3mu}\left(M,\mathbf{r},\mathbf{m}\right)\mspace{-3mu}:=\mspace{-3mu}\left\{\mspace{-3mu}(\mu,\nu) \mspace{-2mu}\in\mspace{-2mu} \mathcal{P}^2_{\mathsf{H}^2}\big(\RR^d\big)  \mspace{-2 mu}:\begin{aligned}  &\mu,\nu \ll \lambda,~p,q \in  L_{\psi}(M),\\
&c_{\mathsf{KB}}^\star\left(f_{\mathsf{H}^2}|_{B_d(r_k)},B_d(r_k)\right)\mspace{-2mu}\vee \mspace{-2mu}\norm{\frac{\dd \mu}{\dd \nu}}_{\infty,B_d(r_k)}  \mspace{-15mu}\leq m_k,~\forall\, k \in \NN \end{aligned}\right\}\mspace{-4mu}. \notag 
\end{align}
Also, consider the following NN class obtained from $\tilde{\cG}_{k,t}^{\mathsf{R}}(\cdot)$ (see \eqref{Helfnclasserr}) by  nullifying the functions outside of $B_d(r)$:
\begin{equation}
    \check{\cG}^{\mspace{2 mu}\mathsf{R}}_{k,t}(a,r):=\Big\{g \ind_{B_d(r)}:g \in \tilde{\cG}_{k,t}^{\mathsf{R}}(a)\Big\}.   \label{helclassnc} 
\end{equation}
The next
theorem provides conditions under which consistency holds for $\mathsf{H}^2$ neural estimation and bounds the effective error; see Appendix \ref{Thm:helsqNCsupp-proof} for the proof. 
\begin{theorem}[Squared Hellinger distance neural estimation]\label{Thm:helsqNCsupp}
Let  $\mathbf{m}$ satisfy $m_k=o(k^{1/4})$.
The  following hold: 
 \begin{enumerate}[label = (\roman*),leftmargin=15 pt] 
\item  
For $(\mu,\nu) \in \bar{\cP}^2_{\mathsf{H}^2,\psi}\left(M,\mathbf{r},\mathbf{m}\right)$, and  $\mathbf{k}$, $\mathbf{r}$, $\mathbf{m}$, $n$  such that $k_n\ \rightarrow  \infty$, $r_{k_n}\ \rightarrow  \infty$, $m_{k_n}\ \rightarrow  \infty$, $k_n^{1/2}m_{k_n}^2r_{k_n}= O \left(n^{(1-\rho)/2}\right)$ for some $0<\rho<1$ and $\cG_n =\check{\cG}^{\mspace{2 mu}\mathsf{R}}_{k_n,m_{k_n}^{-1/2}}(m_{k_n},r_{k_n})$, we have
 \begin{equation}
 \hat{\mathsf{H}}^2_{\cG_n}(X^n,Y^n) \xrightarrow[n \rightarrow \infty]{} \mathsf{H}^2(\mu,\nu),\quad \mathbb{P}-\mbox{a.s.}  \notag
 \end{equation}
\item  For any $M \geq 0$ and $\cG_k =\check{\cG}^{\mspace{2 mu}\mathsf{R}}_{k,m_k^{-1/2}}(m_k,r_k)$, we obtain
 \begin{flalign}
& \sup_{(\mu,\nu) \in \bar{\cP}^2_{\mathsf{H}^2,\psi}\left(M,\mathbf{r},\mathbf{m}\right)} \mathbb{E}\left[  \abs{\hat{\mathsf{H}}^2_{\cG_k}(X^n,Y^n) -\mathsf{H}^2(\mu,\nu)}\right]    \notag \\
& \qquad \qquad  \qquad \qquad \qquad \qquad \qquad \quad  \lesssim_{M,\psi}  m_k^2 d^{\frac 12} k^{-\frac 12}+d^{\frac 32}m_k^2r_k n^{-\frac 12}+\Big(\psi\big(r_kM^{-1}\big)\Big)^{-\frac 12}. 
\notag
 \end{flalign}
\end{enumerate}
\end{theorem}
The proof of Theorem \ref{Thm:helsqNCsupp} follows along similar lines to Theorem \ref{Thm:KLNCsupp}. Notice that the NN class $ \check{\cG}^{\mspace{2 mu}\mathsf{R}}_{k,t}$  is used to overcome the issue of singularity of $f_{\mathsf{H}^2}$ as in Theorem \ref{strongconshel}.
\begin{remark}[Feasible distributions]
 Theorem \ref{Thm:helsqNCsupp} applies for any distributions $ (\mu,\nu) \in \mathcal{P}^2_{\mathsf{H}^2}\big(\RR^d\big)$, $\mu,\nu \ll \lambda$, such that their densities $p,q$ are sufficiently smooth and  bounded (from above and below) on Euclidean balls. To list a few, this includes multivariate Gaussians, mixture Gaussians, Cauchy distributions, etc.
\end{remark}
The next corollary provides effective error bounds for the class of isotropic Gaussian distributions with bounded parameters, $\bar{\cP}^2_{\mathsf{N}}(M)$, considered in Section \ref{kldivncsupp}; see Appendix \ref{cor:helgaussrate-proof} for the proof. 
\begin{corollary}[Gaussian effective error]\label{cor:helgaussrate}
For $1 <M < \infty$,  $r_k=1 \vee M+(M+8M^2)^{-1/2}(\log k)^{1/2}$, $m_k \asymp_{d,M}k^{2M/(1+8M)} (\log k)^{0.5(s_{\mathsf{KB}}+d+1)}$, and $\cG_k=\check{\cG}^{\mspace{2 mu}\mathsf{R}}_{k,m_k^{-1/2}}(m_k,r_k)$,
\begin{flalign}
     \sup_{\substack{(\mu,\nu) \in   \bar{\cP}^2_{\mathsf{N}}(M)}} \mathbb{E}\left[  \abs{\hat{\mathsf{H}}^2_{\cG_k}(X^n,Y^n) -\mathsf{H}^2(\mu,\nu)}\right]  &\lesssim_{d,M} (\log k)^{s_{\mathsf{KB}}+ d+2} k^{-\frac{1}{2+16M}}\left(1+   k^{\frac 12}n^{-\frac 12}\right). \notag
\end{flalign}
\end{corollary}
\begin{remark}[Gaussian error rate]
Setting $k=n$ in the equation above yields an effective error rate of $n^{-1/(2+16M)}(\log n)^{s_{\mathsf{KB}}+ d+2}$. While this rate deteriorates with $M$ and~$d$, in Proposition \ref{helmildCoD} in Appendix \ref{mildCoDhelnc}, we show that a  rate of $ n^{-1/2}$ (up to logarithmic factors) is possible independent of dimension for a certain class of sub-Gaussian distributions with unbounded support.
\end{remark}
\subsection{TV Distance}
Finally, we consider  neural estimation of TV distance for distributions with unbounded support. For $M \geq 0,s \geq 0,b \geq 0$, $N \in \NN$,  sequences $\mathbf{r}$ and $\mathbf{m}$ as above, let 
\begin{align}
&\bar{\mathcal{P}}^2_{\mathsf{TV,\psi}}(M,s,\mathbf{r},\mathbf{m}):=\left\{(\mu,\nu) \in \cP_{\mathsf{TV}}^2(M,\RR^d):\begin{aligned} &  \mu,\nu \ll \lambda, p,q \in  L_{\psi}(M),\\
&f_{\mathsf{TV}}  \ind_{B_d(r_k)} \in \mathsf{Lip}_{s,1,m_k}(B_d(r_k))\end{aligned}\right\}, \notag \\
&\hat{\mathcal{P}}^2_{\mathsf{TV}}(b, M,N):=\left\{(\mu,\nu) \in \cP_{\mathsf{TV}}^2(M,\RR^d): \mu,\nu \in \mathcal{SG}(M), ~ \exists~  f \in  \mathcal{T}_{b,N}\big(\RR^d\big)  \mbox{ s.t. }p-q=f \right\}, \notag 
\end{align}
 where $\cP_{\mathsf{TV}}^2(M,\RR^d)$ and $\mathcal{T}_{b,N}\big(\RR^d\big) $ are defined in \eqref{tvsetdistr} and \eqref{critzerodefset}, respectively. Also, define the NN classes $\vec{\cG}^{\mspace{3 mu }\mathsf{R}}_k(a,r):=\big\{g \ind_{B_d(r)}:g \in \bar{\cG}^{\mspace{3 mu}\mathsf{R}}_k(a)\big\}$,  and $\vec{\cG}^*_k(\phi,r):=\big\{g \ind_{B_d(r)}:g \in \bar{\cG}^*_k(\phi)\big\}$, where  $\bar{\cG}_k$ is given in \eqref{stepNNftv}. 
 
The next 
theorem is the analogue of Theorem \ref{Thm:KLNCsupp} for TV distance neural estimation. Its proof is presented in  Appendix \ref{Thm:TVerrbndNC-proof}. 
\begin{theorem}[TV distance neural estimation] \label{Thm:TVerrbndNC}
The following hold:
 \begin{enumerate}[label = (\roman*),leftmargin=15 pt] 
 \item 
For $\mu,\nu \in \cP\big(\RR^d\big)$, any $0<\rho<1$, and  $\mathbf{k},\mathbf{r},n$ such that $k_n \rightarrow  \infty$, $r_n \rightarrow \infty$, $k_nr_n^{1/2} = O\big(n^{(1-\rho)/2}\big)$ and $\cG_n=\vec{\cG}^*_{k_n}(\phi,r_n)$, we have
 \begin{align}
      \tvf_{\cG_n}(X^n,Y^n)  \xrightarrow[n\rightarrow \infty]{}  \tv{\mu}{\nu}, \quad \mathbb{P}-\mbox{a.s.}  \notag
 \end{align}  
\item For any $M \geq 0$ and  $0<s \leq 1$, $\cG_k =\vec{\cG}^{\mspace{3 mu }\mathsf{R}}_k\big(\vec{c}_{k,d,s,\mathbf{m},\mathbf{r}},r_k\big)$, where    $\vec{c}_{k,d,s,\mathbf{m},\mathbf{r}}$ is given in \eqref{consttvscnc}, we obtain 
\begin{flalign}
&\sup_{(\mu,\nu) \in \bar{\mathcal{P}}^2_{\mathsf{TV},\psi}(M,s,\mathbf{r},\mathbf{m})}\mathbb{E}\left[  \abs{\tvf_{\cG_k}(X^n,Y^n) -\tv{\mu}{\nu}}\right] \notag \\ &\qquad \qquad ~~ \lesssim_{d,M,s,\rho} \Big(m_k^{d+2}r_k^{s(d+1)}k^{-\frac{s}{2}}\Big)^{\frac{1}{s+d+2}}+n^{-\frac 12}\Big(m_kr_k^{s+1}k^{\frac 12}\Big)^{\frac{d+2}{s+d+2}}+ \psi\big((r_kM^{-1})\big)^{-1}. \notag &&
\end{flalign}
\end{enumerate}
\end{theorem}

The following corollary (see Appendix \ref{cor:tvsubgaussrate-proof} for proof) provides effective error bounds for sub-Gaussian distributions such that $p-q$ has finite number of critical zeros pairwise separated by  Euclidean distance bounded away from zero.
\begin{corollary}[Sub-Gaussian effective error]\label{cor:tvsubgaussrate}
For any $0 <s \leq 1$, $b \geq 0,~ M\geq 0,~N \in \NN$, $r_k= M \vee 1+4 \sqrt{dM\log k}$, $m_k= c_{d,s,b,N,r_k}$ (see \eqref{constvsuffcond}) and $\cG_k =\vec{\cG}^{\mspace{3 mu }\mathsf{R}}_k\big(\vec{c}_{k,d,s,\mathbf{m},\mathbf{r}},r_k\big)$, we have 
   \begin{flalign}
    &  \sup_{(\mu,\nu) \in \hat{\mathcal{P}}^2_{\mathsf{TV}}(b, M,N)} \mathbb{E}\left[  \abs{\tvf_{\cG_k}(X^n,Y^n) -\tv{\mu}{\nu}}\right]  \notag \\
     &  \qquad \qquad \qquad \qquad \qquad \qquad \qquad \lesssim_{d,s,b,N}  (\log k)^{\frac{(s+d)(d+2)}{2(s+d+2)}}k^{\frac{-s}{2(s+d+2)}}+(\log k)^{\frac{d+2}{2}}k^{\frac{d+2}{2(s+d+2)}}n^{-\frac 12}.\notag &&
  \end{flalign}
\end{corollary}
\begin{remark}[Sub-Gaussian error rate]
 Setting $k=n$ in the bound above, the effective error rate is  $n^{-s/2(s+d+2)}(\log n)^{(d+2)/2}$. 
\end{remark}
\begin{remark}[Feasible distributions]
$\hat{\mathcal{P}}^2_{\mathsf{TV}}(\cdot,\cdot,\cdot)$ includes generalized Gaussian distributions, mixture Gaussians, and in general,  distributions pairs with smooth bounded densities having finite number of modes and sub-Gaussian tails. 
\end{remark}

\section{Concluding Remarks} \label{sec:conlcusion}
This paper studied  neural estimation of SDs, aiming to characterize the performance of NEs via an approximation-estimation error analysis. We showed that NEs of $\mathsf{f}$-divergences, such as the KL and $\chi^2$ divergences, squared Hellinger distance, and TV distance are consistent, provided the appropriate scaling of the NN size $k$ with the sample size $n$. We further derived non-asymptotic absolute-error upper bounds that quantify the dependence on $k$ and $n$. In the compactly supported case, the derived bounds enabled to establish the minimax optimality of NEs for KL divergence, $\chi^2$ divergence, and $\mathsf{H}^2$ distance.
The key results leading to these bounds are Theorems \ref{THM:approximation} and \ref{thm:optkdepNN}, which, respectively, bound the sup-norm approximation error by NNs and the empirical estimation error of the parametrized~SD. Our theory covers distributions whose densities belong to an appropriate Orlicz class (e.g., sub-Gaussian distributions), but faster  (optimal) parametric rates are attained when supports are compact.

Going forward, we aim to extend our results to additional SDs such as Wasserstein distances and IPMs. While our analysis strategy extends to these examples, new approximation bounds for the appropriate function classes (e.g., 1-Lipschitz) are needed. Generalizing our results to NEs based on deep nets is another natural direction. Recent results on the approximation capabilities of DNNs \citep[e.g.,][]{YAROTSKY-2017} appears useful for this purpose. While our analysis does not account for the optimization error, this is another important component of the overall error and we plan to examine it in the future. Through the results herein and the said future directions, we hope to couple neural estimators with the theory to guarantee their performance and/or elucidate their limitations.

\acks{The work of S. Sreekumar is supported by the
TRIPODS Center for Data Science National Science Foundation Grant CCF-1740822. Z. Goldfeld is supported by the NSF CRII Grant CCF-1947801, the NSF CAREER Award under Grant CCF-2046018, and the 2020 IBM Academic Award. We thank Prof. Kengo Kato for many useful suggestions that helped us improve the manuscript. We also thank Zhengxin Zhang for his involvement in an earlier version of this work.}

 \appendix

\section{Proofs} \label{sec:proofs}
This section contains proofs of the results presented in Section \ref{Sec:prelim}-\ref{sec:extncsupp}, each given in a different subsection. For fluidity, derivations of lemmas used in those proofs are relegated to Appendix \ref{Sec:prooflemmas}.

We first state an auxiliary result which will be useful in several proofs that follow. For $b \geq 0$ and an integer $s \geq 0$, define the function classes:
\begin{flalign}
  \cL^{\mathsf{KB}}_{s,b}\big(\RR^d\big)\mspace{-5 mu}:=\mspace{-5 mu}\left\{\mspace{-3 mu}f \mspace{-3.5 mu}\in \mspace{-3.5 mu}  L^1\big(\RR^d\big) \mspace{-3 mu} \cap \mspace{-2 mu}L^2\big(\RR^d\big)\mspace{-3 mu}: \mspace{-5 mu}\begin{aligned}\mspace{-4 mu} &\abs{f(0)} \vee \norm{\nabla f(0)}_1\mspace{-2 mu}\leq \mspace{-2 mu}b, \big\|D^{\tilde \alpha}f\big\|_1\mspace{-5 mu}<\mspace{-5 mu}\infty, \forall~ \|\tilde \alpha\|_1\mspace{-3 mu}\leq s\\
  & \norm{D^{\alpha}f}_2 \leq b, \forall~\norm{\alpha}_1 \in \{2,s\}  \mspace{-5 mu}\end{aligned}\right\},  \label{squareintclassklu} &&
\end{flalign}
\begin{flalign}
  \cL^{\mathsf{B}}_{s,b}\big(\RR^d\big):=\left\{f \in   L^1\big(\RR^d\big)  \cap L^2\big(\RR^d\big): \begin{aligned} &\abs{f(0)} \leq b, \big\|D^{\tilde \alpha}f\big\|_1<\infty, \forall~ \|\tilde \alpha\|_1\leq s\\
  & \norm{D^{\alpha}f}_2 \leq b, \forall~\norm{\alpha}_1 \in \{1,s\}  \end{aligned}\right\}.  \label{squareintclassbar} &&
\end{flalign}
The next lemma states that functions in  $\cL^{\mathsf{KB}}_{s,b}\big(\RR^d\big)$ (resp.  $\cL^{\mathsf{B}}_{s,b}\big(\RR^d\big)$) with sufficient smoothness order $s$ belong to the  Klusowski-Barron (resp. Barron) class. Its proof is given in  Appendix \ref{app:auxlemma} and borrows arguments from \citep{Barron_1993}. 
\begin{lemma}[Smoothness and Klusowksi-Barron class]
\label{lem:suffcondbar} 
Recall  $s_{\mathsf{KB}}=\lfloor 0.5 d\rfloor+3$ and $s_{\mathsf{B}}:=\lfloor 0.5 d\rfloor+2$. If $f \in  \cL^{\mathsf{KB}}_{s_{\mathsf{KB}},b}\big(\RR^d\big)$, then we have $S_2(f)  \leq bd^{3/2}\kappa_d $, while if $f \in  \cL_{s_{\mathsf{B}},b}\big(\RR^d\big)$, then  $S_1(f)  \leq b d^{1/2}\kappa_d$ , where 
\begin{equation}
\kappa_d^2:=\big(d+d^{s_{\mathsf{B}}}\big)\int_{\RR^d}\left(1+\norm{\omega}^{2(s_{\mathsf{B}}-1)}\right)^{-1}\dd  \omega<\infty.\label{constkappa}
\end{equation}
Consequently, for $\X \subseteq \RR^d$, $\cL_{s_{\mathsf{KB}},b}\big(\RR^d\big) \subseteq \cB_{c,2,\X}\big(\RR^d\big)$ and $\cL_{s_{\mathsf{B}},b}\big(\RR^d\big) \subseteq \cB_{c,1,\X}\big(\RR^d\big)$ with $c= b \vee bd^{3/2}\kappa_d \norm{\X}$ and $c= b \vee bd^{1/2}\kappa_d \norm{\X}$, respectively.
\end{lemma}

\subsection{Proofs for Section \ref{Sec:prelim}}
\subsubsection{Proof of Theorem \ref{THM:approximation}}\label{supnormapprox-proof}
For $\mathbf{a}=(a_1,a_2,a_3,a_4)$, we denote the set of feasible parameters of $\cG_k(\mathbf{a},\phi)$ by $\Theta_k(\mathbf{a})$, i.e.,  
\begin{equation}
  \Theta_{k}(\mathbf{a}):=\left\{\left(\{\beta_i,w_i,b_i\}_{i=1}^k,w_0,b_0\right):\ \ \begin{aligned}&  w_i \in \mathbb{R}^d,~ b_i,\beta_i \in \mathbb{R},\max_{1 \leq i \leq k}\norm{w_{i}}_1 \vee \abs{b_i} \leq a_1,  \\&
 \max_{1 \leq i \leq k}|\beta_i| \leq a_2,~ \abs{b_0} \leq a_3, \norm{w_0}_1 \leq a_4 \end{aligned} 
 \right\}. \label{paramspace}
\end{equation}
Also, throughout this section, we write $g_{\theta}(x)$ to denote $g(x)=\sum_{i=1}^k \beta_i \phi\left(w_i\cdot x+b_i\right)+w_0 \cdot x+b_0$ with $\theta=\big(\{\beta_i,w_i,b_i\}_{i=1}^k,w_0,b_0\big)$, whenever the underlying $\theta$ is to be emphasized.

\medskip

We prove the second claim in Theorem \ref{THM:approximation}. The proof  relies on arguments from \citep{Barron-1992} and \citep{Barron_1993}, along with the uniform central limit theorem (CLT) for uniformly bounded VC-type classes. Fix an arbitrary (small) $\delta>0$, and let  $\tilde f:\RR^d \rightarrow \RR$ be such that $f=\tilde  f|_\cX$ and $\norm{\cX}S_1(\tilde f) \vee \tilde f(0) \leq a+\delta$. Such an $\tilde f$ exists since  $ c_{\mathsf{B}}^\star(f,\X)\leq a$. Then, since $\X$ is compact, it follows from the proof of \citep[Theorem 2]{Barron_1993} that 
\begin{align}
  \tilde f_0(x):=  \tilde f(x)- \tilde f(0)=\int_{\omega \in \mathbb{R}^d \setminus \{0\}} \varrho(x,\omega)  \gamma(\dd\omega),  \notag
\end{align}
where 
\begin{align}
&\varrho(x,\omega):=\frac{L\big(\tilde f,\X\big)}{\sup_{x \in \X} \abs{\omega \cdot x}}\big(\cos(\omega \cdot x+ \zeta(\omega))-\cos(\zeta(\omega))\big), \notag \\
 & \gamma(\dd\omega):= \frac{\sup_{x \in \X} \abs{\omega \cdot x}\big|\tilde  F\big|(\dd\omega)}{L\big(\tilde f,\X\big)}, \notag 
\end{align}
with $L\big(\tilde f,\X\big):= \int_{\RR^d} \sup_{x \in \X} \abs{\omega \cdot x}\big|\tilde  F\big|(\dd\omega)$. Here $\big|\tilde  F\big|(\dd\omega)$ and  $\zeta(\omega)$ are the  magnitude and phase of the complex Borel measure in the Fourier representation of $\tilde f$, respectively. Note that $\gamma$ defined above is a probability measure on $\RR^d$. 

Let $\tilde{\Theta}:=\tilde{\Theta}\big(k,L\big(\tilde f,\X\big)\big):=\Theta_{1}\big(k^{1/2} \log k,2L\big(\tilde f,\X\big),0,0\big)$ (see \eqref{paramspace}). Then, it further follows from the proofs of \citep[Lemma 2-Lemma 4,Theorem 3]{Barron_1993} that  there exists a  probability measure $\gamma_k \in \tilde \cP_k:= \mathcal{P}\big(\tilde{\Theta}\big)$ \citep[see][Eqns. (28)-(32)]{Barron_1993} such that
\begin{align}
\norm{\tilde  f_0- \int_{\tilde \theta \in \tilde{\Theta} }  g_{\tilde \theta}(\cdot)~   \gamma_k \big(\dd\tilde \theta\big)}_{\infty,\X} \lesssim L\big(\tilde f,\X\big) k^{-\frac 12}, \label{approxcc}
\end{align}
where  $ g_{\tilde \theta}(x)=\tilde \beta \phi_{\mathsf{S}}\big(\tilde w\cdot x+\tilde b\big)$ for $\tilde \theta=(\tilde \beta,\tilde w,\tilde b,0,0)$ and $\phi_{\mathsf{S}}$ is the logistic sigmoid. The previous step needs further elaboration. The claims in \citep[Lemma 2- Lemma 4, ~Theorem 3]{Barron_1993}  are stated for  $L^2$ norm, but it is not hard to see from the proof therein that the same also holds for sup-norm, apart from the following subtlety. In the proof of Lemma 3, it is shown that $\varrho(x,\omega)$, $\omega \in \mathbb{R}^d$, lies in the convex closure of a certain class of step functions, whose  discontinuity points are adjusted to coincide with the continuity points of the underlying measure $\eta$. 
While this can be shown to account for universal approximation under the essential supremum w.r.t. $\eta$, to obtain a sup-norm bound one additional step is needed. Specifically, by using modified step functions whose value  at 0 is 0.5 (instead of 1),  using their linear combinations for approximation of the target function in Lemma 3, and   subsequently replacing each such step function by sigmoids with coinciding values at zero, it can be seen that $\varrho(x,\omega)$ lies in the point-wise closure of convex hull of the desired sigmoid function class.

Next, for each fixed $x$, let  $\upsilon_x:\tilde{\Theta} \rightarrow \RR$ be given by $\upsilon_x(\tilde \theta):=\tilde \beta \phi_{\mathsf{S}}\big(\tilde w\cdot x+\tilde b\big)$ for $\tilde \theta=(\tilde \beta,\tilde w,\tilde b,0,0)$, and consider the function class $\tilde {\mathcal{F}}_k :=\left\{\upsilon_x,~x \in \RR^d\right\}$.   
Note that every $ \upsilon_x \in \tilde {\mathcal{F}}_k$ is a composition of an affine function in $(\tilde w,\tilde b)$ with the bounded monotonic function  $\tilde \beta \phi_{\mathsf{S}} (\cdot)$. Hence,  \citep[Lemma 2.6.15, Lemma 2.6.18]{AVDV-book} yields that $\tilde {\mathcal{F}}_k$  is a VC type class  with index at most $d+3$ for each $k \in \NN$. Hence, it follows  from \citep[Theorem 2.6.7]{AVDV-book} that for every $0<\epsilon \leq 1$,
\begin{align}
    \sup_{\gamma \in \tilde \cP_k} N\left( 2\epsilon L\big(\tilde f,\X\big),\tilde {\mathcal{F}}_k,\|\cdot\|_{2,\gamma}\right) \mspace{-4 mu}\leq \mspace{-4 mu} \sup_{\gamma \in \tilde \cP_{\infty}} N\left( 2\epsilon L\big(\tilde f,\X\big),\tilde {\mathcal{F}}_{\infty},\|\cdot\|_{2,\gamma}\right) \mspace{-4 mu}\lesssim \mspace{-4 mu} (d+3)(16e)^{d+3}\epsilon^{-2(d+2)}. \notag 
\end{align}
Moreover, by
\citep[Theorem 2.8.3]{AVDV-book},  $\tilde {\mathcal{F}}_k$ is a uniform Donsker class (in particular, $\gamma_k$-Donsker) for all probability measures $\gamma \in \tilde \cP_k$. Consequently, the uniform CLT \citep{dudley_1999} applied to a VC-type class uniformly bounded by $2L\big(\tilde f,\X\big)$ yields that there exists  $k$ parameter vectors, $\tilde \theta_i:=( \tilde \beta_i, \tilde w_i, \tilde b_i,0,0) \in \tilde{\Theta},~ 1 \leq i \leq k$,
such that \citep[see also][Theorem 2.1]{Yukich-1995} 
\begin{align}
\Bigg\|\int_{\tilde \theta \in \tilde{\Theta}} g_{\tilde \theta}(\cdot)~   \gamma_k(\dd\tilde \theta)- \frac 1k \sum_{i=1}^k  g_{\tilde \theta_i}(\cdot)\Bigg\|_{\infty,\RR^d} \lesssim d^{\frac 12} L\big(\tilde f,\X\big) k^{-\frac 12}. \label{donskerprop}
\end{align}
The RHS above is independent of $\gamma_k$ and depends on $\tilde  f$ and $\X$ only through $L\big(\tilde f,\X\big)$.

From \eqref{approxcc}-\eqref{donskerprop} and triangle inequality, we obtain
\begin{align}
    \norm{ \tilde f_0- \frac 1k \sum_{i=1}^k  g_{\tilde\theta_i}}_{\infty,\X}  \lesssim  d^{\frac 12} L\big(\tilde f,\X\big)k^{-\frac 12}.\notag
\end{align}
Setting $\theta=\big(\big\{(\tilde \beta_i/k,\tilde w_i,\tilde b_i)\big\}_{i=1}^k,0,\tilde  f(0)\big)$ and $g_{\theta}(x)= k^{-1} \sum_{i=1}^k  \tilde \beta_i \phi_{\mathsf{S}}(\tilde w_i \cdot x+\tilde b_i)+\tilde f(0)$ and noting that $L\big(\tilde f,\X\big) \leq  \norm{\cX} S_1(\tilde f)$ by Cauchy-Schwartz,  we have 
\begin{align}
\norm{\tilde f- g_{\theta}}_{\infty,\X}  \lesssim   d^{\frac 12} \norm{\cX} S_1(\tilde f) k^{-\frac 12}  \leq  d^{\frac 12}(a+\delta)k^{-\frac 12}. \notag
\end{align}
Next, note that $  \big\|\tilde f- g_{\theta}\big\|_{\infty,\X}= \norm{ f- g_{\theta}}_{\infty,\X}$ and $ g_{\theta} \in \mathcal{G}_k^{\mathsf{S}}\big( \norm{\cX} S_1(\tilde f) \vee \tilde f(0)\big)$ $\subseteq \mathcal{G}_k^{\mathsf{S}}\left(a+\delta\right)
$. Since  $\delta>0$ is arbitrary and $\phi_{\mathsf{S}}$ is continuous, we obtain that there exists $ g_{\theta} \in  \mathcal{G}_k^{\mathsf{S}}\left(a\right)
$ with
\begin{align}
\norm{f- g_{\theta}}_{\infty,\X}   \lesssim  ad^{\frac 12}\mspace{1 mu} k^{-\frac 12}. \label{finaapperrbarcls}
\end{align}

\subsubsection{Proof of Proposition \ref{prop:bndfourcoeff}}\label{prop:bndfourcoeff-proof}
To prove the first claim, consider $\tilde f \in  \mathsf{C}_b^{s_{\mathsf{KB}}}(\Ucal)$ such that $f=\tilde f|_\cX$. By Theorem \ref{THM:approximation}, it suffices to show that there exists an extension $f_{\mathsf{ext}}$ of $\tilde f$ from $\Ucal$ to $\RR^d$ such that $ \norm{\cX} S_2(f_{\mathsf{ext}}) \vee \abs{f_{\mathsf{ext}}(0)} \vee \norm{\nabla f_{\mathsf{ext}}(0)}_1 \leq \bar c_{b,d,\norm{\cX}} $. 
Let $\alpha_{|j}$ denote a multi-index of order $j$.  
Consider an extension of $D^{\alpha_{|s_{\mathsf{KB}}}}\tilde f$ from $\Ucal$ to $\RR^d$, which is zero outside $\Ucal$. 
Fixing $D^{\alpha_{|s_{\mathsf{KB}}}}\tilde f$ on $\RR^d$ induces an extension 
of all lower  order derivatives $D^{\alpha_{|j}}f,~0 \leq j<s_{\mathsf{KB}}$  to $\RR^d$,  which can be defined recursively as $D^{\alpha_{|1}}D^{\alpha_{|s_{\mathsf{KB}}-j}}\tilde f(x)=D^{\alpha_{|1}+\alpha_{|s_{\mathsf{KB}}-j}}\tilde f(x)$, $x \in \RR^d$, for all $\alpha_{|1}$, $\alpha_{|s_{\mathsf{KB}}-j}$ and  $1\leq j\leq s_{\mathsf{KB}}$. 

Let  $\Ucal':=\big\{x' \in \RR^d: \exists \mspace{2 mu} x\in\cX,\, \norm{x'-x} < 1\big\}$ and first assume the strict inclusion $\Ucal \subsetneq \Ucal'$. In that case, the mean value theorem yields that for any $x,x' \in \Ucal'$ and $1 \leq j \leq s_{\mathsf{KB}}$, we~have
\begin{equation}
  \abs{D^{\alpha_{|s_{\mathsf{KB}}-j}}\tilde f(x')} 
    \leq  \abs{D^{\alpha_{|s_{\mathsf{KB}}-j}}\tilde f(x)}+ \sqrt{d}\max_{\tilde x \in\mspace{2 mu} \Ucal',\,\alpha_{|1}}\abs{D^{\alpha_{|s_{\mathsf{KB}}-j}+\alpha_{|1}}\tilde f(\tilde x)} \norm{x-x'},\label{recuderval}
\end{equation}
where we also used the fact that $\norm{x-x'}_1 \leq \sqrt{d} \norm{x-x'}$. Further, note that 
$ \big\| D^{\alpha_{|s_{\mathsf{KB}}}}\tilde f \big\|_{\infty,\mspace{2 mu}\Ucal'}$ $\leq b$  ($D^{\alpha_{|s_{\mathsf{KB}}}}\tilde f$ equals zero outside $\Ucal$), and since $\tilde f \in \mathsf{C}_b^{s_{\mathsf{KB}}}(\Ucal)$, we have $ \big\|D^{\alpha_{|s_{\mathsf{KB}}-j}}\tilde f(x)\big\|_{\infty,\mspace{2 mu}\Ucal} \leq b$. Then, for any $x' \in \Ucal'$,  taking $x \in \X$ with $\norm{x-x'}\leq 1$ (such an $x$ exists by definition of $\Ucal'$) in \eqref{recuderval} yields $\big|D^{\alpha_{|s_{\mathsf{KB}}-1}}\tilde f(x')\big| \leq b+b\sqrt{d}$. Having this, we recursively apply  \eqref{recuderval} to obtain for $1 \leq j \leq s_{\mathsf{KB}}$ that
\begin{align}
   \big\|D^{\alpha_{|s_{\mathsf{KB}}-j}}\tilde f\big\|_{\infty,\Ucal'} &\leq b \sum_{i=1}^j d^{\frac{i-1}{2}} +b d^{\frac{j}{2}} \leq b\frac{1-d^{\frac{s_{\mathsf{KB}}}{2}}}{1-\sqrt{d}}+bd^{\frac{s_{\mathsf{KB}}}{2}}=: \tilde b. \label{bndderallpar}
\end{align}
If $\Ucal' \subseteq \Ucal$, then $ \big\|D^{\alpha_{|s_{\mathsf{KB}}-j}}\tilde f\big\|_{\infty,\Ucal'} \leq b$ by definition since $\tilde f \in  \mathsf{C}_b^{s_{\mathsf{KB}}}(\Ucal)$. Hence, \eqref{bndderallpar} holds in both cases as $\tilde b \geq b$.

 \medskip
 
The desired final extension is $f_{\mathsf{ext}}:=\tilde f\cdot f_{\mathsf{c}}$, where $f_{\mathsf{c}}$ is the smooth cut-off function
\begin{equation}
    f_{\mathsf{c}}(x):=\ind_{\X'} \ast \Psi_{\frac 12} (x):=\int_{\RR^d} \ind_{\X'}(y) \Psi_{\frac 12}(x-y)\dd y,~x\in \RR^d,\label{cutofffndef} 
\end{equation}
with $\X':=\left\{x' \in \RR^d:\exists \mspace{2 mu} x\in\cX,\, \big\|x'-x\big\| \leq 0.5\right\}$ and $\Psi (x)\propto \exp\Big(-\frac{1}{0.5-\norm{x}^2}\Big)\ind_{\{\|x\|<0.5\}}$ as the canonical mollifier normalized to have unit mass. 
Since $\Psi \in \mathsf{C}^{\infty}\big(\RR^d\big)$, we have $f_{\mathsf{c}} \in \mathsf{C}^{\infty}\big(\RR^d\big)$. Also, observe that $f_{\mathsf{c}}(x)=1$ for $x \in \X$, $f_{\mathsf{c}}(x)=0$ for $x \in \RR^d \setminus \Ucal'$ and $f_{\mathsf{c}}(x)  \in (0,1)$ for $x \in  \Ucal'  \setminus \X$. Hence, $f_{\mathsf{ext}}(x)=\tilde f(x)$ for $x \in \X$, $f_{\mathsf{ext}}(x)=0$ for $x \in \RR^d \setminus \Ucal'$ and $\abs{f_{\mathsf{ext}}(x)} \leq  \big|\tilde f(x)\big|$ for $x \in  \Ucal'  \setminus \X$, thus satisfying $f_{\mathsf{ext}}|_\cX=\tilde f|_\cX= f$ as required. Moreover, for all $0\leq j\leq s_{\mathsf{KB}}$, 
 we have $D^{\alpha_{|j}}f_{\mathsf{ext}}(x)=0$, for $x\notin \Ucal'$, and 
 \begin{equation}
    \big\|D^{\alpha_{|j}}f_{\mathsf{ext}}\big\|_{\infty,\mspace{2 mu}\Ucal'} \leq  2^j\tilde b \max_{\substack{\alpha:\norm{\alpha}_1 \leq j }} \big\|D^{\alpha}f_{\mathsf{c}}\big\|_{\infty,\mspace{2 mu}\Ucal'}  \leq 2^{s_{\mathsf{KB}}}\tilde b \max_{\substack{\alpha:\norm{\alpha}_1\leq s_{\mathsf{KB}} }} \big\|D^{\alpha}\Psi\big\|_{\infty,\mspace{2 mu}B_d(0.5)}=:\hat b, \label{derbndinu}
 \end{equation}
 where the first inequality follows using product rule for differentiation and \eqref{bndderallpar}, while the second is due to  \eqref{cutofffndef}. 

Consequently, for $0 \leq j \leq s_{\mathsf{KB}}$ and $i=1,2$, we have
\begin{equation}
\norm{D^{\alpha_{|j}}f_{\mathsf{ext}}}_{i}^i =\int_{\Ucal'} (D^{\alpha_{|j}}f_{\mathsf{ext}})^i(x) \dd x \leq \hat b^i~ \lambda\big(B_d(\mathsf{rad}(\X)+1)\big)=\hat b^i \frac{\pi^{\frac{d}{2}}}{\Gamma(0.5d+1)} \big(\mathsf{rad}(\X)+1\big)^d, \label{allderivsqint}
\end{equation}
where $\lambda$ denotes the Lebesgue measure,  $\mathsf{rad}(\X)=0.5 \sup_{x,x' \in \X}\norm{x-x'}$, and $\Gamma$ denotes the gamma function. Defining $b':=\hat b d \pi^{d/2}\Gamma(d/2+1)^{-1} \big(\mathsf{rad}(\X)+1\big)^{d}$ and noting that $b' \geq \hat b$, we have from \eqref{derbndinu}-\eqref{allderivsqint} that $ f_{\mathsf{ext}} \in \tilde{\cL}_{s_{\mathsf{KB}},b'}\big(\RR^d\big) $, where 
\begin{flalign}
 & \tilde{\cL}_{s_{\mathsf{KB}},b'}\big(\RR^d\big)\mspace{-4 mu}:=\mspace{-4 mu}\left\{f \mspace{-4 mu} \in \mspace{-4 mu} L^1\big(\RR^d\big) \mspace{-2 mu} \cap \mspace{-2 mu} L^2\big(\RR^d\big): \mspace{-4 mu}\begin{aligned} \mspace{-4 mu}&\abs{f(0)}\leq b', \norm{D^{\alpha}f}_2 \leq b' \mbox{ for } 1 \leq \norm{\alpha}_1 \leq s_{\mathsf{KB}}\\
  & \norm{\nabla f(0)}_1 \mspace{-2 mu}\leq \mspace{-2 mu} b',\norm{D^{\tilde \alpha}f}_1\mspace{-4 mu}<\mspace{-4 mu}\infty \mbox{ for } \norm{\tilde \alpha}_1 \leq s_{\mathsf{KB}}\end{aligned}\right\}. \label{squareintclassallder} &&
\end{flalign}
Since $\tilde{\cL}_{s_{\mathsf{KB}},b'}\big(\RR^d\big) \subseteq \cL^{\mathsf{KB}}_{s_{\mathsf{KB}},b'}\big(\RR^d\big)$ (see \eqref{squareintclassklu}), Lemma \ref{lem:suffcondbar} yields $S_2(f_{\mathsf{ext}}) \leq \kappa_d d^{3/2}b'$ and 
\begin{align}
    f_{\mathsf{ext}} \in \cB_{ \bar c_{b,d,\norm{\cX}},2,\X}\big(\RR^d\big) \cap  \tilde{\cL}_{s_{\mathsf{KB}},b'}\big(\RR^d\big) \subseteq \cB_{\bar c_{b,d,\norm{\cX}},2,\X}\big(\RR^d\big) \cap  \cL_{s_{\mathsf{KB}},b'}^{\mathsf{KB}}\big(\RR^d\big), \label{finextchar}
\end{align}
where 
 \begin{align}
     \bar c_{b,d,\norm{\cX}}
     &:=(\kappa_d d^{\frac 32} \norm{\X} \vee 1)  \notag \\
     &  \times\underbrace{\pi^{\frac{d}{2}}\Gamma\bigg(\frac{d}{2}+1\bigg)^{-1} \big(\mathsf{rad}(\X)+1\big)^{ d}2^{s_{\mathsf{KB}}}bd \left(\frac{1-d^{\frac{s_{\mathsf{KB}}}{2}}}{1-\sqrt{d}}+d^{\frac{s_{\mathsf{KB}}}{2}}\right) \max_{\substack{\norm{\alpha}_1\leq s_{\mathsf{KB}} }} \big\|D^{\alpha}\Psi\big\|_{\infty,B_d(0.5)}}_{=:b'}, \label{constapproxhold} 
 \end{align}
and $\kappa_d^2:=  \big(d+d^{s_{\mathsf{B}}}\big)\int_{\RR^d}\Big(1+\norm{\omega}^{2(s_{\mathsf{B}}-1)}\Big)^{-1} \dd \omega$. It then  follows from Theorem \ref{THM:approximation}  that there exists $g \in \mathcal{G}_k^{\mathsf{R}}\big(\bar c_{b,d,\norm{\cX}}\big)$ such that $\norm{f-g}_{\infty,\X} \lesssim  \bar c_{b,d,\norm{\cX}}d^{1/2} k^{-1/2}. $
This proves the first claim of the proposition. Repeating the same arguments starting with $\tilde f \in  \mathsf{C}_b^{s_{\mathsf{B}}}(\Ucal)$, the second claim follows again from Theorem \ref{THM:approximation},  thus completing  the proof.
\subsubsection{Proof of Theorem \ref{empesterrbnd}} \label{empesterrbnd-proof}
We require the following theorem which gives a tail probability bound for the deviation of supremum of a sub-Gaussian process from its associated entropy integral. 
\begin{theorem}{\citep[Theorem 5.29]{VanHandel-book}}\label{thm:tailineq}
Let $(X_{\theta})_{\theta \in \Theta }$ be a separable sub-Gaussian process on the metric space $(\Theta,\mathsf{d})$. Then, there exists $c>0$ such that for any $\theta_0 \in \Theta$ and $\delta \geq 0$, we have
\begin{align}
  &  \mathbb{P}\left(\sup_{\theta \in \Theta} X_{\theta}-X_{\theta_0} \geq c \int_{0}^{\infty} \sqrt{\log N(\epsilon,\Theta,\mathsf{d})}\dd\epsilon+\delta \right) \leq c\mspace{2 mu}e^{-\frac{\delta^2}{c\mspace{2 mu}\mathsf{diam}(\Theta,\mathsf{d})^2}}, \notag
\end{align}
where $\mathsf{diam}(\Theta,\mathsf{d}):= \sup\limits_{\theta,\tilde{\theta} \in \Theta}\mathsf{d}(\theta,\tilde{\theta})$. 
\end{theorem}
We will also use the following lemma which bounds the covering number of $\cG_k(\mathbf{a}_k,\phi)$ w.r.t. to metric induced by $\norm{\cdot}_{\infty,\X}$.
\begin{lemma} \label{lem:covnumbnd}
Let $\phi$ be a continuous monotone  activation whose Lipschitz constant is bounded by $L$, and $U_{a,\cX}(\phi):=\phi\big(a(\norm{\cX}+1)\big) \vee \phi\big(-a(\norm{\cX}+1)\big)$. Then
\begin{align}
    N\big(\epsilon,\cG_k(\mathbf{a}_k,\phi), \norm{\cdot}_{\infty,\X}\big)  \leq& \big(1+10ka_{2,k}U_{a_{1,k},\cX}(\phi)\epsilon^{-1}\big)^k \big(1+10a_{4,k}\norm{\cX}\epsilon^{-1}\big)^d\big(1+10a_{3,k}\epsilon^{-1}\big) \notag \\
& \quad \times  \big(1+10Lka_{1,k}a_{2,k}\norm{\cX}\epsilon^{-1}\big)^{dk} \big(1+10Lka_{1,k}a_{2,k}\norm{\cX}\epsilon^{-1}\big)^{k}. \notag\end{align}
In particular, for $\phi \in \{\phi_{\mathsf{R}},\phi_{\mathsf{S}}\}$, we have
\begin{align}
& N\big(\epsilon,\cG_k^{\mathsf{R}}(a),\norm{\cdot}_{\infty,\X}\big) \leq 
 \big(1+20a(\norm{\cX}+1)\epsilon^{-1}\big)^{(d+2)k+d+1}, \label{covbndscrelu} \\
& N\big(\epsilon,\cG_k^{\mathsf{S}}(a),\norm{\cdot}_{\infty,\X}\big) \leq  \Big(1+20a(\norm{\cX}+1)k^{\frac 12}(\log k+1)\epsilon^{-1}\Big)^{(d+2)k+1}, \label{covbndscsig} \\
& N\big(\epsilon,\cG_k^*(\phi),\norm{\cdot}_{\infty,\X}\big) \leq  \big(1+10k(\norm{\cX}+1)\epsilon^{-1}\big)^{(d+2)k+1}. \label{covbndforcons}
\end{align}
\end{lemma}
The proof of Lemma \ref{lem:covnumbnd} (see Appendix \ref{lem:covnumbnd-proof}) is based on the fact that the covering number of $B_d^m(r)$ w.r.t. $\|\cdot\|_m$ norm, $m \geq 1$, satisfies 
\begin{align}
   N\big(\epsilon,B_d^m(r),\norm{\cdot}_m\big) \leq \left(2r\epsilon^{-1}+1\right)^d.\label{euclideancovnumb}
\end{align}
\medskip
 Continuing with the proof of Theorem \ref{empesterrbnd}, we will show that the claim holds with 
\begin{align}
& V_{k,h,\phi,\cX} \lesssim \bar C\big(|\cG_k^*(\phi)|,\cX\big)^2 \left( \bar C\left(\abs{h'\circ \cG_k^*(\phi)},\cX\right)+1\right)^2, \label{Vkconstdef} \\
   & E_{k,h,\phi,\cX} \lesssim k\sqrt{d(\norm{\cX}+1)}  \big(\bar C\big(\abs{h'\circ \cG_k^*(\phi)},\cX\big)+1\big) \sqrt{\bar C\big(|\cG_k^*(\phi)|,\cX\big)}, \label{Ekconstdef} 
\end{align}
where we recall that $ \bar C\left(|\cF|,\cX\right):=\sup_{x \in \X, f \in  \cF}|f(x)|$. In the following,  we will suppress the dependence of $\phi$, $h$, and $\cX$   for simplicity (unless explicitly needed), e.g.,  $\cG_k(\mathbf{a}_k)$ instead of $\cG_k(\mathbf{a}_k,\phi)$.

\medskip
Fix $\mu,\nu \in \cP(\X)$ such that $\mathsf{D}_{h,\cG_k(\mathbf{a}_k)}(\mu,\nu)<\infty$. We have 
\begin{flalign}
&\hat{\mathsf{D}}_{h,\cG_k(\mathbf{a}_k)}(x^n,y^n)-\mathsf{D}_{h, \cG_k(\mathbf{a}_k)}(\mu,\nu) \notag\\
&= \sup_{g_{\theta} \in \cG_k(\mathbf{a}_k)} \frac 1n \sum_{i=1}^n g_{\theta}(x_i)-\frac 1n \sum_{i=1}^n h\circ g_{\theta}(y_i)-\left(\sup_{g_{\theta} \in \cG_k(\mathbf{a}_k)}\EE_{\mu}\big[g_{\theta}\big]-\EE_{\nu}\big[h\circ g_{\theta}\big]\right) \notag \\
    &\leq   \sup_{g_{\theta} \in \cG_k(\mathbf{a}_k)} \frac 1n \sum_{i=1}^n g_{\theta}(x_i)-\frac 1n \sum_{i=1}^n h\circ g_{\theta}(y_i)-\EE_{\mu}\big[g_{\theta}\big]+\EE_{\nu}\big[h\circ g_{\theta}\big]. \label{supdiffer}&&
\end{flalign}

Consider the stochastic process $(Z_{g_{\theta}})_{g_{\theta}\in \cG_k(\mathbf{a}_k)}$ defined by 
\begin{align}
Z_{g_{\theta}}:=\frac 1n \sum_{i=1}^n g_{\theta}(X_i)-\frac 1n \sum_{i=1}^n h\circ g_{\theta}(Y_i)-\EE_{\mu}\big[g_{\theta}\big]+\EE_{\nu}\big[h\circ g_{\theta}\big]. \label{defnztheta}
\end{align}
To apply Theorem \ref{thm:tailineq}, we now show that $(Z_{g_{\theta}})_{g_{\theta}\in \cG_k(\mathbf{a}_k)}$ is a separable sub-Gaussian process on $\big(\cG_k(\mathbf{a}_k),\mathsf{d}_{k,\mathbf{a}_k,n}\big)$, where $\mathsf{d}_{k,\mathbf{a}_k,n}$ will be defined below. Note that  $\mathbb{E}\left[Z_{g_{\theta}}\right]=0$ 
for all $g_{\theta} \in \cG_k(\mathbf{a}_k)$, and
\begin{flalign}
 & \abs{   Z_{g_{\theta}}-  Z_{g_{\tilde{\theta}}}} \leq  \sum_{i=1}^n \frac 1n \abs{g_{\theta}(X_i)-g_{\tilde{\theta}}(X_i)-\EE_{\mu}\big[g_{\theta}-g_{\tilde{\theta}}\big]} \notag \\
 &\qquad \qquad \qquad \qquad \qquad ~~ +  \frac 1n \abs{h\circ g_{\theta}(Y_i)-h\circ g_{\tilde{\theta}}(Y_i)-\EE_{\nu}\left[h\circ g_{\theta}-h\circ g_{\tilde{\theta}}\right]}. \label{sumzthet} &&
\end{flalign}
By an application of the mean value theorem, we have for all $g_{\theta},g_{\tilde{\theta}} \in \cG_k(\mathbf{a}_k)$,
    \begin{flalign}
   \abs{h \circ g_{\theta}(x)-h \circ g_{\tilde{\theta}}(\tilde{x})} &\leq  \bar C\left(\abs{h'\circ \cG_k(\mathbf{a}_k)}\right) \abs{g_{\theta}(x)-g_{\tilde{\theta}}(\tilde{x})}. \label{applymeanvalthm} 
    \end{flalign}

Hence, we  have that almost surely
\begin{flalign}
& \frac 1n \abs{g_{\theta}(X_i)-g_{\tilde{\theta}}(X_i)-\EE_{\mu}\big[g_{\theta}-g_{\tilde{\theta}}\big]}+ \frac 1n \big|h\circ g_{\theta}(Y_i)-h\circ g_{\tilde{\theta}}(Y_i)  -\EE_{\nu}\big[h\circ g_{\theta}\notag  -h\circ g_{\tilde{\theta}}\big]\big| \notag \\
&  \leq \frac 1n \Big[ \abs{g_{\theta}(X_i)-g_{\tilde{\theta}}(X_i)}+ \abs{\EE_{\mu}\big[g_{\theta}-g_{\tilde{\theta}}\big]} + \abs{h\circ g_{\theta}(Y_i)-h\circ g_{\tilde{\theta}}(Y_i)}+ \abs{\EE_{\nu}\left[h\circ g_{\theta}-h\circ g_{\tilde{\theta}}\right]} \Big]\notag \\
 &\leq  2\mspace{2 mu}n^{-1} \big( \bar C\left(\abs{h'\circ \cG_k(\mathbf{a}_k)}\right)+1\big) \norm{g_{\theta}-g_{\tilde{\theta}}}_{\infty,\cX}. \label{bndtermstv} &&
\end{flalign}
  Let  $\mathsf{d}_{k,\mathbf{a}_k,n}\big(g_{\theta},g_{\tilde{\theta}}\big):=R_{k,\mathbf{a}_k} \|g_{\theta}-g_{\tilde{\theta}}\|_{\infty,\cX}n^{-\frac 12}$, where $R_{k,\mathbf{a}_k}:= 2\left( \bar C\left(\abs{h'\circ \cG_k(\mathbf{a}_k)}\right)+1\right)$. 
Then, it follows from \eqref{sumzthet} and \eqref{bndtermstv} via Hoeffding's lemma that 
\begin{align}
    \mathbb{E}\bigg[e^{t\big(Z_{g_{\theta}}-Z_{g_{\tilde{\theta}}}\big)} \bigg]\leq   e^{\frac{1}{2}t^2 \mathsf{d}_{k,\mathbf{a}_k,n}\big(g_{\theta},g_{\tilde{\theta}}\big)^2}. \notag
\end{align}
  Thus, $(Z_{g_{\theta}})_{g_{\theta}\in\cG_k(\mathbf{a}_k)}$ is a separable sub-Gaussian process on the metric space $\big( \cG_k(\mathbf{a}_k), \mathsf{d}_{k,\mathbf{a}_k,n}\big)$, where the separability  follows from  \eqref{bndtermstv} by the denseness of the countable subset of $\cG_k(\mathbf{a}_k)$ obtained by quantizing each of the finite number of bounded NN parameters to rational numbers (recall that a finite union of countable sets is countable and the activation $\phi$ is assumed continuous).

\medskip

Specializing to the NN class $\cG_k^*(\phi):=\cG_k(\mathbf{a}^*,\phi)$, we next bound its covering number w.r.t. $\mathsf{d}_{k,\mathbf{a}^*,n}$, where  $\mathbf{a}^*=(1,1,1,0)$. 
We have 
\begin{flalign}
    N\big(\epsilon,\cG_k^*, \mathsf{d}_{k,\mathbf{a}^*,n}\big) &:=N\big(\epsilon,\cG_k^*, R_{k,\mathbf{a}^*} n^{-\frac 12}\|\cdot\|_{\infty,\cX}\big)\notag \\
    &=N\big(\epsilon/(R_{k,\mathbf{a}^*} n^{-\frac 12}),\cG_k^*, \|\cdot\|_{\infty,\cX}\big)\notag \\
   & \leq \big(1+10k(\norm{\cX}+1)R_{k,\mathbf{a}^*} n^{-\frac 12}\epsilon^{-1}\big)^{(d+2)k+1},    \notag &&
\end{flalign}
where the last inequality uses \eqref{covbndforcons}. Also,  we have that $N\big(\epsilon,\cG_k^*,$ $ \mathsf{d}_{k,\mathbf{a}^*,n}\big)=1$  for $\epsilon \geq \mathsf{diam}\big(\cG_k^*,\mathsf{d}_{k,\mathbf{a}^*,n}\big):=\max_{g_{\theta},g_{\tilde{\theta}} \in \cG_k^*}\mathsf{d}_{k,\mathbf{a}^*,n}\big(g_{\theta},g_{\tilde{\theta}}\big)$. 
Then, 
\begin{flalign}
&\int_{0}^{\infty} \sqrt{\log N\big(\epsilon,\cG_k^*, \mathsf{d}_{k,\mathbf{a}^*,n}\big)}\dd\epsilon \notag \\
&\qquad \qquad\qquad \qquad= \int_{0}^{\mathsf{diam}\big(\cG_k^*,\mathsf{d}_{k,\mathbf{a}^*,n}\big)} \sqrt{\log N\big(\epsilon,\cG_k^*, \mathsf{d}_{k,\mathbf{a}^*,n}\big)}\dd\epsilon \notag \\
    & \qquad \qquad\qquad \qquad\lesssim  \sqrt{kd}\int_{0}^{\mathsf{diam}\big(\cG_k^*,\mathsf{d}_{k,\mathbf{a}^*,n}\big)} \sqrt{\log \big(1+10k(\norm{\cX}+1)R_{k,\mathbf{a}^*}n^{-\frac 12}\epsilon^{-1}}\big)\dd\epsilon \notag \\
    &\qquad \qquad\qquad \qquad\lesssim  k\sqrt{d(\norm{\cX}+1)} R_{k,\mathbf{a}^*} \sqrt{\bar C\big(|\cG_k^*|\big)}n^{-\frac{1}{2}}, \notag &&
\end{flalign}
where the last step uses $\log(1+x) \leq x,~x \geq -1$, and $\mathsf{diam}\big(\cG_k^*,\mathsf{d}_{k,\mathbf{a}^*,n}\big) \leq 2 R_{k,\mathbf{a}^*}\bar C\big(|\cG_k^*|\big) n^{-1/2}$. 

It follows from Theorem \ref{thm:tailineq} with $Z_{0}=0$ and the definitions of $V_{k}$ and $E_{k}$ (see \eqref{Vkconstdef} and \eqref{Ekconstdef})  that there exists a constant $c>0$ such that
\[
\mathbb{P}\left(\sup_{g_{\theta} \in \cG_k^*}Z_{g_{\theta}} \geq cE_{k}n^{-\frac 12}+\delta\right) 
\leq c\mspace{2 mu}e^{-\frac{\delta^2}{c\mspace{2 mu}\mathsf{diam}\left(\cG_k^*,\mathsf{d}_{k,\mathbf{a}^*,n}\right)^2}}=c\mspace{2 mu} e^{-\frac{n\delta^2}{V_{k} }},\quad \forall~\delta \geq 0. 
\]
Noting that this also holds with $-Z_{g_{\theta}}$ in place of $Z_{g_{\theta}}$, the union bound gives
\begin{align}
\mathbb{P}\left(\sup_{g_{\theta} \in \cG_k^*}|Z_{g_{\theta}}| \geq \delta+c E_{k}n^{-\frac 12}\right) \leq 2c e^{-\frac{n\delta^2}{V_{k} }}. \notag
\end{align}
From \eqref{supdiffer}-\eqref{defnztheta} and the above equation, we obtain that for $\delta \geq 0$
\begin{flalign}
    \mathbb{P}\mspace{-3 mu}\left(\abs{\mathsf{D}_{h, \cG_k^*}(\mu,\nu)\mspace{-2 mu}-\mspace{-2 mu}\hat{\mathsf{D}}_{h,\cG_k^*}(X^n,Y^n)}\mspace{-2 mu} \geq \mspace{-2 mu}\delta \mspace{-2 mu}+ \mspace{-2 mu}c E_{k} n^{-\frac 12}\mspace{-2 mu}\right)
    &\mspace{-2 mu}\leq\mspace{-2 mu} \mathbb{P}\mspace{-2 mu}\left(\sup_{g_{\theta} \in \cG_k^*}\mspace{-2 mu}|Z_{g_{\theta}}| \geq \delta+c E_{k} n^{-\frac 12}\mspace{-2 mu}\right) \leq 2c e^{-\frac{n\delta^2}{V_{k}}}. \notag&& 
\end{flalign}
Taking supremum over $\mu,\nu \in \cP(\X)$ such that $\mathsf{D}_{h,\cG_k^*}(\mu,\nu)<\infty$ yields \eqref{bndesterremp}.

\medskip

By following  similar steps with \eqref{covbndscrelu} and \eqref{covbndscsig} in place of \eqref{covbndforcons}, we have  for $\cG_k \in \big\{\cG_k^{\mathsf{R}}(a),\cG_k^{\mathsf{S}}(a)\big\}$ that
\begin{align}
 \mathbb{P}\left(\abs{\mathsf{D}_{h, \cG_k}(\mu,\nu)-\hat{\mathsf{D}}_{h,\cG_k}(X^n,Y^n)} \geq \delta +c \bar E_{k,a,h,\cX} n^{-\frac 12}\right) \leq 2\,c\, e^{-n\delta^2/\bar V_{k,a,h,\cX}},  \label{tailineqncsup}  
\end{align}
where $\bar V_{k,a,h,\cX} \lesssim \bar C\big(|\cG_k^{\mathsf{R}}(a)|,\cX\big)^2 \big( \bar C\big(\big|h'\circ \cG_k^{\mathsf{R}}(a)\big|,\cX\big)+1\big)^2$ and  $\bar E_{k,a,h,\cX} \lesssim \sqrt{dk\log k} a(\norm{\cX}+1) \big(\bar C\big(\big|h'\circ \cG_k^{\mathsf{R}}(a)\big|,\cX\big)+1\big)$. To establish \eqref{tailineqncsup}, we use   
\begin{align}
          \int_{0}^{\delta} \sqrt{\log \big(1+A\epsilon^{-1}\big)} \lesssim \delta \sqrt{\log\big((A+\delta)/\delta\big)}, \label{intbypartsineq}
  \end{align}
  for $A \geq e$ and $0 \leq \delta \leq 1$, which can be shown via integration by parts.

\subsubsection{Proof of Theorem \ref{thm:optkdepNN}} \label{thm:optkdepNN-proof}
We establish a more general upper bound with $\cG_k$ replaced by an arbitrary VC-type class $\cF_k$ satisfying certain assumptions. This result is also applicable to deep NNs with finite width  in each layer,  continuous activation and bounded parameters, and hence, may be of independent interest.
\begin{theorem}[Estimation error bound] \label{empesterrorbndgen}
Let $\mu,\nu \in \cP(\X)$ and  $X^n \sim \mu^{\otimes n}$ and $Y^n\sim \nu^{\otimes n}$. Suppose  $h:\RR\to\bar{\RR}$ and $(\cF_k)_{k \in \NN}$ (with domain $\X$) satisfy the following conditions for each $k \in \NN$:
\begin{enumerate}[label = (\roman*),leftmargin=15 pt]
    \item $h$ is differentiable  at every point in $\left[\underaccent{\bar}{C}(\cF_k,\cX), \bar C\left(\cF_k,\cX\right)\right]$ with derivative $h'$;
    \item $ \bar C\left(\abs{h'\circ \cF_k},\cX\right) \vee \bar C\left(\abs{\cF_k},\cX\right) <\infty$;
    \item $\cF_k$ is a VC-type class with constants $l_{\mathsf{vc}}(\cF_k) \geq e$ and $u_{\mathsf{vc}}(\cF_k) \geq 1$ satisfying \eqref{def:VC-type class} w.r.t. a constant  envelope $M_k$ (note that this implies $\bar C(|\cF_k|,\cX)\leq M_k$);
    \item $\mathcal{F}_k$ is point-wise measurable, i.e.,  there exists a countable subclass  $\mathcal{F}_k' \subseteq \mathcal{F}_k$ of  measurable functions such that for any $f \in \mathcal{F}_k$, there is a sequence of functions $\{f_j\}_{j \in \NN}\subset \mathcal{F}_k'$ for which $\lim_{j \rightarrow \infty}f_j(x)=f(x),~\forall x \in \X$.
\end{enumerate}
Then, for every $k,n \in \NN$, we have
\begin{align}
&\sup_{\substack{\mu,\nu \in \cP(\X):\\\mathsf{D}_{h,\cF_k}(\mu,\nu) <\infty}} \EE\left[\abs{ \hat{\mathsf{D}}_{h,\cF_k}(X^n,Y^n)-\mathsf{D}_{h, \cF_k}(\mu,\nu)}\right] \notag \\
& \qquad  \qquad \lesssim    M_k\Big(\bar C\left(\abs{h' \circ \cF_k},\cX\right)+1\Big)n^{-\frac 12} \int_{0}^{1}\sqrt{ \sup_{\gamma \in \cP(\X)} \log N\left(M_k\epsilon,\cF_k,\|\cdot\|_{2,\gamma}\right)}\dd\epsilon \label{entrpyintgenvccls}\\
& \qquad  \qquad  \lesssim   \big(u_{\mathsf{vc}}(\cF_k)\log l_{\mathsf{vc}}(\cF_k)\big)^{\frac 12}\mspace{2 mu}M_k\left(\bar C\left(\abs{h'\circ \cF_k},\cX\right)+1\right) n^{-\frac 12}. \label{genNNubndwrtn}
\end{align}
\end{theorem}
The proof of this theorem is based on  standard maximal inequalities from empirical process theory, and is presented in Appendix \ref{empesterrorbndgen-proof} below. 
\medskip

To prove \eqref{empestuppbnd}, we first verify that the relevant assumptions given in Theorem \ref{empesterrorbndgen} hold with $\cF_k=\cG_k^{\mathsf{R}}(a)$ and a constant envelope $M_k=3a(\norm{\cX}+1)$. The proof for $\cG_k^{\mathsf{S}}(a)$ is similar, and hence omitted. Conditions  $(i)$ and $(ii)$ are satisfied by the hypotheses in the theorem. 
Condition $(iii)$ holds as
\begin{flalign}
\sup_{\gamma \in \cP(\X)}N\big(M_k\epsilon,\cG_k^{\mathsf{R}}(a),\|\cdot\|_{2,\gamma}\big)  \leq  N\big(M_k\epsilon,\cG_k^{\mathsf{R}}(a),\norm{\cdot}_{\infty,\X}\big) \leq 
 \big(1+7\epsilon^{-1}\big)^{(d+2)k+d+1},  \label{covuppbndshallnn}
\end{flalign}
 for any $0<\epsilon \leq 1$, where the last inequality  follows from \eqref{covbndscrelu}.
To verify condition $(iv)$, note that $g \in \cG_k^{\mathsf{R}}(a)$ is measurable since it is a finite linear combination of compositions of an affine function with a continuous activation. Moreover, point-wise measurability of $\cG_k^{\mathsf{R}}(a)$ follows by the continuity of activation and the fact that each of the finite number of parameters of $\cG_k^{\mathsf{R}}(a)$ can be approximated arbitrary well by rational numbers.

\medskip 

Next, we evaluate the entropy integral term  in \eqref{entrpyintgenvccls} by bounding $N\big(M_k\epsilon,\cG_k^{\mathsf{R}}(a),\|\cdot\|_{2,\gamma}\big)$. 
For this purpose, let $\cG_k^{\dagger}(a):=\cG_k(1,2k^{-1}a,0,0,\phi_{\mathsf{R}})$. For any $g_{\theta}, g_{\tilde \theta} \in \cG_k^{\mathsf{R}}(a)$, where $g_{\theta}=\sum_{i=1}^k \beta_i \phi_{\mathsf{R}}\left(w_i\cdot x+b_i\right)+w_0 \cdot x + b_0$ and $g_{\tilde \theta}=\sum_{i=1}^k \tilde \beta_i \phi_{\mathsf{R}}\left(\tilde w_i\cdot x+\tilde b_i\right)+\tilde w_0 \cdot x + \tilde b_0$, we have
\begin{align}
    \norm{g_{\theta}-g_{\tilde \theta} }_{2, \gamma} \mspace{-3 mu}\leq\mspace{-3 mu} \norm{\sum_{i=1}^k \beta_i \phi_{\mathsf{R}}\left(w_i\cdot x+b_i\right)\mspace{-2 mu}-\mspace{-2 mu}\sum_{i=1}^k \tilde \beta_i \phi_{\mathsf{R}}(\tilde w_i\cdot x+\tilde b_i)}_{2, \gamma}\mspace{-10 mu}+\norm{w_0-\tilde w_0}_1\mspace{-3 mu}\norm{\cX} \mspace{-2 mu}+\mspace{-2 mu}|b_0-\tilde b_0|. \notag 
\end{align}
Hence,
\begin{align}
    N\big(\epsilon,\cG_k^{\mathsf{R}}(a),\|\cdot\|_{2,\gamma}\big)  &\leq  N\big(\epsilon/3,\cG_k^{\dagger}(a),\|\cdot\|_{2,\gamma}\big) N(\epsilon/3, B_d^1(a),\norm{\cX}\norm{\cdot}_1) N\big(\epsilon/3,B_1(a),|\cdot|\big) \notag \\
    & \leq N\big(\epsilon/3,\cG_k^{\dagger}(a),\|\cdot\|_{2,\gamma}\big) (1+6a(\norm{\cX}+1)\epsilon^{-1})^{d+1}, \label{covnumbbndfin}
\end{align}
where \eqref{covnumbbndfin} uses \eqref{euclideancovnumb}.

Consider $g=\sum_{i=1}^k \beta_i \phi_{\mathsf{R}}\left(w_i\cdot x+b_i\right) \in \cG_k^{\dagger}(a)$. Let $\cF=2a\phi_{\mathsf{R}} \circ \tilde \cF$, where $\tilde  \cF=\{f:\cX \rightarrow \RR: f=w\cdot x+b, w \in \RR^d,b \in \RR,\norm{w}_1 \vee \abs{b} \leq 1\}$. By considering the $d$ coordinate projections $f_i(x)=x_i$, $1 \leq i \leq d$, and $f_{d+1}(x)=1$ spanning the finite dimensional vector space $\tilde \cF$, we have from  \citep[Lemma 2.6.15]{AVDV-book} that $\cF$ is a VC subgraph class of index atmost $d+3$. This uses  the fact that if $\tilde \cF$ is a VC subgraph class of index $v$ and $\phi$ is monotone, then $\phi \circ \tilde  \cF$ is a VC subgraph class of index at most $v$, which follows from the proof of \citep[Lemma 2.6.18 (viii)]{AVDV-book}. Then, \citep[Theorem 2.6.7]{AVDV-book} yields $N\big(2a(\norm{\cX}+1)\epsilon,\cF',\|\cdot\|_{2,\gamma}\big) \lesssim (d+3)(16 e)^{d+3}\epsilon^{-2(d+2)}$, where $\cF':=\cF \cup -\cF$. Further, by a careful inspection of the proof of \citep[Theorem 3.6.17]{gine_nickl_2015}, we obtain that
$\log N\big(2a(\norm{\cX}+1)\epsilon,\overline{\mathsf{co}}(\cF'),\|\cdot\|_{2,\gamma}\big) \lesssim d\epsilon^{-2(d+2)/(d+3)}$, where $\overline{\mathsf{co}}(\cF')$ denotes the sequential closure of the convex hull of  $\cF'$ given by $\mathsf{co}(\cF'):=\{\sum_{i=1}^k \lambda_i f_i: f_i \in\cF',\sum_{i=1}^k \lambda_i=1, \lambda_i \geq 0, ~\forall~1 \leq i \leq k, ~k \in \NN\}$. Since $\cG_k^{\dagger}(a) \subseteq \overline{\mathsf{co}}(\cF')$, we have $\log N\big(2a(\norm{\cX}+1)\epsilon,\cG_k^{\dagger}(a),\|\cdot\|_{2,\gamma}\big) \lesssim d\epsilon^{-2(d+2)/(d+3)}$. Hence,
\begin{flalign}
& \int_{0}^{1}\sqrt{ \sup_{\gamma \in \cP(\X)} \log N\left(M_k\epsilon,\cG_k^{\mathsf{R}}(a),\|\cdot\|_{2,\gamma}\right)}\dd\epsilon\notag \\
& \stackrel{(a)}{\leq} \int_{0}^{1}\sqrt{ \sup_{\gamma \in \cP(\X)} \log N\left(M_k\epsilon/3,\cG_k^{\dagger}(a),\|\cdot\|_{2,\gamma}\right)}\dd\epsilon+\sqrt{d+1}\int_{0}^1 \sqrt{  \log\big(1+2\epsilon^{-1}\big)}\dd\epsilon \notag \\
 &\lesssim\sqrt{d}\int_{0}^{1} (0.5\epsilon)^{-(d+2)/(d+3)}\dd\epsilon+\sqrt{d}\int_{0}^1 \sqrt{  \log\big(1+2\epsilon^{-1}\big)}\dd\epsilon \notag \\
 &\lesssim d^{\frac 32},\label{entropyintscclass} &&
\end{flalign}
where $(a)$ uses \eqref{covnumbbndfin}  and $\sqrt{x+y} \leq \sqrt{x}+\sqrt{y}$ for $x,y \geq 0$; and the second integral in the penultimate step can be evaluated by applying $\log(1+x) \leq x$ for $x \geq 0$.  This completes the proof of \eqref{empestuppbnd} via \eqref{entrpyintgenvccls}.

\subsubsection{Proof of Theorem \ref{empesterrorbndgen}}\label{empesterrorbndgen-proof}
To simplify notation, we will denote $\bar C\left(\abs{\cF_k},\cX\right)$ by $\bar C\left(\abs{\cF_k}\right)$. Fix $\mu,\nu \in \cP(\X)$ such that $\mathsf{D}_{h,\cF_k}(\mu,\nu) <\infty$. Note that 
\begin{align*}
\hat{\mathsf{D}}_{h,\cF_k}(X^n,Y^n)-\mathsf{D}_{h, \cF_k}(\mu,\nu)&\leq \frac{1}{\sqrt{n}}\sup_{f \in \cF_k} \frac{1}{\sqrt{n}} \sum_{i=1}^n \big(f(X_i)-\EE_{\mu}[f]- h \circ f(Y_i)+\EE_{\nu}[h\circ f]\big).
\end{align*}
  Let $\mu_n$ and $\nu_n$ denote the empirical measures $n^{-1} \sum_{i=1}^n\delta_{X_i}$ and $n^{-1}\sum_{i=1}^n \delta_{Y_i}$, where $\delta_x$ denotes the Dirac measure centered at $x\in\cX$. Then, we have 
\begin{flalign}
   &\EE \left[\abs{\hat{\mathsf{D}}_{h,\cF_k}(X^n,Y^n)-\mathsf{D}_{h, \cF_k}(\mu,\nu)} \right] \notag \\
   & \leq  n^{-\frac 12}~ \EE \left[\sup_{f \in \cF_k} n^{-\frac 12} \abs{\sum_{i=1}^n \big(f(X_i)-\EE_{\mu}[f]}\right]+n^{-\frac 12}~ \EE \left[\sup_{f \in \cF_k} n^{-\frac 12}\abs{\sum_{i=1}^n h \circ f(Y_i)-\EE_{\nu}[h \circ f]\big)}\right] \notag \\
        &\stackrel{(a)}{\lesssim} n^{-\frac 12} \EE \left[\int_{0}^{\infty}\sqrt{ \log N\big(\epsilon,\cF_k,\|\cdot\|_{2,\mu_n}\big)}\dd\epsilon+\int_{0}^{\infty}\sqrt{ \log N\big(\epsilon,h \circ \cF_k,\|\cdot\|_{2,\nu_n}\big)} \dd\epsilon\right]\notag \\
            &\stackrel{(b)}{\leq}  n^{-\frac 12} \int_{0}^{\infty}\sqrt{ \sup_{\gamma \in \cP(\X)} \log N\left(\epsilon,\cF_k,\|\cdot\|_{2,\gamma}\right)}\dd\epsilon \notag \\
                        & \qquad \qquad  \qquad  \qquad \qquad   \qquad  \quad   +n^{-\frac 12}  \int_{0}^{\infty}\sqrt{ \sup_{\gamma \in \cP(\X)} \log N\left(\epsilon,\cF_k,\bar C\left(\abs{h' \circ \cF_k}\right)\|\cdot\|_{2,\gamma}\right)}\dd\epsilon \notag \\
                        &\stackrel{(c)}{=}  n^{-\frac 12} \int_{0}^{2M_k}\sqrt{ \sup_{\gamma \in \cP(\X)} \log N\left(\epsilon,\cF_k,\|\cdot\|_{2,\gamma}\right)}\dd\epsilon\notag \\
                        & \qquad \qquad  \qquad  ~  +n^{-\frac 12} \int_{0}^{2M_k\bar C\left(\abs{h' \circ \cF_k}\right)}\sqrt{ \sup_{\gamma \in \cP(\X)} \log N\left(\epsilon \big(\bar C\left(\abs{h' \circ \cF_k}\right)\big)^{-1},\cF_k,\|\cdot\|_{2,\gamma}\right)}\dd\epsilon \notag \\
                         &\lesssim M_k\Big(\bar C\left(\abs{h' \circ \cF_k}\right)+1\Big)n^{-\frac 12} \int_{0}^{1}\sqrt{ \sup_{\gamma \in \cP(\X)} \log N\left(M_k\epsilon,\cF_k,\|\cdot\|_{2,\gamma}\right)}\dd\epsilon \label{entrintupbnd} \\
            & \stackrel{(d)}{\leq} M_k (u_{\mathsf{vc}}(\cF_k))^{\frac 12}\mspace{2 mu}\Big(\bar C\left(\abs{h' \circ \cF_k}\right)+1\Big)n^{-\frac 12} \int_{0}^{1}\sqrt{\log\big(1+l_{\mathsf{vc}}(\cF_k)\epsilon^{-1}\big)}\dd\epsilon \notag \\
            & \stackrel{(e)}{\lesssim} M_k \big(u_{\mathsf{vc}}(\cF_k)\log l_{\mathsf{vc}}(\cF_k)\big)^{\frac 12}\Big(\bar C\left(\abs{h' \circ \cF_k}\right)+1\Big)n^{-\frac 12}, \label{upbndgenNNest}&&
\end{flalign}
where 
\begin{enumerate}[label = (\alph*),leftmargin=15 pt]
\item follows  
via an application of \citep[Corollary 2.2.8]{AVDV-book} since for fixed $(X^n,Y^n)=(x^n,y^n)$, Hoeffding's inequality implies that $n^{-\frac 12 } \sum_{i=1}^n \sigma_i f(x_i)$ and $n^{-\frac 12 } \sum_{i=1}^n h \circ f(y_i) \mspace{2 mu}\sigma_i$ are sub-Gaussian w.r.t. pseudo-metrics  $\|\cdot\|_{2,\mu_n}$ and $\|\cdot\|_{2,\nu_n}$, respectively; 
\item is due to
\begin{align}
      N\big(\epsilon,h \circ \cF_k,\|\cdot\|_{2,\gamma}\big)  &\leq N\big(\epsilon,\cF_k,\bar C\mspace{-4 mu}\left(\abs{h' \circ \cF_k}\right) \|\cdot\|_{2,\gamma}\big) \notag \\
      &=N\Big(\epsilon\big(\bar C\big(\abs{h' \circ \cF_k}\big)\big)^{-1},\cF_k,\|\cdot\|_{2,\gamma}\Big), \label{hcompgvcfn}
\end{align}
which in turn follows from \eqref{applymeanvalthm}, and  taking supremum w.r.t. to $\gamma \in \cP(\X)$;
 \item follows since  $N\left(\epsilon,\cF_k,\bar C\left(\abs{h' \circ \cF_k}\right)\|\cdot\|_{2,\gamma}\right)=1$ for $\epsilon \geq 2M_k\bar C\left(\abs{h' \circ \cF_k}\right)$, $N\left(\epsilon,\cF_k,\|\cdot\|_{2,\gamma}\right) = 1$ for $\epsilon \geq 2M_k$, and  $N\left(\epsilon,\cF_k,\bar C\left(\abs{h' \circ \cF_k}\right)\|\cdot\|_{2,\gamma}\right) = N\left((\bar C\left(\abs{h' \circ \cF_k}\right))^{-1}\epsilon,\cF_k,\|\cdot\|_{2,\gamma}\right) $ (note that both sides equal 1 when $\bar C\left(\abs{h' \circ \cF_k}\right)=0$).
 \item is because $\cF_k$ is assumed to be a VC-type class with constants $l_{\mathsf{vc}}(\cF_k) \geq e$ and $u_{\mathsf{vc}}(\cF_k)$ corresponding to envelope $M_k$;
 \item is since $ \int_{0}^{\delta} \sqrt{\log(A/\epsilon)} \dd \epsilon\lesssim \delta \sqrt{\log(A/\delta)}$ for $A \geq e$ and $0 \leq \delta \leq 1$, which in turn follows via integration by parts.
\end{enumerate}
 Taking supremum on both sides of \eqref{upbndgenNNest} over $\mu,\nu$ such that $\mathsf{D}_{h,\cF_k}(\mu,\nu) <\infty$  proves \eqref{genNNubndwrtn}.

\subsection{Proofs for Section  \ref{Sec:NESDs}} \label{Sec:NESDs-proofs}
\subsubsection{Proof of Theorem  \ref{strongcons}}  \label{strongcons-proof}
Let $\mathsf{D}_{\mathcal{G}_k(\mathbf{a}_k,\phi)}(\mu,\nu):=\mathsf{D}_{h_{\mathsf{KL}}, \mathcal{G}_k(\mathbf{a}_k,\phi)}(\mu,\nu)$ be the parametrized (by the NN class $\mathcal{G}_k(\mathbf{a}_k,\phi)$) KL divergence. 
We will use the  following lemma which proves consistency of parametrized KL divergence estimator.  
\begin{lemma}[Parametrized KL divergence estimation] \label{lem:consicomp}
Let  $(\mu,\nu) \in\mathcal{P}^2_{\mathsf{KL}}(\X)$. Then, for any  $0<\rho<1$, and   $n, k_n$,   such that   $k_n^{3/2}(\norm{\cX}+1)e^{k_n (\norm{\cX}+1)} =O \left(n^{(1-\rho)/2}\right)$,
\begin{align}
  \hat{\mathsf{D}}_{\cG_k^*(\phi)}(X^n,Y^n)  \xrightarrow[n\rightarrow \infty]{}    \mathsf{D}_{\cG_k^*(\phi)}(\mu,\nu),\quad  \mathbb{P}-\mbox{a.s.}\label{errbndest}
\end{align}
\end{lemma}
Lemma \ref{lem:consicomp} is proven using Theorem \ref{empesterrbnd}; see Appendix  \ref{lem:consicomp-proof} for details.
\medskip

We proceed with the proof of \eqref{finbndascon}.  Since $\X$ is compact and $f_{\mathsf{KL}} \in \mathsf{C}\left(\X\right)$, it follows from \citep[Theorem 2.1 and 2.8]{stinchcombe1990approximating} that for any $\epsilon>0$, there is a $k_0(\epsilon) \in \NN$, such that for any $k \geq k_0(\epsilon)$, there exists a $g_{\theta_k} \in \cG_k^*(\phi)$ with 
  \begin{align}
      \norm{f_{\mathsf{KL}}-g_{ \theta_k}}_{\infty,\X} \leq \epsilon. \label{approxbndwt}
  \end{align}
  This implies 
  \begin{align}
     \lim_{k \rightarrow \infty} \mathsf{D}_{\cG_k^*(\phi)}(\mu,\nu)  =\kl{\mu}{\nu}. \label{approxlim}
  \end{align}
To see this, note that
\begin{align}
 \mathsf{D}_{\cG_k^*(\phi)}(\mu,\nu) \leq  \kl{\mu}{\nu},\quad \forall~  k \in \mathbb{N}, \label{limitapprkl}
\end{align}
 by \eqref{CC-charact} since $g \in \cG_k^*(\phi)$ is continuous and  bounded ($\norm{g}_{\infty,\X} \leq k(\norm{\cX}+1)+1 \leq 2k+1$ for $\cX=[0,1]^d$). Moreover, the left-hand side (LHS) 
 of \eqref{limitapprkl} is monotonically increasing in $k$, and being bounded, it has a limit point. Thus, to establish \eqref{approxlim}, it suffices to show that this limit point is $\kl{\mu}{\nu}$. 
 
Assume to the contrary that $\lim_{k \rightarrow \infty}   \mathsf{D}_{\cG_k^*(\phi)}(\mu,\nu)<\kl{\mu}{\nu}$. Note that $\cG_k^*(\phi)$ is a compact set and hence the supremum in the variational form of the LHS of \eqref{limitapprkl} is a maximum. Then, defining 
$D(g):=1+\EE_{\mu}[g]-\EE_{\nu}[e^g]$, it follows that there exists $\delta>0$ and $g_{ \bar \theta_{k}} \in \argmax_{g_{\theta} \in \cG_k^*(\phi)} D(g_{\theta})$ 
such that for all $k$,
\begin{align}
 \kl{\mu}{\nu}- D\left(g_{ \bar\theta_{k}}\right) \geq \delta. \label{contradineqkl}
\end{align}
However, we have for all $k \geq k_0(\epsilon)$ that 
\begin{flalign}
   \kl{\mu}{\nu}- D(g_{ \bar \theta_{k}}) &\leq \kl{\mu}{\nu}- D(g_{\theta_k}) \notag\\
   &\leq \EE_{\mu}\left[\abs{f_{\mathsf{KL}}- g_{\theta_k}}\right]+\EE_{\nu}\left[\Big|e^{f_{\mathsf{KL}}}-e^{g_{\theta_k}}\Big|\right] \notag\\
     &\leq  \EE_{\mu}\left[\abs{f_{\mathsf{KL}}- g_{\theta_k}}\right]+\EE_{\nu}\left[\frac{\dd \mu}{\dd \nu}\right]\Big\|1-e^{g_{\theta_k}-f_{\mathsf{KL}}}\Big\|_{\infty,\nu}\notag \\ 
  &\leq   \epsilon+ e^{\epsilon}-1, \notag &&
\end{flalign}
where the final inequality follows from  \eqref{approxbndwt} and $\EE_{\nu}\left[\dd \mu/ \dd \nu\right] \leq 1$. Then, 
taking $\epsilon$ sufficiently small  contradicts \eqref{contradineqkl}, thus proving \eqref{approxlim}.  From this and \eqref{errbndest} with $k=k_n  \rightarrow \infty$,  \eqref{finbndascon}  follows since $k^{3/2}e^{k(\norm{\cX}+1)} < e^{k(2+\delta)}$ for  $\cX=[0,1]^d$, any $\delta>0$, and $k$ sufficiently large. 

\medskip 
Next, we prove \eqref{KLeffbndsimp}. Fix $(\mu,\nu) \in \mathcal{P}^2_{\mathsf{KL}}(M,\X)$, and with some abuse of notation, let $\mathbf{m}=(m_k)_{k \in \NN}$ be a non-decreasing  positive divergent sequence, and note that since $c_{\mathsf{KB}}^\star\left(f_{\mathsf{KL}},\X\right)  \leq M$, we have from \eqref{approxrateklubar}  that for $k$ such that $m_k \geq M$, there exists $g_{\theta_k} \in \mathcal{G}_{k}^{\mathsf{R}}(m_k)$ and $c>0$ satisfying 
 \begin{align}
\big\| f_{\mathsf{KL}}-  g_{\theta_k}\big\|_{\infty,\X} \leq c d^{\frac 12} M k^{-\frac 12}. \label{approxerrKL} 
 \end{align}
Also, since $g_{\theta_k} \in \mathcal{G}_{k}^{\mathsf{R}}(m_k)$ is bounded, we have that $ \kl{\mu}{\nu} \geq \mathsf{D}_{\mathcal{G}_{k}^{\mathsf{R}}(m_k)}(\mu,\nu)$. Then, the following  hold for $k$ such that $m_k \geq M$ and $c^2 d M^2 \leq  k/2$: 
\begin{flalign}
     \abs{\kl{\mu}{\nu}-\mathsf{D}_{\mathcal{G}_{k}^{\mathsf{R}}(m_k)}(\mu,\nu)} &= \kl{\mu}{\nu}- \mathsf{D}_{\mathcal{G}_{k}^{\mathsf{R}}(m_k)}(\mu,\nu) \notag\\
      & \leq \EE_{\mu}\big[\big|f_{\mathsf{KL}}- g_{\theta_k}\big|\big]+\Big\|1-e^{g_{\theta_k}-f_{\mathsf{KL}}}\Big\|_{\infty,\nu}\EE_{\nu}\left[e^{f_{\mathsf{KL}}}\right] \notag \\
    & \lesssim d^{\frac 12}  M k^{-\frac 12}, \notag && 
\end{flalign}
where the last bound follows from \eqref{approxerrKL},  $\EE_{\nu}\big[e^{f_{\mathsf{KL}}}\big]=\EE_{\nu}\left[\dd \mu/\dd \nu\right]=1$, and since
\begin{align}
     \norm{1-e^{g_{\theta_k}-f_{\mathsf{KL}}}}_{\infty,\nu} & \leq \sum_{j=1}^{\infty}\frac{\Big(c d^{\frac 12} M k^{-\frac 12}\Big)^j}{j !} \leq  \sum_{j=1}^{\infty}\left(cd^{\frac 12} M k^{-\frac 12}\right)^j \lesssim d^{\frac 12} M k^{-\frac 12}. \label{simpKLbndgeom}
    \end{align}
Next, note that   $\mathsf{D}_{\mathcal{G}_{k}^{\mathsf{R}}(m_k)}(\mu,\nu) \geq 0$ as $g=0 \in \mathcal{G}_{k}^{\mathsf{R}}(m_k)$. This implies that for $k$ with $m_k < M$ or $ c^2 d M^2 > k/2$, we have $\big|\kl{\mu}{\nu}-\mathsf{D}_{\mathcal{G}_{k}^{\mathsf{R}}(m_k)}(\mu,\nu)\big|\leq  \kl{\mu}{\nu} \leq M$. Consequently
\begin{align}
     \abs{\kl{\mu}{\nu}-\mathsf{D}_{\mathcal{G}_{k}^{\mathsf{R}}(m_k)}(\mu,\nu)} \lesssim_{\mathbf{m},M}  d^{\frac 12}  k^{-\frac 12},\quad \forall~ k\in\NN.\notag
\end{align}
\medskip
On the other hand, since $\bar C\left(\abs{\mathcal{G}_{k}^{\mathsf{R}}(m_k)},\cX\right) \leq 3m_k(\norm{\cX}+1)$ and $\bar C\left(\abs{h_{\mathsf{KL}} \circ \mathcal{G}_{k}^{\mathsf{R}}(m_k)},\cX\right) \leq e^{3m_k(\norm{\cX}+1)}$,  it follows from the above,  \eqref{entrpyintgenvccls} and \eqref{entropyintscclass}
that
\begin{flalign}
&  \mathbb{E}\left[  \abs{\hat{\mathsf{D}}_{\mathcal{G}_{k}^{\mathsf{R}}(m_k)}(X^n,Y^n) -\kl{\mu}{\nu}}\right]  \notag \\
&\qquad \qquad \leq \abs{\mathsf{D}_{\mathcal{G}_{k}^{\mathsf{R}}(m_k)}(\mu,\nu) -\kl{\mu}{\nu}}   + \mathbb{E}\left[\abs{\mathsf{D}_{\mathcal{G}_{k}^{\mathsf{R}}(m_k)}(\mu,\nu)-
    \hat{\mathsf{D}}_{\mathcal{G}_{k}^{\mathsf{R}}(m_k)}(X^n,Y^n)}\right] \notag\\
&\qquad\qquad \lesssim_{\mathbf{m},M}  d^{\frac 12}  k^{-\frac 12}+ d^{\frac 32}m_k(\norm{\cX}+1)e^{3m_k(\norm{\cX}+1)}~ n^{-\frac 12}. \label{finerrbndklapp}&&
 \end{flalign}
Since $\norm{\cX}=1$, choosing $m_k=\log \log k \vee 1$  in \eqref{finerrbndklapp}  yields
\begin{flalign}
 &  \mathbb{E}\left[  \abs{\hat{\mathsf{D}}_{\mathcal{G}_{k}^{\mathsf{R}}\left(m_k\right)}(X^n,Y^n)  -\kl{\mu}{\nu}}\right] \lesssim_{M} d^{\frac 12} k^{-\frac 12} +d^{\frac 32}(\log k)^7~ n^{-\frac 12}. \notag 
\end{flalign} 
Noting that the above bound holds independent of $(\mu,\nu) \in \mathcal{P}^2_{\mathsf{KL}}(M,\X)$, the proof is completed by taking supremum w.r.t. such $\mu,\nu$. 
Note that setting $m_k=M$ in \eqref{finerrbndklapp} and taking supremum w.r.t. such $\mu,\nu$ yields \eqref{finbnderrrate-case1} from Remark~\ref{KLNEratessimp}.

 \subsubsection{Proof of Corollary \ref{minimaxoptKL}} \label{minimaxoptKL-proof}
To show that the minimax risk is $\Omega(n^{-1/2})$, it suffices to consider $d=1$. 
Recall that the minimax risk for differential entropy estimation over the class of one-dimensional Gaussian distributions with unknown variance in a non-empty interval is $\Omega(n^{-1/2})$ \citep[cf., e.g., Appendix A,][]{Goldfeld-2020}. Take $\X=[a,b]$, for some $a,b \in \RR$ with $b>a$, and let $\cP_{\sf{AC}}(\cX)$ be the class of (Lebesgue absolutely continuous) distributions on $\cX$. The differential entropy of $\mu \in \cP_{\sf{AC}}(\X)$ is defined as $\sf{h}(\mu):=-\EE_\mu\left[ \log\left(\dd\mu/\dd\lambda\right)\right]$, and can be equivalently written as $\mathsf{h}(\mu)=\log (b-a)-{\mathsf{D}}_{\mathsf{KL}}\big(\mu\|u_{[a,b]}\big)$, where $u_{[a,b]}$ is the uniform distribution on $\X$. Hence, the minimax rate of KL divergence estimation for distributions in the class $\{(\mu,u_{[a,b]}):\mu \in \cP_{\sf{AC}}(\X)\}$ for any $\tilde{\cP}_{\sf{AC}}(\X) \subseteq \cP_{\sf{AC}}(\X)$ is the same as that of differential entropy estimation for distributions in $\tilde{\cP}_{\sf{AC}}(\X)$.

Let $\cP_{\mathsf{TG}}(\cX)\subset\tilde{P}_{\sf{AC}}(\cX)$ be a class of truncated Gaussians supported on $\cX$ with zero mean and variance in an non-empty interval. Note that the minimax rate for differential entropy estimation over $\cP_{\mathsf{TG}}(\cX)$ equals to that over untrucated Gaussian distribution with zero mean and the same variance constraints. This is since both differential entropies are elementary functions of the variance parameter, when the mean (equals zero) and $a,b$ are given. 
By Proposition \ref{prop:distconddirect} (see Remark \ref{kldivclassdist}), $\mathcal{P}^2_{\mathsf{KL}}(M,\X)$ contains pairs of truncated Gaussians (with variance and means within an interval that depends on $M$) and uniform distributions, which implies that the associated KL divergence minimax estimation risk is $\Omega(n^{-1/2})$. The corollary then follows by noting that the NE achieves $ O\big(n^{-1/2}\big)$ error rate by setting $k=n$ in \eqref{finbnderrrate-case1}.

\subsubsection{Proof of Proposition \ref{prop:distconddirect}}\label{prop:distconddirect-proof} 
The proof of Proposition \ref{prop:bndfourcoeff} (see \eqref{finextchar}) shows that there exists extensions  $\tilde p_{\mathsf{ext}},\tilde q_{\mathsf{ext}} \in \cB_{\bar c_{b,d,\norm{\cX}},2,\X}\big(\RR^d\big) $ $\cap  \mspace{2 mu} \cL^{\mathsf{KB}}_{s_{\mathsf{KB}},b'}\big(\RR^d\big)$ 
of $\tilde p,\tilde q$, respectively, where $\bar c_{b,d,\norm{\cX}}=(\kappa_d d^{3/2}\norm{\X} \vee 1)b'$, with $b'$ as defined in \eqref{constapproxhold}. Set 
$f_{\mathsf{KL}}^{\mspace{1 mu}\mathsf{ext}\mspace{1 mu}}:=\tilde p_{\mathsf{ext}}-\tilde q_{\mathsf{ext}}$, and note that since $\tilde p_{\mathsf{ext}},\tilde q_{\mathsf{ext}} \in \cL^{\mathsf{KB}}_{s_{\mathsf{KB}},b'}\big(\RR^d\big)$, their Fourier transforms exist and the corresponding Fourier inversion formulas hold (see proof of Lemma \ref{lem:suffcondbar}). Also,  we have  
\[ S_2 \left(f_{\mathsf{KL}}^{\mspace{1 mu}\mathsf{ext}\mspace{1 mu}}\right)\norm{\cX} \leq S_2 \left(\tilde p_{\mathsf{ext}}\right)\norm{\cX}+ S_2 \left(\tilde q_{\mathsf{ext}}\right)\norm{\cX} \leq 2 \bar c_{b,d,\norm{\cX}},
\]
where the first inequality uses the definition in \eqref{Cfconstdefval} and linearity of the Fourier transform, while the second is because $\tilde p_{\mathsf{ext}},\tilde q_{\mathsf{ext}} \in \cB_{ \bar c_{b,d,\norm{\cX}},2,\X}\big(\RR^d\big) $. 
Moreover, note that
 \begin{align}
     \kl{\mu}{\nu}=\EE_{\mu}\left[f_{\mathsf{KL}}\right]=\EE_{\mu}\left[\log p-\log q\right] \leq 2b,\notag
 \end{align}
 where the final inequality is due to $\log p=\tilde p|_{\X}$ and $\log q=\tilde q|_{\X}$, for $\tilde p,\tilde q \in \mathsf{C}_b^{s_{\mathsf{KB}}}(\Ucal)$. 
Lastly, since $f_{\mathsf{KL}}=f_{\mathsf{KL}}^{\mspace{1 mu}\mathsf{ext}\mspace{1 mu}}\big|_\cX$, it follows that $(\mu,\nu)\in \mathcal{P}^2_{\mathsf{KL}}(M,\X) $ with $M= 2 \bar c_{b,d,\norm{\cX}} \vee 2b $, and the proposition then follows from Theorem \ref{strongcons}.

\subsubsection{Proof of Theorem \ref{strongconschisq}} \label{strongconschisq-proof}
Let $\chi^2_{\mathcal{G}_k(\mathbf{a}_k,\phi)}(\mu,\nu):=\mathsf{D}_{h_{\chi^2}, \mathcal{G}_k(\mathbf{a}_k,\phi)}(\mu,\nu)$. 
We will use the  lemma below which proves consistency of parametrized $\chi^2$ divergence estimator  (see Appendix \ref{lem:consicompchisq-proof} for proof).
 \begin{lemma}[Parametrized $\chi^2$ divergence estimation] \label{lem:consicompchisq}
Let  $(\mu,\nu) \in\mathcal{P}_{\chi^2}^2(\X)$. Then,  for any $0<\rho<1$, and $n,k_n$ such that $k_n^{5/2}(\norm{\cX}+1)^2= O \left(n^{(1-\rho)/2}\right)$, we have
\begin{align}
  \hat{\chi}^2_{\cG_k^*(\phi)}(X^n,Y^n)   \xrightarrow[n\rightarrow \infty]{}     \chi^2_{\cG_k^*(\phi)}(\mu,\nu),\quad  \mathbb{P}-\mbox{a.s.} \label{errbndestchisq}
\end{align}
\end{lemma}

\medskip

Proceeding with the proof of Theorem \ref{strongconschisq},  \eqref{finbndasconchisq} follows from \eqref{errbndestchisq}  using arguments similar to those used to establish \eqref{finbndascon} and steps leading to \eqref{chisqfinbndrate} below; details are omitted. 

To prove \eqref{chisqsimprate}, fix $(\mu,\nu)\in \mathcal{P}^2_{\chi^2}(M,\X)$, and  let $\mathbf{m}=(m_k)_{k \in \NN}$ be a non-decreasing positive divergent  sequence. 
Since $c_{\mathsf{KB}}^\star\left(f_{\chi^2},\X\right)  \leq M$, we have from \eqref{approxrateklubar}  that for $k$ such that $m_k \geq M$, there exists $g_{\theta_k} \in \mathcal{G}_{k}^{\mathsf{R}}(m_k)$ with
\begin{align}
     & \big\| f_{\chi^2}-  g_{\theta_k}\big\|_{\infty,\X} \lesssim M d^{\frac 12}  k^{-\frac 12}. \label{sqrtbndappchisq}
\end{align}
Also, $ \chisq{\mu}{\nu} \geq \chi^2_{\mathcal{G}_{k}^{\mathsf{R}}(m_k)}(\mu,\nu) $ because $g \in \mathcal{G}_{k}^{\mathsf{R}}(m_k)$ is bounded. 
  Then, we have
\begin{flalign}
    \big|\chisq{\mu}{\nu}- \chi^2_{\mathcal{G}_{k}^{\mathsf{R}}(m_k)}(\mu,\nu)\big|
    &= \chisq{\mu}{\nu}- \chi^2_{\mathcal{G}_{k}^{\mathsf{R}}(m_k)}(\mu,\nu) \notag\\
      &\leq \chisq{\mu}{\nu}-  \EE_{\mu}\big[g_{\theta_k}\big]-\EE_{\nu}\Big[g_{\theta_k}+0.25g_{\theta_k}^2\Big] \notag\\
     &\leq \EE_{\mu}\left[\big|f_{\chi^2}- g_{\theta_k}\big|\right]+\EE_{\nu}\left[\big|f_{\chi^2}- g_{\theta_k}\big|+0.25\big|f_{\chi^2}^2- g_{\theta_k}^2\big|\right] \label{intermchisqbnd}\\
       &\lesssim M d^{\frac 12}  k^{-\frac 12}+\EE_{\nu}\left[0.25\big|f_{\chi^2}- g_{\theta_k}\big|\big|f_{\chi^2}+ g_{\theta_k}\big|\right]\notag\\
     &\lesssim  M d^{\frac 12}  k^{-\frac 12}+\EE_{\nu}\Big[0.25\big|f_{\chi^2}- g_{\theta_k}\big|^2\mspace{-3 mu}+\mspace{-3 mu}0.5\big|f_{\chi^2}- g_{\theta_k}\big|\abs{f_{\chi^2}}\Big] \notag \\
     & \lesssim  M d^{\frac 12} k^{-\frac 12} +M^2dk^{-1} +0.5\big\|f_{\chi^2}- g_{\theta_k}\big\|_{\infty,\nu} \EE_{\nu}\big[|f_{\chi^2}|\big]\label{chisqstepexp} \\
     & \lesssim M(M+1) d k^{-\frac 12}, \notag 
\end{flalign}
where the final inequality is due to \eqref{sqrtbndappchisq} and  since $\EE_{\nu}\left[\abs{f_{\chi^2}}\right] \leq \EE_{\nu}\left[2(\dd \mu/\dd \nu)+2\right] \leq 4$. 

Since $g=0 \in \mathcal{G}_{k}^{\mathsf{R}}(m_k)$, for $k$ such that $m_k<M$, we have 
\begin{flalign}
    & \abs{\chisq{\mu}{\nu}- \chi^2_{\mathcal{G}_{k}^{\mathsf{R}}(m_k)}(\mu,\nu)} = \chisq{\mu}{\nu}- \chi^2_{\mathcal{G}_{k}^{\mathsf{R}}(m_k)}(\mu,\nu) \leq \chisq{\mu}{\nu} \leq M. \notag
\end{flalign}
Hence,
\begin{flalign}
    & \abs{\chisq{\mu}{\nu}- \chi^2_{\mathcal{G}_{k}^{\mathsf{R}}(m_k)}(\mu,\nu)} \lesssim_{\mathbf{m},M} d k^{-\frac 12},\quad \forall \mspace{2 mu} k \in \NN . \label{chisqfinbndrate}
\end{flalign}
Since $\bar C\left(\abs{\mathcal{G}_{k}^{\mathsf{R}}(m_k)},\cX\right) \leq 3m_k(\norm{\cX}+1)$ and $\bar C\big(\big|h_{\chi^2}' \circ \mathcal{G}_{k}^{\mathsf{R}}(m_k)\big|,\cX\big) \leq 1.5m_k(\norm{\cX}+1)+1$,
 \begin{flalign}
&     \mathbb{E}\left[  \abs{\hat{\chi}^2_{\mathcal{G}_{k}^{\mathsf{R}}(m_k)}(X^n,Y^n) -\chisq{\mu}{\nu}}\right] \notag \\
&\qquad\qquad\leq \abs{ \chi^2_{\mathcal{G}_{k}^{\mathsf{R}}(m_k)}(\mu,\nu) -\chisq{\mu}{\nu}} +    \mathbb{E}\left[\abs{\chi^2_{\mathcal{G}_{k}^{\mathsf{R}}(m_k)}(\mu,\nu)-
    \hat{\chi}^2_{\mathcal{G}_{k}^{\mathsf{R}}(m_k)}(X^n,Y^n)}\right] \notag \\
&\qquad\qquad\lesssim_{M,\mathbf{m}} d k^{-\frac 12}+ d^{\frac 32}m_k^2(\norm{\cX}+1)^2 n^{-\frac 12}, \label{chisqfinupbnd}  &&
 \end{flalign}
 where the last inequality uses \eqref{entrpyintgenvccls}, \eqref{entropyintscclass} and \eqref{chisqfinbndrate}. Setting $m_k=\log k$ in \eqref{chisqfinupbnd} and taking supremum w.r.t. $(\mu,\nu)\in \mathcal{P}^2_{\chi^2}(M,\X)$  yields \eqref{chisqsimprate}.

 \subsubsection{Proof of Proposition \ref{prop:distconddirect-chisq}} \label{prop:distconddirect-chisq-proof}
It follows from \eqref{finextchar} that there exists  extensions $\tilde p_{\mathsf{ext}},\tilde q_{\mathsf{ext}} \in \cB_{ \bar c_{b,d,\norm{\cX}},2,\X}\big(\RR^d\big) \cap \tilde{\cL}_{s_{\mathsf{KB}},b'}\big(\RR^d\big)$ of $\tilde p, \tilde q \in \mathsf{C}_b^{s_{\mathsf{KB}}}(\Ucal)$, respectively, where $ \tilde{\cL}_{s_{\mathsf{KB}},b'}\big(\RR^d\big)$ is defined in \eqref{squareintclassallder} and  $\bar c_{b,d,\norm{\cX}}:=(\kappa_d d^{3/2}\norm{\X} \vee 1)b'$, with $b'$ from \eqref{constapproxhold}. Let $f_{\chi^2}^{\mspace{1 mu}\mathsf{ext}\mspace{1 mu}}=2 \big(\tilde p_{\mathsf{ext}} \mspace{2 mu} \tilde q_{\mathsf{ext}}-1\big) $ and recall that   $\alpha_{|j}$ denotes a multi-index of order $j$. We have from the product rule of derivatives that for any  $j \in \ZZ_{\geq 0}$,
\[D^{\alpha_{|j}}f_{\chi^2}^{\mspace{1 mu}\mathsf{ext}\mspace{1 mu}}=2\sum_{\alpha_{|j_1}+\alpha_{|j_2}=\alpha_{|j}}~\frac{\alpha_{|j}!}{\alpha_{|j_1}!\,\alpha_{|j_2}!}D^{\alpha_{|j_1}}\tilde p_{\mathsf{ext}}  D^{\alpha_{|j_2}}\tilde q_{\mathsf{ext}}-D^{\alpha_{|j}}2,\]
where $\alpha !:=\prod_{i=1}^d\alpha_i!$. Also, note  from \eqref{derbndinu} and \eqref{allderivsqint} that for $0 \leq j \leq s_{\mathsf{KB}}$, $\tilde p_{\mathsf{ext}}$,  $\tilde q_{\mathsf{ext}}$ satisfies
\begin{subequations}
\begin{equation}
   \big\|D^{\alpha_{|j}}\tilde p_{\mathsf{ext}}\big\|_{\infty,\RR^d} \vee \big\|D^{\alpha_{|j}}\tilde q_{\mathsf{ext}}\big\|_{\infty,\RR^d} \leq \hat b \leq b',\notag
\end{equation}
\begin{equation}
 \big\|D^{\alpha_{|j}}\tilde p_{\mathsf{ext}}\big\|_{2,\RR^d} \vee \big\|D^{\alpha_{|j}}\tilde q_{\mathsf{ext}}\big\|_{2,\RR^d}   \leq b'. \notag
\end{equation}
\end{subequations}
Combining these observations, we have  for $0 \leq j \leq s_{\mathsf{KB}}$ that
\begin{flalign}
  \Big\|D^{\alpha_{|j}}f_{\chi^2}^{\mspace{1 mu}\mathsf{ext}\mspace{1 mu}}\Big\|_{2,\RR^d}& \leq 2+2\Bigg\|\sum_{\alpha_{|j_1}+\alpha_{|j_2}=\alpha_{|j}}~\frac{\alpha_{|j}!}{\alpha_{|j_1}!\,\alpha_{|j_2}!}D^{\alpha_{|j_1}}\tilde p_{\mathsf{ext}}  D^{\alpha_{|j_2}}\tilde q_{\mathsf{ext}}\Bigg\|_{2,\RR^d} \notag \\
  &\leq 2+2^{j+1} b'^2. && \label{chisqconstbndprop}
\end{flalign}
Similarly, we have $ \big\|D^{\alpha_{|j}}f_{\chi^2}^{\mspace{1 mu}\mathsf{ext}\mspace{1 mu}}\big\|_1<\infty$ for  $0 \leq j \leq s_{\mathsf{KB}}$. Hence, $f_{\chi^2}^{\mspace{1 mu}\mathsf{ext}\mspace{1 mu}} \in \tilde{\cL}_{s_{\mathsf{KB}},2+2^{s_{\mathsf{KB}}+1} b'^2}\big(\RR^d\big)$. From Lemma    \ref{lem:suffcondbar}, it follows that $S_2 \big(f_{\chi^2}^{\mspace{1 mu}\mathsf{ext}\mspace{1 mu}}\big) \leq (2+2^{s_{\mathsf{KB}}+1} b'^2) \kappa_d d^{3/2}$.  
Since $f_{\chi^2}=f_{\chi^2}^{\mspace{1 mu}\mathsf{ext}\mspace{1 mu}}\big|_\cX$, this implies that $c_{\mathsf{KB}}^\star\left(f_{\chi^2},\X\right)  \leq (2+2^{s_{\mathsf{KB}}+1} b'^2) (\kappa_d d^{3/2}\norm{\X} \vee 1)$. 
Also,
\begin{align}
   \chisq{\mu}{\nu}=\mathbb{E}_{\nu}\left[\left(pq^{-1}-1\right)^2 \right] \leq \mathbb{E}_{\nu}\left[p^2q^{-2}+1 \right] \leq b^2+1.\notag
\end{align}
The claim then follows from Theorem \ref{strongconschisq} by noting that $b'^2 \leq \bar c_{b,d,\norm{\cX}}^2$ and $(\mu,\nu)\in \mathcal{P}^2_{\chi^2}(M,\X) $ with $M= \big(2+2^{s_{\mathsf{KB}}+1} \bar c_{b,d,\norm{\cX}}^2\big ) (\kappa_d d^{3/2}\norm{\X} \vee 1)\vee ( b^2+1)$.

\subsubsection{Proof of Theorem \ref{strongconshel}} \label{strongconshel-proof}
Let $\mathsf{H}^2_{\tilde {\mathcal{G}}_{k,t}(\mathbf{a}_k,\phi)}(\mu,\nu):=\mathsf{D}_{h_{\mathsf{H}^2}, \tilde {\mathcal{G}}_{k,t}(\mathbf{a}_k,\phi)}(\mu,\nu)$.
We need the following lemma (see Appendix \ref{lem:consicomphel-proof} for proof) which  shows that parametrized $\mathsf{H}^2$  distance estimation is consistent.
\begin{lemma}[Parametrized $\mathsf{H}^2$  distance estimation] \label{lem:consicomphel}
Let  $(\mu,\nu) \in\mathcal{P}_{\mathsf{H}^2}^2(\X)$. Then, for any $0<\rho<1$, $t_n \rightarrow 0$, and $n,k_n,$ such that $k_n^{3/2}(\norm{\cX}+1)t_n^{-2}= O \left(n^{(1-\rho)/2}\right)$,
\begin{align}
  \hat{\mathsf{H}}^2_{\tilde{\cG}^*_{k_n,t_n}(\phi)}(X^n,Y^n)   \xrightarrow[n\rightarrow \infty]{}     \mathsf{H}^2_{\tilde{\cG}^*_{k_n,t_n}(\phi)}(\mu,\nu),\quad  \mathbb{P}-\mbox{a.s.} \label{errbndesthel}
\end{align}

\end{lemma}
\medskip

Continuing  with the proof of Theorem \ref{strongconshel}, we first prove \eqref{finbndasconhel}. 
Fix $(\mu,\nu) \in \mathcal{P}_{\mathsf{H}^2}^2(\X)$. Recall that $f_{\mathsf{H}^2}=1-\left(\dd \mu /\dd \nu\right)^{- 1/2}$. Since $\norm{\dd \mu/\dd \nu }_{\infty,\eta} \leq M$ by assumption,  we have $\norm{(1-f_{\mathsf{H}^2})}_{\infty,\eta}\geq  M^{-1/2}$. 
It follows from \citep[Theorem 2.1 and 2.8]{stinchcombe1990approximating} and the definition of $ \tilde{\cG}^*_{k,t}(\phi)$ that for any $\epsilon>0$, there exists $k_0(\epsilon) \in \NN$  and $g_{\theta_k} \in \tilde{\cG}^*_{k,M^{-1/2}}(\phi)$ such that for all $k \geq k_0(\epsilon)$,
  \begin{align}
\norm{f_{\mathsf{H}^2}-g_{\theta_k}}_{\infty,\eta} \leq \epsilon. \label{approxbndwthel}
  \end{align}
  Then, noting that $ \mathsf{H}^2(\mu,\nu) \geq   \mathsf{H}^2_{\tilde{\cG}^*_{k,M^{-1/2}}(\phi)}\mspace{-2 mu}(\mu,\nu)$, we have 
  \begin{flalign}
    \Big|\mathsf{H}^2(\mu,\nu)- \mathsf{H}^2_{\tilde{\cG}^*_{k,M^{-1/2}}(\phi)}\mspace{-2 mu}(\mu,\nu)\Big|  &= \mathsf{H}^2(\mu,\nu)- \mathsf{H}^2_{\tilde{\cG}^*_{k,M^{-1/2}}(\phi)}\mspace{-2 mu}(\mu,\nu)  \notag\\
     &\leq \EE_{\mu}\Big[\big|f_{\mathsf{H}^2}- g_{\theta_k}\big|\Big]+\EE_{\nu}\Big[\Big| f_{\mathsf{H}^2}(1-f_{\mathsf{H}^2})^{-1}-g_{\theta_k}\big(1-g_{\theta_k}\big)^{-1}\Big|\Big] \notag\\
     &\leq  \EE_{\mu}\Big[\big|f_{\mathsf{H}^2}- g_{\theta_k}\big|\Big]\mspace{-4 mu}+\mspace{-3 mu}\EE_{\nu}\Big[ \Big|\big(f_{\mathsf{H}^2}- g_{\theta_k}\big)(1-f_{\mathsf{H}^2})^{-1}\big(1-g_{\theta_k}\big)^{-1}\Big|\Big]\notag\\
   &\leq  \epsilon+ M\epsilon, \label{consishel}&&
\end{flalign}
where the final inequality uses \eqref{approxbndwthel}, $\big\|1-f_{\mathsf{H}^2}\big\|_{\infty,\eta} \wedge \big\|1-g_{\theta_k}\big\|_{\infty,\eta} \geq  M^{-1/2}$.  Since $\epsilon>0$ is arbitrary, this implies (similarly to \eqref{approxlim} in Theorem \ref{strongcons}) that
  \begin{align}
     \lim_{k \rightarrow \infty}  \mathsf{H}^2_{\tilde{\cG}^*_{k,M^{-1/2}}(\phi)}\mspace{-2 mu}(\mu,\nu) =\mathsf{H}^2(\mu,\nu).\notag
  \end{align} 
  Then, \eqref{finbndasconhel}  follows from  \eqref{errbndesthel} and \eqref{consishel}.
  
  \medskip

Next, we prove \eqref{helsimperrbnd}. Fix $(\mu,\nu) \in \mathcal{P}^2_{\mathsf{H}^2}(M,\X)$.  By some abuse of notation, let $\mathbf{m}=(m_k)_{k \in \NN}$ and $\mathbf{t}=(t_k)_{k \in \NN}$ denote a non-decreasing positive divergent sequence and a non-increasing sequence tending to zero, respectively. Since $\norm{\dd \mu/\dd \nu }_{\infty,\eta} \leq M$, we have $\norm{1-f_{\mathsf{H}^2}}_{\infty,\eta} \geq M^{-1/2}$.  Using $t_k \rightarrow 0$, it then  follows from \eqref{approxrateklubar} that for $k$ such that $t_k  \leq M^{-1/2}$ and $m_k \geq M$, there exists $g_{\theta_k} \in \tilde{\mathcal{G}}^{\mathsf{R}}_{k,t_k}(m_k)$ with
\begin{align}
     &  \norm{ f_{\mathsf{H}^2}-  g_{\theta_k}}_{\infty,\eta} \lesssim M d^{\frac 12} k^{-\frac 12 }. \label{sqrtbndhel}
\end{align}
 Then, following the arguments leading to the penultimate step in \eqref{consishel}, we have
\begin{flalign}
    &\Big|\mathsf{H}^2(\mu,\nu)- \mathsf{H}^2_{\tilde{\mathcal{G}}^{\mathsf{R}}_{k,t_k}(m_k)}\mspace{-2 mu}(\mu,\nu)\Big|  \notag \\
     &\qquad \qquad \qquad  \leq  \EE_{\mu}\left[\big|f_{\mathsf{H}^2}- g_{\theta_k}\big|\right]+\EE_{\nu}\left[ \Big|\big(f_{\mathsf{H}^2}- g_{\theta_k}\big)\big(1-f_{\mathsf{H}^2}\big)^{-1}\big(1-g_{\theta_k}\big)^{-1}\Big|\right]\notag \\
   &\qquad \qquad \qquad  \leq  \norm{f_{\mathsf{H}^2}- g_{\theta_k}}_{\infty,\mu}+\norm{f_{\mathsf{H}^2}- g_{\theta_k}}_{\infty,\nu}\EE_{\nu}\left[ \abs{\big(1-f_{\mathsf{H}^2}\big)^{-1}\big(1-g_{\theta_k}\big)^{-1}}\right]\notag\\
      &\qquad \qquad \qquad  \lesssim_M  d^{\frac 12} (1+t_k^{-1}) k^{-\frac 12},
     \notag &&
\end{flalign}
where the final inequality is  due to  \eqref{sqrtbndhel}, $1-g_{\theta_k}(x) \geq t_k$ for any $x \in \RR^d$,  and 
\begin{align}
   \EE_{\nu}\left[ \abs{\big(1-f_{\mathsf{H}^2}\big)^{-1}}\right]  = \EE_{\nu}\left[\sqrt{\frac{\dd \mu}{\dd \nu}}\right] \leq \sqrt{\EE_{\nu}\left[\frac{\dd \mu}{\dd \nu}\right]}=1.\notag
\end{align}
Moreover,  since $g=0 \in \tilde{\mathcal{G}}^{\mathsf{R}}_{k,t_k}(m_k)$, for $k$ such that $m_k<M$ or $t_k  > M^{-1/2}$, we obtain 
\begin{flalign}
    &\Big| \mathsf{H}^2(\mu,\nu)- \mathsf{H}^2_{\tilde{\mathcal{G}}^{\mathsf{R}}_{k,t_k}(m_k)}(\mu,\nu)\Big|=\mathsf{H}^2(\mu,\nu)- \mathsf{H}^2_{\tilde{\mathcal{G}}^{\mathsf{R}}_{k,t_k}(m_k)}(\mu,\nu) \leq  \mathsf{H}^2(\mu,\nu) \leq 2, \notag
\end{flalign}
where the last inequality follows from
\begin{align}
     \mathsf{H}^2(\mu,\nu)=\EE_{\nu}\left[\left(\sqrt{\frac{\dd \mu}{\dd \nu}}-1\right)^2\right] \leq \EE_{\nu}\left[\frac{\dd \mu}{\dd \nu}+1 \right] \leq 2.\notag
\end{align}
Thus, for all $k$, we have
\begin{flalign}
    \Big|\mathsf{H}^2(\mu,\nu)- \mathsf{H}^2_{\tilde{\mathcal{G}}^{\mathsf{R}}_{k,t_k}(m_k)}(\mu,\nu)\Big|  
      & \lesssim_{M,\mathbf{m},\mathbf{t}} d^{\frac 12} (1+t_k^{-1}) k^{-\frac 12}. \label{approxerrhelcase2}
\end{flalign}
Noting that $\bar C\big(\big|\tilde{\mathcal{G}}^{\mathsf{R}}_{k,t_k}(m_k)\big|,\cX\big) \leq 3m_k(\norm{\cX}+1)$ and $\bar C\big(\big|h_{\mathsf{H}^2}' \circ \tilde{\mathcal{G}}^{\mathsf{R}}_{k,t_k}(m_k)\big|,\cX\big) \leq t_k^{-2}$, it follows from  \eqref{entrpyintgenvccls}, \eqref{entropyintscclass} and \eqref{approxerrhelcase2} that
 \begin{flalign}
  \mathbb{E}&\left[  \Big|\hat{\mathsf{H}}^2_{\tilde{\mathcal{G}}^{\mathsf{R}}_{k,t_k}(m_k)}(X^n,Y^n) -\mathsf{H}^2(\mu,\nu)\Big|\right] \notag \\
&\qquad\leq \Big|\mathsf{H}^2(\mu,\nu)- \mathsf{H}^2_{\tilde{\mathcal{G}}^{\mathsf{R}}_{k,t_k}(m_k)}(\mu,\nu)\Big|   + \mathbb{E}\left[\Big|\hat{\mathsf{H}}^2_{\tilde{\mathcal{G}}^{\mathsf{R}}_{k,t_k}(m_k)}(X^n,Y^n)-
    \mathsf{H}^2_{\tilde{\mathcal{G}}^{\mathsf{R}}_{k,t_k}(m_k)}(\mu,\nu)\Big|\right] \notag \\
 &\qquad\lesssim_{M,\mathbf{m},\mathbf{t}}   d^{\frac 12}(1+t_k^{-1}) k^{-\frac 12} + d^{\frac 32} (\norm{\cX}+1)m_k t_k^{-2} n^{-\frac 12}. \label{finbndratehelcase2} &&
 \end{flalign}
Noting that the above bound holds for any $(\mu,\nu) \in \mathcal{P}^2_{\mathsf{H}^2}(M,\X)$, and setting $m_k= \log k$,  $t_k=(\log k)^{-1}$, we obtain  \eqref{helsimperrbnd}.

 \subsubsection{Proof of Proposition \ref{prop:distconddirect-hel}} \label{prop:distconddirect-hel-proof}
As in the proof of Proposition \ref{prop:distconddirect-chisq}, \eqref{finextchar} yields that there exists extensions $\tilde p_{\mathsf{ext}},\tilde q_{\mathsf{ext}} \in \cB_{\bar c_{b,d,\norm{\cX}},2,\X}\big(\RR^d\big) \cap \tilde{\cL}_{s_{\mathsf{KB}},b'}\big(\RR^d\big)$ of $\tilde p,\tilde q$, respectively. Let $f_{\mathsf{H}^2}^{\mspace{2 mu}\mathsf{ext}}=1-\tilde p_{\mathsf{ext}} \cdot \tilde q_{\mathsf{ext}} $. Then, following steps leading to \eqref{chisqconstbndprop}, we obtain for $0 \leq j \leq s_{\mathsf{KB}}$ that
\begin{flalign}
  \norm{D^{\alpha_{|j}}f_{\mathsf{H}^2}^{\mspace{2 mu}\mathsf{ext}}}_2 \leq 1+\Bigg\|\sum_{\alpha_{|j_1}+\alpha_{|j_2}=\alpha_{|j}}\frac{\alpha_{|j}!}{\alpha_{|j_1}!\,\alpha_{|j_2}!}D^{\alpha_{|j_1}}\tilde p_{\mathsf{ext}}~ D^{\alpha_{|j_2}}\tilde q_{\mathsf{ext}}\Bigg\|_2 
  &\leq 1+2^{j} b'^2. \notag 
\end{flalign}
Similarly, $\big\|D^{\alpha_{|j}}f_{\mathsf{H}^2}^{\mspace{2 mu}\mathsf{ext}}\big\|_1<\infty$ for $0 \leq j \leq s_{\mathsf{KB}}$. Hence, $f_{\mathsf{H}^2}^{\mspace{2 mu}\mathsf{ext}} \in \tilde{\cL}_{s_{\mathsf{KB}},1+2^{s_{\mathsf{KB}}} b'^2}\big(\RR^d\big)$, which yields via Lemma    \ref{lem:suffcondbar} that $S_2 \big(f_{\mathsf{H}^2}^{\mspace{2 mu}\mathsf{ext}}\big) \leq (1+2^{s_{\mathsf{KB}}} b'^2) \kappa_d d^{3/2}$.  
Since $f_{\mathsf{H}^2}=f_{\mathsf{H}^2}^{\mspace{2 mu}\mathsf{ext}}\big|_\cX$, this implies that $c_{\mathsf{KB}}^\star\left(f_{\mathsf{H}^2},\X\right) \leq (1+2^{s_{\mathsf{KB}}} b'^2)(\kappa_d d^{3/2}\norm{\X} \vee 1)$.  Moreover,  we have $\norm{\dd \mu/\dd \nu}_{\infty,\eta} =   \norm{ pq^{-1}}_{\infty,\eta} \leq  b^2$.
 Hence, $(\mu,\nu) \in \mathcal{P}^2_{\mathsf{H}^2}(M,\X)$ with  $M=(\kappa_dd^{3/2}\norm{\X} \vee 1)\big(1+2^{s_{\mathsf{KB}}} \bar c^2_{b,d,\norm{\X}}\big) \vee b^2$ since $b'^2 \leq \bar c_{b,d,\norm{\cX}}^2$. The claim then follows from Theorem \ref{strongconshel}. 

\subsubsection{Proof of Theorem \ref{TVerrbnd}} \label{TVerrbnd-proof}
Let $\tvn{\mu}{\nu}{\bar{\mathcal{G}}_{k}(\mathbf{a},\phi)}:=\mathsf{D}_{h_{\mathsf{TV}}, \bar{\mathcal{G}}_k(\mathbf{a},\phi)}(\mu,\nu)$. The proof of  Theorem \ref{TVerrbnd} is based on the following lemma which establishes consistency of the parametrized TV distance estimator (see Appendix \ref{lem:consicomptv-proof} for proof).
 \begin{lemma}[Parametrized TV distance estimation] \label{lem:consicomptv}
Let $\mu,\nu \in \cP(\X)$. Then, for any $0<\rho<1$, and  $n,k_n$ such that $k_n(\norm{\cX}+1)^{1/2} = O \left(n^{(1-\rho)/2}\right)$,
\begin{align}
 \tvf_{\bar{\cG}_{k_n}^*(\phi)}(X^n,Y^n) \xrightarrow[n\rightarrow \infty]{}     \tvn{\mu}{\nu}{\bar{\cG}_{k_n}^*(\phi)},\quad  \mathbb{P}-\mbox{a.s.}\label{errbndesttv}
\end{align}
\end{lemma}
\medskip

Equipped with Lemma \ref{lem:consicomptv}, we first prove \eqref{finbndascontv}. Since $f_{\mathsf{TV}}$ is not continuous, the universal approximation property of NNs used  in the consistency proofs until now cannot be used directly in this case. However, we will show that there exists a  continuous function approximating $f_{\mathsf{TV}}$ to any desired accuracy, which can in turn be approximated by $\bar{\cG}_k^*(\phi)$ arbitrary well.

Fix $\mu,\nu\in\cP(\cX)$. Let  $p$ and $q$ denote the densities of $\mu$ and $\nu$ w.r.t. $\eta=0.5(\mu+\nu) \in \cP(\X)$, and let $\mathcal{C}^*$ be the set defined in \eqref{setdefntvopt}. Note that $\norm{p \vee q}_{\infty,\eta} \leq 2$. Also, observe that $\mathcal{C}^*$ and $\X \setminus \mathcal{C}^*$ are Borel sets, since $p(x)$ and $q(x)$ are Borel measurable by definition, and hence so is $p(x)-q(x)$. Since $\eta \in \cP(\X)$ is a regular probability measure,   for any $\epsilon>0$, there exists  compact sets $\mathcal{C}$, $\bar{\mathcal{C}}$, open sets $\mathcal{U},\bar{\mathcal{U}}$ such that $\mathcal{C} \subseteq \mathcal{C}^* \subseteq \mathcal{U} $, $\bar{\mathcal{C}} \subseteq \X \setminus \mathcal{C}^* \subseteq \bar{\mathcal{U}} $  and 
\begin{align}
    \eta(\Ucal \setminus \cC) \vee \eta(\bar{\Ucal} \setminus \bar{\cC}) \vee \eta(\bar{\Ucal} \cap  \mathcal{C}^*) \vee\eta(\Ucal \cap (\X \setminus \mathcal{C}^*)) \leq 0.25 \epsilon, \notag
\end{align}
along with continuous (Urysohn) functions $\zeta_{\mathcal{C}^*}:\mathbb{R}^d \rightarrow [0,1]$, $\zeta_{\X \setminus \mathcal{C}^*}: \mathbb{R}^d \rightarrow [0,1]$ such that
\begin{subequations}
\begin{equation}
    \zeta_{\mathcal{C}^*}(x)=\begin{cases}
    1,&x \in \mathcal{C}, \\
    0,&x \in \RR^d \setminus \mathcal{U},
    \end{cases} \notag
\end{equation}
\begin{equation}
 \zeta_{\X \setminus \mathcal{C}^*}(x)=\begin{cases}
    1,&x \in  \bar{\mathcal{C}}, \\
    0,&x \in \RR^d \setminus \bar{\mathcal{U}}.
    \end{cases} \notag
\end{equation}
\end{subequations}
Hence, 
\begin{align}
      & \EE_{\mu} \left[\abs{\ind_{\mathcal{C}^*}-\zeta_{\mathcal{C}^*}}\right] \vee  \EE_{\nu} \left[\abs{\ind_{\mathcal{C}^*}-\zeta_{\mathcal{C}^*}}\right] \leq  \mathbb{E}_{\eta} \left[(p \vee q)\abs{\ind_{\mathcal{C}^*}-\zeta_{\mathcal{C}^*}}\right] \leq 0.25 \norm{p \vee q}_{\infty,\eta}  \epsilon, \label{triangineq3}  \\
      &  \EE_{\mu} \left[\abs{\ind_{\X \setminus  \mathcal{C}^*}-\zeta_{\X \setminus  \mathcal{C}^*}}\right] \vee \EE_{\nu} \left[\abs{\ind_{\X \setminus  \mathcal{C}^*}-\zeta_{\X \setminus  \mathcal{C}^*}}\right] \leq  \mathbb{E}_{\eta} \left[(p \vee q)\abs{\ind_{\X \setminus  \mathcal{C}^*}-\zeta_{\X \setminus  \mathcal{C}^*}}\right] \notag \\
     & \qquad \qquad  \qquad  \qquad \qquad \qquad  \qquad  \qquad \qquad \quad    \leq 0.25 \norm{p \vee q}_{\infty,\eta} \epsilon. \label{triangineq4}
\end{align}
Let  $\zeta(x)=\zeta_{\mathcal{C}^*}(x)-\zeta_{\X \setminus \mathcal{C}^*}(x)$. Note that $\zeta(x) \in [-1,1]$, $\zeta(x)=1$ for $x \in \cC \setminus \bar \Ucal$ and $\zeta(x)=-1$ for $x \in \bar \cC \setminus \Ucal$. Since $\zeta(\cdot)$ is a continuous function, it follows from \citep[Theorem 2.1 and 2.8]{stinchcombe1990approximating} that for any $\epsilon>0$ and $k \geq k_0(\epsilon)$, there exists a $\tilde g \in \cG_k^*(\phi)$ such that $  \norm{\zeta-\tilde g}_{\infty,\X} \leq  \epsilon$.  Since $\norm{\zeta}_{\infty} \leq 1$, it then follows from the definition of $\bar{\cG}_k^*(\phi)$ that there exists $g^* \in \bar{\cG}_k^*(\phi)$ such that
    \begin{align}
      \norm{\zeta-g^*}_{\infty,\X} \leq  \epsilon.  \label{contfnapptv}
  \end{align}
 
Let $ \tilde \delta_{\mathsf{TV}}(g):=\EE_{\mu}[g]-\EE_{\nu}\left[g\right]$. Then, we have for $k \geq k_0(\epsilon)$ that
 \begin{flalign}
 &\big|\tv{\mu}{\nu}- \tvn{\mu}{\nu}{ \bar{\cG}_k^*(\phi)} \big| \notag \\
 &\qquad= \tv{\mu}{\nu}- \tvn{\mu}{\nu}{\bar{\cG}_k^*(\phi)} \notag \\
 &\qquad\leq \tv{\mu}{\nu}-  \tilde \delta_{\mathsf{TV}}(g^*) \notag \\
 & \qquad\leq  \EE_{\mu} \left[\abs{f_{\mathsf{TV}}-g^*}\right]+\EE_{\nu} \left[\abs{f_{\mathsf{TV}}-g^*}\right] \notag\\
 &\qquad\leq  \EE_{\mu} \left[\abs{f_{\mathsf{TV}}-\zeta}+\abs{\zeta-g^*}\right]+\EE_{\nu} \left[\abs{f_{\mathsf{TV}}-\zeta}+\abs{\zeta-g^*}\right] \notag \\
  &\qquad\leq   \EE_{\mu} \left[\abs{\ind_{\mathcal{C}^*}-\zeta_{\mathcal{C}^*}}\right] + \EE_{\nu} \left[\abs{\ind_{\mathcal{C}^*}-\zeta_{\mathcal{C}^*}}\right]+ \EE_{\mu} \left[\abs{\ind_{\X \setminus  \mathcal{C}^*}-\zeta_{\X \setminus  \mathcal{C}^*}}\right] +\EE_{\nu} \left[\abs{\ind_{\X \setminus  \mathcal{C}^*}-\zeta_{\X \setminus  \mathcal{C}^*}}\right]\notag \\
  &\qquad \qquad +\EE_{\mu} \left[\abs{\zeta-g^*}\right]+\EE_{\nu} \left[\abs{\zeta-g^*}\right] \notag \\
 &\qquad\leq   \epsilon\big(\norm{p \vee q}_{\infty,\eta} +2\big) \leq 4 \epsilon, \label{finasconvbndtv} &&
\end{flalign}
where \eqref{finasconvbndtv} follows from \eqref{triangineq3}, \eqref{triangineq4}, \eqref{contfnapptv} and $\norm{p \vee q}_{\infty,\eta} \leq 2$. Since  $\epsilon >0$ is arbitrary,  we have from \eqref{finasconvbndtv} that
  \begin{align}
     \lim_{k \rightarrow \infty}  \tvn{\mu}{\nu}{ \bar{\cG}_k^*(\phi)}=\tv{\mu}{\nu}. \notag
  \end{align}
Taking  $k_n,n$ satisfying $k_n=O\left(n^{(1-\rho)/2}\right)$,  \eqref{finbndascontv} follows from  the above equation and   \eqref{errbndesttv}.

\medskip

Next, we prove \eqref{finbnderrratetv}. Fix $(\mu,\nu) \in\cP_{\mathsf{TV}}^2(M,\X)$ such that $f_{\mathsf{TV}} \in \mathsf{Lip}_{s,1,M}(\X)$. Since $f_{\mathsf{TV}}$ does not belong to the Klusowski-Barron class, we consider approximation of an intermediate function $f_{\mathsf{TV}}^{(t)}$, which is a smoothed version of $f_{\mathsf{TV}}$ and belongs to this class. The smoothing parameter $t$ is then decreased as a function of $k$ at an appropriate rate such that the $L^1$ error between  $f_{\mathsf{TV}}^{(t)}$ and $f_{\mathsf{TV}}$ vanishes as $k \rightarrow \infty$.  For this purpose, consider a non-negative  smoothing kernel $\Phi \in L^1\big(\RR^d\big)$, $\Phi \geq 0$, such that $\int_{\RR^d} \Phi(x) \dd x=1$.  Let $\Phi_t(x):=t^{-d} \Phi(t^{-1}x)$, $t>0$, and 
\begin{align}
f_{\mathsf{TV}}^{(t)}(x):=f_{\mathsf{TV}}* \Phi_t(x)= \int_{\RR^d}  f_{\mathsf{TV}}(x-y) \Phi_t(y)\dd y, \notag 
\end{align}
denote the smoothing of $f_{\mathsf{TV}}$ using $\Phi_t$.

Recalling that $ \tilde \delta_{\mathsf{TV}}(f):=\EE_{\mu}[f]-\EE_{\nu}\left[f\right]$, we have
\begin{align}
 \big|\tv{\mu}{\nu}- \tvn{\mu}{\nu}{\bar{\mathcal{G}}_{k}(\mathbf{a},\phi)} \big|&=\tv{\mu}{\nu}- \tvn{\mu}{\nu}{\bar{\mathcal{G}}_{k}(\mathbf{a},\phi)}\notag \\ 
 &=\tv{\mu}{\nu}-\tilde \delta_{\mathsf{TV}}\Big(f_{\mathsf{TV}}^{(t)}\Big)+\tilde \delta_{\mathsf{TV}}\Big(f_{\mathsf{TV}}^{(t)}\Big)- \tvn{\mu}{\nu}{\bar{\mathcal{G}}_{k}(\mathbf{a},\phi)},\label{splittermtvbnd}
\end{align}
 The first term in \eqref{splittermtvbnd} can be written as follows:
\begin{align}
   \tv{\mu}{\nu}-\tilde \delta_{\mathsf{TV}}\Big(f_{\mathsf{TV}}^{(t)}\Big) = \EE_{\mu}\left[f_{\mathsf{TV}}- f_{\mathsf{TV}}^{(t)}\right]-\EE_{\nu}\left[f_{\mathsf{TV}}- f_{\mathsf{TV}}^{(t)}\right]. \label{sumbndfact1}
\end{align}
Denoting by $p,q$, the respective densities of $\mu,\nu$ w.r.t. Lebesgue measure, we have
\begin{flalign}
\EE_{\mu}\left[f_{\mathsf{TV}}- f_{\mathsf{TV}}^{(t)}\right] &= \int_{\RR^d} \left[f_{\mathsf{TV}}(x)-t^{-d}\int_{\RR^d}  f_{\mathsf{TV}}(y) \Phi\left((x-y)t^{-1}\right)\dd y \right]p(x)\dd x \notag\\
&=\int_{\RR^d} \left[f_{\mathsf{TV}}(x)-\int_{\RR^d}  f_{\mathsf{TV}}(x-tu) \Phi(u) \dd u \right]p(x)\dd x \notag\\
&=\int_{\RR^d} \left[\int_{\RR^d}\left[ f_{\mathsf{TV}}(x)\Phi(u)- f_{\mathsf{TV}}(x-tu) \Phi(u)\right] \dd u \right] p(x)\dd x \notag \\
& \leq \int_{\RR^d} \left[\int_{\RR^d}\abs{ f_{\mathsf{TV}}(x)- f_{\mathsf{TV}}(x-tu) }p(x)\dd x \right] \Phi(u) \dd u  \notag \\
& = \int_{\RR^d} \left[\int_{\RR^d}\abs{ f_{\mathsf{TV}}(x+tu)- f_{\mathsf{TV}}(x) }p(x+tu)\dd x \right] \Phi(u) \dd u \notag \\
& \leq \norm{p}_{\infty,\X} \int_{\RR^d} \left[\int_{\RR^d}\abs{ f_{\mathsf{TV}}(x+tu)- f_{\mathsf{TV}}(x) }\dd x \right] \Phi(u) \dd u \notag \\
& \stackrel{(a)}{\leq}  M \int_{\RR^d}  \xi_{1,1}\left(f_{\mathsf{TV}},t \norm{u}\right)\Phi(u) \dd u \notag \\
& \stackrel{(b)}{\leq}  M^2 \int_{\RR^d}  t^s\norm{u}^s \Phi(u) \dd u, \label{pdisttvconerrbnd1} &&
\end{flalign}
where $(a)$ and $(b)$ are due to $(\mu,\nu) \in\cP_{\mathsf{TV}}^2(M,\X)$ and $f_{\mathsf{TV}} \in \mathsf{Lip}_{s,1,M}(\X)$, respectively.
Since \eqref{pdisttvconerrbnd1} also holds for $\nu$ in place of $\mu$, we have from \eqref{sumbndfact1} that 
\begin{align}
  \abs{ \tv{\mu}{\nu}-\tilde \delta_{\mathsf{TV}}\Big(f_{\mathsf{TV}}^{(t)}\Big) } \leq 2M^2 \int_{\RR^d}  t^s\norm{u}^s \Phi(u) \dd u. \label{combbndtvconerr}
\end{align}
Next, note that 
\begin{align}
    \norm{f_{\mathsf{TV}}^{(t)}}_1 \leq \int_{\RR^d}  \int_{\RR^d} \abs{ f_{\mathsf{TV}}(x-y) \Phi_t(y)}\dd y \dd x 
    &\stackrel{(a)}{\leq} \norm{f_{\mathsf{TV}}}_1 \norm{\Phi_t}_1 \stackrel{(b)}{\leq} \norm{f_{\mathsf{TV}}}_1 \stackrel{(c)}{<}\infty, \notag
\end{align}
where 
\begin{enumerate}[label = (\alph*),leftmargin=15 pt]
\item follows from Minkowski's integral inequality;
\item is due to $\int_{\RR^d} \abs{\Phi_t(y)}\dd y=1$;
\item is since $f_{\mathsf{TV}} \in L^1(\X)$.
   \end{enumerate}
 Hence, the Fourier transform of $f_{\mathsf{TV}}^{(t)}$ exists, and is  given by 
\begin{align}
    \mathfrak{F}\left[f_{\mathsf{TV}}^{(t)}\right] =  \mathfrak{F}[f_{\mathsf{TV}}]   \mathfrak{F}[\Phi_t].\label{convformprod}
\end{align}
Choose $\Phi$ to be standard Gaussian kernel, i.e., $\Phi=\Phi^{\cN}:= (2\pi)^{-d/2} e^{-0.5\norm{x}^2}$.
Then, we have 
\begin{align}
   \norm{\mathfrak{F}\left[f_{\mathsf{TV}}^{(t)}\right]}_1& \stackrel{(a)}{\leq} \norm{f_{\mathsf{TV}}}_1 \int_{\RR^d}  \abs{\mathfrak{F}[\Phi_t](\omega)}\dd\omega \stackrel{(b)}{\leq} M \int_{\RR^d}  \abs{\mathfrak{F}[\Phi](t\omega)}\dd\omega \stackrel{(c)}{\leq}  M \int_{\RR^d}   e^{-\frac 12 t^2\norm{\omega}^2}\dd\omega <\infty, \notag
\end{align}
where
\begin{enumerate}[label = (\alph*),leftmargin=15 pt]
\item follows from \eqref{convformprod} and $ \big\|\mathfrak{F}\left[f_{\mathsf{TV}}\right]\big\|_{\infty} \leq \norm{f_{\mathsf{TV}}}_1$;
\item is via the formula $\mathfrak{F}\big[\Phi\left(t^{-1}\cdot\right)\big](\omega)=t^d \mathfrak{F}[\Phi]\left(t \omega\right) $, and $\norm{f_{\mathsf{TV}}}_1 \leq M$ by the definition of Lipschitz seminorm;
\item is since $\mathfrak{F}\big[\Phi^{\cN}\big](\omega)=e^{-\frac 12 \norm{\omega}^2}$.
   \end{enumerate}
Hence, the Fourier representation in Definition \ref{def:barronclass} holds via the Fourier inversion formula for $f_{\mathsf{TV}}^{(t)}$. Then, we can bound the spectral norm as 
\begin{flalign}
    S_2\Big(f_{\mathsf{TV}}^{(t)}\Big) &:= \int_{\RR^d}  \norm{\omega}_1^2\big|\mathfrak{F}\Big[f_{\mathsf{TV}}^{(t)}\Big](\omega)\big|d\omega \notag \\
    & \leq \norm{f_{\mathsf{TV}}}_1 \int_{\RR^d} \norm{\omega}_1^2 \abs{\mathfrak{F}[\Phi_t](\omega)}\dd\omega \notag  \\
    &\leq  Md \int_{\RR^d} \norm{\omega}^2 e^{-\frac 12 t^2\norm{\omega}^2}\dd\omega. \notag &&
\end{flalign}
 Evaluating the integral above by converting to  hyperspherical coordinates, we obtain  
\begin{flalign}
     \norm{\cX}S_2\Big(f_{\mathsf{TV}}^{(t)}\Big)   &\leq   \norm{\X}Md \int_{\RR^d} \norm{\omega}^2 e^{-\frac 12 t^2\norm{\omega}^2}\dd\omega  
   =: c_{d,M,\norm{\X},t}, \label{barronconstftv} 
\end{flalign}
where
\begin{align} 
  c_{d,M,\norm{\X},t}&:= \begin{cases}
    (2 \pi)^{\frac 12} \norm{\X}Mt^{-3}, &d=1,\\
       2^{\frac{d+3}{2}}\pi\norm{\X}Mdt^{-(d+2)} \Gamma((d+2)/2) \prod_{j=1}^{d-2}\int_{0}^{\pi}\sin^{d-1-j}(\varphi_j)\dd\varphi_j, &d \geq 2.
    \end{cases}\notag 
\end{align}
Moreover,  $\norm{f_{\mathsf{TV}}}_{\infty} \leq 1$ and $\int_{\RR^d} \abs{\Phi_t(y)}\dd y=1$ implies
\begin{align}
    \Big| f_{\mathsf{TV}}^{(t)}(x)\Big| \leq  \int_{\RR^d} \abs{ f_{\mathsf{TV}}(x-y) \Phi_t(y)}\dd y \leq  \int_{\RR^d} \abs{\Phi_t(y)}\dd y=1.
\end{align}
Since $\big|f_{\mathsf{TV}}^{(t)}(0)\big| \vee \big\|\nabla f_{\mathsf{TV}}^{(t)}(0)\big\| \leq 1 \vee (2 d \pi^{-d})^{1/2}\Gamma\big(0.5(d+1)\big)t^{-1}$ and \eqref{barronconstftv} holds,  there exists $ g_{\theta_k} \in \bar \cG_k^{\mathsf{R}}\left(\hat c_{d,M,\norm{\X},t}\right)$ such that for all $0 <t \leq 1$, 
\begin{align}
\Big\|f_{\mathsf{TV}}^{(t)}- g_{\theta_k}\Big\|_{\infty,\X}\lesssim \hat c_{d,M,\norm{\X},t}d^{\frac 12} k^{-\frac 12}, \label{bndapproxTV}
\end{align}
where $\hat c_{d,M,\norm{\X},t}:= c_{d,M,\norm{\X},t} \vee 1 \vee (2 d \pi^{-d})^{1/2}\Gamma\big(0.5(d+1)\big)t^{-1}$. The existence of  $g_{\theta_k}$ follows by  truncating $g \in \cG_k^{\mathsf{R}}\left(\hat c_{d,M,\norm{\X},t}\right)$ satisfying  \eqref{approxrateklubar} to $[-1,1]$, and noting that  truncation only decreases the approximation error as $\big\|f_{\mathsf{TV}}^{(t)}\big\|_{\infty} \leq 1$. 
Hence, we have 
\begin{flalign}
    \tilde \delta_{\mathsf{TV}}\Big(f_{\mathsf{TV}}^{(t)}\Big)- \tvn{\mu}{\nu}{\bar{\mathcal{G}}_{k}^{\mathsf{R}}(\hat c_{d,M,\norm{\X},t})} &\leq  \tilde \delta_{\mathsf{TV}}\Big(f_{\mathsf{TV}}^{(t)}\Big)- \tilde \delta_{\mathsf{TV}}\big(g_{\theta_k}\big) \notag \\
    & \leq  \EE_{\mu} \left[\abs{f_{\mathsf{TV}}^{(t)}-g_{\theta_k}}\right]+\EE_{\nu} \left[\abs{f_{\mathsf{TV}}^{(t)}-g_{\theta_k}}\right]  \label{bndtvapperrint}\\
  &  \lesssim   \hat c_{d,M,\norm{\X},t}d^{\frac 12} k^{-\frac 12}.\label{errbndtvapprox} &&
\end{flalign}
Next, observe that  \eqref{combbndtvconerr} with $\Phi=\Phi^{\cN}$ yields
\begin{align}
  \abs{ \tv{\mu}{\nu}-\tilde \delta_{\mathsf{TV}}\Big(f_{\mathsf{TV}}^{(t)}\Big) } \leq c_{d,M,s} t^s, \notag
\end{align}
where $c_{d,M,s}:= 2 M^2(2\pi)^{-d/2} \int_{\RR^d}  \norm{u}^s e^{-0.5\norm{u}^2} \dd u$. From this,  \eqref{splittermtvbnd} and \eqref{errbndtvapprox}, we obtain
   \begin{flalign}
 \abs{\tv{\mu}{\nu}- \tvn{\mu}{\nu}{\bar{\mathcal{G}}_{k}^{\mathsf{R}}(\hat c_{d,M,\norm{\X},t})} }&=\tv{\mu}{\nu}- \tvn{\mu}{\nu}{\bar{\mathcal{G}}_{k}^{\mathsf{R}}(\hat c_{d,M,\norm{\X},t})} \lesssim_{d,M,s}  t^s+ \norm{\X} t^{-(d+2)}k^{-\frac 12}.  \notag
   \end{flalign}
   Setting $t=t_k^*:=k^{-1/2(s+d+2)}$ and \begin{align}
   \tilde c_{k,d,s,M,\norm{\X}}:=\hat c_{d,M,\norm{\X},t_k^*}=O_{d,M}\big(\norm{\cX}k^{(d+2)/2(s+d+2)}\big),   \label{consNNparTV}  
   \end{align}
 yields
   \begin{align}
    \abs{\tv{\mu}{\nu}- \tvn{\mu}{\nu}{\bar{\mathcal{G}}_{k}^{\mathsf{R}}\big(\tilde c_{k,d,s,M,\norm{\X}}\big)} }\lesssim_{d,M,s} (\norm{\cX}+1)  k^{-s/2(s+d+2)}.  \label{approxerrtv}
   \end{align}

   \medskip
Finally, we bound the expected empirical estimation error. Note that $\bar C\big(\big|\bar{\mathcal{G}}_{k}^{\mathsf{R}}(a)\big|,\cX\big) \leq 1$ and  $\bar C\big(\big|h_{\mathsf{TV}}' \circ \bar{\mathcal{G}}^{\mathsf{R}}_{k}(a)\big|,\cX\big) =1$. Then, we have 
\begin{flalign}
   & \mathbb{E}\left[  \abs{\tvf_{\bar{\cG}_k^{\mathsf{R}}\big(\tilde c_{k,d,s,M,\norm{\X}}\big)}(X^n,Y^n) -\tvn{\mu}{\nu}{\bar{\mathcal{G}}_{k}^{\mathsf{R}}\big(\tilde c_{k,d,s,M,\norm{\X}}\big)} }\right] \notag \\
   &  \stackrel{(a)}{\lesssim}   n^{-\frac 12} \int_{0}^{1}\sqrt{ \sup_{\gamma \in \cP(\X)} \log N\left(\epsilon,\bar{\mathcal{G}}^{\mathsf{R}}_{k}\big(\tilde c_{k,d,s,M,\norm{\X}}\big),\|\cdot\|_{2,\gamma}\right)}\dd\epsilon \notag \\
    &  \stackrel{(b)}{\leq}   n^{-\frac 12} \int_{0}^{1}\sqrt{ \sup_{\gamma \in \cP(\X)} \log N\left(\epsilon,\mathcal{G}^{\mathsf{R}}_{k}\big(\tilde c_{k,d,s,M,\norm{\X}}\big),\|\cdot\|_{2,\gamma}\right)}\dd\epsilon \notag \\
   &\stackrel{(c)}{\leq}n^{-\frac 12}\int_{0}^{1}\sqrt{ \sup_{\gamma \in \cP(\X)} \log N\left(\epsilon/3,\cG_k^{\dagger}\big(\tilde c_{k,d,s,M,\norm{\X}}\big),\|\cdot\|_{2,\gamma}\right)}\dd\epsilon\notag \\
   & \qquad \qquad \qquad \qquad\qquad \qquad +n^{-\frac 12} (d+1)^{\frac 12} \int_{0}^{1}\sqrt{\log(1+6\tilde c_{k,d,s,M,\norm{\X}}(\norm{\cX}+1))}\dd\epsilon  \notag \\
   &\stackrel{(d)}{\lesssim}n^{-\frac 12}d^{\frac 32} \big(\tilde c_{k,d,s,M,\norm{\X}}(\norm{\cX}+1)\big)^{\frac{d+2}{d+3}} +n^{-\frac 12} d^{\frac 12} \big(\tilde c_{k,d,s,M,\norm{\X}}(\norm{\cX}+1)\big)^{\frac 12},  \notag \\
      &\lesssim n^{-\frac 12}d^{\frac 32} \big(\tilde c_{k,d,s,M,\norm{\X}}(\norm{\cX}+1)+1),  \label{empesterrortvfin} &&
\end{flalign}
where $(a)$ follows from \eqref{entrpyintgenvccls}; $(b)$ is since the pointwise difference between functions in $\bar{\mathcal{G}}^{\mathsf{R}}_{k}(a)$ (with range $[-1,1]$) is less than the difference between the corresponding untruncated functions in $\mathcal{G}^{\mathsf{R}}_{k}(a)$;
$(c)$ is due to \eqref{covnumbbndfin}; and $(d)$ is via steps leading to \eqref{entropyintscclass}.   Then,  \eqref{approxerrtv} and \eqref{empesterrortvfin} implies that 
  \begin{flalign}
      & \mathbb{E}\left[  \abs{\tvf_{\bar{\cG}_k^{\mathsf{R}}(\tilde c_{k,d,s,M,\norm{\X}})}(X^n,Y^n) -\tv{\mu}{\nu}}\right] \notag \\ &\qquad \qquad \qquad \lesssim_{d,M,s} (\norm{\cX}+1)  k^{-s/2(s+d+2)} +n^{-\frac 12}k^{(d+2)/2(s+d+2)}\big(\norm{\cX}^2+1\big)^{\frac 12}. \label{explictbndtvpar}&&
  \end{flalign}
 Recalling that  $\cX=[0,1]^d$ and taking supremum over $(\mu,\nu) \in\cP_{\mathsf{TV}}^2(M,\X)$ such that $f_{\mathsf{TV}} \in \mathsf{Lip}_{s,1,M}(\X)$, we obtain \eqref{finbnderrratetv}.

\subsubsection{Proof of Proposition \ref{TVpropsuffcond}} \label{TVpropsuffcond-proof}
 Since  $p-q \in  \mathcal{T}_{b,N}(\X)$ and $f_{\mathsf{TV}}=\ind_{\{p-q \geq 0\}}-\ind_{\{p-q< 0\}}$, the definition of $\xi_{1,1}(f_{\mathsf{TV}},t)$ yields
 \begin{equation}
 \xi_{1,1}(f_{\mathsf{TV}},t) \leq \begin{cases}
     2 N \lambda(B_d(t)), & t \leq b \\
     2 \norm{f_{\mathsf{TV}}}_1, & \mbox{otherwise}.
     \end{cases}      \notag
 \end{equation}
Hence, for any $0<s\leq 1$, it holds that
\begin{flalign}
    \norm{f_{\mathsf{TV}}}_{\mathsf{Lip}(s,1)} &=\norm{f_{\mathsf{TV}}}_{1}+\sup_{t>0} t^{-s}\xi_{1,1}(f_{\mathsf{TV}},t) \notag \\
    &= \norm{f_{\mathsf{TV}}}_{1}+\sup_{0<t \leq b} t^{-s}\xi_{1,1}(f_{\mathsf{TV}},t) \vee  \sup_{t > b} t^{-s}\xi_{1,1}(f_{\mathsf{TV}},t)\notag \\
        &= \norm{f_{\mathsf{TV}}}_{1}+ \sup_{0<t \leq b} t^{-s} 2 N \lambda(B_d(t)) \vee  \sup_{t > b} t^{-s}2\norm{f_{\mathsf{TV}}}_{1}\notag \\
    &=\lambda(\X)+2N \pi^{\frac d 2}b^{d-s} \left(\Gamma(0.5d+1)\right)^{-1} \vee 2 b^{-s}\lambda(\X), \label{genbndsuppMval} &&
\end{flalign}
where $\lambda$ denotes the Lebesgue measure and $\Gamma$ is the gamma function. Hence, $f_{\mathsf{TV}} \in \mathsf{Lip}_{s,1,M}(\X)$ with $M=\lambda(\X)+2N \big(\Gamma(0.5d+1)\big)^{-1}\pi^{\frac d 2}b^{d-s} \vee 2 b^{-s}\lambda(\X)$ and any $0 <s \leq 1$, thus proving the claim via Theorem \ref{TVerrbnd}.

\subsection{Proofs for Section \ref{sec:extncsupp}} \label{sec:extncsupp-proof}

\subsubsection{Proof of Theorem \ref{Thm:KLNCsupp}} \label{Thm:KLNCsupp-proof}

To prove part $(i)$, fix some $0<\epsilon<1$. Let $B_d^c(r)=\RR^d \setminus B_d(r)$, and  $r(\epsilon)$ be sufficiently large such that $\EE_{\mu}\big[\abs{f_{\mathsf{KL}}} \ind_{B_d^c(r(\epsilon))}\big] \vee \EE_{\nu}\big[|\dd \mu/\dd \nu-1|\ind_{B_d^c(r(\epsilon))}\big] \leq \epsilon$. Since $f_{\mathsf{KL}} \in \mathsf{C}(\RR^d)$, from \citep[Theorem 2.1 and 2.8]{stinchcombe1990approximating}, there is a $k_0(\epsilon,r(\epsilon)) \in \NN$, such that for any $k \geq k_0(\epsilon,r(\epsilon))$, there exists a $g_{\theta_k} \in \hat \cG_{k}^*(\phi,r(\epsilon))$ with 
  \begin{align}
      \norm{f_{\mathsf{KL}}-g_{ \theta_k}}_{\infty,B_d(r(\epsilon))} \leq \epsilon. \label{approxerrgencont}
  \end{align}
  Then, we have
\begin{flalign}
  &   \abs{\kl{\mu}{\nu}-\mathsf{D}_{\hat{\cG}_{k}^*(\phi,r(\epsilon))}(\mu,\nu)} \notag \\
      & \leq  \EE_{\mu}\left[\abs{f_{\mathsf{KL}}- g_{\theta_k}}\right]+\EE_{\nu}\left[\big| e^{f_{_{\mathsf{KL}}}}-e^{g_{\theta_k}}\big|\right] \notag \\
   &= \EE_{\mu}\left[\abs{f_{\mathsf{KL}}- g_{\theta_k}} \ind_{B_d(r(\epsilon))}\right]+\EE_{\mu}\left[\abs{f_{\mathsf{KL}}- g_{\theta_k}} \ind_{B_d^c(r(\epsilon))}\right] +\EE_{\nu}\left[\big|e^{f_{\mathsf{KL}}}-e^{g_{\theta_k}}\big|\ind_{B_d^c(r(\epsilon))}\right]\notag \\ & \qquad \qquad  \qquad \qquad \qquad  \qquad  \qquad \qquad  \qquad \qquad \qquad  ~~\quad  + \EE_{\nu}\left[\big|e^{f_{\mathsf{KL}}}-e^{g_{\theta_k}}\big|\ind_{B_d(r(\epsilon))}\right]  \notag \\
    & \leq \norm{(f_{\mathsf{KL}}- g_{\theta_k})\ind_{B_d(r(\epsilon))}}_{\infty,\mu}   + \EE_{\mu}\left[\abs{f_{\mathsf{KL}}} \ind_{B_d^c(r(\epsilon))}\right]+\EE_{\nu}\left[\abs{\frac{\dd \mu}{\dd \nu}-1}\ind_{B_d^c(r(\epsilon))}\right] \notag \\
    & \qquad \qquad  \qquad \qquad \qquad  \qquad   ~ +\EE_{\nu}\left[\big|e^{f_{\mathsf{KL}}}\big|\ind_{B_d(r(\epsilon))}\right] \norm{\left(1-e^{g_{\theta_k}-f_{\mathsf{KL}}}\right)\ind_{B_d(r(\epsilon))}}_{\infty,\nu} \label{genbndnckld}\\
       & \lesssim \epsilon, \label{approxerrbndnc}
\end{flalign}  
where the final inequality is due to \eqref{approxerrgencont}, the choice of $r(\epsilon)$, and  $\EE_{\nu}\left[\abs{e^{f_{\mathsf{KL}}}}\ind_{B_d(r(\epsilon))}\right] \leq 1$.
  
On the other hand,  for any  $0<\rho<1$, and   $n, k_n,r_n$  such that   $k_n^{3/2}(r_n+1)e^{k_n(r_n+1)} =O \left(n^{(1-\rho)/2}\right)$,  Lemma \ref{lem:consicomp} yields 
\begin{align}
  \hat{\mathsf{D}}_{\hat{\cG}_{k_n}^*(\phi,r_n)}(X^n,Y^n)  \xrightarrow[n\rightarrow \infty]{}    \mathsf{D}_{\hat{\cG}_{k_n}^*(\phi,r_n)}(\mu,\nu),\quad  \mathbb{P}-\mbox{a.s.}\notag 
\end{align}
  This along with \eqref{approxerrbndnc} completes the proof of Part $(i)$.  
\medskip

To prove part $(ii)$, we first state a general error bound for KL neural estimation based on the tail behaviour of random variables $f_{\mathsf{KL}}(X)$ and $h_{\mathsf{KL}} \circ f_{\mathsf{KL}}(Y):= e^{f_{\mathsf{KL}}(Y)}-1$ outside $B_d(r)$  for $X \sim \mu$ and $Y \sim \nu$. For an increasing positive  divergent sequence $\mathbf{r}=(r_k)_{k \in \NN}$, $r_k \geq 1$, ($r_k \rightarrow \infty$), a positive non-decreasing sequence $\mathbf{m}=(m_k)_{k \in \NN}$, $m_k \geq 1$, and a non-increasing non-negative sequence $\mathbf{v}=(v_k)_{k \in \NN}$ with $v_k \rightarrow 0$, set
\begin{align}
 &  \breve{\mathcal{P}}^2_{\mathsf{KL}}(M,\mathbf{r},\mathbf{m}, \mathbf{v})\notag \\
 &:=\left\{(\mu,\nu) \in \mathcal{P}^2_{\mathsf{KL}}\big(\RR^d\big): \begin{aligned} 
 & \EE_{\mu}\Big[\abs{f_{\mathsf{KL}}}\ind_{B_d^c(r_k)}\Big] \vee \EE_{\nu}\left[\abs{h_{\mathsf{KL}} \circ f_{\mathsf{KL}}}\ind_{B_d^c(r_k)}\right] \leq v_k,\\
&\kl{\mu}{\nu} \leq M,~c_{\mathsf{KB}}^\star\left(f_{\mathsf{KL}}|_{B_d(r_k)},B_d(r_k)\right)  \leq m_k, ~k \in \NN  \end{aligned}\right\}. \notag
\end{align}
Then, we have the following lemma.
\begin{lemma}[KL divergence neural estimation] \label{lem:genmvncKL}
 Suppose there exists $M \geq 0$, $0<\rho<1$, and  $\mathbf{r},\mathbf{m}, \mathbf{v}$ as above satisfying $1 \leq m_k \lesssim k^{(1-\rho)/2}$, such that $(\mu,\nu) \in \breve{\mathcal{P}}^2_{\mathsf{KL}}(M,\mathbf{r},\mathbf{m}, \mathbf{v})$.  Then, for $\cG_k =\hat{\mathcal{G}}_{k}^{\mathsf{R}}(m_k, r_k)$, we have
\begin{align}
&\sup_{\substack{(\mu,\nu) \in \\ \breve{\mathcal{P}}^2_{\mathsf{KL}}(M,\mathbf{r},\mathbf{m}, \mathbf{v})}} \mathbb{E}\left[  \abs{\hat{\mathsf{D}}_{\cG_k}(X^n,Y^n) -\kl{\mu}{\nu}}\right]  
  \lesssim_{d,M,\rho}
 m_k  k^{-\frac 12}  +v_k+ m_kr_ke^{3m_k(r_k+1)}~ n^{-\frac 12}.\notag 
\end{align}
\end{lemma} 
The proof of the above lemma is based on  an application of Theorem \ref{THM:approximation} to bound the NN approximation error on balls $B_d(r_k)$, leveraging tail integrability assumptions in the definition of $   \breve{\mathcal{P}}^2_{\mathsf{KL}}$ to bound the approximation error outside $B_d(r_k)$, and using Theorem  \ref{thm:optkdepNN} to control the empirical estimation error. Its proof is given in Appendix \ref{lem:genmvncKL-proof}.

Continuing with the proof of the Theorem, we  will show that $(\mu,\nu) \in \bar{\cP}^2_{\mathsf{KL},\psi}\left(M,\ell,\mathbf{r},\mathbf{m}\right)$ implies  $(\mu,\nu) \in  \breve{\mathcal{P}}^2_{\mathsf{KL}}( M,\mathbf{r},\mathbf{m}, \mathbf{v})$ for some $\mathbf{v}$ that will be specified below. Then, Part  $(ii)$  will follow from Lemma \ref{lem:genmvncKL}.

 Note that $\norm{f_{\mathsf{KL}}}_{\ell,\mu} \leq M$, where  $\ell >1$ (or equivalently $\ell \geq 2$ since $\ell \in \NN$), implies 
\begin{align}
   \kl{\mu}{\nu} = \EE_{\mu} \left[f_{\mathsf{KL}}\right] \leq \sqrt{\EE_{\mu} \left[f^2_{\mathsf{KL}}\right]} \leq M. \label{kldivsqintbnd}
\end{align}
Also,
 \begin{flalign}
   \EE_{\mu}\left[\abs{f_{\mathsf{KL}}}\ind_{B_d^c(r_k)}\right] &\stackrel{(a)}{\leq}    \norm{f_{\mathsf{KL}}}_{\ell,\mu} \big(\mu \left(\norm{X} > r_k\right)\big)^{\frac{1}{\ell^*}} \label{tail1klbnd}\\
   & \stackrel{(b)}{\leq}   M  \mu \Big(\psi\big(\norm{X}M^{-1}\big) > \psi\big(r_kM^{-1}\big)\Big)^{\frac{1}{\ell^*}} \notag \\
   &\stackrel{(c)}{\leq} M \Big(\EE_{\mu}\left[\psi\big(\norm{X}M^{-1}\big)\right]\Big)^{\frac{1}{\ell^*}} \Big(\psi\big(r_kM^{-1}\big)\Big)^{-\frac{1}{\ell^*}} \notag \\
   &\stackrel{(d)}{\leq} M \Big(\psi\big(r_kM^{-1}\big)\Big)^{-\frac{1}{\ell^*}}, \notag &&
 \end{flalign}
 \vspace{-10 pt}
  \begin{flalign}
   \EE_{\nu}\left[\abs{h_{\mathsf{KL}} \circ f_{\mathsf{KL}}}\ind_{B_d^c(r_k)}\right]&= \EE_{\nu}\left[\abs{\frac{\dd \mu }{\dd \nu}-1 }\ind_{B_d^c(r_k)}\right] \notag \\
   & \leq    \EE_{\nu}\left[\left(\frac{\dd \mu }{\dd \nu}+1\right) \ind_{B_d^c(r_k)}\right] \notag \\
   &=\mu \big(B_d^c(r_k)\big)+\nu\big(B_d^c(r_k)\big) \label{tail2klbnd} \\
   & \stackrel{(e)}{\leq}    \Big(\EE_{\mu}\left[\psi\big(\norm{X}M^{-1}\big)\right] +\EE_{\nu}\left[\psi\big(\norm{Y}M^{-1}\big)\right]\Big) \Big(\psi\big(r_kM^{-1}\big)\Big)^{-1}\notag \\
   & \stackrel{(f)}{\leq}    2\Big(\psi\big(r_kM^{-1}\big)\Big)^{-1},\notag &&
 \end{flalign}
 where
 \begin{enumerate}[label = (\alph*),leftmargin=17 pt]
\item follows by H\"{o}lder's inequality;
 \item is since $\norm{f_{\mathsf{KL}}}_{\ell,\mu} \leq M$ and $\psi$ is increasing;
 \item and $(e)$ are due to Markov's inequality;
  \item and $(f)$ are since $p,q \in L_{\psi}(M)$ implies that $\EE_{\mu}\left[\psi\big(\norm{X}M^{-1}\big)\right] \vee \EE_{\nu}\left[\psi\big(\norm{Y}M^{-1}\big)\right] \leq 1$.
\end{enumerate}
 It follows that $(\mu,\nu) \in  \breve{\mathcal{P}}^2_{\mathsf{KL}}(M,\mathbf{r},\mathbf{m}, \mathbf{v})$ with  $v_k\asymp_{M,\psi,\ell}\big(\psi(r_kM^{-1})\big)^{-1/\ell^*}$ since $r_k \geq 1$. Note that $v_k \rightarrow 0$ as $r_k \rightarrow \infty$. This completes the proof of Part  $(ii)$ via Lemma \ref{lem:genmvncKL}.

\subsubsection{Proof of Corollary \ref{cor:klgaussrate}} \label{cor:klgaussrate-proof}
 Fix $(\mu,\nu)=\left(\cN(\mathsf{m}_p,\sigma_p^2 \mathrm{I}_d),\cN(\mathsf{m}_q,\sigma_q^2 \mathrm{I}_d)\right) \in \bar{\cP}^2_{\mathsf{N}}(M)$ and $\mathbf{r}=\big(r_k)_{k \in \NN}=\big(1 \vee M+\tilde r_k\big)_{k \in \NN}$, where $\tilde r_k \geq 0$, $k \in \NN$, will be specified below.   Note that
\begin{align}
  & f_{\mathsf{KL}}(x)=  d \log\left(\frac{\sigma_q}{\sigma_p}\right)+\frac{\norm{x-\mathsf{m}_q}^2}{2 \sigma_q^2}-\frac{\norm{x-\mathsf{m}_p}^2}{2 \sigma_p^2}, \notag \\
  & \kl{\mu}{\nu}=d \log \left(\frac{\sigma_q}{\sigma_p}\right)-0.5 d+0.5 d\frac{\sigma_p^{2}}{\sigma_q^{2}}+\frac{\norm{\mathsf{m}_p-\mathsf{m}_q}^2}{2 \sigma_q^{2}}. \notag
\end{align}
Also, $f_{\mathsf{KL}}$ is infinitely differentiable on $\RR^d$, and it can be seen by computing derivatives that for any multi-index $\alpha$ of dimension $d$ and arbitrary order $\norm{\alpha}_1  \in \NN$,
\begin{align}
  \norm{D^{\alpha}f_{\mathsf{KL}}}_{\infty, B_d(r_k)} \leq b_{k,d,M}^*:= 
    c_{d,M} \left(1 +\tilde r_k^2\right), \notag
\end{align}
 for some constant $ c_{d,M}$ (polynomial in $M$). Hence, $f_{\mathsf{KL}}|_{B_d(r_k)} \in \mathsf{C}^{s_{\mathsf{KB}}}_{b_{k,d,M}^*}$, which implies via Proposition \ref{prop:bndfourcoeff} that 
\begin{align}
    c_{\mathsf{KB}}^\star\left(f_{\mathsf{KL}}|_{B_d(r_k)},B_d(r_k)\right) \leq  m_k^{\mathsf{KL}}:= 
   c_{d,M}\left(1+ \tilde r_k^{d+3}\right). \notag 
\end{align}

By a straightforward calculation by using $1/M <\sigma_p,\sigma_q <M, \norm{\mathsf{m}_p} \vee \norm{\mathsf{m}_q} \leq M$, it follows from Gaussian integral formulas that there exists some $c_{d,M}$ such that
\begin{align}
  & \norm{f_{\mathsf{KL}}}_{2,\mu} \vee \norm{p}_{\psi_2} \vee \norm{q}_{\psi_2} \leq   c_{d,M},\notag 
\end{align}
where $\psi_2(z)=e^{z^2}-1$.
Hence, $ \bar{\cP}^2_{\mathsf{N}}(M) \subseteq \bar{\cP}^2_{\mathsf{KL},\psi_2}\left(c_{d,M},2,\mathbf{r},\mathbf{m}^{\mathsf{KL}}\right)$, and we have from Part $(ii)$ of Theorem \ref{Thm:KLNCsupp} with $m_k= m_k^{\mathsf{KL}}$ and $\cG_k =\hat{\mathcal{G}}_{k}^{\mathsf{R}}\left(m_k^{\mathsf{KL}}, r_k\right)$ that  
 \begin{flalign}
  \mathbb{E}\mspace{-3 mu}\left[  \abs{\hat{\mathsf{D}}_{\cG_k}\mspace{-2 mu}(X^n,Y^n) -\kl{\mu}{\nu}}\right]  \mspace{-3 mu}&\lesssim_{d,M,\rho} m_k k^{-\frac 12}  \mspace{-2 mu}+e^{-\frac{\tilde r_k^2}{c_{d,M}^2}}+m_k\tilde r_k e^{3m_k\tilde r_k}\mspace{2 mu} n^{-\frac 12}. \notag 
 \end{flalign}
Then,
setting $r_k=r_k^{\mathsf{KL}}:= 1 \vee M+\tilde r_k$ with $\tilde  r_k=
  \left(c\log k/ 3c_{d,M}\right)^{1/(d+4)}$ yields for $\cG_k =\hat{\mathcal{G}}_{k}^{\mathsf{R}}\left(m_k^{\mathsf{KL}}, r_k^{\mathsf{KL}}\right)$ that
 \begin{flalign}
&  \mathbb{E}\left[  \abs{\hat{\mathsf{D}}_{\cG_k}(X^n,Y^n) -\kl{\mu}{\nu}}\right]  \lesssim_{d,M}
    ck^{-\frac 12} \log k +e^{-\left(c\log k/c_{d,M}\right)^{\frac{2}{d+4}}}+ck^{c }\log k~n^{-\frac 12}. \notag
 \end{flalign}
Solving for the value of $c$ such that the first two terms in the RHS of the equation above are equal (up to logarithmic factors) yields  $c=c_{d,M}2^{-(d+4)/2} (\log k)^{(d+2)/2}$. 
Substituting $c$ and taking supremum over $(\mu,\nu)\in \bar{\cP}^2_{\mathsf{N}}(M)$,  we obtain
the claim in the Corollary.

\subsubsection{Proof of Theorem \ref{Thm:chisqNCsupp}} \label{Thm:chisqNCsupp-proof}
Let $r(\epsilon)$ be sufficiently large such that $\EE_{\mu}\big[|f_{\chi^2}|\ind_{B_d^c(r(\epsilon))}\big] \vee  \EE_{\nu}\big[|h_{\chi^2} \circ f_{\chi^2}|\ind_{B_d^c(r(\epsilon))}\big] \leq \epsilon$. Similar to \eqref{approxerrgencont}, there  exists $g_{\theta_k} \in \hat \cG_{k}^*(\phi,r(\epsilon))$  satisfying  $\norm{f_{\chi^2}-g_{ \theta_k}}_{\infty,B_d(r(\epsilon))} \leq \epsilon$ for $k \geq k_0(\epsilon,r(\epsilon))$. Then, we have
\begin{flalign}
  &   \Big|\chisq{\mu}{\nu}-\chi^2_{\hat{\cG}_{k_n}^*(\phi,r_n)}(\mu,\nu)\Big| \notag \\
   &\leq \EE_{\mu}\left[\abs{f_{\chi^2}- g_{\theta_k}}\right]+\EE_{\nu}\left[\abs{h_{\chi^2} \circ f_{\chi^2}- h_{\chi^2} \circ g_{\theta_k}}\right] \notag \\
    &= \EE_{\mu}\left[\abs{f_{\chi^2}- g_{\theta_k}} \ind_{B_d(r(\epsilon))}\right] + \EE_{\nu}\left[\abs{h_{\chi^2} \circ f_{\chi^2}- h_{\chi^2} \circ g_{\theta_k}}\ind_{B_d(r(\epsilon))}\right] \notag \\
    &\qquad \qquad \qquad +\EE_{\mu}\left[\abs{f_{\chi^2}- g_{\theta_k}} \ind_{B_d^c(r(\epsilon))}\right]+ \EE_{\nu}\left[\abs{h_{\chi^2} \circ f_{\chi^2}- h_{\chi^2} \circ g_{\theta_k}}\ind_{B_d^c(r(\epsilon))}\right] \notag \\
   &\leq  \norm{\big(f_{\chi^2}- g_{\theta_k}\big)\ind_{B_d(r(\epsilon))}}_{\infty,\mu}+   \EE_{\nu}\left[\Big|h_{\chi^2} \circ f_{\chi^2}- h_{\chi^2} \circ g_{\theta_k}\Big|\ind_{B_d(r(\epsilon))}\right]  \notag \\
  &\qquad  \qquad \qquad \qquad  \qquad \qquad +\EE_{\mu}\left[\abs{f_{\chi^2}} \ind_{B_d^c(r(\epsilon))}\right] +\EE_{\nu}\left[\abs{h_{\chi^2} \circ f_{\chi^2}}\ind_{B_d^c(r(\epsilon))}\right] \label{chisqlststp} \\
   &\stackrel{(a)}{\lesssim}  \epsilon+   \EE_{\nu}\left[\abs{h_{\chi^2} \circ f_{\chi^2}- h_{\chi^2} \circ g_{\theta_k}}\ind_{B_d(r(\epsilon))}\right] \notag \\
    &\stackrel{(b)}{\lesssim} \epsilon+\EE_{\nu}\left[\abs{f_{\chi^2}- g_{\theta_k}}\ind_{B_d(r(\epsilon))}\right]+ \EE_{\nu}\Big[0.25\abs{f_{\chi^2}- g_{\theta_k}}^2  \ind_{B_d(r(\epsilon))}\Big]  \notag \\
    & \qquad \qquad \qquad  \qquad \qquad \qquad  \qquad \qquad \qquad  \quad  +0.5\EE_{\nu}\Big[\abs{f_{\chi^2}- g_{\theta_k}}\abs{f_{\chi^2}}\ind_{B_d(r(\epsilon))}\Big]  \notag \\
     & \lesssim  \epsilon+ \norm{\big(f_{\chi^2}- g_{\theta_k}\big)\ind_{B_d(r(\epsilon))}}_{\infty,\nu} \EE_{\nu}\left[\abs{f_{\chi^2}}\right] \notag \\ 
       &\stackrel{(c)}{\lesssim}  \epsilon,  \notag &&
\end{flalign}
where  $(a)$ follows by definition of $r(\epsilon)$ and $g_{\theta_k}$ above; $(b)$
is via  steps leading to \eqref{chisqstepexp}; $(c)$ is due to definition of $r(\epsilon)$, $g_{\theta_k}$ and $\EE_{\nu}\big[\abs{f_{\chi^2}}\big] \leq 4$. From this and Lemma \ref{lem:consicompchisq},   Part $(i)$ follows.

\medskip

Next, we prove Part $(ii)$. For sequences $\mathbf{m}, ~\mathbf{r}$ and  $\mathbf{v}$  as in Appendix \ref{Thm:KLNCsupp-proof}, let
\begin{align}
 \breve{\mathcal{P}}^2_{\chi^2}(\mathbf{r},\mathbf{m}, \mathbf{v})&:=\left\{(\mu,\nu) \in  \mathcal{P}^2_{\chi^2}\big(\RR^d\big): \begin{aligned} 
 & \EE_{\mu}\mspace{-4 mu}\left[\abs{f_{\chi^2}}\mspace{-2 mu}\ind_{B_d^c(r_k)}\right] \mspace{-2 mu}\mspace{-2 mu}\vee\mspace{-2 mu} \EE_{\nu}\mspace{-4 mu}\left[\abs{h_{\chi^2} \circ f_{\chi^2}}\mspace{-2 mu}\ind_{B_d^c(r_k)}\right] \leq  v_k\\&
~ c_{\mathsf{KB}}^\star\left(f_{\chi^2}|_{B_d(r_k)},B_d(r_k)\right)  \leq m_k,~~k \in \NN  \end{aligned}\mspace{-1 mu}\right\}. \notag
\end{align}
We will use the following lemma which bounds the $\chi^2$ neural estimation error for distributions satisfying general tail integrability conditions (see Appendix \ref{chisqgenmvbnd-proof} for proof).
\begin{lemma}[$\chi^2$ neural estimation error] \label{chisqgenmvbnd}
For $\cG_k =\hat{\mathcal{G}}_{k}^{\mathsf{R}}(m_k, r_k)$, we have
 \begin{flalign}
& \sup_{(\mu,\nu) \in  \breve{\mathcal{P}}^2_{\chi^2}(\mathbf{r},\mathbf{m}, \mathbf{v})} \mathbb{E}\left[  \abs{\hat{\chi}^2_{\cG_k}(X^n,Y^n) -\chisq{\mu}{\nu}}\right]\lesssim m_k d^{\frac 12} k^{-\frac 12}+m_k^2 dk^{-1}  +v_k+d^{\frac 32}m_k^2r_k^2 n^{-\frac 12}. \label{chisqgenbndmv} 
 \end{flalign}
\end{lemma}
Armed with Lemma \ref{chisqgenmvbnd}, we next show that $(\mu,\nu) \in \bar{\cP}^2_{\chi^2,\psi}\left(M,\ell,\mathbf{r},\mathbf{m}\right)$ implies that $(\mu,\nu) \in  \breve{\mathcal{P}}^2_{\chi^2}(\mathbf{r},\mathbf{m}, \mathbf{v})$ for some $\mathbf{v}$ that will be identified  below. We have
\begin{flalign}
    \EE_{\mu}\left[\abs{f_{\chi^2}}\ind_{B_d^c(r_k)}\right] &\stackrel{(a)}{\leq}\norm{f_{\chi^2}}_{\ell,\mu} \big(\mu \left(\norm{X} > r_k\right)\big)^{\frac{1}{\ell^*}} \stackrel{(b)}{\leq} M \Big(\psi\big(r_kM^{-1}\big)\Big)^{-\frac{1}{\ell^*}}, \label{chisqtail1} 
    \end{flalign}
    \vspace{-20 pt}
  \begin{flalign}  
     \EE_{\nu}\left[\abs{h_{\chi^2} \circ f_{\chi^2}}\ind_{B_d^c(r_k)}\right]&= \EE_{\nu}\left[2\abs{\frac{\dd \mu }{\dd \nu}-1 }\ind_{B_d^c(r_k)}\right] + \EE_{\nu}\left[\left(\frac{\dd \mu }{\dd \nu}-1 \right)^2\ind_{B_d^c(r_k)}\right]\notag \\
   & \leq  2  \EE_{\nu}\left[\left(\frac{\dd \mu }{\dd \nu}+1\right) \ind_{B_d^c(r_k)}\right]+\EE_{\nu}\left[\left(\frac{\dd \mu }{\dd \nu}\right)^2 \ind_{B_d^c(r_k)}\right]+\nu\big(B_d^c(r_k)\big) \notag \\
   &=2\mu\big(B_d^c(r_k)\big)+3\nu\big(B_d^c(r_k)\big) +\EE_{\mu}\left[\frac{\dd \mu }{\dd \nu} \ind_{B_d^c(r_k)}\right]\notag \\
   &=2\mu\big(B_d^c(r_k)\big)+3\nu\big(B_d^c(r_k)\big) +\EE_{\mu}\left[(0.5 f_{\chi^2}+1) \ind_{B_d^c(r_k)}\right]\notag \\
      &\stackrel{(c)}{=} 3\mu\big(B_d^c(r_k)\big)+3\nu\big(B_d^c(r_k)\big) +0.5 \norm{f_{\chi^2}}_{\ell,\mu}\big(\mu\big(B_d^c(r_k)\big)\big)^{\frac{1}{\ell^*}}\label{chisqtail2} \\
   & \stackrel{(d)}{\leq}    6 \Big(\psi\big(r_kM^{-1}\big)\Big)^{-1}+0.5 M \Big(\psi\big(r_kM^{-1}\big)\Big)^{-\frac{1}{\ell^*}},\notag 
\end{flalign}
where 
\begin{enumerate}[label = (\alph*),leftmargin=17 pt]
\item and (c) is by H\"{o}lder's inequality;
\item and (d) follows via Markov's inequality since $\EE_{\mu}\left[\psi\big(\norm{X}M^{-1}\big)\right] \vee \mspace{2 mu}\EE_{\nu}\left[\psi\big(\norm{Y}M^{-1}\big)\right] \leq 1$, and $\norm{f_{\chi^2}}_{\ell,\mu} \leq M$ by assumption. 
\end{enumerate}
Hence,  $(\mu,\nu) \in  \breve{\mathcal{P}}^2_{\chi^2}(\mathbf{r},\mathbf{m}, \mathbf{v})$ with  
\begin{align}
    v_k=6 \Big(\psi\big(r_kM^{-1}\big)\Big)^{-1}+ M \Big(\psi\big(r_kM^{-1}\big)\Big)^{-\frac{1}{\ell^*}}\lesssim_{M,\psi,\ell} \Big(\psi\big(r_kM^{-1}\big)\Big)^{-\frac{1}{\ell^*}}. \notag
\end{align}
 This implies  Part $(ii)$ via Lemma \ref{chisqgenmvbnd} (since $m_k k^{-\frac 12}+m_k^2k^{-1}  \leq 2 m_k^2 k^{-\frac 12}$ due to $m_k \geq 1$).

\subsubsection{Proof of Corollary \ref{cor:chisqgaussrate}} \label{cor:chisqgaussrate-proof}
 We will require the  following  lemma which  bounds the tail probability of an isotropic Gaussian distribution outside a Euclidean ball $ B_d(r)$ of radius $r$. This is a straighforward consequence of Gaussian concentration \citep[Eqn. 1.4]{LT-1991} and the fact that $\norm{\cdot}$ is $1$-Lipschitz function on the metric space $(\RR^d,\norm{\cdot})$.
\begin{lemma}[Gaussian tail integral bound] \label{lem:gaussmom}
For any $\mathsf{m}_p \in \RR^d$ such that $\norm{\mathsf{m}_p} \leq M$, $\sigma^2>0$ and  $r \geq M$,  
\begin{align}
    & \big(2 \pi \sigma^2\big)^{-\frac d2}\int_{B_d^c(r)} e^{-\frac{\norm{x-\mathsf{m}_p}^2}{2 \sigma^2}} \dd x \leq  2 e^{-\frac{(r-M)^2}{2 \sigma^2}}. \label{gint1} 
\end{align}
\end{lemma}

\medskip

Proceeding with the proof of the corollary, fix $(\mu,\nu)=\left(\cN(\mathsf{m}_p,\sigma^2 \mathrm{I}_d),\cN(\mathsf{m}_q,\sigma^2 \mathrm{I}_d)\right) \in \bar{\cP}^2_{\chi^2,\mathsf{N}}(M)$, and $\mathbf{r}=\big(r_k)_{k \in \NN}=\big(1 \vee M+\tilde r_k\big)_{k \in \NN}$, where $\tilde r_k \geq 0$, $k \in \NN$, will be specified below. Note that since
 \begin{align}
     f_{\chi^2}(x)=2\left(\frac{p(x)}{q(x)}-1\right)=2\left(\left(\frac{\sigma_q}{\sigma_p}\right)^{d}e^{\frac{\norm{x-\mathsf{m}_q}^2}{{2 \sigma_q^2}}-\frac{\norm{x-\mathsf{m}_p}^2}{2 \sigma_p^2}}-1\right), \notag
 \end{align}
   it is  infinitely differentiable on $\RR^d$. A straightforward  computation  shows that for any multi-index $\alpha \in \ZZ_{\geq 0}^d$ of order $\norm{\alpha}_1 \leq s_{\mathsf{KB}}$, 
\begin{align}
  \norm{D^{\alpha}f_{\chi^2}}_{\infty,B_d(r_k)} \leq \tilde b_k^{\mathsf{R}}:= 
        c_{d,M}\left(1+\tilde r_k^{s_{\mathsf{KB}}}\right)~e^{2M^2\tilde r_k^2}. \notag
\end{align}
Hence, $f_{\chi^2}|_{B_d(r_k)} \in \mathsf{C}^{s_{\mathsf{KB}}}_{\tilde b_k^{\mathsf{R}}}$, which implies via  Proposition \ref{prop:bndfourcoeff} that 
\begin{align}
  &  c_{\mathsf{KB}}^\star\left(f_{\chi^2}|_{B_d(r_k)},B_d(r_k)\right) \leq  m_k^{\chi^2}:=
    c_{d,M} \left(1+\tilde r_k^{s_{\mathsf{KB}}+d+1}\right)e^{2M^2\tilde r_k^2}.  \label{optratemchisq}
\end{align}
Furthermore, letting $\tilde \sigma^{-2}:=\left(\sigma_p^{-2}-0.5\sigma_q^{-2}\right)\wedge 0.5\sigma_q^{-2}\wedge 0.5\sigma_p^{-2}$ and noting that $\tilde \sigma^{-2} \geq 0.5 M^{-3}$ by definition of $\bar{\cP}^2_{\chi^2,\mathsf{N}}(M)$ and $M > 1$, we have
\begin{flalign}
    \EE_{\mu}\left[\abs{f_{\chi^2}}\ind_{B_d^c(r_k)}\right]  &\leq \frac{2}{ (2\pi \sigma_p^2)^{d/2}}\int_{B_d^c(r_k)}\left(\left(\frac{\sigma_q}{\sigma_p}\right)^{d}e^{\frac{\norm{x-\mathsf{m}_q}^2}{{2 \sigma_q^2}}-\frac{\norm{x-\mathsf{m}_p}^2}{2 \sigma_p^2}}+1\right)e^{-\frac{\norm{x-\mathsf{m}_p}^2}{2 \sigma_p^2}}\dd x \notag \\
    &  \leq \frac{2}{ (2\pi \sigma_p^2)^{d/2}}\int_{B_d^c(r_k)}\left(\frac{\sigma_q}{\sigma_p}\right)^{d}e^{\frac{\norm{x-\mathsf{m}_q}^2}{{2 \sigma_q^2}}-\frac{\norm{x-\mathsf{m}_p}^2}{ \sigma_p^2}} \dd x+e^{-\frac{\norm{x-\mathsf{m}_p}^2}{2 \sigma_p^2}}\dd x \notag \\
    & \stackrel{(a)}{\lesssim_{d,M}}  e^{-\frac{\tilde r_k^2}{\tilde \sigma^2}} \leq e^{-\frac{\tilde r_k^2}{2M^3}},
    \notag &&
    \end{flalign}
    \vspace{-10 pt}
    \begin{flalign}
   &  \EE_{\nu}\left[\abs{h_{\chi^2} \circ f_{\chi^2}}\ind_{B_d^c(r_k)}\right] \notag \\
   & = \frac{1}{ (2\pi \sigma_q^2)^{d/2}} \Bigg(\int_{B_d^c(r_k)}2\left(\left(\frac{\sigma_q}{\sigma_p}\right)^{d}e^{\frac{\norm{x-\mathsf{m}_q}^2}{{2 \sigma_q^2}}-\frac{\norm{x-\mathsf{m}_p}^2}{2 \sigma_p^2}}-1\right)e^{-\frac{\norm{x-\mathsf{m}_q}^2}{2 \sigma_q^2}} \dd x\notag \\
   & \qquad \qquad \qquad \qquad \qquad \qquad \qquad  +\int_{B_d^c(r_k)}\left(\left(\frac{\sigma_q}{\sigma_p}\right)^{d}e^{\frac{\norm{x-\mathsf{m}_q}^2}{{2 \sigma_q^2}}-\frac{\norm{x-\mathsf{m}_p}^2}{2 \sigma_p^2}}-1\right)^2e^{-\frac{\norm{x-\mathsf{m}_q}^2}{2 \sigma_q^2}}\dd x\Bigg) \notag \\
     & \stackrel{(b)}{\lesssim_{d,M}} \int_{B_d^c(r_k)} e^{-\frac{\norm{x-\mathsf{m}_p}^2}{2 \sigma_p^2}} \dd x+\int_{B_d^c(r_k)}e^{\frac{\norm{x-\mathsf{m}_q}^2}{{2 \sigma_q^2}}-\frac{\norm{x-\mathsf{m}_p}^2}{ \sigma_p^2}}\dd x+\int_{B_d^c(r_k)}e^{-\frac{\norm{x-\mathsf{m}_q}^2}{2 \sigma_q^2}}\dd x \notag \\
    &\stackrel{(c)}{\lesssim_{d,M}} e^{-\frac{\tilde r_k^2}{\tilde \sigma^2}} \leq  e^{-\frac{\tilde r_k^2}{2M^3}}, \notag &&
\end{flalign}
where 
\begin{enumerate}[label = (\alph*),leftmargin=17 pt]
\item and (c) follows by an application of Lemma \ref{lem:gaussmom} via completion of squares since $\sigma_p^2 < 2 \sigma_q^2$ by assumption;
 \item uses $(ae^x-1)^2 \leq a^2e^{2x}+1$ for $x \in \RR^d$ and $a\geq 0$.
\end{enumerate}
Hence, $(\mu,\nu) \in \breve{\mathcal{P}}^2_{\chi^2}\big(\mathbf{r},\mathbf{m}^{\chi^2}, \mathbf{v}^{\chi^2}\big)$ with $\mathbf{m}^{\chi^2}$ as defined in \eqref{optratemchisq} and  $v_k^{\chi^2}:=c_{d,M} e^{-\tilde r_k^2/2M^3}$, and the error bound in  \eqref{chisqgenbndmv} applies. 
Setting $   r_k= r_k^{\chi^2} :=1 \vee M+\tilde r_k=1 \vee M+ 2^{-0.5}M^{-1}\sqrt{c\log k}$ for some constant $c$ in  \eqref{chisqgenbndmv}, optimizing the resulting bound  w.r.t. $c$ (achieved at $c=2M^5/(4M^5+1)<0.5$), 
we obtain for $\cG_k=\hat{\mathcal{G}}_{k}^{\mathsf{R}}\big(m_k^{\chi^2}, r_k^{\chi^2}\big)$ that 
\begin{flalign}
&  \mathbb{E}\left[  \abs{\hat{\chi}^2_{\cG_k}(X^n,Y^n) -\chisq{\mu}{\nu}}\right]   \lesssim_{d,M} (\log k)^{ 2(s_{\mathsf{KB}}+ d+1)}\left(k^{-\frac{1}{2+8M^5}}+ k^{\frac{4M^5}{1+4M^5}}n^{-\frac{1}{2}}\right). \notag
 \end{flalign}
Taking supremum over $(\mu,\nu) \in \bar{\cP}^2_{\chi^2,\mathsf{N}}(M)$ completes the proof.

 \subsubsection{Proof of Theorem \ref{Thm:helsqNCsupp}} \label{Thm:helsqNCsupp-proof}
 For sequences $\mathbf{m}, ~\mathbf{r}$ and  $\mathbf{v}$  as in Appendix \ref{Thm:KLNCsupp-proof}, let
\begin{flalign}
 &  \breve{\mathcal{P}}^2_{\mathsf{H}^2}(\mathbf{r},\mathbf{m}, \mathbf{v}):=\left\{\mspace{-5 mu}(\mu,\nu) \in \mathcal{P}^2_{\mathsf{H}^2}\big(\RR^d\big): \begin{aligned} &\EE_{\mu}\left[\abs{f_{\mathsf{H}^2}}\ind_{B_d^c(r_k)}\right] \vee \EE_{\nu}\left[\abs{h_{\mathsf{H}^2} \circ f_{\mathsf{H}^2}}\ind_{B_d^c(r_k)}\right] \leq  v_k,\\
 &c_{\mathsf{KB}}^\star\left(f_{\mathsf{H}^2}|_{B_d(r_k)},B_d(r_k)\right)  \vee \norm{\frac{\dd \mu}{\dd \nu}}_{\infty,B_d(r_k)}   \leq m_k,~k \in \NN  \end{aligned}\right\}. \notag &&
\end{flalign}
 The following lemma proves consistency of the NE for $\mathsf{H}^2$ estimation and bounds its effective error for distributions satisfying general tail integrability conditions; see Appendix \ref{lem:sqhelmv-proof} for proof. 
 \begin{lemma}[$\mathsf{H}^2$ neural estimation] \label{lem:sqhelmv}
 Let $(\mu,\nu) \in \breve{\mathcal{P}}^2_{\mathsf{H}^2}(\mathbf{r},\mathbf{m}, \mathbf{v})$, where $\mathbf{m}$ satisfies  $m_k=o(k^{1/4})$. 
 Then, the following hold: 
 \begin{enumerate}[label = (\roman*),leftmargin=15 pt]
 \item For $k_n,m_{k_n},r_{k_n},n$ satisfying  $k_n\ \rightarrow  \infty$, $r_{k_n}\ \rightarrow  \infty$,  $k_n^{1/2}m_{k_n}^2r_{k_n} =O\left(n^{(1-\rho)/2}\right)$, and $\cG_n =\check{\cG}^{\mspace{2 mu}\mathsf{R}}_{k_n,m_{k_n}^{-1/2}}(m_{k_n},r_{k_n})$, we have
 \begin{equation}
  \hat{\mathsf{H}}^2_{\cG_n}(X^n,Y^n) \xrightarrow[n \rightarrow \infty]{} \mathsf{H}^2(\mu,\nu),\quad \mathbb{P}-\mbox{a.s.} \label{helconsisncgen}
 \end{equation}
\item For $\cG_k =\check{\cG}^{\mspace{2 mu}\mathsf{R}}_{k,m_k^{-1/2}}(m_k,r_k)$, we have
\begin{flalign}
 &  \sup_{\substack{(\mu,\nu) \in  \breve{\mathcal{P}}^2_{\mathsf{H}^2}(\mathbf{r},\mathbf{m}, \mathbf{v})}} \mathbb{E}\left[  \abs{\hat{\mathsf{H}}^2_{\cG_k}(X^n,Y^n) -\mathsf{H}^2(\mu,\nu)}\right]  \lesssim m_k^2 d^{\frac 12} k^{-\frac 12}+v_k+d^{\frac 32}m_k^2r_k n^{-\frac 12}. \label{genhelbndmktk} &&
\end{flalign} 
\end{enumerate}
\end{lemma}
To prove the theorem, we  will show that $(\mu,\nu) \in \bar{\cP}^2_{\mathsf{H}^2,\psi}\left(M,\mathbf{r},\mathbf{m}\right)$ implies that $(\mu,\nu) \in  \breve{\mathcal{P}}^2_{\mathsf{H}^2}(\mathbf{r},\mathbf{m}, \mathbf{v})$ for some $\mathbf{v}$ stated below. Then, Part $(i)$ and $(ii)$  will follow from the corresponding Parts in the above lemma.
We have
\begin{flalign}
\EE_{\mu}\left[\abs{f_{\mathsf{H}^2}}\ind_{B_d^c(r_k)}\right] &= \EE_{\mu}\left[\abs{1-\sqrt{qp^{-1}}}\ind_{B_d^c(r_k)}\right]\notag \\
& \leq \mu\big(B_d^c(r_k)\big)+\EE_{\mu}\left[\sqrt{qp^{-1}}\ind_{B_d^c(r_k)}\right]\notag \\
&\stackrel{(a)}{\leq} \mu\big(B_d^c(r_k)\big)+\sqrt{\nu\big(B_d^c(r_k)\big)} \label{heltailineq1}\\
& \stackrel{(b)}{\leq} \Big(\psi\big(r_kM^{-1}\big)\Big)^{-1}+\Big(\psi\big(r_kM^{-1}\big)\Big)^{-\frac 12}, \notag &&
\end{flalign}
\vspace{-10 pt}
\begin{flalign}
\EE_{\nu}\left[\abs{h_{\mathsf{H}^2} \circ f_{\mathsf{H}^2}}\ind_{B_d^c(r_k)}\right]&= \EE_{\nu}\left[\abs{\sqrt{pq^{-1}}-1}\ind_{B_d^c(r_k)}\right]\notag \\
&\stackrel{(c)}{\leq} \nu \big(B_d^c(r_k)\big)+\sqrt{\mu \big(B_d^c(r_k)\big)} \label{heltailineq2}\\
& \stackrel{(d)}{\leq} \Big(\psi\big(r_kM^{-1}\big)\Big)^{-1}+\Big(\psi\big(r_kM^{-1}\big)\Big)^{-\frac 12}, \notag &&
\end{flalign}
where 
\begin{enumerate}[label = (\alph*),leftmargin=17 pt]
\item and (c) follows from Cauchy-Schwarz inequality and $\EE_{\mu}\left[qp^{-1}\right]=\EE_{\nu}\left[pq^{-1}\right]=1$;
 \item and (d) follows from Markov's inequality as $\EE_{\mu}\left[\psi\big(\norm{X}M^{-1}\big)\right] \vee \mspace{2 mu}\EE_{\nu}\left[\psi\big(\norm{Y}M^{-1}\big)\right] \leq 1$.
\end{enumerate}
Hence, $ (\mu,\nu) \in  \breve{\mathcal{P}}^2_{\mathsf{H}^2}(\mathbf{r},\mathbf{m}, \mathbf{v})$ with $v_k=\big(\psi\big(r_kM^{-1}\big)\big)^{-1}+\big(\psi\big(r_kM^{-1}\big)\big)^{-1/2} \lesssim_{\psi,M}\\ \big(\psi\big(r_kM^{-1}\big)\big)^{-1/2}\rightarrow 0$. This  completes the proof  via Lemma \ref{lem:sqhelmv}.

\subsubsection{Proof of Corollary \ref{cor:helgaussrate}} \label{cor:helgaussrate-proof}
Fix  $(\mu,\nu)=\left(\cN(\mathsf{m}_p,\sigma^2 \mathrm{I}_d),\cN(\mathsf{m}_q,\sigma^2 \mathrm{I}_d)\right) \in \bar{\cP}^2_{\mathsf{N}}(M)$, and $\mathbf{r}=\big(r_k)_{k \in \NN}=\big(1 \vee M+\tilde r_k\big)_{k \in \NN}$, where $\tilde r_k \geq 0$, $k \in \NN$, will be specified below.  Observe that 
 \begin{align}
     f_{\mathsf{H}^2}(x)=1-\left(\frac{p(x)}{q(x)}\right)^{-\frac 12}=1-\left(\frac{\sigma_p}{\sigma_q}\right)^{d/2}e^{\frac{\norm{x-\mathsf{m}_p}^2}{{4 \sigma_p^2}}-\frac{\norm{x-\mathsf{m}_q}^2}{4 \sigma_q^2}}, \notag
 \end{align}
 is infinitely differentiable on $\RR^d$. Then, for any multi-index $\alpha \in \ZZ_{\geq 0}^d$ of  order $\norm{\alpha}_1 \leq s_{\mathsf{KB}}$, it is easy to see by computing partial derivatives that
\begin{align}
  \norm{D^{\alpha}f_{\mathsf{H}^2}} _{\infty,B_d(r_k)}\leq  \hat b_k^{\mathsf{R}}:=
        c_{d,M}\big(1+\tilde r_k^{s_{\mathsf{KB}}}\big)~e^{M^2 \tilde r_k^2}. \notag
\end{align}
Hence,  $f_{\mathsf{H}^2}|_{B_d(r_k)} \in \mathsf{C}^{s_{\mathsf{KB}}}_{\hat b_k^{\mathsf{R}}}$, which yields via  Proposition \ref{prop:bndfourcoeff} that
\begin{align}
    c_{\mathsf{KB}}^\star\left(f_{\mathsf{H}^2}|_{B_d(r_k)},B_d(r_k)\right) \leq 
    c_{d,M} \left(1+\tilde r_k^{s_{\mathsf{KB}}+d+1}\right)e^{M^2 \tilde r_k^2}.
\notag
\end{align}
Also,  we have
\begin{flalign}
&\norm{\frac{\dd \mu}{\dd \nu}}_{\infty,B_d(r_k)} =  \sup_{x \in B_d(r_k)}\left(\frac{\sigma_q}{\sigma_p}\right)^{d} e^{\frac{\norm{x-\mathsf{m}_q}^2}{2 \sigma_q^2}-\frac{\norm{x-\mathsf{m}_p}^2}{{2 \sigma_p^2}}} \leq 
    c_{d,M} \left(1+e^{2M^2r_k^2}\right). \notag 
    \end{flalign}
Furthermore, defining $\hat \sigma^2:=4 \sigma_p^2 \sigma_q^2/(\sigma_p^2+\sigma_q^2) \vee 2 \sigma_p^2 \vee 2 \sigma_q^2 =2 \sigma_p^2 \vee 2 \sigma_q^2 \geq 2M^{-1}$, we obtain
\begin{flalign}
       \EE_{\mu}\left[\abs{f_{\mathsf{H}^2}}\ind_{B_d^c(r_k)}\right]&\leq \frac{1}{ (2\pi \sigma_p^2)^{d/2}}\int_{B_d^c(r_k)}\left(1+\left(\frac{\sigma_p}{\sigma_q}\right)^{d/4}e^{\frac{\norm{x-\mathsf{m}_p}^2}{4 \sigma_p^2}-\frac{\norm{x-\mathsf{m}_q}^2}{{4 \sigma_q^2}}}\right)e^{-\frac{\norm{x-\mathsf{m}_p}^2}{2 \sigma_p^2}}\dd x \notag \\
    & \leq \frac{1}{ (2\pi \sigma_p^2)^{d/2}}\int_{B_d^c(r_k)}\left(e^{-\frac{\norm{x-\mathsf{m}_p}^2}{2 \sigma_p^2}}+\left(\frac{\sigma_p}{\sigma_q}\right)^{d/4}e^{-\frac{\norm{x-\mathsf{m}_q}^2}{{4 \sigma_q^2}}-\frac{\norm{x-\mathsf{m}_p}^2}{ 4\sigma_p^2}}\right)\dd x \notag \\
    & \stackrel{(a)}{\lesssim_{d,M}} e^{-\frac{\tilde r_k^2}{\hat \sigma^2}} \leq e^{-0.5 M \tilde r_k^2}, \notag  &&
    \end{flalign}
    \begin{flalign}
     \EE_{\nu}\left[\abs{h_{\mathsf{H}^2} \circ f_{\mathsf{H}^2}}\ind_{B_d^c(r_k)}\right] 
   &= \EE_{\nu}\left[\abs{\sqrt{\frac{\dd \mu}{\dd \nu}}-1}\ind_{B_d^c(r_k)}\right] \notag \\
   & \leq  \frac{1}{ (2\pi \sigma_q^2)^{d/2}}\int_{B_d^c(r_k)}\left(\left(\frac{\sigma_q}{\sigma_p}\right)^{d/4}e^{\frac{\norm{x-\mathsf{m}_q}^2}{4 \sigma_q^2}-\frac{\norm{x-\mathsf{m}_p}^2}{{4 \sigma_p^2}}}+1\right)e^{-\frac{\norm{x-\mathsf{m}_q}^2}{2 \sigma_q^2}}\dd x \notag \\
   &\leq \frac{1}{ (2\pi \sigma_q^2)^{d/2}}\int_{B_d^c(r_k)}\left(\frac{\sigma_p}{\sigma_q}\right)^{d/4}e^{-\frac{\norm{x-\mathsf{m}_q}^2}{{4 \sigma_q^2}}-\frac{\norm{x-\mathsf{m}_p}^2}{ 4\sigma_p^2}} \dd x+e^{-\frac{\norm{x-\mathsf{m}_q}^2}{2 \sigma_q^2}}\dd x \notag  \\
    &\stackrel{(b)}{\lesssim_{d,M}} e^{-\frac{\tilde r_k^2}{\hat \sigma^2}} \leq e^{-0.5 M\tilde r_k^2}, \notag &&
\end{flalign}
where   $(a)$ and $(b)$  above follows from  Lemma \ref{lem:gaussmom}.
Hence, $ \bar{\cP}^2_{\mathsf{N}}(M)\subseteq \breve{\mathcal{P}}^2_{\mathsf{H}^2}\big(\mathbf{r},\mathbf{m}^{\mathsf{H}^2}, \mathbf{v}^{\mathsf{H}^2}\big)$  with $m_k^{\mathsf{H}^2} \asymp_{d,M} \big(1+\tilde r_k^{s_{\mathsf{KB}}+d+1}\big)e^{2M^2 \tilde r_k^2}$ and $v_k^{\mathsf{H}^2} \asymp_{d,M} e^{-0.5 M\tilde r_k^2}$, and \eqref{genhelbndmktk} applies. Setting $r_k=1 \vee M+\tilde r_k$ with $\tilde r_k=\sqrt{2cM^{-1}\log k}$, $c>0$, and optimizing the resulting bound in   \eqref{genhelbndmktk} w.r.t. $c$ (optimum achieved at $c=0.5/(1+8M)$) yields
with $m_k=m_k^{\mathsf{H}^2}$ and $\cG_k=\check{\cG}^{\mspace{2 mu}\mathsf{R}}_{k,m_k^{-1/2}}(m_k,r_k)$  that 
\begin{align}
      \mathbb{E}\left[  \abs{\hat{\mathsf{H}}^2_{\cG_k}(X^n,Y^n) -\mathsf{H}^2(\mu,\nu)}\right]  \lesssim_{d,M}  (\log k)^{s_{\mathsf{KB}}+ d+2} k^{-\frac{1}{2+16M}}\left(1+ k^{\frac 12} n^{-\frac 12}\right).  \notag 
\end{align}
Taking supremum over $(\mu,\nu) \in \bar{\cP}^2_{\mathsf{N}}(M)$ completes the  proof.

 \subsubsection{Proof of Theorem \ref{Thm:TVerrbndNC}}\label{Thm:TVerrbndNC-proof}
Fix $\epsilon>0$ and let $r(\epsilon)$ denote $r$ such that $\mu\big(B_d^c(r)\big) \vee \nu\big(B_d^c(r)\big) \leq \epsilon$.  Then, following steps leading to \eqref{finasconvbndtv}, there exists  $g^* \in \vec{\cG}^*_{k}(\phi,r(\epsilon))$ for $k \geq k_0(\epsilon)$ such that the following holds:
 \begin{flalign}
 &\big|\tv{\mu}{\nu}- \tvn{\mu}{\nu}{ \vec{\cG}^*_{k}(\phi,r(\epsilon))} \big|
 \notag \\
 & \leq  \EE_{\mu} \left[\abs{f_{\mathsf{TV}}-g^*}\ind_{B_d(r(\epsilon))}\right]+\EE_{\nu} \left[\abs{f_{\mathsf{TV}}-g^*}\ind_{B_d(r(\epsilon))}\right]+\EE_{\mu} \big[\abs{f_{\mathsf{TV}}-g^*}\ind_{B_d^c(r(\epsilon))}\big]\notag \\
 & \qquad \qquad\qquad\qquad\qquad\qquad\qquad\qquad\qquad\qquad \qquad\qquad +\EE_{\nu} \big[\abs{f_{\mathsf{TV}}-g^*}\ind_{B_d^c(r(\epsilon))}\big] \notag\\
 &\lesssim   \epsilon+\EE_{\mu} \big[\abs{f_{\mathsf{TV}}}\ind_{B_d^c(r(\epsilon))}\big]+\EE_{\nu} \big[\abs{f_{\mathsf{TV}}}\ind_{B_d^c(r(\epsilon))}\big] \notag \\
  &\leq   \epsilon+\mu\big(B_d^c(r)\big) +\nu\big(B_d^c(r)\big)  \lesssim  \epsilon, \notag  &&
\end{flalign}
This combined with \eqref{errbndesttv} proves Part $(i)$. 

\medskip

Next, we prove  Part $(ii)$. Fix $(\mu,\nu) \in \bar{\mathcal{P}}^2_{\mathsf{TV},\psi}(M,s,\mathbf{r},\mathbf{m})$. For $t>0$, let  $f_{\mathsf{TV},r_k}:=f_{\mathsf{TV}} \ind_{B_d(r_k)}$ and $f_{\mathsf{TV},r_k}^{(t)}=f_{\mathsf{TV},r_k}*\mspace{2 mu} \Phi_t^{\cN}$, where $\Phi_t^{\cN}(x):= t^{-d} \Phi^{\cN}(t^{-1}x)$ and $\Phi^{\cN}=(2\pi)^{-d/2} e^{-0.5\norm{x}^2}$.  Then, similar to \eqref{barronconstftv}, we have
\begin{flalign}
    S_2\Big(f_{\mathsf{TV},r_k}^{(t)}\Big)r_k & :=  r_k \int_{\RR^d} \norm{\omega}_1^2 \abs{\mathfrak{F}\left[f_{\mathsf{TV},r_k}^{(t)}\right](\omega)}\dd\omega \notag \\
    & \leq r_k\norm{f_{\mathsf{TV},r_k}}_1d \int_{\RR^d} \norm{\omega}^2 \abs{\mathfrak{F}[\Phi_t](\omega)}\dd\omega \notag  \\
    &= r_k^{d+1} \frac{d\pi^{0.5d}}{\Gamma(0.5d+1)}\int_{\RR^d} \norm{\omega}^2 e^{-\frac 12 t^2\norm{\omega}^2}\dd\omega, \notag \\
    &=:\check c_{d,r_k,t},\notag &&
\end{flalign}
where 
\begin{align}
    \check c_{d,r_k,t}&:= \begin{cases}
    \sqrt{2}\pi r_k^{2}t^{-3} \big(\Gamma(3/2)\big)^{-1},& d=1,\\
    2^{\frac{d+3}{2}}\pi^{0.5d+1}d r_k^{d+1} t^{-(d+2)} \prod_{j=1}^{d-2}\int_{0}^{\pi}\sin^{d-1-j}(\varphi_j)\dd\varphi_j,& d \geq 2.\end{cases} \notag
\end{align}
 Then, noting that $\big|f_{\mathsf{TV}}^{(t)}(0)\big| \vee \big\|\nabla f_{\mathsf{TV}}^{(t)}(0)\big\| \leq 1 \vee (2 d \pi^{-d})^{1/2}\Gamma\big(0.5(d+1)\big)t^{-1}$, it  follows from \eqref{approxrateklubar}
that  there exists 
 $ g_{\theta_k} \in \vec{\cG}_k^{\mathsf{R}}\big( \breve  c_{d,r_k,t},r_k\big)$  such that 
\begin{align} 
    \abs{f_{\mathsf{TV},r_k}^{(t)}(x)- g_{\theta_k}(x)}  &\lesssim \begin{cases}
 \breve  c_{d,r_k,t}d^{\frac 12}k^{-\frac 12}, &x \in B_d(r_k),\\
    1, & \mbox{otherwise},
    \end{cases} \label{approxcmpctsuptv}
\end{align}
where $\breve  c_{d,r_k,t}:=\check  c_{d,r_k,t} \vee 1 \vee (2 d \pi^{-d})^{1/2}\Gamma\big(0.5(d+1)\big)t^{-1}.$

On the other hand, we have similar to steps leading to \eqref{pdisttvconerrbnd1} that
\begin{flalign}
\Big|\EE_{\mu}\left[f_{\mathsf{TV},r_k}- f_{\mathsf{TV},r_k}^{(t)}\right] \Big|
& \leq \int_{\RR^d} \left[\int_{\RR^d}\abs{ f_{\mathsf{TV},r_k}(x)- f_{\mathsf{TV},r_k}(x-tu) }p(x)\dd x \right] \Phi(u) \dd u  \notag \\
& = \int_{\RR^d} \left[\int_{\RR^d}\abs{ f_{\mathsf{TV},r_k}(x+tu)- f_{\mathsf{TV},r_k}(x) }p(x+tu)\dd x \right] \Phi(u) \dd u \notag \\
& \leq \norm{p}_{\infty,\RR^d} \int_{\RR^d} \left[\int_{\RR^d}\abs{ f_{\mathsf{TV},r_k}(x+tu)- f_{\mathsf{TV},r_k}(x) }\dd x \right] \Phi(u) \dd u  \notag \\
& \leq M \int_{\RR^d}  \xi_{1,1}(f,t \norm{u})\Phi(u) \dd u  \notag\\
& \leq M m_k \int_{\RR^d}  t^s\norm{u}^s \Phi(u) \dd u =c^*_{s,d}Mm_kt^s, \notag &&
\end{flalign}
where $c^*_{s,d}=\int_{\RR^d} \norm{u}^s \Phi(u) \dd u$. Then, defining $v_k=\mu \big(B_d^c(r_k)\big) \vee \nu \big(B_d^c(r_k)\big)$,  we have
\begin{flalign}
\abs{ \EE_{\mu}\left[f_{\mathsf{TV}}- f_{\mathsf{TV},r_k}^{(t)}\right]} &\leq \big|\EE_{\mu}\left[f_{\mathsf{TV}}- f_{\mathsf{TV},r_k}\right] \big|+\abs{\EE_{\mu}\left[f_{\mathsf{TV},r_k}- f_{\mathsf{TV},r_k}^{(t)}\right] } \notag \\
& \leq 2 \mu \big(B_d^c(r_k)\big)+c^*_{s,d}Mm_kt^s \notag \\
&\leq 2v_k+c^*_{s,d}Mm_kt^s. \notag &&
\end{flalign}
 Noting that the above holds with $\nu$ in place of $\mu$,
 we obtain
 \begin{align}
  \abs{ \tv{\mu}{\nu}-\tilde \delta_{\mathsf{TV}}\Big(f_{\mathsf{TV},r_k}^{(t)}\Big) } \leq  4v_k+2c^*_{s,d}Mm_kt^s. \label{apprxintrnctv}
\end{align}
 Recalling that $ \tilde \delta_{\mathsf{TV}}(g):=\EE_{\mu}[g]-\EE_{\nu}\left[g\right]$, we have
 \begin{flalign}
 & \Big|\tv{\mu}{\nu}- \tvn{\mu}{\nu}{\vec{\mathcal{G}}_{k}^{\mspace{3 mu}\mathsf{R}}\mspace{-2 mu}\big( \breve  c_{d,r_k,t},r_k\big)} \Big| \notag \\
 &\qquad \qquad\qquad\qquad\quad\stackrel{(a)}{=}\tv{\mu}{\nu}- \tvn{\mu}{\nu}{\vec{\mathcal{G}}_{k}^{\mspace{3 mu}\mathsf{R}}\mspace{-2 mu}\big( \breve  c_{d,r_k,t},r_k\big)}  \notag \\
 & \qquad \qquad\qquad\qquad \quad = \tv{\mu}{\nu}-\tilde \delta_{\mathsf{TV}}\Big(f_{\mathsf{TV},r_k}^{(t)}\Big)+\tilde \delta_{\mathsf{TV}}\Big(f_{\mathsf{TV},r_k}^{(t)}\Big)- \tvn{\mu}{\nu}{\vec{\mathcal{G}}_{k}^{\mspace{3 mu}\mathsf{R}}\mspace{-2 mu}(\breve c_{d,r_k,t})} \notag \\
 & \qquad \qquad\qquad\qquad \quad\stackrel{(b)}{\leq} 4v_k+2c^*_{s,d}Mm_kt^s+  \EE_{\mu} \left[\abs{f_{\mathsf{TV},r_k}^{(t)}-g_{\theta_k}}\right]+\EE_{\nu} \left[\abs{f_{\mathsf{TV},r_k}^{(t)}-g_{\theta_k}}\right] \notag \\
 & \qquad \qquad\qquad\qquad \quad \stackrel{(c)}{\lesssim_{d,M,s}} v_k+m_kt^s+ r_k^{d+1} t^{-(d+2)}k^{-\frac 12}+ \mu \big(B_d^c(r_k)\big)+ \nu \big(B_d^c(r_k)\big) \notag \\
  & \qquad \qquad\qquad\qquad \quad  \lesssim v_k+m_kt^s+  r_k^{d+1} t^{-(d+2)}k^{-\frac 12}, \notag  &&
 \end{flalign}
 where $(a)$ follows since $\big\|g\big\|_{\infty}\leq 1$ for $g \in \vec{\mathcal{G}}_{k}^{\mspace{3 mu}\mathsf{R}}\big( \breve  c_{d,r_k,t},r_k\big)$ and \eqref{tvdistvarcharac};
 $(b)$  uses \eqref{bndtvapperrint} and  \eqref{apprxintrnctv}; and 
$(c)$ is due to \eqref{approxcmpctsuptv}. Setting $t=t_{k,s}^*=\big(r_k^{d+1}k^{-1/2}m_k^{-1}\big)^{1/(s+d+2)}$ yields
\begin{align}
  \bigg|\tv{\mu}{\nu}- \tvn{\mu}{\nu}{\vec{\mathcal{G}}_{k}^{\mspace{3 mu}\mathsf{R}}\big( \breve  c_{d,r_k,t^*_{k,s}},r_k\big)} \bigg| \lesssim_{d,M,s} m_k^{\frac{d+2}{s+d+2}}r_k^{\frac{s(d+1)}{s+d+2}}k^{-\frac{s}{2(s+d+2)}}+  v_k. \notag
\end{align}
Then, defining 
  \begin{align}
      \vec{c}_{k,d,s,\mathbf{m},\mathbf{r}}:=\breve c_{d,r_k,t^*_{k,s}}= O_d\Big(\big(r_k^{s(d+1)}k^{0.5(d+2)}m_k^{d+2}\big)^{\frac{1}{s+d+2}}\Big),\label{consttvscnc}
  \end{align}
we have  from the above equation and  \eqref{empesterrortvfin}  that 
  \begin{align}
       \mathbb{E}\bigg[  \Big|\tvf_{\vec{\cG}^{\mspace{3 mu }\mathsf{R}}_k\big(\vec{c}_{k,d,s,\mathbf{m},\mathbf{r}},r_k\big)}(X^n,Y^n) -\tv{\mu}{\nu}\Big|\bigg] 
     &\mspace{-3 mu}\lesssim_{d,M,s,\rho} m_k^{\frac{d+2}{s+d+2}}r_k^{\frac{s(d+1)}{s+d+2}}k^{-\frac{s}{2(s+d+2)}}+v_k\notag \\
     & \qquad \quad +n^{-\frac 12}\Big(m_kr_k^{s+1}k^{\frac 12}\Big)^{\frac{d+2}{s+d+2}}.\label{errbndtvfin}
  \end{align}
  This completes the proof of Part $(ii)$ by taking supremum w.r.t. $(\mu,\nu) \in \bar{\mathcal{P}}^2_{\mathsf{TV},\psi}(M,s,\mathbf{r},\mathbf{m})$ and noting that $v_k \leq  \big(\psi\big(r_kM^{-1}\big)\big)^{-1}$ by Markov's inequality.

  \subsubsection{Proof of Corollary \ref{cor:tvsubgaussrate}} \label{cor:tvsubgaussrate-proof}
 We will use the  following relation between sub-Gaussian and norm sub-Gaussian distributions. $\mu \in \cP\big(\RR^d\big)$ is $\sigma^2$-norm sub-Gaussian for $\sigma>0$ if $X \sim \mu$ satisfies
\begin{align}
    \mu\big(\norm{X-\EE [X]} > t\big) \leq 2e^{\frac{-t^2}{2\sigma^2}},\quad \forall~ t \in \RR. \notag
\end{align}
\begin{lemma}{\citep[Lemma 1]{jin-2019-subgaussnorm}} \label{lem:normsubgauss}
If $\mu \in \cP\big(\RR^d\big)$ is $\sigma^2$-sub-Gaussian, then it is $8d \sigma^2$-norm sub-Gaussian.
\end{lemma}

\medskip

Continuing with the proof of the Corollary, fix $(\mu,\nu) \in \hat{\mathcal{P}}^2_{\mathsf{TV}}(b, M,N)$.  From the above lemma, we have for $\mu \in \mathcal{SG}(M)$ and $t \geq M$ that 
\begin{align}
    \mu \big(B_d^c(t)\big) \leq \mu\left(\norm{X-\EE_{\mu}[X]}+\norm{\EE_{\mu}[X]} > t\right) \leq  2e^{\frac{-(t-\norm{\EE_{\mu}[X]})^2}{16d\sigma^2}} \leq 2e^{\frac{-(t-M)^2}{16dM}}.  \label{subgaussbndmu}
\end{align}
Similar bound holds with $\nu$ in place of $\mu$. Next, since $(\mu,\nu) \in \hat{\mathcal{P}}^2_{\mathsf{TV}}(b, M,N)$,  following the steps leading to \eqref{genbndsuppMval} yields 
  \begin{flalign}
      \norm{f_{\mathsf{TV},r_k}}_{\mathsf{Lip}(s,1)}&= \lambda(B_d(r_k))+2N \frac{\pi^{\frac d 2}b^{d-s}}{\Gamma\Big(\frac{d}{2}+1\Big)} \vee 2 b^{-s}\lambda(B_d(r_k)) \notag \\
      &= \frac{\pi^{\frac d 2}r_k^d}{\Gamma\Big(\frac{d}{2}+1\Big)} +2N \frac{\pi^{\frac d 2}b^{d-s}}{\Gamma\Big(\frac{d}{2}+1\Big)} \vee 2 b^{-s} \frac{\pi^{\frac d 2}r_k^d}{\Gamma\Big(\frac{d}{2}+1\Big)}=: c_{d,s,b,N,r_k}.  \label{constvsuffcond} &&
  \end{flalign}
Then, it follows from \eqref{errbndtvfin} with
$r_k= M \vee 1+4 \sqrt{dM\log k}$, $v_k=2e^{-(r_k-M)^2/16dM}$ and  $m_k= c_{d,s,b,N,r_k}$ that 
   \begin{flalign}
     &  \mathbb{E}\left[  \abs{\tvf_{\vec{\cG}^{\mspace{3 mu }\mathsf{R}}_k\big(\vec{c}_{k,d,s,\mathbf{m},\mathbf{r}},r_k\big)}(X^n,Y^n) -\tv{\mu}{\nu}}\right] \notag \\
     &\qquad \qquad \lesssim_{d,s,b,N}(\log k)^{\frac{(s+d)(d+2)}{2(s+d+2)}}k^{-\frac{s}{2(s+d+2)}}+ k^{-1} +(\log k)^{\frac{d+2}{2}}k^{\frac{d+2}{2(s+d+2)}}n^{-\frac 12}  \notag \\
     &\qquad \qquad \lesssim_{d,s,b,N}(\log k)^{\frac{(s+d)(d+2)}{2(s+d+2)}}k^{-\frac{s}{2(s+d+2)}}+(\log k)^{\frac{d+2}{2}}k^{\frac{d+2}{2(s+d+2)}}n^{-\frac 12}.\notag &&
  \end{flalign}
This completes the proof by taking supremum w.r.t. $(\mu,\nu) \in \hat{\mathcal{P}}^2_{\mathsf{TV}}(b, M,N)$.

 \section{Proofs of Lemmas in Appendix \ref{sec:proofs} } \label{Sec:prooflemmas}
 \subsection{Proof of  Lemma  \ref{lem:suffcondbar}}
\label{app:auxlemma}
Suppose $f \in  \cL^{\mathsf{KB}}_{s_{\mathsf{KB}},b}\big(\RR^d\big)$. Since $f \in L^1\big(\RR^d\big)$, its Fourier transform $\mathfrak{F}[f]:\RR^d \rightarrow \RR$ is well-defined.  Also,
 \begin{flalign}
 \int_{\RR^d}  \abs{\mathfrak{F}[f](\omega)}\dd\omega &\stackrel{(a)}{\leq}  \left(\int_{\RR^d}  \frac{\dd \omega}{1+\norm{\omega}^{2s_{\mathsf{KB}}}}\right)^{\frac 12} \left(\int_{\RR^d}  \left(1+\norm{\omega}^{2s_{\mathsf{KB}}}\right)\abs{\mathfrak{F}[f](\omega)}^2\dd\omega\right)^{\frac 12} \notag \\
   & \stackrel{(b)}{\leq } \left(\int_{\RR^d}  \frac{\dd \omega}{1+\norm{\omega}^{2s_{\mathsf{KB}}}}\right)^{\frac 12} \left(\norm{f}^2_2+d^{s_{\mathsf{KB}}}\max_{\alpha:\norm{\alpha}_1=s_{\mathsf{KB}}}\norm{D^{\alpha}f}^2_2\right)^{\frac 12} \stackrel{(c)}{<}\infty, \notag &&
\end{flalign}
where 
\begin{enumerate}[label = (\alph*),leftmargin=*]
    \item follows from  Cauchy-Schwarz inequality;
    \item is by Plancherel's theorem since  $ \mathfrak{F}[D^{\alpha}f](\omega)=\mathfrak{F}[f](\omega)\prod_{j=1}^d(i \omega_j)^{\alpha_{j}},~\forall~\norm{\alpha}_1 \leq  s_{\mathsf{KB}},$
     where $i$ denotes the imaginary unit $\sqrt{-1}$, and $f \in  \cL^{\mathsf{KB}}_{s_{\mathsf{KB}},b}\big(\RR^d\big)$. 
     The above identity holds because  $\norm{D^{\alpha}f}_1<\infty$ for all $\norm{\alpha}_1 \leq s_{\mathsf{KB}}$ by assumption.
     \item follows since the first integral is finite and $f \in  \cL^{\mathsf{KB}}_{s_{\mathsf{KB}},b}\big(\RR^d\big)$.
\end{enumerate}
  Hence, $\mathfrak{F}[f] \in L^1\big(\RR^d\big)$  and the Fourier inversion formula holds (at every $x \in \RR^d$ since $f \in  \cL_{s_{\mathsf{KB}},b}\big(\RR^d\big)$ is necessarily continuous) with $F(d\omega)=\mathfrak{F}[f](\omega)\dd\omega$, i.e., $ f(x)=\int_{0}^{\infty} e^{i \omega \cdot x} \mathfrak{F}[f](\omega)\dd \omega$. 
   Then,  it follows from $\norm{\omega}_1 \leq \sqrt{d} \norm{\omega}$ that
    \begin{align}
  & S_2(f):=  \int_{\RR^d} \norm{\omega}_1^2 \big|\mathfrak{F}[f](\omega)\big|\dd\omega \leq d \int_{\RR^d} \norm{\omega}^2 \big|\mathfrak{F}[f](\omega)\big|\dd\omega. \label{CSapprox}
\end{align}

If $\norm{D^{\alpha}f}_2 \leq b$ for all $\alpha$ with $\norm{\alpha}_1 \in \{1,s_{\mathsf{KB}}\}$, then  we have
\begin{flalign}
 \int_{\RR^d} \norm{\omega}^2 \big|\mathfrak{F}[f](\omega)\big|\dd\omega &\stackrel{(a)}{\leq}  \left(\int_{\RR^d}  \frac{\dd \omega}{1+\norm{\omega}^{2(s_{\mathsf{B}}-1)}}\right)^{\frac 12} \left(\int_{\RR^d}  \left(\norm{\omega}^4+\norm{\omega}^{2s_{\mathsf{KB}}}\right)\big|\mathfrak{F}[f](\omega)\big|^2\dd\omega\right)^{\frac 12} \notag\\
   &  \stackrel{(b)}{\leq} \left(\int_{\RR^d}  \frac{\dd \omega}{1+\norm{\omega}^{2(s_{\mathsf{B}}-1)}}\right)^{\frac 12} \big(d^2+d^{s_{\mathsf{KB}}}\big)^{\frac 12} b, && \label{bndfcfval}
\end{flalign}
where 
\begin{enumerate}[label = (\alph*),leftmargin=*]
    \item follows from  Cauchy-Schwarz inequality;
    \item is due to Plancherel's theorem and $f \in  \cL^{\mathsf{KB}}_{s_{\mathsf{KB}},b}\big(\RR^d\big)$.
\end{enumerate}
Combining \eqref{CSapprox} and \eqref{bndfcfval} yields $S_2(f)  \leq bd^{3/2}\kappa_d$. Following similar steps, it can be shown that if $f \in  \cL^{\mathsf{B}}_{s_{\mathsf{B}},b}\big(\RR^d\big)$, then $S_1(f)  \leq bd^{1/2}\kappa_d$. The final claims follows from these and definition of the classes $\cL^{\mathsf{B}}_{s_{\mathsf{B}},b}\big(\RR^d\big)$, $\cL^{\mathsf{KB}}_{s_{\mathsf{KB}},b}\big(\RR^d\big)$, $ \cB_{c,1,\cX}\big(\RR^d\big)$ and $ \cB_{c,2,\cX}\big(\RR^d\big)$. 

\subsection{Proof of Lemma \ref{lem:covnumbnd}} \label{lem:covnumbnd-proof}
Assume that $\phi$ is monotone increasing. Let $g, \tilde g \in \cG_k(\mathbf{a}_k,\phi)$ be arbitrary, where $g(x)=\sum_{i=1}^k \beta_i \phi\left(w_i\cdot x+b_i\right)+w_0 \cdot x + b_0$ and $\tilde g(x)=\sum_{i=1}^k \tilde \beta_i \phi\left(\tilde w_i\cdot x+\tilde b_i\right)+\tilde w_0 \cdot x + \tilde b_0$. Define $\boldsymbol{\beta}:=(\beta_1, \ldots, \beta_k)$, $\tilde{\boldsymbol{\beta}}:=(\tilde \beta_1, \ldots, \tilde \beta_k)$, $\mathbf{w}=(w_1, \ldots,w_k)$, $\tilde{\mathbf{w}}=(\tilde w_1, \ldots, \tilde w_k)$, $\mathbf{b}=(b_1,\ldots,b_k)$ and $\tilde{\mathbf{b}}=(\tilde b_1,\ldots,\tilde b_k)$. Note that $\boldsymbol{\beta},\tilde{\boldsymbol{\beta}},\mathbf{b}, \tilde{\mathbf{b}} \in \RR^k$ and $\mathbf{w}, \tilde{\mathbf{w}} \in \RR^{kd}$. For any $x \in \cX$, we have
\begin{flalign}
   & \abs{g(x)-\tilde g(x)} \notag \\
   &\leq \abs{\sum_{i=1}^k \beta_i \phi\left(w_i\cdot x+b_i\right)-\sum_{i=1}^k \tilde \beta_i \phi\left(\tilde w_i\cdot x+\tilde b_i\right)}+ \abs{(w_0-\tilde w_0)\cdot x}+\big|b_0-\tilde b_0\big| \notag \\
    & \leq \abs{\sum_{i=1}^k \beta_i \phi\left(w_i\cdot x+b_i\right)-\sum_{i=1}^k \tilde \beta_i \phi\left( w_i\cdot x+ b_i\right)}+\norm{w_0-\tilde w_0}_1 \norm{\cX}+\big|b_0-\tilde b_0\big|\notag \\
    & \qquad +\abs{\sum_{i=1}^k \tilde \beta_i \phi\left(w_i\cdot x+b_i\right)-\sum_{i=1}^k \tilde \beta_i \phi\left(\tilde w_i\cdot x+\tilde b_i\right)}\notag \\
     & \stackrel{(a)}{\leq}\norm{\boldsymbol{\beta}-\tilde{\boldsymbol{\beta}}}_1 \phi\left(\sup_{x \in \cX, 1 \leq i \leq k} |w_i\cdot x+b_i|\right)+\norm{w_0-\tilde w_0}_1 \norm{\cX}+\big|b_0-\tilde b_0\big|\notag \\
    & \qquad +L \sum_{i=1}^k \big|\tilde \beta_i\big| \big|(w_i-\tilde w_i)\cdot x+b_i-\tilde b_i\big| \notag \\
         & \stackrel{(b)}{\leq} \norm{\boldsymbol{\beta}-\tilde{\boldsymbol{\beta}}}_1 \phi\big(a_{1,k}(\norm{\cX}+1)\big)+\norm{w_0-\tilde w_0}_1 \norm{\cX}+\big|b_0-\tilde b_0\big|+La_{2,k}\norm{\cX}\norm{\mathbf{w}-\tilde{\mathbf{w}}}_1\notag \\
         & \qquad +La_{2,k}\big\|\mathbf{b}-\tilde{\mathbf{b}}\big\|_1\notag, &&
\end{flalign}
where 
\begin{enumerate}[label = (\alph*),leftmargin=*]
    \item is since $\phi$ is monotone increasing function with Lipschitz constant bounded by $L$;
    \item is because $\max_{1 \leq i \leq k}\norm{w_i}_1 \vee |b_i|\leq a_{1,k}$ and $\max_{1 \leq i \leq k}|\tilde \beta_i| \leq a_{2,k}$.
\end{enumerate}
Defining $u_k=\phi\big(a_{1,k}(\norm{\cX}+1)\big)$, it follows by application of \eqref{euclideancovnumb} that
\begin{align}
N\big(\epsilon,\cG_k(\mathbf{a}_k,\phi),\norm{\cdot}_{\infty,\X}\big) &\leq N\big(\epsilon/5,[-a_{2,k},a_{2,k}]^k,u_k\norm{\cdot}_1\big)N\big(\epsilon/5,B_{d}^1(a_{4,k}),\norm{\cX}\norm{\cdot}_1\big) \notag \\
& \qquad  N\big(\epsilon/5,[-a_{3,k},a_{3,k}],|\cdot|\big)N\big(\epsilon/5,B_{kd}^1(ka_{1,k}),La_{2,k}\norm{\cX}\norm{\cdot}_1\big) \notag \\
& \qquad N\big(\epsilon/5,B_{k}^1(ka_{1,k}),La_{2,k}\norm{\cdot}_1\big) \notag\\
& \leq \big(1+10ka_{2,k}u_k\epsilon^{-1}\big)^k \big(1+10a_{4,k}\norm{\cX}\epsilon^{-1}\big)^d\big(1+10a_{3,k}\epsilon^{-1}\big) \notag \\
& \qquad \big(1+10Lka_{1,k}a_{2,k}\norm{\cX}\epsilon^{-1}\big)^{dk} \big(1+10Lka_{1,k}a_{2,k}\epsilon^{-1}\big)^{k}.\notag
\end{align}
If $\phi$ is monotone decreasing, the above holds with $u_k=\phi\big(-a_{1,k}(\norm{\cX}+1)\big)$. This proves the first bound in Lemma \ref{lem:covnumbnd}.  Specializing to NN classes $\cG_k^{\mathsf{R}}(a)$, $\cG_k^{\mathsf{S}}(a)$, $\cG_k^*(\phi_{\mathsf{R}})$, and $\cG_k^*(\phi_{\mathsf{S}})$ by noting that the Lipschitz constant $L \leq 1$ for $\phi_{\mathsf{R}}$ and $\phi_{\mathsf{S}}$, $\abs{\phi_{\mathsf{R}}(x)} \leq x$, and  $\abs{\phi_{\mathsf{S}}(x)} \leq 1$, yields \eqref{covbndscrelu}-\eqref{covbndforcons}.

\subsection{Proof of Lemma \ref{lem:consicomp}}\label{lem:consicomp-proof}
We will use Theorem \ref{empesterrbnd} for the proof. Fix any $(\mu,\nu) \in\mathcal{P}^2_{\mathsf{KL}}(\X)$. Note that for $h_{\mathsf{KL}}(x)=e^x-1$, we have $\bar C\big(|\cG_k^*(\phi)|,\cX\big) \leq k(\norm{\cX}+1)+1$,
 \begin{align}
    & \bar C\left(\abs{h_{\mathsf{KL}}' \circ \cG_k^*(\phi)},\cX\right)\leq e^{k(\norm{\cX}+1)+1},  \label{derbndkl} \\
   &V_{k,h,\phi,\cX} \lesssim  \big(k(\norm{\cX}+1)+1 \big)^2\big(e^{k(\norm{\cX}+1)+1}+1\big)^2 , \notag
 \end{align}
 where $h_{\mathsf{KL}}'$ denotes the derivative of $h_{\mathsf{KL}}$ and $V_{k,h,\phi,\cX}$ is given in \eqref{Vkconstdef}. Also, observe that since $g \in \cG_k^*(\phi)$ is continuous and bounded, $\mathsf{D}_{\cG_k^*(\phi)}(\mu,\nu) \leq \kl{\mu}{\nu} <\infty$. Then, since
 \begin{align}
 E_{k,h,\phi,\cX}\mspace{2 mu}n^{-\frac 12} \lesssim n^{-\frac 12}  k\sqrt{d(\norm{\cX}+1)}\sqrt{k(\norm{\cX}+1)+1} \left(e^{k(\norm{\cX}+1)+1}+1\right) \xrightarrow[n\rightarrow \infty]{} 0,   \notag
 \end{align}
 for  $k$ such that  $k^{3/2}(\norm{\cX}+1) e^{k(\norm{\cX}+1)} = O \left(n^{(1-\rho)/2}\right)$ for $0<\rho<1$, it follows from \eqref{bndesterremp} that for any $k \in \mathbb{N}$, $\delta>0$, and $n$ sufficiently large, 
 \begin{align}
     \mathbb{P}\left(\abs{\mathsf{D}_{\cG_k^*(\phi)}(\mu,\nu) -
     \hat{\mathsf{D}}_{\cG_k^*(\phi)}(X^n,Y^n) } \geq \delta \right) \leq  c e^{-\frac{n\big(\delta-E_{k,h,\phi,\cX}n^{-1/2}\big)^2}{V_{k,h,\phi,\cX}}}. \notag
 \end{align}
 Hence,  for  $k_n$  such that   $k_n^{3/2}(\norm{\cX}+1)e^{k_n (\norm{\cX}+1)} =O \left(n^{(1-\rho)/2}\right)$,
\begin{flalign}
  \sum_{n=1}^{\infty}  \mathbb{P}\left(\abs{\mathsf{D}_{\cG_k^*(\phi)}(\mu,\nu) -
     \hat{\mathsf{D}}_{\cG_k^*(\phi)}(X^n,Y^n)} \geq \delta \right) &\leq  c\sum_{n=1}^{\infty}e^{\frac{-n\big(\delta-E_{k_n,h,\phi,\cX}n^{-1/2}\big)^2}{V_{k_n,h,\phi,\cX}}} <\infty, \label{finiteasbc}
\end{flalign}
where the final inequality in  \eqref{finiteasbc} can be established via integral test for sum of series. This implies \eqref{errbndest} via the first Borel-Cantelli lemma by  taking supremum w.r.t. $(\mu,\nu) \in\mathcal{P}^2_{\mathsf{KL}}(\X)$. 

 \subsection{Proof of Lemma \ref{lem:consicompchisq}} \label{lem:consicompchisq-proof}
Fix $(\mu,\nu)\in\mathcal{P}_{\chi^2}^2(\X)$. Recalling the quantities defined in Theorem \ref{empesterrbnd}, we have for $h_{\chi^2}(x)=x+0.25 x^2$ that
\begin{align}
   &  \bar C\big(\big|h_{\chi^2}' \circ \cG_k^*(\phi)\big|,\cX\big)  \leq 0.5k(\norm{\cX}+1)+1.5, \label{derbndchisq} \\
    & V_{k,h,\phi,\cX} \lesssim  (k(\norm{\cX}+1)+1)^2(0.5k(\norm{\cX}+1)+1.5)^2,\notag \\
   & E_{k,h,\phi,\cX} \lesssim k\sqrt{d(\norm{\cX}+1)}  \big(0.5k(\norm{\cX}+1)+1.5\big) \sqrt{k(\norm{\cX}+1)+1}, \notag
\end{align}
where $h_{\chi^2}'$ denotes the derivative of $h_{\chi^2}$. Also, note that $ \chi^2_{\cG_k^*(\phi)}(\mu,\nu) \leq \chisq{\mu}{\nu}<\infty$. Then, since
\begin{align}
  0 \leq  E_{k,h,\phi,\cX}\mspace{2 mu}n^{-\frac 12} \lesssim  k^{\frac 52}(\norm{\cX}+1)^2 n^{-\frac 12} \xrightarrow[n\rightarrow \infty]{}  0, \notag
\end{align}
for $k^{5/2}(\norm{\cX}+1)^2 = O \left(n^{(1-\rho)/2}\right)$, $0<\rho<1$, it follows from \eqref{bndesterremp} that for any $k \in \mathbb{N}$, $\delta>0$, and $n$ sufficiently large, 
 \begin{align}
     \mathbb{P}\left(\abs{\hat{\chi}^2_{\cG_k^*(\phi)}(X^n,Y^n)-\chi^2_{\cG_k^*(\phi)}(\mu,\nu)} \geq \delta \right) \leq c e^{-\frac{n\big(\delta-E_{k,h,\phi,\cX}n^{- 1/2}\big)^2}{V_{k,h,\phi,\cX}}}. \notag
 \end{align}
 Then, \eqref{errbndestchisq} follows using similar steps used to prove \eqref{errbndest} (see  \eqref{finiteasbc}).  This completes the proof.

 \subsection{Proof of Lemma \ref{lem:consicomphel}} \label{lem:consicomphel-proof}
Fix $(\mu,\nu) \in\mathcal{P}_{\mathsf{H}^2}^2(\X)$. Note that $h_{\mathsf{H}^2}(x)=x/(1-x)$ and 
\begin{align}
   & \bar C\big(\big|h_{\mathsf{H}^2}' \circ \tilde{\cG}^*_{k,t}(\phi)\big|\big)  =\sup_{\substack{g_{\theta} \in \tilde{\cG}^*_{k,t}(\phi),x \in \X} }(1-g_{\theta}(x))^{-2} \leq t^{-2}, \notag
\end{align}
where $h_{\mathsf{H}^2}'$ denotes  derivative of $h_{\mathsf{H}^2}$.
 By examining the proof, it can be seen  that   Theorem \ref{empesterrbnd} continues to hold with $\cG_k^*(\phi)$ in \eqref{maxdergamma} and \eqref{bndesterremp} replaced with $\tilde{\cG}^*_{k,t}(\phi)$. We have $ V_{k,h,\phi,\cX} \lesssim  (k(\norm{\cX}+1)+1)^2\left(t_k^{-2}+1\right)^2$, and 
\begin{align}
  0 \leq   E_{k,h,\phi,\cX}n^{-\frac 12}\lesssim n^{-\frac 12}k\sqrt{d(\norm{\cX}+1)}  \big(t_k^{-2}+1\big) \sqrt{k(\norm{\cX}+1)+1} \xrightarrow[n\rightarrow \infty]{}  0, \notag
\end{align}
 for  $k,t_k$ such that  $k^{3/2}(\norm{\cX}+1)t_k^{-2}= O \left(n^{(1-\rho)/2}\right)$. Further, $\mathsf{H}^2_{\tilde{\cG}^*_{k,t}(\phi)}(\mu,\nu) <\mathsf{H}^2(\mu,\nu) \leq 2$. It then follows from \eqref{bndesterremp} that for any $k \in \mathbb{N}$, $\delta>0$, and $n$ sufficiently large, 
 \begin{align}
     \mathbb{P}\left(\Big|\hat{\mathsf{H}}^2_{\tilde{\cG}^*_{k,t_k}(\phi)}(X^n,Y^n)-\mathsf{H}^2_{\tilde{\cG}^*_{k,t_k}(\phi)}(\mu,\nu) \Big| \geq \delta \right) \leq  c e^{-\frac{n\big(\delta-E_{k,h,\phi,\cX}n^{- 1/2}\big)^2}{V_{k,h,\phi,\cX}}}. \notag
 \end{align}
 Then, \eqref{errbndesthel} follows via similar steps used to prove \eqref{errbndest} (see  \eqref{finiteasbc}).

\subsection{Proof of Lemma \ref{lem:consicomptv}} \label{lem:consicomptv-proof}
Fix $\mu,\nu \in \cP(\X)$. We have $\tvn{\mu}{\nu}{\bar\cG_k^*(\phi)} \leq \tv{\mu}{\nu} \leq 2$, $ \bar C\left(\abs{\bar\cG_k^*(\phi)}\right) \leq 1$, and
\begin{align}
   \bar C\left(\abs{h_{\mathsf{TV}}' \circ \bar\cG_k^*(\phi)}\right)  =1, \notag
\end{align} 
where $h_{\mathsf{TV}}'$ denotes the derivative of $h_{\mathsf{TV}}$. Also, it can be seen from the proof of Theorem \ref{empesterrbnd} that it holds with $\cG_k^*(\phi)$ in \eqref{maxdergamma} and \eqref{bndesterremp} replaced by $\bar\cG_k^*(\phi)$.
Further, $V_{k,h,\phi,\cX} \lesssim  1$, and 
\begin{align}
  0 \leq    E_{k,h,\phi,\cX} n^{-\frac 12}\lesssim  n^{-\frac 12}k\sqrt{d(\norm{\cX}+1)} \xrightarrow[n\rightarrow \infty]{}  0, \notag
\end{align}
 for  $k,n$ such that  $k(\norm{\cX}+1)^{1/2}= O \left(n^{(1-\rho)/2}\right)$.
It follows from \eqref{bndesterremp} that for any $k \in \mathbb{N}$, $\delta>0$, and $n$ sufficiently large, 
 \begin{align}
     \mathbb{P}\left(\abs{\tvf_{\bar\cG_k^*(\phi)}(X^n,Y^n)-
   \tvn{\mu}{\nu}{\bar\cG_k^*(\phi)}} \geq \delta \right) \leq  c e^{-\frac{n\big(\delta- E_{k,h,\phi,\cX}n^{-1/2}\big)^2}{ V_{k,h,\phi,\cX}}}. \notag
 \end{align}
 Then, \eqref{errbndesttv}  follows using similar steps used to prove \eqref{errbndest}.  This completes the proof.

  \subsection{Proof of Lemma \ref{lem:genmvncKL}}\label{lem:genmvncKL-proof}
Fix $(\mu,\nu) \in \breve{\mathcal{P}}^2_{\mathsf{KL}}(M,\mathbf{r},\mathbf{m}, \mathbf{v})$. Recall that $
\hat \cG_k^{\mathsf{R}}(a,r)=\{g \ind_{B_d(r)}: g\in \cG_k^{\mathsf{R}}(a)\}$. 
Since $c_{\mathsf{KB}}^\star\left(f_{\mathsf{KL}}|_{B_d(r_k)},B_d(r_k)\right) \leq m_k$, 
it follows from \eqref{approxrateklubar} that there exists $g_{\theta_k} \in \hat{\cG}_{k}^{\mathsf{R}}(m_k,r_k)$ and $c>0$ such that
\begin{align} 
\Big\|f_{\mathsf{KL}}-g_{\theta_k}\Big\|_{\infty,B_d(r_k)} \leq c
   d^{\frac 12} m_kk^{-\frac 12}. \label{approxerrgennc}
\end{align}
Then, following steps leading to \eqref{genbndnckld}, we have for $k$ with $c^2d m_k^2<0.5 k$ that
\begin{flalign}
  &   \abs{\kl{\mu}{\nu}-\mathsf{D}_{\hat{\cG}_{k}^{\mathsf{R}}(m_k,r_k)}(\mu,\nu)} \notag \\
    & \leq \norm{(f_{\mathsf{KL}}- g_{\theta_k})\ind_{B_d(r_k)}}_{\infty,\mu}   + \EE_{\mu}\left[\abs{f_{\mathsf{KL}}} \ind_{B_d^c(r_k)}\right]+\EE_{\nu}\left[\abs{\frac{\dd \mu}{\dd \nu}-1}\ind_{B_d^c(r_k)}\right] \notag \\
    & \qquad \qquad  \qquad \qquad \qquad  \qquad  \quad \qquad  +\EE_{\nu}\left[\big|e^{f_{\mathsf{KL}}}\big|\ind_{B_d(r_k)}\right] \norm{\left(1-e^{g_{\theta_k}-f_{\mathsf{KL}}}\right)\ind_{B_d(r_k)}}_{\infty,\nu} \notag \\
       & \lesssim m_k d^{\frac 12} k^{-\frac 12}  +v_k, \notag&&
\end{flalign}
where the final inequality is due to \eqref{approxerrgennc}, $e^{c d^{\frac 12} m_k k^{-1/2}}-1 \leq c d^{\frac 12}m_k k^{-1/2}$ which follows similar to \eqref{simpKLbndgeom} (since $c^2d m_k^2<0.5 k$),  $\EE_{\mu}\big[|f_{\mathsf{KL}}| \ind_{B_d^c(r_k)}\big] \vee \EE_{\nu}\big[|(\dd \mu/\dd \nu)-1|\ind_{B_d^c(r_k)}\big] \leq v_k $, and $\EE_{\nu}\big[|e^{f_{\mathsf{KL}}}|\ind_{B_d(r_k)}\big] \leq 1.$

On the other hand, for $k$ such that $c^2 d m_k^2 \geq 0.5 k$,  $g=0 \in \hat{\cG}_{k}^{\mathsf{R}}(m_k,r_k)$ implies that
\begin{flalign}
  &   \abs{\kl{\mu}{\nu}-\mathsf{D}_{\hat{\cG}_{k}^{\mathsf{R}}(m_k,r_k)}(\mu,\nu)}=\kl{\mu}{\nu}-\mathsf{D}_{\hat{\cG}_{k}^{\mathsf{R}}(m_k,r_k)}(\mu,\nu)  \leq \kl{\mu}{\nu} \leq M. \notag
  \end{flalign}
 Since $m_k^2 \lesssim k^{1-\rho}$, $k$ such that $c^2d m_k^2 \geq 0.5 k$ necessarily satisfies $k^{\rho} \lesssim  d$. Thus, for all $k \in \NN$, 
  \begin{flalign}
  &   \abs{\kl{\mu}{\nu}-\mathsf{D}_{\hat{\cG}_{k}^{\mathsf{R}}(m_k,r_k)}(\mu,\nu)} \lesssim_{d,M,\rho} m_k k^{-\frac 12}  +v_k.  \label{kldivbndgeninterm}
  \end{flalign}
  Note that the RHS above tends to zero as $k \rightarrow \infty$ since $v_k \rightarrow 0$ and $m_k^2 \lesssim k^{1-\rho}$.

Next, it follows from \eqref{entrpyintgenvccls}, \eqref{entropyintscclass},  and \eqref{kldivbndgeninterm} that for $k,m_k$ satisfying  $m_k^2 \lesssim k^{1-\rho}$, 
 \begin{flalign}
&  \mathbb{E}\left[  \abs{\hat{\mathsf{D}}_{\hat{\mathcal{G}}_{k}^{\mathsf{R}}(m_k, r_k)}(X^n,Y^n) -\kl{\mu}{\nu}}\right]  \notag \\
&\qquad \leq \abs{\mathsf{D}_{\hat{\mathcal{G}}_{k}^{\mathsf{R}}(m_k, r_k)}(\mu,\nu) -\kl{\mu}{\nu}}   + \mathbb{E}\left[\abs{\mathsf{D}_{\hat{\mathcal{G}}_{k}^{\mathsf{R}}(m_k, r_k)}(\mu,\nu)-
    \hat{\mathsf{D}}_{\hat{\mathcal{G}}_{k}^{\mathsf{R}}(m_k, r_k)}(X^n,Y^n)}\right] \notag\\
  &\qquad  \lesssim_{d,M,\rho}
 m_k  k^{-\frac 12}  +v_k+ m_kr_ke^{3m_k(r_k+1)}n^{-\frac 12}. \notag&&
 \end{flalign}
 Taking supremum w.r.t. $(\mu,\nu) \in \breve{\mathcal{P}}^2_{\mathsf{KL}}(M,\mathbf{r},\mathbf{m}, \mathbf{v})$ completes the proof.

 \subsection{Proof of Lemma \ref{chisqgenmvbnd}} \label{chisqgenmvbnd-proof}
Fix $(\mu,\nu) \in \breve{\mathcal{P}}^2_{\chi^2}(\mathbf{r},\mathbf{m}, \mathbf{v})$. Since $c_{\mathsf{KB}}^\star\left(f_{\chi^2}|_{B_d(r_k)},B_d(r_k)\right) \leq m_k$, there exists $g_{\theta_k} \in \hat{\cG}_{k}^{\mathsf{R}}(m_k,r_k)$ 
such that 
\begin{align} 
\norm{f_{\chi^2}-g_{\theta_k}}_{\infty,B_d(r_k)}\lesssim 
    d^{\frac 12} m_kk^{-\frac 12}.
    \label{approxerrgenncchisq}
\end{align}
Then, following steps leading to \eqref{chisqlststp}, we have for all $k \in \NN$ that
\begin{flalign}
  &   \abs{\chisq{\mu}{\nu}-\chi^2_{\hat{\cG}_{k}^{\mathsf{R}}(m_k,r_k)}(\mu,\nu)} \notag \\
   &\leq  \norm{\big(f_{\chi^2}- g_{\theta_k}\big)\ind_{B_d(r_k)}}_{\infty,\mu} +\EE_{\mu}\left[\abs{f_{\chi^2}} \ind_{B_d^c(r_k)}\right] +   \EE_{\nu}\left[\big|h_{\chi^2} \circ f_{\chi^2}- h_{\chi^2} \circ g_{\theta_k}\big|\ind_{B_d(r_k)}\right]\notag \\
  &\qquad  \qquad \qquad \qquad  \qquad \qquad \qquad  \qquad \qquad \qquad  \qquad \qquad \qquad  +\EE_{\nu}\left[\abs{h_{\chi^2} \circ f_{\chi^2}}\ind_{B_d^c(r_k)}\right] \notag \\
   &\stackrel{(a)}{\lesssim}  d^{\frac 12} m_k k^{-\frac 12} +v_k +   \EE_{\nu}\left[\big|h_{\chi^2} \circ f_{\chi^2}- h_{\chi^2} \circ g_{\theta_k}\big|\ind_{B_d(r_k)}\right] \notag \\
    &\stackrel{(b)}{\lesssim} d^{\frac 12} m_k k^{-\frac 12} +v_k +\EE_{\nu}\left[\big|f_{\chi^2}- g_{\theta_k}\big|\ind_{B_d(r_k)}\right]+ \EE_{\nu}\Big[0.25\big|f_{\chi^2}- g_{\theta_k}\big|^2  \ind_{B_d(r_k)}\Big]  \notag \\
    & \qquad \qquad \qquad  \qquad \qquad \qquad  \qquad \qquad \qquad  \qquad \qquad  +0.5\EE_{\nu}\Big[\big|f_{\chi^2}- g_{\theta_k}\big|\abs{f_{\chi^2}}\ind_{B_d(r_k)}\Big]  \notag \\
     & \lesssim d^{\frac 12} m_k k^{-\frac 12} +v_k +  dm_k^2k^{-1}+ \norm{(f_{\chi^2}- g_{\theta_k})\ind_{B_d(r_k)}}_{\infty,\nu}  \EE_{\nu}\left[\abs{f_{\chi^2}}\right] \notag \\ 
       &\stackrel{(c)}{\lesssim}  d^{\frac 12} m_k k^{-\frac 12}  +dm_k^2 k^{-1}+v_k,  \notag&&
\end{flalign}
where 
 \begin{enumerate}[label = (\alph*),leftmargin=15 pt]
 \item follows from \eqref{approxerrgenncchisq} and since $ (\mu,\nu) \in \breve{\mathcal{P}}^2_{\chi^2}(\mathbf{r},\mathbf{m}, \mathbf{v})$;
 \item is via  steps leading to \eqref{chisqstepexp};
\item is due to \eqref{approxerrgenncchisq}  and $\EE_{\nu}\big[\abs{f_{\chi^2}}\big] \leq 4$.
\end{enumerate}
Then, it follows from the above equation, \eqref{entrpyintgenvccls} and \eqref{entropyintscclass} that  
 \begin{flalign}
&  \mathbb{E}\left[  \abs{\hat{\chi}^2_{\hat{\mathcal{G}}_{k}^{\mathsf{R}}(m_k, r_k)}(X^n,Y^n) -\chisq{\mu}{\nu}}\right]  \notag \\
&\qquad \qquad  \leq \abs{\chi^2_{\hat{\mathcal{G}}_{k}^{\mathsf{R}}(m_k, r_k)}(\mu,\nu) -\chisq{\mu}{\nu}}   + \mathbb{E}\left[\abs{\chi^2_{\hat{\mathcal{G}}_{k}^{\mathsf{R}}(m_k, r_k)}(\mu,\nu)-
    \hat{\chi}_{\hat{\mathcal{G}}_{k}^{\mathsf{R}}(m_k, r_k)}(X^n,Y^n)}\right] \notag\\
  &\qquad \qquad \lesssim
d^{\frac 12}m_k  k^{-\frac 12}+dm_k^2k^{-1}  +v_k+d^{\frac 32}m_k^2r_k^2 n^{-\frac 12}.  \notag &&
 \end{flalign}
   Taking supremum w.r.t. $(\mu,\nu) \in \breve{\mathcal{P}}^2_{\chi^2}(\mathbf{r},\mathbf{m}, \mathbf{v})$ completes the proof.

 \subsection{Proof of Lemma \ref{lem:sqhelmv}} \label{lem:sqhelmv-proof}
 Fix $(\mu,\nu) \in \breve{\mathcal{P}}^2_{\mathsf{H}^2}(\mathbf{r},\mathbf{m}, \mathbf{v})$. Since $\norm{\frac{\dd \mu}{\dd \nu}}_{\infty,B_d(r_k)}    \leq m_k$, we have
 \begin{align}
     1-f_{\mathsf{H}^2}(x)=\left(\frac{\dd \mu}{\dd \nu}(x)\right)^{-\frac{1}{2}} \geq m_k^{-\frac 12},~ x \in B_d(r_k). \label{heloptfnprop}
 \end{align}
Hence,  $c_{\mathsf{KB}}^\star\left(f_{\mathsf{H}^2}|_{B_d(r_k)},B_d(r_k)\right) \leq m_k$ implies via   \eqref{approxrateklubar} and \eqref{helclassnc} that there exists  $g_{\theta_k} \in \check{\cG}^{\mspace{2 mu}\mathsf{R}}_{k,m_k^{-1/2}}(m_k,r_k)$  such that 
\begin{align} 
    \norm{f_{\mathsf{H}^2}-g_{\theta_k}}_{\infty,B_d(r_k)} &\lesssim  m_k d^{\frac 12}k^{-\frac 12}. \label{approxerrheltr}
\end{align}
Following the derivation leading to the penultimate step in \eqref{consishel}, we have
\begin{flalign}
    &\bigg|\mathsf{H}^2(\mu,\nu)- \mathsf{H}^2_{ \check{\cG}^{\mspace{2 mu}\mathsf{R}}_{k,m_k^{-1/2}}(m_k,r_k)}(\mu,\nu)\bigg|  \notag \\
     &\leq  \EE_{\mu}\left[\big|f_{\mathsf{H}^2}- g_{\theta_k}\big|\ind_{B_d(r_k)}\right]+\EE_{\nu}\left[ \abs{\frac{f_{\mathsf{H}^2}- g_{\theta_k}}{(1-f_{\mathsf{H}^2})(1-g_{\theta_k})}}\ind_{B_d(r_k)}\right]  +  \EE_{\mu}\left[\abs{f_{\mathsf{H}^2}}\ind_{B_d^c(r_k)}\right]\notag \\
     & \qquad \qquad \qquad \qquad \qquad \qquad \qquad \qquad \qquad\qquad \qquad \qquad +\EE_{\nu}\left[ \abs{\frac{f_{\mathsf{H}^2}}{(1-f_{\mathsf{H}^2})}}\ind_{B_d^c(r_k)}\right] \notag \\
     & \stackrel{(a)}{\lesssim}   m_k d^{\frac 12}k^{-\frac 12}+  m_k^2d^{\frac 12} k^{-\frac 12}+v_k \notag \\
     &\stackrel{(b)}{\lesssim} m_k^2 d^{\frac 12} k^{-\frac 12}+v_k,\label{helapprxbndgenvm} &&
  \end{flalign}   
  where $(a)$
 follows from \eqref{heloptfnprop}, \eqref{approxerrheltr}, $(\mu,\nu) \in \breve{\mathcal{P}}^2_{\mathsf{H}^2}(\mathbf{r},\mathbf{m}, \mathbf{v})$, and $1-g_{\theta_k}(x) \geq  m_k^{-1/2} $ by the definition of $\check{\cG}^{\mspace{2 mu}\mathsf{R}}_{k,m_k^{-1/2}}(m_k,r_k)$, while $(b)$ is due to $m_k \geq 1$.

 Next, using \eqref{tailineqncsup} and  following steps similar to proof of Lemma \ref{lem:consicomp}, we obtain that for $k,m_{k},r_{k},n$ such that $k^{1/2}m_{k}^2r_{k} =O\left(n^{(1-\rho)/2}\right)$,
 \begin{align}
  \hat{\mathsf{H}}^2_{ \check{\cG}^{\mspace{2 mu}\mathsf{R}}_{k,m_{k}^{-1/2}}(m_{k},r_{k})}(X^n,Y^n) \xrightarrow[n \rightarrow \infty]{} \mathsf{H}^2_{\check{\cG}^{\mspace{2 mu}\mathsf{R}}_{k,m_k^{-1/2}}(m_k,r_k)}(\mu,\nu),\quad \mathbb{P}-\mbox{a.s.} \notag
 \end{align}
Then, \eqref{helconsisncgen} follows from this and \eqref{helapprxbndgenvm} since $m_k=o(k^{1/4})$  and $v_k \rightarrow 0$ by assumption.
 
Also,
 \begin{flalign}
&  \mathbb{E}\left[  \abs{\hat{\mathsf{H}}^2_{\check{\cG}^{\mspace{2 mu}\mathsf{R}}_{k,m_k^{-1/2}}\mspace{-2 mu}(m_k,r_k)}(X^n,Y^n) \mspace{-3 mu}-\mspace{-3 mu}\mathsf{H}^2(\mu,\nu)}\right] \notag \\
&\leq\mspace{-4 mu} \abs{\mathsf{H}^2(\mu,\nu)- \mathsf{H}^2_{\check{\cG}^{\mspace{2 mu}\mathsf{R}}_{k,m_k^{-1/2}}\mspace{-2 mu}(m_k,r_k)}(\mu,\nu)}  \mspace{-3 mu} +\mspace{-3 mu} \mathbb{E}\mspace{-4 mu}\left[\abs{\hat{\mathsf{H}}^2_{\check{\cG}^{\mspace{2 mu}\mathsf{R}}_{k,m_k^{-1/2}}\mspace{-2 mu}(m_k,r_k)}(X^n,Y^n)\mspace{-3 mu}-\mspace{-3 mu}
    \mathsf{H}^2_{\check{\cG}^{\mspace{2 mu}\mathsf{R}}_{k,m_k^{-1/2}}\mspace{-2 mu}(m_k,r_k)}(\mu,\nu)}\right] \notag \\
 &\lesssim  m_k^2 d^{\frac 12} k^{-\frac 12}+v_k+d^{\frac 32}m_k^2r_k n^{-\frac 12}, \notag&&
 \end{flalign}
 where the final inequality uses \eqref{entrpyintgenvccls}, \eqref{entropyintscclass} and \eqref{helapprxbndgenvm} to bound the last term.  Taking supremum over $(\mu,\nu) \in \breve{\mathcal{P}}^2_{\mathsf{H}^2}(\mathbf{r},\mathbf{m}, \mathbf{v})$ yields \eqref{genhelbndmktk}.

 \section{Consistency and effective error bounds for DV-NE} \label{App:DVNE-consefferr}
Defining $\mathsf{D}_{\mathsf{DV},\cG}(\mu,\nu):=\sup_{g \in \cG} \big( \EE_{\mu}[g]-\log \EE_{\nu}[e^{g}]\big) $ and
\begin{align}
\tilde Z_{g}:=\frac 1n \sum_{i=1}^n g(X_i)-\log \left(\frac 1n \sum_{i=1}^n e^{g(Y_i)}\right)-\EE_{\mu}\big[g\big]+\log \EE_{\nu}\big[e^g\big],\notag
\end{align}
we have similar to \eqref{supdiffer} that 
\begin{align}
     &  \check{\mathsf{D}}_{\mathsf{DV},\cG}(X^n,Y^n)-\mathsf{D}_{\mathsf{DV},\cG}(\mu,\nu)  \leq \sup_{g \in \cG} \tilde Z_{g}.\notag
     \end{align}
Moreover, since the Lipschitz constant of logarithm is bounded by $e^{\bar C(|\cG|,\cX)}$ in $\big[e^{-\bar C(|\cG|,\cX)},$ $e^{\bar C(|\cG|,\cX)}\big]$, we have almost surely that
\begin{flalign}
 \abs{Z_g-  Z_{\tilde g}} &\leq  n^{-1} \sum_{i=1}^n \abs{g(X_i)-\tilde g(X_i)-\EE_{\mu}\big[g-\tilde g\big]} +  e^{\bar C(\cG,\cX)} \abs{e^{g(Y_i)}-e^{\tilde g(Y_i)}-\EE_{\nu}\left[ e^g- e^{\tilde g}\right]}, \notag 
\end{flalign}
where each term inside the summation is bounded by $
2 \big( e^{2\bar C\left(\abs{\cG}\right)}+1\big) \big\|g_{\theta}-g_{\tilde{\theta}}\big\|_{\infty,\cX}$ similar to \eqref{bndtermstv}. Then, following the steps in the proof of Theorem \ref{empesterrbnd}, we have
\begin{flalign}
   & \sup_{\substack{\mu,\nu \in \cP(\X):\\\mathsf{D}_{\mathsf{DV},\cG_k^*(\phi)}(\mu,\nu)<\infty}} \mathbb{P}\Big(\abs{ \check{\mathsf{D}}_{\mathsf{DV},\cG_k^*(\phi)}(X^n,Y^n)-\mathsf{D}_{\mathsf{DV},\cG_k^*(\phi)}(\mu,\nu)}\geq \delta+ \tilde E_{k,h,\phi,\cX}n^{-\frac 12}\Big) \leq  c\, e^{-\frac{n\delta^2}{\tilde V_{k,h,\phi,\cX}}}, \notag &&
\end{flalign}
where $\tilde  V_{k,h,\phi,\cX} \lesssim (k(\norm{\cX}+1)+1)^2  e^{4k(\norm{\cX}+1)}$ and $\tilde E_{k,h,\phi,\cX} \lesssim k^{3/2}d^{1/2}(\norm{\cX}+1)  e^{2k(\norm{\cX}+1)}$. 
Then, similar to Lemma
\ref{lem:consicomp}, we obtain that 
for any  $0<\rho<1$, and   $n, k_n$  such that   $k_n^{3/2}(\norm{\cX}+1)e^{2k_n(\norm{\cX}+1)} =O \left(n^{(1-\rho)/2}\right)$,
\begin{align}
  \check{\mathsf{D}}_{\mathsf{DV},\cG_k^*(\phi)}(X^n,Y^n) \xrightarrow[n\rightarrow \infty]{}    \mathsf{D}_{\mathsf{DV},\cG_k^*(\phi)}(\mu,\nu),\quad  \mathbb{P}-\mbox{a.s.}\notag 
\end{align}
Moreover,  $\lim_{n \rightarrow \infty} \mathsf{D}_{\mathsf{DV},\cG_k^*(\phi)}(\mu,\nu) =\kl{\mu}{\nu}$ follows identical to \eqref{approxlim} provided $f_{\mathsf{KL}} \in \mathsf{C}(\X)$. Hence, for $\cX=[0,1]^d$, we obtain that for  any $0<\rho<1$,   $(k_n)_{n \in \NN}$ with $k_n\ \rightarrow  \infty$ and $k_n \leq \frac 18(1-\rho) \log n$, we have
 \begin{equation}
    \check{\mathsf{D}}_{\mathsf{DV},\cG_k^*(\phi)}(X^n,Y^n)  \xrightarrow[n\rightarrow \infty]{} \kl{\mu}{\nu},\quad \mathbb{P}-\mbox{a.s.} \label{DVconsisproof}
 \end{equation}
Next, we bound the expected error of the DV-NE estimator. Note that
\begin{flalign}
   & \check{\mathsf{D}}_{\mathsf{DV},\cG_k^{\mathsf{R}}(a)}(X^n,Y^n)-\mathsf{D}_{\mathsf{DV},\cG_k^{\mathsf{R}}(a)}(\mu,\nu)\notag \\
   &= \sup_{g \in \cG_k^{\mathsf{R}}(a)} \frac 1n\sum_{i=1}^n g(X_i)-\log \left(\frac 1n  \sum_{i=1}^n e^{g(Y_i)}\right)-\sup_{g \in \cG_k^{\mathsf{R}}(a)} \left(\EE_{\mu}[g]-\log \EE_{\nu}[e^g]\right)\notag \\
    & \leq   \sup_{g \in \cG_k^{\mathsf{R}}(a)} \frac 1n\sum_{i=1}^n g(X_i)-\log \left(\frac 1n  \sum_{i=1}^n e^{g(Y_i)}\right)- \left(\EE_{\mu}[g]-\log \EE_{\nu}[e^g]\right). \notag &&
\end{flalign}
   Thus, 
\begin{flalign}
   &\EE \left[\abs{\check{\mathsf{D}}_{\mathsf{DV},\cG_k^{\mathsf{R}}(a)}(X^n,Y^n)-\mathsf{D}_{\mathsf{DV},\cG_k^{\mathsf{R}}(a)}(\mu,\nu)} \right] \notag \\
   &  \leq  \EE \left[\sup_{g \in \cG_k^{\mathsf{R}}(a)} \abs{\frac 1n\sum_{i=1}^n g(X_i)-\EE_{\mu}[g]}\right]  +\EE \left[ \sup_{g \in \cG_k^{\mathsf{R}}(a)} \abs{\log \left(\frac 1n  \sum_{i=1}^n e^{g(Y_i)}\right)-\log \EE_{\nu}[e^g]}\right]\notag \\ 
    & \stackrel{(a)}{\leq}   \EE \left[\sup_{g \in \cG_k^{\mathsf{R}}(a)} \abs{\frac 1n\sum_{i=1}^n g(X_i)-\EE_{\mu}[g]}\right]  +e^{3a(\norm{\cX}+1)}\EE \left[\sup_{g \in \cG_k^{\mathsf{R}}(a)}  \abs{ \frac 1n  \sum_{i=1}^n e^{g(Y_i)}-\EE_{\nu}[e^g]}\right]\notag \\ 
   &\stackrel{(b)}{\lesssim} a(\norm{\cX}+1)\big(e^{6a(\norm{\cX}+1)}+1\big)n^{-\frac 12} \int_{0}^{1}\sqrt{ \sup_{\gamma \in \cP(\X)} \log N\left(3a(\norm{\cX}+1)\epsilon,\cG_k^{\mathsf{R}}(a),\|\cdot\|_{2,\gamma}\right)}\dd\epsilon, \notag \\
   & \stackrel{(c)}{\lesssim} a(\norm{\cX}+1)\big(e^{6a(\norm{\cX}+1)}+1\big)d^{\frac 32}n^{-\frac 12}, \label{DVNEexperrorbnd} &&
        \end{flalign}
        where 
         \begin{enumerate}[label = (\alph*),leftmargin=15 pt]
 \item is since $\bar C(|\cG_k^{\mathsf{R}}(a)|,\cX) \leq 3a(\norm{\cX}+1)$ and the Lipschitz constant of $\log x$ is bounded by $e^{3a(\norm{\cX}+1)}$ in $\big[e^{-\bar C(|\cG_k^{\mathsf{R}}(a)|,\cX)},$ $e^{\bar C(|\cG_k^{\mathsf{R}}(a)|,\cX)}\big]$;
 \item follows using steps akin to \eqref{entrintupbnd} and \citep[Corollary 2.2.8]{AVDV-book};
 \item is due to \eqref{entropyintscclass}.
\end{enumerate}

\section{CoD-Free Error Rate in the Unbounded Support Case}

\subsection{KL Divergence}\label{app:mildcodkl}
 Consider the NN class $\hat{\cG}_k^{\mathsf{S}}(a, r)=\big\{g \ind_{B_d(r)}: g \in \cG_k^{\mathsf{S}}(a)\big\}$ (see Definition \ref{def:NNclass}) and the following class of sub-Gaussian distributions:
\begin{flalign}
&\hat{\cP}^2_{\mathsf{KL}}(M,\ell):=\left\{(\mu,\nu) \in \mathcal{P}^2_{\mathsf{KL}}\big(\RR^d\big)   :  \mu,\nu \in \mathcal{SG}(M),~ f_{\mathsf{KL}} \in \hat{\mathcal{I}}(M),\norm{f_{\mathsf{KL}}}_{\ell,\mu} \leq M \right\}, \notag \\
&    \hat{\mathcal{I}}(M):=\left\{f: S_1(f) \vee |f(0)|\leq M  \right\}.\label{setKLncfn}
\end{flalign}

 \begin{proposition}[KL CoD-free error bound] \label{klmildCoD}
 Let $M \geq 0$, $\ell >1$ and $\ell^*=\ell/(\ell-1)$. Then, for   $z_k=12\sqrt{\ell^*d}M^{3/2}(\log k)^{-1/2}$ and  $r_k=M \vee 1+4\sqrt{dM \ell^*\log k}$, 
    \begin{flalign}
 \sup_{(\mu,\nu) \in \hat{\cP}^2_{\mathsf{KL}}(M,\ell)}\mathbb{E}\left[  \abs{\hat{\mathsf{D}}_{\hat{\mathcal{G}}_{k}^{\mathsf{S}}(Mr_k, r_k)}(X^n,Y^n)  -\kl{\mu}{\nu}}\right] 
& \lesssim_{d,M,\ell} (\log k)^{\frac 32}\Big(k^{-\frac 12} +k^{z_k} ~ n^{-\frac 12}\Big).\notag
\end{flalign}
Setting  $k=n$ in the above bound gives an effective error bound  $O\big(n^{-1/3}\big)$.
 \end{proposition}
 \begin{proof}
Fix $(\mu,\nu) \in \hat{\cP}^2_{\mathsf{KL}}(M,\ell)$.  From \eqref{subgaussbndmu}, we have  for $\mu,\nu \in \mathcal{SG}(M)$ and $r \geq M$ that
\begin{align}
    \mu \big(B_d^c(r)\big) \vee \nu \big(B_d^c(r)\big) \leq 2e^{\frac{-(r-M)^2}{16dM}}.  \label{subgtailmunu}
\end{align}
  Then, it follows  from \eqref{tail1klbnd} and \eqref{tail2klbnd} that for $r_k \geq M$,
  \begin{align}
        \EE_{\mu}\left[\abs{f_{\mathsf{KL}}}\ind_{B_d^c(r_k)}\right] \vee  \EE_{\nu}\left[\abs{h_{\mathsf{KL}} \circ f_{\mathsf{KL}}}\ind_{B_d^c(r_k)}\right] \lesssim_{M} e^{\frac{-(r_k-M)^2}{16dM \ell^*}}. \notag
  \end{align}
  Moreover, $f_{\mathsf{KL}} \in  \hat{\mathcal{I}}(M)$ implies $ c_{\mathsf{B}}^\star\left(f_{\mathsf{KL}}|_{B_d(r_k)},B_d(r_k) \right) \leq   Mr_k$ for $r_k \geq 1$. 
  
  Next, note that $\bar C(|\hat{\cG}_k^{\mathsf{S}}(m_k, r_k)|,B_d(r_k)) \leq 3Mr_k$, and $\bar C\big(\big|h_{\mathsf{KL}}' \circ \hat\cG^{\mathsf{S}}_{k}(m_k,r_k)\big|,B_d(r_k)\big)$ $ \leq e^{3Mr_k}$. Also, similar to \eqref{entropyintscclass}, we  have
  \begin{align}
      &\int_{0}^{1}\sqrt{ \sup_{\gamma \in \cP(\X)} \log N\left(3Mr_k\epsilon,\hat\cG^{\mathsf{S}}_{k}(Mr_k,r_k),\|\cdot\|_{2,\gamma}\right)}\dd\epsilon \lesssim d^{\frac 32} , \label{entrpyintsigmcls}
  \end{align}
Then, 
\eqref{entrpyintgenvccls} implies 
\begin{align}
    \mathbb{E}\left[\abs{\mathsf{D}_{\hat{\cG}_k^{\mathsf{S}}(Mr_k, r_k)}(\mu,\nu)-
    \hat{\mathsf{D}}_{\hat{\mathcal{G}}_{k}^{\mathsf{S}}(Mr_k, r_k)}(X^n,Y^n)}\right] \lesssim_M d^{\frac 32}r_k e^{3Mr_k} n^{-\frac 12}.\notag
\end{align}
Thus, we have similar to Lemma \ref{lem:genmvncKL} (by using Theorem \ref{THM:approximation} for the sigmoid NN class)  for $1 \leq Mr_k \lesssim k^{(1-\rho)/2}$ for some $\rho>0$ that
\begin{align}
 \mathbb{E}\left[  \abs{\hat{\mathsf{D}}_{\hat{\mathcal{G}}_{k}^{\mathsf{S}}(Mr_k, r_k)}(X^n,Y^n) -\kl{\mu}{\nu}}\right]  
  &\lesssim_{d,M,\rho} r_k k^{-\frac 12}+r_k e^{3Mr_k} n^{-\frac 12}  +e^{\frac{-(r_k-M)^2}{16dM \ell^*}}.\notag 
\end{align}
Taking $r_k=M \vee 1+4\sqrt{ dM \ell^*\log k}$ and noting that $1 \leq Mr_k \lesssim k^{1/4}$ (say), we obtain
\begin{flalign}
 & \mathbb{E}\left[  \abs{\hat{\mathsf{D}}_{\hat{\mathcal{G}}_{k}^{\mathsf{S}}(Mr_k, r_k)}(X^n,Y^n)  -\kl{\mu}{\nu}}\right] \notag \\
 & \qquad \qquad \qquad  \qquad \qquad \qquad \lesssim_{d,M,\ell} k^{-\frac 12}(\log k)^{\frac 12} +k^{-1}+ (\log k)^{\frac 32}~e^{12M^{3/2}\sqrt{\ell^*d\log k}}~ n^{-\frac 12} \notag \\
 &\qquad \qquad \qquad  \qquad \qquad \qquad  \lesssim_{d,M,\ell} k^{-\frac 12}(\log k)^{\frac 12} +k^{\frac{12\sqrt{\ell^*d}M^{3/2}}{\sqrt{\log k}}} (\log k)^{\frac 32}~ n^{-\frac 12}.&&\notag
 \end{flalign}
 Taking supremum w.r.t. $(\mu,\nu) \in \hat{\cP}^2_{\mathsf{KL}}(M,\ell)$ yields the claim.
 \end{proof}
 \begin{remark}[CoD-free rate] \label{rem:KLNCfeasible}
 $\hat{\cP}^2_{\mathsf{KL}}(M,\ell)$, for example, includes $M$-sub-Gaussian distributions $(\mu,\nu)$ such that $\norm{f_{\mathsf{KL}}}_{\ell,\mu} \leq M$ and $f_{\mathsf{KL}} \in \cL^{\mathsf{B}}_{s_{\mathsf{B}},b}\big(\RR^d\big)$ (for appropriate value of $b$),  where $s_{\mathsf{B}}=\lfloor d/2\rfloor+2$ and $ \cL^{\mathsf{B}}_{s_{\mathsf{B}},b}\big(\RR^d\big)$ is given in \eqref{squareintclassbar}. 
It also contains certain $M$-sub-Gaussian distributions $(\mu,\nu)$ such that $f_{\mathsf{KL}}=c+f$ for some $c \in \RR$ and  $f \in \cS\big(\RR^d\big)$, where $\cS\big(\RR^d\big)=\big\{f \in \mathsf{C}^{\infty}\big(\RR^d\big): \sup_{x \in \RR^d} \abs{x^{\alpha}D^{\tilde{\alpha}}f(x)}<\infty,~\forall \alpha, \tilde{\alpha} \in \ZZ_{\geq 0}^d \big\}$ is the Schwartz space of rapidly decreasing functions and $\alpha, \tilde{\alpha}$ are multi-indices of dimension $d$. An example would be some $M$-sub-Gaussian distributions $(\mu,\nu)$ with $pq^{-1}=c e^{e^{-x^2}}$, where $c$ is normalization constant (e.g., take $q$ to be multivariate Gaussian, $p(x)=c e^{e^{-x^2}}q(x)$ and $c$ such that $\int_{\RR^d}q(x)\dd x=1$ ). We note that $f \in \cS\big(\RR^d\big)$ implies existence of Fourier transforms and Fourier inversion formula such that   $S_1(f)<\infty$. 
\end{remark}
 \subsection{$\chi^2$ Divergence} \label{mildcodchisq}
 With $\hat{\mathcal{I}}(M)$ as defined in \eqref{setKLncfn}, let 
 \begin{align} &\hat{\mathcal{P}}^2_{\chi^2}(M,\ell):=\left\{(\mu,\nu) \in \mathcal{P}^2_{\chi^2}\big(\RR^d\big) :\mu,\nu \in \mathcal{SG}(M),~f_{\chi^2} \in \hat{\mathcal{I}}(M),\norm{f_{\chi^2}}_{\ell,\mu} \leq M \right\}.\notag
  \end{align}

 \begin{proposition}[$\chi^2$ CoD-free error bound] \label{chisqmildCoD}
 Let $M \geq 0$, $\ell>1$ and $\ell^*=\ell/(\ell-1)$. Then, for  $r_k=M \vee 1+ 4\sqrt{dM \ell^*\log k}$, 
    \begin{flalign}
 & \sup_{(\mu,\nu) \in \hat{\mathcal{P}}^2_{\chi^2}(M,\ell)} \mathbb{E}\left[  \abs{\hat{\chi}^2_{\hat{\mathcal{G}}_{k}^{\mathsf{S}}(Mr_k, r_k)}(X^n,Y^n) -\chisq{\mu}{\nu}}\right]  \lesssim_{d,M,\ell}~ k^{-\frac 12}(\log k)^{\frac 12}+ \log k~ n^{-\frac 12}.  \notag
\end{flalign}
Setting $k=n$ yields an effective error bound  $\tilde O\big(n^{-1/2}\big)$. 
 \end{proposition}
\begin{proof}
Fix $(\mu,\nu) \in \hat{\mathcal{P}}^2_{\chi^2}(M,\ell)$. From \eqref{subgtailmunu}, \eqref{chisqtail1} and \eqref{chisqtail2}, we have 
\begin{align}
  & \EE_{\mu}\left[\abs{f_{\chi^2}}\ind_{B_d^c(r_k)}\right]  \vee  \EE_{\nu}\left[\abs{h_{\chi^2} \circ f_{\chi^2}}\ind_{B_d^c(r_k)}\right]  \lesssim_{M} e^{\frac{-(r_k-M)^2}{16dM \ell^*}}. \notag
\end{align}
Note that  $ c_{\mathsf{B}}^\star\left(f_{\chi^2}|_{B_d(r_k)},B_d(r_k) \right) \leq   Mr_k$ for $r_k \geq 1$, 
$\bar C(|\hat{\cG}_k^{\mathsf{S}}(m_k, r_k)|,B_d(r_k)) \leq 3Mr_k$, and $\bar C\big(\big|h_{\chi^2}' \circ \bar\cG^{\mathsf{S}}_{k}(m_k,r_k)\big|,B_d(r_k)\big)$ $ \leq 1.5 Mr_k+1$. Then
we obtain similar to \eqref{chisqgenbndmv} using  Theorem \ref{THM:approximation} and \eqref{entrpyintsigmcls}  that
 \begin{flalign}
&  \mathbb{E}\left[  \abs{\hat{\chi}^2_{\hat{\mathcal{G}}_{k}^{\mathsf{S}}(Mr_k, r_k)}(X^n,Y^n) -\chisq{\mu}{\nu}}\right]  
  \lesssim_{M}
r_k d^{\frac 12} k^{-\frac 12}+r_k^2 d k^{-1}  +e^{\frac{-(r_k-M)^2}{16dM \ell^*}}+d^{\frac 32} r_k^2 n^{-\frac 12}. \notag
 \end{flalign}
Setting $r_k=M \vee 1+ 4\sqrt{dM \ell^*\log k}$, and taking supremum w.r.t. $(\mu,\nu) \in \hat{\mathcal{P}}^2_{\chi^2}(M,\ell)$ proves the claim.
\end{proof}
\begin{remark}[CoD-free rate]\label{rem:mildcodchisqapp}
$\hat{\mathcal{P}}^2_{\chi^2}(M,\ell)$ contains certain $M$-sub-Gaussian distributions $(\mu,\nu)$ such that $\norm{f_{\chi^2}}_{\ell,\mu} \leq M$, and $f_{\chi^2} \in \cS\big(\RR^d\big) \cup  \cL^{\mathsf{B}}_{s_{\mathsf{B}},b}\big(\RR^d\big)$ for  appropriate value of $b$. In particular, this includes certain Gaussian distributions pairs  $\big(\cN(\mathsf{m}_p,\sigma_p^2 \mathrm{I}_d),$ $\cN(\mathsf{m}_q,\sigma_q^2\mathrm{I}_d)\big)$  with $0<\sigma_p <\sigma_q \leq M$ and $\norm{\mathsf{m}_p} \vee \norm{\mathsf{m}_q} \leq M$. To see this, recall  $f_{\chi^2}=2(pq^{-1}-1)$, and note $\sigma_q >\sigma_p$ ensures that $\norm{f_{\chi^2}}_{\infty,\RR^d} <\infty$ implying that $\norm{f_{\chi^2}}_{\ell,\mu}<\infty$. Also,  since $pq^{-1}$ is again (upto constants) a Gaussian density, $\mathfrak{F}[pq^{-1}]$ exist which is again a Gaussian density  (upto constants). Hence, $\mathfrak{F}[pq^{-1}]$ is integrable and this implies the Fourier inversion formula holds. Moreover, it is easy to verify  that $S_1\big(pq^{-1}\big)<\infty$. Hence, such Gaussian  pairs satisfies the conditions defining $\hat{\mathcal{P}}^2_{\chi^2}(M,\ell)$ for large enough $M$, and the claim follows.
\end{remark}
 
 \subsection{Squared Hellinger Distance}\label{mildCoDhelnc}
 Let $
\check{\cG}^{\mathsf{S}}_{k,t}(a,r) :=\big\{ g \ind_{B_d(r)}: g \in \cG_k\big(k^{1/2}\log k, 2k^{-1}a,a, 0,\phi_{\mathsf{S}}\big)\big\}$, and 
 \begin{align}
  &\hat{\mathcal{P}}^2_{\mathsf{H}^2}(M):=\left\{(\mu,\nu) \in \mathcal{P}^2_{\mathsf{H}^2}\big(\RR^d\big):\begin{aligned} \mu,\nu \in \mathcal{SG}(M),~f_{\mathsf{H}^2} \in \hat{\mathcal{I}}(M), \norm{\frac{\dd \mu}{\dd \nu}}_{\infty,\RR^d} \leq M \end{aligned}\right\}, \notag
\end{align}
where   $\hat{\mathcal{I}}(m)$ is given in   \eqref{setKLncfn}.
 \begin{proposition}[$\mathsf{H}^2$ CoD-free error bound] \label{helmildCoD}
For $M \geq 0$, $m_k=Mr_k$ and  $r_k=M \vee 1+ \sqrt{32dM\log k}$, 
    \begin{flalign}
 &  \sup_{(\mu,\nu) \in \hat{\mathcal{P}}^2_{\mathsf{H}^2}(M)}  \mathbb{E}\left[  \abs{\hat{\mathsf{H}}^2_{\check{\cG}^{\mathsf{S}}_{k,m_k^{-1/2}}(m_k,r_k)}(X^n,Y^n) -\mathsf{H}^2(\mu,\nu)}\right]  \lesssim_{d,M} k^{-\frac 12} \log k  + n^{-\frac 12} \log k.  \notag 
\end{flalign}
Setting $k=n$ yields an effective error bound $\tilde O(n^{-1/2})$.
 \end{proposition}
\begin{proof}
Fix $(\mu,\nu) \in \hat{\mathcal{P}}^2_{\mathsf{H}^2}(M)$. From \eqref{subgtailmunu}, \eqref{heltailineq1} and \eqref{heltailineq2}, we obtain
\begin{flalign}
\EE_{\mu}\left[\abs{f_{\mathsf{H}^2}}\ind_{B_d^c(r_k)}\right] \vee
\EE_{\nu}\left[\abs{h_{\mathsf{H}^2} \circ f_{\mathsf{H}^2}}\ind_{B_d^c(r_k)}\right]& \leq  e^{\frac{-(r_k-M)^2}{32dM}}. \notag
\end{flalign}
We have $ c_{\mathsf{B}}^\star\left(f_{\mathsf{H}^2}|_{B_d(r_k)},B_d(r_k) \right) \leq   Mr_k$ for $r_k \geq 1$,  
$\bar C(|\check{\cG}^{\mathsf{S}}_{k,t}(m_k, r_k)|,B_d(r_k)) \leq 3Mr_k$, and $\bar C\big(\big|h_{\mathsf{H}^2}' \circ \check{\cG}^{\mathsf{S}}_{k,t}(m_k, r_k)\big|,B_d(r_k)\big) \leq t^{-2} $.
Then,  for $k,r_k$ satisfying 
$r_k=o(k^{1/4})$, we have similar to \eqref{genhelbndmktk} using  Theorem \ref{THM:approximation} and \eqref{entrpyintsigmcls}  that 
\begin{flalign}
 &   \mathbb{E}\left[  \abs{\hat{\mathsf{H}}^2_{\check{\cG}^{\mathsf{S}}_{k,m_k^{-1/2}}(m_k,r_k)}(X^n,Y^n) -\mathsf{H}^2(\mu,\nu)}\right]   \lesssim_{M} r_k^2 d^{\frac 12} k^{-\frac 12}+ e^{\frac{-(r_k-M)^2}{32dM}}+d^{\frac 32}r_k^2 n^{-\frac 12}.\notag 
\end{flalign} 
Setting  $r_k=M \vee 1+ \sqrt{32dM\log k}$ and taking supremum w.r.t. $(\mu,\nu) \in \hat{\mathcal{P}}^2_{\mathsf{H}^2}(M)$, we obtain the claim in the Proposition.
\end{proof}
 \begin{remark}[CoD-free rate] \label{rem:hsqapp}
$\hat{\mathcal{P}}^2_{\mathsf{H}^2}(M)$ includes certain $M$-sub-Gaussian  pairs $(\mu,\nu)$ such that $\norm{pq^{-1}}_{\infty,\RR^d} \leq M$ and $qp^{-1}=(e^{f}+c)^2$ for some $f \in \cS\big(\RR^d\big)$, where $c$ is the normalization constant to ensure that $p$ and $q$ are probability densities. To see this, note that $\sqrt{qp^{-1}}$ and $\sqrt{pq^{-1}}$ are both bounded on $\RR^d$. Moreover, $f_{\mathsf{H}^2}=1-\sqrt{qp^{-1}}=-c+1-e^f$. Noting that $1-e^f \in \cS\big(\RR^d\big)$ if $f \in \cS\big(\RR^d\big)$, it follows as discussed in Remark \ref{rem:KLNCfeasible} that $S_1(f_{\mathsf{H}^2})<\infty$, thus implying the claim.  
\end{remark}

\bibliography{SS_ref,ZG_ref_new}

\end{document}